\documentclass[11pt,a4paper]{article}
\usepackage[dvipsnames]{xcolor}
\usepackage{amsmath,amssymb,amsbsy,amsthm,calc,enumerate,url,framed,fancyhdr,latexsym,lipsum,tikz,bbm,enumitem,natbib,url}
\usepackage[mathscr]{euscript}
\usepackage{xcite}
\usepackage{xr-hyper}
\usepackage[pageanchor]{hyperref}
\usepackage{xcolor}
\hypersetup{colorlinks,linkcolor={red!70!black},citecolor={blue!65!black},urlcolor={blue!70!black}}
\usepackage{graphicx}
\usepackage{epstopdf}
\usepackage[font=small,skip=0pt]{caption}
\usepackage[all]{hypcap}

\newcommand{\hfparskip}{\vspace{-0.5\parskip}}
\newcommand{\unparskip}{\vspace{-\parskip}}
\newcommand{\umparskip}{\vspace{-1.5\parskip}}
\newcommand{\deparskip}{\vspace{-2\parskip}}

\newcommand{\abs}[1]{\lvert #1 \rvert}
\newcommand{\norm}[1]{\| #1 \|}

\newcommand{\rbr}[1]{\left(#1\right)}

\newcommand{\cbr}[1]{\left\{#1\right\}}
\newcommand{\abr}[1]{\langle #1\rangle}
\newcommand{\inv}[1]{#1^{-1}}
\newcommand{\N}{\mathbb{N}}

\newcommand{\Q}{\mathbb{Q}}
\newcommand{\R}{\mathbb{R}}

\newcommand{\E}{\mathbb{E}}
\renewcommand{\Pr}{\mathbb{P}}
\newcommand{\Ind}{\mathbbm{1}}
\newcommand{\ipr}[2]{\langle #1,#2 \rangle}
\newcommand{\restr}[2]{\left.#1\right\vert_{#2}}

\newcommand{\sperp}{\raisebox{-0.7pt}{\scalebox{0.7}[0.35]{$\perp$}}}
\newcommand{\barp}[1]{\overset{\sperp}{#1}}
\newcommand{\barpp}[2]{\barp{#1}^{\raisebox{-2pt}{$\scriptstyle #2$}}}

\DeclareMathOperator\sgn{sgn}
\DeclareMathOperator\Var{Var}
\DeclareMathOperator\Cov{Cov}
\DeclareMathOperator\Img{Im}

\DeclareMathOperator\Span{span}
\DeclareMathOperator*\rank{rank}
\DeclareMathOperator\diag{diag}

\DeclareMathOperator*\argmin{argmin}
\DeclareMathOperator\tr{tr}

\newcommand{\iid}{\overset{\mathrm{iid}}{\sim}}
\newcommand{\cvc}{\overset{c}{\to}}
\DeclareMathOperator*\clim{c-lim}
\newcommand{\cvas}{\overset{a.s.}{\to}}
\newcommand{\cvp}{\overset{p}{\to}}
\newcommand{\cvd}{\overset{d}{\to}}

\newcommand{\eqd}{\overset{d}{=}}
\newcommand{\eqdcond}[1]{\mathrel{\eqd\!\vert_{#1}}}
\newcommand{\indep}{\mathrel{\perp\!\!\!\perp}}

\newcommand{\PL}{\mathrm{PL}}
\newcommand{\GOE}{\mathrm{GOE}}
\newcommand{\AMP}{\mathrm{AMP}}
\DeclareMathOperator\mmse{mmse}
\DeclareMathOperator\ST{ST}
\DeclareMathOperator\prox{prox}

\theoremstyle{plain} 
\newtheorem{theorem}{Theorem}[section]
\newtheorem{proposition}[theorem]{Proposition}
\newtheorem{lemma}[theorem]{Lemma}
\newtheorem{corollary}[theorem]{Corollary}
\newtheorem{definition}[theorem]{Definition}

\theoremstyle{definition} 
\newtheorem{remark}[theorem]{Remark}
\newtheorem{example}{Example}

\begin{document}

\title{A unifying tutorial on Approximate Message Passing}
\author{  
Oliver Y. Feng$^\ast$, Ramji Venkataramanan$^\dagger$, Cynthia Rush$^\ddagger$ and Richard J. Samworth$^\ast$ \\
$^\ast$Statistical Laboratory, University of Cambridge \\
$^\dagger$Department of Engineering, University of Cambridge \\
$^\ddagger$Department of Statistics, Columbia University
}

\maketitle

\begin{abstract}
Over the last decade or so, Approximate Message Passing (AMP) algorithms have become extremely popular in various structured high-dimensional statistical problems.  The fact that the origins of these techniques can be traced back to notions of belief propagation in the statistical physics literature lends a certain mystique to the area for many statisticians.  Our goal in this work is to present the main ideas of AMP from a statistical perspective, to illustrate the power and flexibility of the AMP framework.  Along the way, we strengthen and unify many of the results in the existing literature.
\end{abstract}

\section{Introduction}

Approximate Message Passing (AMP) refers to a class of iterative algorithms that have been successfully applied to a number of statistical estimation tasks such as linear regression~\citep{donohoMM2009,BM11,krzakala2012}, generalised linear models~\citep{RanganGAMP,schniter2014compressive,MM20} and low-rank matrix estimation~\citep{Ryosuke2013,Deshpande2014,deshpande2016asymptotic,montanari2016non,kabashima2016phase,lesieur2017constrained,rangan2018iterative,MV21}.  Moreover, these techniques are also popular and practical in a variety of engineering and computer science applications such as imaging \citep{fletcher2014scalable,VSM15,metzler2017learned}, communications \citep{SchniterBICM2011,JGMS15,Barbier2017,rushGV2017} and deep learning~\citep{pandit2019asymptotics,Yan19,ESPRF20,PSRSF20}. AMP algorithms have two features that make them particularly attractive. First, they can easily be tailored to take advantage of prior information on the structure of the signal, such as sparsity or other constraints. Second, under suitable assumptions on a design or data matrix, AMP theory provides precise asymptotic guarantees for statistical procedures in the high-dimensional regime where the ratio of the number of observations~$n$ to dimensions~$p$ converges to a constant~\citep{BM12,donohoJM_SC2013,sur2017likelihood}. More generally, AMP has been also used to obtain lower bounds on the estimation error of first-order methods~\citep{CMW20}, and in linear regression and low rank matrix estimation, it plays a fundamental role in understanding the performance gap between information-theoretically optimal and computationally feasible estimators~\citep{reeves2019replica,barbier2019optimal,lelarge2019fundamental}. In these settings, it is conjectured that AMP achieves the optimal asymptotic estimation error among all polynomial-time algorithms~\citep[cf.][]{CM19}. 


The purpose of this article is to give a comprehensive and rigorous introduction to what AMP can offer, as well as to unify and formalise the core concepts within the large body of recent work in the area. In fact, many of the original ideas of AMP were developed in the physics and engineering literature, and involved notions such as `loopy belief propagation'~\citep[e.g.][Section~11.3]{KF2009} and the `replica method'~\citep[e.g.][]{Tan02,GuoV05,MM2009,RFG2009,krzakala2012}, which will be unfamiliar to many statisticians. In our view, these notions provide useful heuristics, but are not essential to understand the sharp asymptotic characterisations of AMP algorithms in statistical applications, and are omitted here. Instead, the starting point for our development will be an abstract AMP recursion, whose form depends on whether or not the data matrix is symmetric; we will study the symmetric case in detail, and then present the asymmetric version, which can be handled via a reduction argument. The striking and crucial feature of this recursion is that when the dimension is large, the empirical distribution of the coordinates of each iterate is approximately Gaussian, with limiting variance given by a scalar iteration called `state evolution'.

Rigorous formulations of the key AMP property are given in Theorems~\ref{thm:AMPmaster} and \ref{thm:masterext} (for the symmetric case) and Theorem~\ref{thm:AMPnonsym} (for the asymmetric case), which can be found in Sections~\ref{sec:AMPsym} and~\ref{sec:AMPnonsym} respectively. Here, we both strengthen earlier related results, and seek to make the underlying arguments more transparent. These `master theorems', which can be viewed as asymptotic results on Gaussian random matrices, can be adapted to analyse variants of the original AMP recursion that are geared towards more statistical problems. In this aspect, we focus on two canonical statistical settings, namely estimation of low-rank matrices in Section~\ref{Sec:LowRank}, and estimation in generalised linear models (GLMs) in Section~\ref{Sec:GAMP}. The former encompasses Sparse Principal Component Analysis~\citep{JTU03,ZHT06,Deshpande2014,WBS16,gataric2020sparse}, submatrix detection~\citep{ma2015computational}, hidden clique detection~\citep{alon1998clique,deshpande2015clique}, spectral clustering~\citep{vonLux07}, matrix completion \citep{candes2009exact,ZWS19}, topic modelling \citep{BNJ2003} and collaborative filtering \citep{SK09}. The latter provides a holistic approach to studying a suite of popular modern statistical methods, including penalised M-estimators such as the Lasso \citep{Tib96} and SLOPE \citep{bogdan2015slope}, as well as more traditional techniques such as logistic regression. A novel aspect of our presentation in Section~\ref{Sec:GAMP} is that we formalise the connection between AMP and a broad class of convex optimisation problems, and then show how to systematically derive exact expressions for the asymptotic risk of estimators in GLMs. We expect that our general recipe can be applied to a wider class of GLMs than have been studied in the AMP literature to date.

To preview the statistical content in the paper and highlight some recurring themes, we now discuss two prototypical applications of AMP that form the basis of Sections~\ref{Sec:LowRank} and~\ref{Sec:GAMP} respectively. First, suppose that we wish to estimate an unknown signal $v\in\R^n$ based on an observation
\[
A=\frac{\lambda}{n}vv^\top+W,
\]
where $\lambda>0$ is fixed and $W\in\R^{n\times n}$ is a symmetric Gaussian noise matrix. In this so-called spiked Wigner model (see Section~\ref{sec:rankone} and the references therein), a popular and well-studied estimator of $v$ is the leading eigenvector $\hat{\varphi}$ of $A$, which can be approximated via the power method, with iterates 
\[
v^{k+1}=\frac{Av^k}{\norm{Av^k}}.
\]
An AMP algorithm in this context can be interpreted as a generalised power method that produces a sequence of estimates $\hat{v}^k$ of $v$ via iterative updates of the form 
\[
\hat{v}^k =g_k(v^k),\qquad v^{k+1}=A\hat{v}^k-b_k\hat{v}^{k-1}
\]
for $k\in\N_0$, where we emphasise the following two characteristic features:

\unparskip
\begin{enumerate}[label=(\roman*)]
\item Each `denoising' function $g_k\colon\R\to\R$ is applied componentwise to vectors, and can be chosen appropriately to exploit different types of prior information about the structure of $v$ (e.g.\ to encourage $\hat{v}^k$ to be sparse).
\item In the `memory' term $-b_k\hat{v}^{k-1}$, which is called an `Onsager' correction in the AMP literature~\citep[e.g.][]{donohoMM2009,BM11}, the scalar $b_k$ is defined as a specific function of $v^k$ to ensure that the iterates $v^{k+1}$ have desirable statistical properties; see~\eqref{eq:AMPmatsym} below.
\end{enumerate}

\unparskip
One way to incorporate additional structural information on $v$ into the spiked model is to assume that its entries are drawn independently from some prior distribution $\pi$ on $\R$; for example, we can enforce sparsity through priors that place strictly positive mass at 0. Then under appropriate conditions, AMP theory guarantees that, for each $k$, the components of the estimate $\hat{v}^k$ have approximately the same empirical distribution as those of $g_k(\mu_k v+\sigma_k\xi)$; here, $\xi\sim N_n(0,I_n)$ is a `noise' vector that is independent of the signal $v\in\R^n$, and the `signal' and `noise' parameters $\mu_k\in\R$, $\sigma_k>0$ are determined by a scalar state evolution recursion that depends on $(g_k)$ and the prior distribution~$\pi$; see~\eqref{eq:statevolsymat}. This distributional characterisation enables us to choose the functions $g_k$ in such a way that the `effective signal-to-noise ratios' $(\mu_k/\sigma_k)^2$ are large and the resulting AMP estimates $\hat{v}^k=g_k(v^k)$ have low asymptotic estimation error as $n\to\infty$. 

For instance, suppose that the entries of $v$ are drawn uniformly at random from $\{-1,1\}$. Then it turns out that the asymptotic mean squared error (MSE) of $\hat{v}^k$ is minimised by choosing $g_k$ to be the function $x\mapsto\tanh(\mu_k x/\sigma_k^2)$; see Section~\ref{sec:lowrankgk}. Figure~\ref{fig:amse} illustrates that the limiting MSE of the AMP estimates $\hat{v}^k$ decreases with the iteration number $k$, and in particular that they improve on the pilot spectral estimator $\hat{v}^{-1}$ (which is agnostic to the structure of $v$).
\pdfsuppresswarningpagegroup=1
\begin{figure}[htbp!]
\begin{center}
\includegraphics[width=0.475\textwidth]{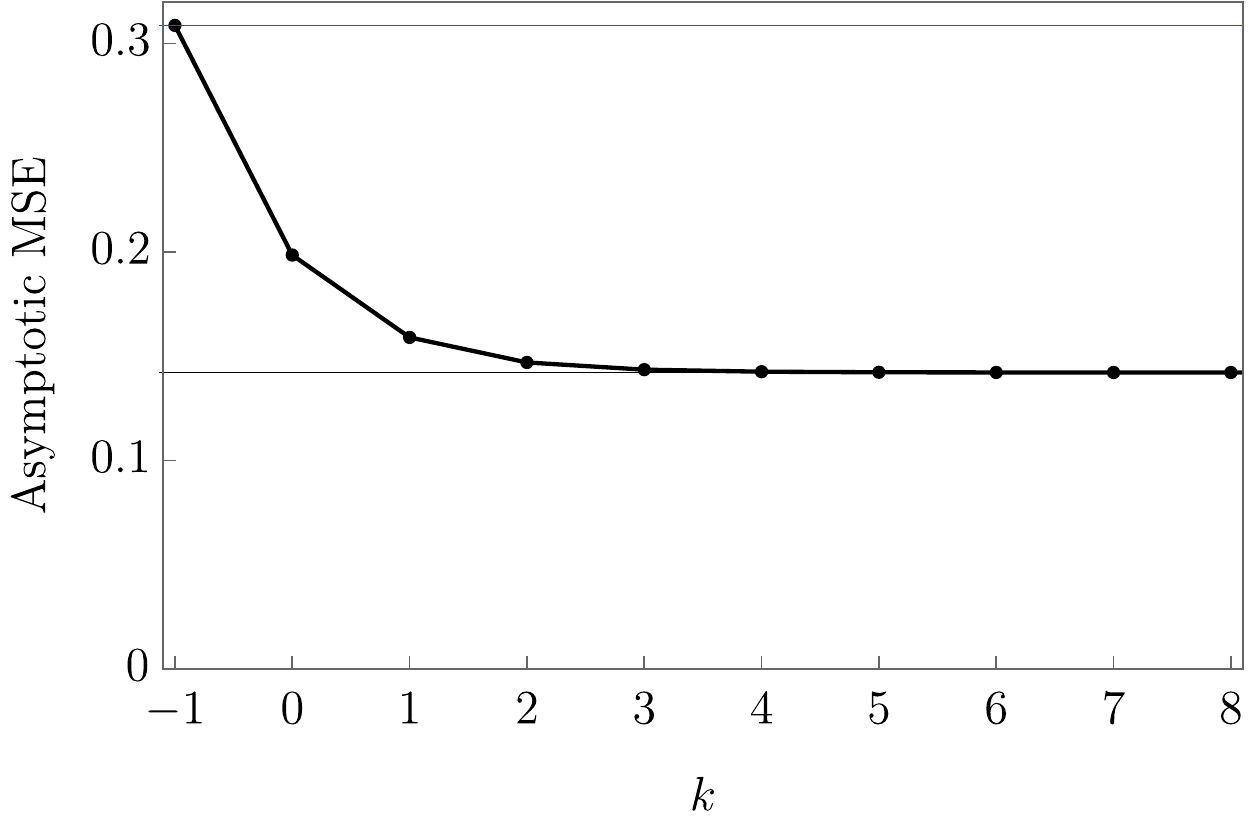}
\hspace{10pt}
\includegraphics[width=0.475\textwidth]{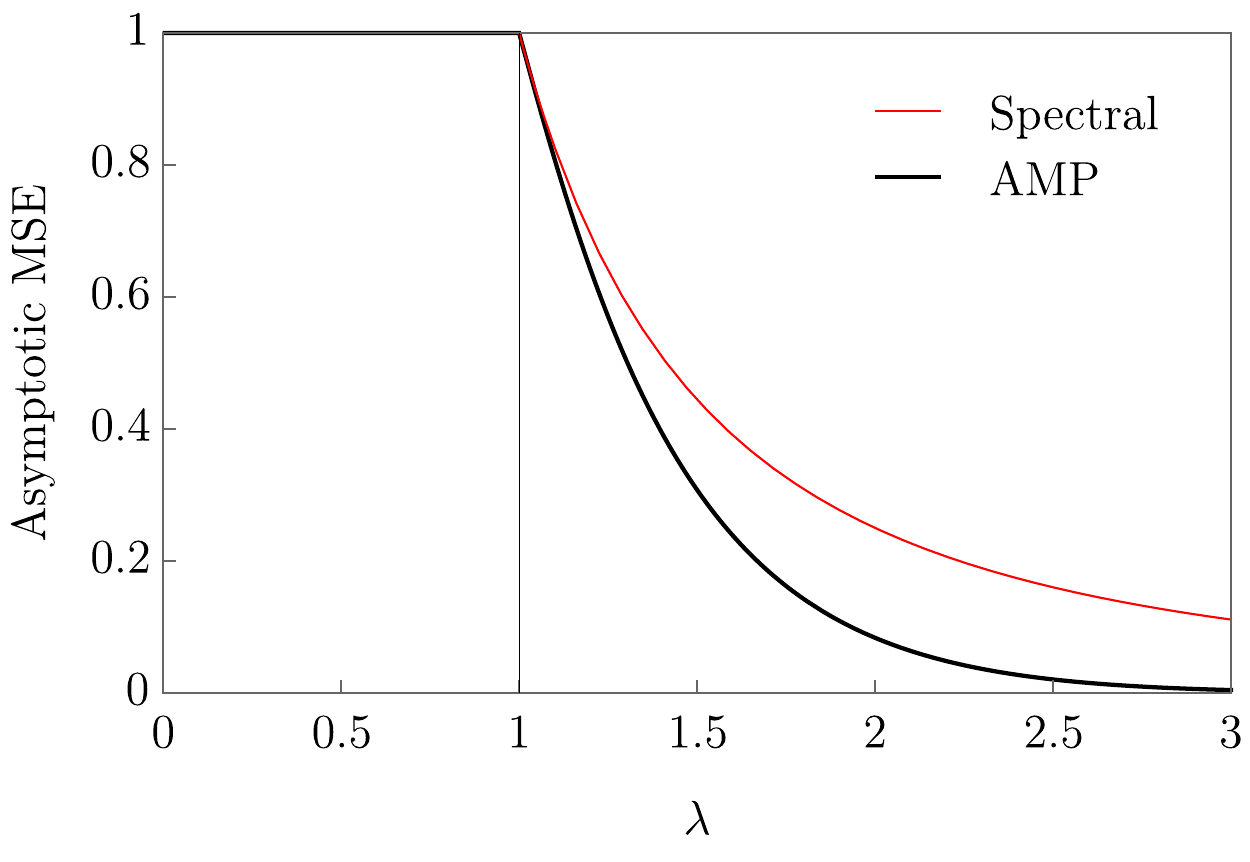}
\end{center}
\caption{\label{fig:amse}Asymptotic mean-squared error plots for estimation of a signal $v\in\R^n$ with i.i.d.\ $U\{-1,1\}$ entries in the rank-one spiked model, based on an AMP algorithm with denoising functions $g_k\colon x\mapsto\tanh(\mu_k x/\sigma_k^2)$ and  spectral initialisation ($v^0=\hat{\varphi}$ and $\hat{v}^{-1}=\lambda^{-1}\hat{\varphi}$ with $\norm{\hat{\varphi}}=\sqrt{n\lambda^2(\lambda^2-1)_+}$); see Section~\ref{sec:spectral}.\\[1ex]\emph{Left}: Plot of $\mathrm{AMSE}_k(\lambda):=\lim_{n\to\infty}\norm{\hat{v}^k-v}^2/n$ against the iteration number $k$ for the AMP estimates $\hat{v}^k\equiv\hat{v}_\lambda^k(n)$, when $\lambda=1.7$. $\mathrm{AMSE}_k(\lambda)$ decreases monotonically to some $\mathrm{AMSE}_\infty(\lambda)$ as $k\to\infty$; see Theorem~\ref{thm:bayesamp}(c). \\[1ex]\emph{Right}: Plots of $\mathrm{AMSE}_{-1}(\lambda)=1\vee\lambda^{-2}$ for the pilot spectral estimator $\hat{v}^{-1}$ and $\mathrm{AMSE}_\infty(\lambda)$ for AMP, with $\lambda\in [0,3]$. The spectral estimator undergoes the so-called BBP phase transition at $\lambda=1$; see Section~\ref{sec:rankone}.}
\unparskip
\end{figure}

As a second example, consider the linear model $y=X\beta+\varepsilon$, where $\beta\in\R^p$ is the target of inference, $\varepsilon\in\R^n$ is a noise vector, and $X\in\R^{n\times p}$ is a random design matrix with independent $N(0,1/n)$ entries. In high-dimensional regimes where $p$ is comparable in magnitude to, or even much larger than $n$, a popular (sparse) estimator is the Lasso~\citep{Tib96}, which for $\lambda>0$ is defined by 
\[
\hat{\beta}_\lambda^{\mathrm L}\in\argmin_{\tilde{\beta}\in\R^p}\,\biggl\{\frac{1}{2}\norm{y-X\tilde{\beta}}^2+\lambda\norm{\tilde{\beta}}_1\biggr\}.
\]
In the literature on high-dimensional estimation, upper bounds on the prediction and estimation error of the Lasso have been obtained under suitable conditions on the design matrix $X$, such as the restricted isometry property or compatibility conditions~\citep[e.g.][]{BvG11}. AMP offers complementary guarantees by providing exact formulae for the asymptotic risk in the `large system limit' where $n,p\to\infty$ with $n/p\to\delta\in (0,\infty)$, and with the components of $\beta$ drawn independently from a prior distribution on $\mathbb{R}$. To motivate the form of the AMP algorithm in this setting, first consider the iterative soft thresholding algorithm (ISTA) for solving the Lasso optimisation problem, whose update steps can be written as
\begin{equation}
\label{eq:ISTA}
\hat{r}^k=y-X\hat{\beta}^k,\qquad\hat{\beta}^{k+1}=\ST_{\lambda\eta_k}\bigl(\hat{\beta}^k + \eta_k X^\top\hat{r}^k\bigr)\qquad\text{for }k\in\N_0;
\end{equation}
here, $\hat{r}^k$ is the current residual, $\eta_k>0$ is a deterministic step size, and for $t>0$, the soft-thresholding function $\ST_t\colon w\mapsto\sgn(w)(\abs{w}-t)_+$ is applied componentwise to vectors. This is an instance of the general-purpose proximal gradient method~\citep[Sections~4.2 and~4.3]{PB13}. An `accelerated' version of~\eqref{eq:ISTA} called FISTA~\citep{BT09} bears a closer resemblance to an AMP algorithm, whose iterates are given by
\begin{equation}
\label{eq:AMPlasso}
\hat{r}^k=y-X\hat{\beta}^k+\frac{\norm{\hat{\beta}^k}_0}{n}\,\hat{r}^{k-1},\qquad\hat{\beta}^{k+1}=\ST_{t_{k+1}}\bigl(\hat{\beta}^k + X^\top\hat{r}^k\bigr)\qquad\text{for }k\in\N_0.
\end{equation}
Here, each $t_k>0$ is a deterministic threshold and $\norm{\hat{\beta}^k}_0$ denotes the number of non-zero entries of $\hat{\beta}^k\in\R^p$. By comparison with~\eqref{eq:ISTA}, we observe that~$\hat{r}^k$ in~\eqref{eq:AMPlasso} is a corrected residual, whose definition includes an additional memory term that is crucial for ensuring that the empirical distribution of the iterates can be characterised exactly. Indeed, for each fixed $k\in\N$, the entries of the AMP estimate $\hat{\beta}^k$ of $\beta$ have approximately the same empirical distribution as those of $\ST_{t_k}(\beta+\sigma_k\xi)$ when $p$ is large; here $\xi\sim N_p(0,I_p)$ is a noise vector that is independent of $\beta$, the noise level $\sigma_k>0$ is determined by the state evolution recursion defined in~\eqref{eq:statevolasso} below, and the scalar denoising function $\ST_{t_k}$ induces sparsity.

\citet{BM12} proved that in the asymptotic regime above, the AMP iterates $(\hat{r}^k,\hat{\beta}^k)$ converge in a suitable sense to a fixed point $(\hat{r}^*,\hat{\beta}^*)$, and a key property of~\eqref{eq:AMPlasso} is that for any such fixed point, $\hat{\beta}^*$ is a Lasso solution; see~\eqref{eq:lassofixed1} below.  It follows that the performance of the Lasso is precisely characterised by a fixed point of the state evolution recursion~\eqref{eq:statevolasso}; see Theorem~\ref{thm:lasso}. Since the above properties are proved under a Gaussian design, the main utility of AMP in this setting is not so much as an efficient Lasso computational algorithm, but rather as a device for gaining insight into the statistical properties of the estimator. In Section~\ref{Sec:GAMP}, the above theory is developed as part of an overarching AMP framework for linear models and generalised linear models (GLMs). 

Note that in both of the examples above, the limiting empirical distributions of the entries of the AMP iterates can be decomposed into independent `signal' and `noise' components, and the effective signal strength and noise level are determined by a state evolution recursion. 
In Sections~\ref{Sec:LowRank} and~\ref{Sec:GAMP}, we show how to derive these asymptotic guarantees by applying the master theorems in Section~\ref{Sec:Master} to suitable abstract recursions, which track the evolution of the asymptotically Gaussian `noise' components of the AMP iterates.  We discuss various extensions in Section~\ref{Sec:Conclusions}, and provide proofs in the Appendix (Section~\ref{Sec:Proofs}), with supplementary mathematical background deferred to Section~\ref{Sec:Supp}.  As a guide to the reader, we remark that rigorous formulations of the results in this paper require a number of technical conditions.  While we take care to state these precisely, and discuss them at appropriate places, we emphasise that these should generally be regarded as mild.  We therefore recommend that the reader initially focuses on the main conclusions of the results.

The statistical roots of AMP lie in compressed sensing~\citep{donohoMM2009,donohoJM_SC2013}. A reader approaching the subject from this perspective can consult~\citet{Mon12},~\citet{TKGM14} and \citet{schniter2019simple} for accessible expositions of the motivating ideas and the connections with message passing algorithms on dense graphs. 
Alternatively, for comprehensive reviews of AMP from a physics perspective, see~\citet{ZK16},~\citet{krzakala2012} and~\citet{lesieur2017constrained}.

In statistical physics, a specific form of AMP was proposed as an iterative algorithm to solve the Thouless--Anderson--Palmer (TAP) equations corresponding to the Sherrington--Kirkpatrick model in spin glass theory~\citep{MPV1987,MM2009,Tal11,Bol14}. The estimation problem here is equivalent to one of reconstructing a symmetric rank-one matrix in a Gaussian spiked model.~\citet{Bol14} proved a rigorous state evolution result for AMP in this specific setting, by introducing a conditioning argument that became an essential ingredient in subsequent analyses of AMP~\citep{BM11,JM13,BMN20,Fan20}. See Section~\ref{sec:AMPcond} for a detailed discussion of this proof technique.


In this article, we restrict our focus to AMP recursions in which the random matrices are Gaussian. However, several recent works have extended AMP and its state evolution recursion to more general non-Gaussian settings. For matrices with independent sub-Gaussian entries, results on the `universality' of AMP were first established by \citet{bayati2015universality} and later in greater generality by~\citet{CL20}. In addition, to accommodate the class of rotationally invariant random matrices, a number of extensions of the original AMP framework have recently been proposed, including Orthogonal AMP~\citep{ma2017,Tak19} and Vector AMP~\citep{SRF16,ranganVAMP19}, 
as well as the general iterative schemes of~\citet{OCW16},~\citet{CO19} and~\citet{Fan20}. Some of these are closely related to expectation propagation~\citep{OW05,KV14}. In all of the above variants of AMP, the recursion is tailored to the spectrum of the random matrix.

\hfparskip
\subsection{Notation and preliminaries}
\label{sec:notation}

Here, we introduce some notation used throughout the paper, and present basic properties of Wasserstein distances, pseudo-Lipschitz functions, as well as the complete convergence of random sequences.

\textbf{General notation}: For $n\in\N$, let $e_1,\dotsc,e_n$ be the standard basis vectors in $\R^n$. For $r\in [1,\infty]$, we write $\norm{x}_r$ for the $\ell_r$ norm of $x\equiv (x_1,\dotsc,x_n)\in\R^n$, so that $\norm{x}_r=(\sum_{i=1}^n\,\abs{x_i}^r)^{1/r}$ when $r\in [1,\infty)$ and $\norm{x}_\infty=\max_{1\leq i\leq n}\,\abs{x_i}$. We also define $\norm{x}_{n,r}:=n^{-1/r}\norm{x}_r=(n^{-1}\sum_{i=1}^n\,\abs{x_i}^r)^{1/r}$ for $r\in (1,\infty)$. Let $\ipr{\cdot\,}{\cdot}$ and $\norm{{\cdot}}:=\norm{{\cdot}}_2$ be the standard Euclidean inner product and norm on $\R^n$ respectively, and define $\ipr{\cdot\,}{\cdot}_n$ to be the scaled Euclidean inner product on $\R^n$ given by $\ipr{x}{y}_n:=\inv{n}\ipr{x}{y}$ for $x,y\in\R^n$, which induces the norm $\norm{{\cdot}}_n:=\norm{{\cdot}}_{n,2}$. We denote by $\mathbf{1}_n:=(1,\dotsc,1)\in\R^n$ the all-ones vector and write $\abr{x}_n:=\ipr{x}{\mathbf{1}_n}_n=n^{-1}\sum_{i=1}^n x_i$ for each $x\in\R^n$.

For $D\in\N$ and $x^1,\dotsc,x^D\in\R^n$, we denote by $\nu_n(x^1,\dotsc,x^D):=n^{-1}\sum_{i=1}^n\delta_{(x_i^1,\dotsc,x_i^D)}$ the joint empirical distribution of their components, and for a function $f\colon\R^D\to\R$, 
write $f(x^1,\dotsc,x^D):=\bigl(f(x_i^1,\dotsc,x_i^D):1\leq i\leq n\bigr)\in\R^n$ for the row-wise application of $f$ to $(x^1\;\cdots\;x^D)$. 

By a \emph{Euclidean space} $(E,\norm{{\cdot}}_E)$ we mean a finite-dimensional inner product space over $\R$, equipped with the norm induced by its inner product; examples include $(\R^n,\norm{{\cdot}})$ for $n\in\N$ and $(\R^{k\times\ell},\norm{{\cdot}}_{\mathrm{F}})$ for $k,\ell\in\N$, where $\norm{{\cdot}}_{\mathrm{F}}$ is the Frobenius norm induced by the trace inner product $(A,B)\mapsto\tr(A^\top B)$. 

\textbf{Gaussian orthogonal ensemble}: We write $W\sim\GOE(n)$ if $W=(W_{ij})_{1\leq i,j\leq n}$ takes values in the space of all symmetric $n\times n$ matrices, and has the property that $(W_{ij})_{1\leq i\leq j\leq n}$ are independent, with $W_{ij}\sim N(0,1/n)$ for $1\leq i<j\leq n$ and $W_{ii}\sim N(0,2/n)$ for $i=1,\dotsc,n$. Writing $\mathbb{O}_n$ for the set of all $n \times n$ orthogonal matrices, we note the orthogonal invariance property of the $\GOE(n)$ distribution: if $Q\in\mathbb{O}_n$ and $W \sim \GOE(n)$, then $Q^\top WQ\sim\GOE(n)$.

\textbf{Complete convergence of random sequences}: The asymptotic results in this paper are formulated in terms of the notion of \emph{complete convergence}~\citep[e.g.][Chapter~1.3]{HR47,Ser80}. This is a stronger mode of stochastic convergence than almost sure convergence, and is denoted throughout using the symbol $\cvc$. In Definition~\ref{def:compconv} and Proposition~\ref{prop:compequiv} below, we give two equivalent characterisations of complete convergence and introduce some associated stochastic $O$ symbols. 
\begin{definition}
\label{def:compconv}
Let $(X_n)$ be a sequence of random elements taking values in a Euclidean space $(E,\norm{{\cdot}}_E)$. We say that $X_n$ \emph{converges completely} to a deterministic limit $x\in E$, and write $X_n\cvc x$ or $\clim_{n\to\infty}X_n=x$, if $Y_n\to x$ almost surely for any sequence of $E$-valued random elements $(Y_n)$ with $Y_n\eqd X_n$ for all $n$. 

We write $X_n=o_c(1)$ if $X_n\cvc 0$, and write $X_n=O_c(1)$ 
if $Y_n=O_{a.s.}(1)$ (i.e.\ $\limsup_{n\to\infty}\norm{Y_n}_E<\infty$ almost surely) for any sequence of $E$-valued random elements $(Y_n)$ with $Y_n\eqd X_n$ for all $n$.
\end{definition}
\begin{proposition}
\label{prop:compequiv}
For a sequence $(X_n)$ of random elements taking values in a Euclidean space $(E,\norm{{\cdot}}_E)$, we have

\unparskip
\begin{enumerate}[label=(\alph*)]
\item $X_n=o_c(1)$ if and only if $\sum_n\Pr(\norm{X_n}_E>\varepsilon)<\infty$ for all $\varepsilon>0$;
\item $X_n=O_c(1)$ if and only if there exists $C>0$ such that $\sum_n\Pr(\norm{X_n}_E>C)<\infty$.
\end{enumerate}

\unparskip
\end{proposition}

\unparskip
For a deterministic $x\in E$, we see that $X_n\cvc x$ if and only if $\sum_n\Pr(\norm{X_n-x}_E>\varepsilon)<\infty$ for all $\varepsilon>0$. Moreover, if $X_n\cvc x$, then $X_n=O_c(1)$. The proof of Proposition~\ref{prop:compequiv}, along with various other properties of complete convergence and a calculus for $o_c(1)$ and $O_c(1)$ notation, is given in Section~\ref{sec:cc}; see also Remark~\ref{rem:AMPconv}.

\textbf{Wasserstein distances and pseudo-Lipschitz functions}: For $D\in\N$ and $r\in [1,\infty)$, we write $\mathcal{P}(r)\equiv\mathcal{P}_D(r)$ for the set of all Borel probability measures $P$ on $\R^D$ with $\int_{\R^D}\norm{x}^r\,dP(x)<\infty$. For $P,Q\in\mathcal{P}_D(r)$, the \emph{$r$-Wasserstein distance} between $P$ and $Q$ is defined by \[d_r(P,Q):=\inf_{(X,Y)}\E(\norm{X-Y}^r)^{1/r},\]
where the infimum is taken over all pairs of random vectors $(X,Y)$ defined on a common probability space with $X\sim P$ and $Y\sim Q$. For $P,P_1,P_2,\dotsc\in\mathcal{P}_D(r)$, we have $d_r(P_n,P)\to 0$ if and only if both $\int_{\R^D}\norm{x}^r\,dP_n(x)\to\int_{\R^D}\norm{x}^r\, dP(x)$ and $P_n\to P$ weakly~\citep[e.g.][Theorem~7.12]{Vil03}. Furthermore, for $L>0$, we write $\PL_D(r,L)$ for the set of functions $\psi\colon\R^D\to\R$ such that
\begin{equation}
\label{eq:PL}
\abs{\psi(x)-\psi(y)}\leq L\norm{x-y}\,(1+\norm{x}^{r-1}+\norm{y}^{r-1})
\end{equation}
for all $x,y\in\R^D$, and denote by $\PL_D(r):=\bigcup_{L>0}\PL_D(r,L)$ the class of \emph{pseudo-Lipschitz functions $f\colon\R^D\to\R$ of order $r$}. Note that $\PL_D(1,L)$ is precisely the class of all $(3L)$-Lipschitz functions on $\R^D$, and that $\PL_D(s)\subseteq\PL_D(r)$ for any $1\leq s\leq r$. Moreover, for any probability measure $P\in\mathcal{P}_D(r)$, we have $\abs{\int_{\R^d}\psi\,dP}\leq L\int_{\R^D}(\norm{x}+\norm{x}^r)\,dP(x)+\abs{\psi(0)}<\infty$ for all $\psi\in\PL_D(r,L)$. Now for $P,Q\in\mathcal{P}_D(r)$, we define
\begin{equation}
\label{eq:Wrtilde}
\widetilde{d}_r(P,Q):=\sup_{\psi\in\PL_D(r,1)}\,\biggl|\,\int_{\R^D}\psi\,dP-\int_{\R^D}\psi\,dQ\,\biggr|.
\end{equation}
In Section~\ref{sec:Wr}, we show (among other things) that $\widetilde{d}_r,d_r$ are metrics on $\mathcal{P}_D(r)$ that induce the same topology (Remark~\ref{rem:Wrmetric}).

\section{Master theorems for abstract AMP recursions}
\label{Sec:Master}
\subsection{Symmetric AMP}
\label{sec:AMPsym}
In this subsection, we present an abstract AMP recursion that was first studied by~\citet{Bol14} in a special case\footnote{In a 2009 workshop, Bolthausen presented his analysis of AMP for the TAP equations, which inspired the work of~\citet{BM11}; see Section~3 of the latter.}, and subsequently by~\citet[Section~4]{BM11} and~\citet{JM13} in greater generality. Let $(f_k)_{k=0}^\infty$ be a sequence of Lipschitz functions $f_k\colon\R^2\to\R$, and for $n\in\N$, let $W\equiv W(n)\in\R^{n\times n}$ be a symmetric matrix and $\gamma\equiv\gamma(n)\in\R^n$ be a vector of auxiliary information. Given $m^{-1}\equiv m^{-1}(n):=0\in\R^n$ and an initialiser $h^0\equiv h^0(n)\in\R^n$, recursively define $m^k\equiv m^k(n)\in\R^n$, $b_k\equiv b_k(n)\in\R$ and $h^{k+1}\equiv h^{k+1}(n)\in\R^n$ by
\begin{equation}
\label{eq:AMPsym}
m^k:=f_k(h^k,\gamma),\qquad b_k:=\abr{f_k'(h^k,\gamma)}_n=\frac{1}{n}\sum_{i=1}^n f_k'(h_i^k,\gamma_i),\qquad h^{k+1}:=Wm^k-b_k m^{k-1}
\end{equation}
for $k\in\N_0$. Here, $f_k'\colon\R^2\to\R$ is a bounded, Borel measurable function that agrees with the partial derivative of $f_k$ with respect to its first argument, wherever the latter is defined.
Note that for each $y\in\R$, the Lipschitz function $x\mapsto f_k(x,y)$ is differentiable Lebesgue almost everywhere~\citep[e.g.][Theorem~3.1.6]{Fed96} with weak derivative $x\mapsto f_k'(x,y)$. 

In its generic form,~\eqref{eq:AMPsym} is not intended for use as an algorithm to solve any particular estimation problem, but for the following reasons, it underpins the statistical framework for AMP:

\unparskip
\begin{enumerate}[label=(\roman*)]
\item \emph{State evolution characterisation of limiting Gaussian distributions}: In an asymptotic regime where conditions~\ref{ass:A0}--\ref{ass:A5} below are satisfied (in particular where~\ref{ass:A0} requires $W$ to be Gaussian), the key mathematical property of~\eqref{eq:AMPsym} is given by~\eqref{eq:AMPmaster} below: for fixed $k\in\N$, the empirical distributions of the components of $h^k\equiv h^k(n)$ converge completely in Wasserstein distance to a Gaussian limit $N(0,\tau_k^2)$ as $n\to\infty$. The variances $\tau_k^2$ are determined by the state evolution recursion~\eqref{eq:statevolsym} below, which depends on the choice of Lipschitz functions $(f_k:k\in\N_0)$. As we will discuss later in this subsection, the so-called Onsager correction term $-b_k m^{k-1}$ plays a pivotal role in ensuring that the asymptotic distributions are indeed Gaussian. 
\item \emph{Basis for the construction and analysis of AMP algorithms}: In statistical settings,~\eqref{eq:AMPsym} cannot be used as a practical procedure when $\gamma$ and/or $W$ are unobservable; for example, in Section~\ref{Sec:LowRank} on low-rank matrix estimation, $\gamma$ represents the unknown target of inference and $W$ is a noise matrix. Instead, one can replace $\gamma$ and/or $W$ in~\eqref{eq:AMPsym} with observed quantities to design an AMP algorithm that produces a sequence of valid estimates of $\gamma$; see~\eqref{eq:AMPmatsym} for instance.
Exact expressions for the asymptotic estimation error can often be derived by subsequently recasting the algorithm as an abstract recursion of the form~\eqref{eq:AMPsym}, whose state evolution characterisation makes it a powerful theoretical tool; see for example Corollary~\ref{cor:AMPlowsym} and (the sketch of) its proof. Moreover, through judicious choices of the Lipschitz functions $f_k$, the AMP estimates can be tailored to different types of prior information about the structure of $\gamma$.

\item \emph{Precursor to other abstract AMP recursions}: By generalising and transforming~\eqref{eq:AMPsym}, we can obtain state evolution descriptions of the limiting behaviour in a number of related abstract AMP iterations, including those in which the input matrix need not be symmetric (Section~\ref{sec:AMPnonsym}) and/or the iterates themselves are matrices rather than vectors (Section~\ref{sec:AMPmatrix}). These facilitate the analysis of a wider class of AMP algorithms that are not covered directly by~\eqref{eq:AMPsym} alone; see for example Section~\ref{Sec:GAMP} on GAMP.
\end{enumerate}

\unparskip
We will now formalise point (i) above through Theorems~\ref{thm:AMPmaster} and~\ref{thm:masterext} below, which establish the Wasserstein limits of the joint empirical distributions of the components of $h^k,\gamma$ and $h^1,\dotsc,h^k,\gamma\in\R^n$ respectively for each fixed $k$ as $n\to\infty$. In view of (ii) and (iii), we 
will refer to these results as `master theorems' for symmetric AMP. 

We will consider a probabilistic setup where for each $n\in\N$, we have an AMP recursion~\eqref{eq:AMPsym} based on a random triple $(m^0,\gamma,W)\equiv\bigl(m^0(n),\gamma(n),W(n)\bigr)$ such that

\unparskip
\begin{enumerate}[label=(A0)]
\item \label{ass:A0} $W\equiv W(n)\sim\GOE(n)$ and is independent of $(m^0,\gamma)\equiv\bigl(m^0(n),\gamma(n)\bigr)$.
\end{enumerate}

\unparskip
Recalling the concepts and definitions from Section~\ref{sec:notation}, we assume that for some $r\in [2,\infty)$ and $\tau_1\in (0,\infty)$, the inputs to~\eqref{eq:AMPsym} also satisfy the following conditions as $n\to\infty$: 

\unparskip
\begin{enumerate}[label=(A\arabic*)]
\item \label{ass:A1} There exists a probability distribution $\pi\in\mathcal{P}_1(r)$ such that the empirical distribution $\nu_n(\gamma)$ of the components of $\gamma\equiv\gamma(n)$ satisfies $d_r\bigl(\nu_n(\gamma),\pi\bigr)\cvc 0$.
\item \label{ass:A2} $\norm{m^0}_n\equiv (n^{-1}\sum_{i=1}^n\,\abs{m_i^0}^2)^{1/2}\cvc\tau_1$ and $\norm{m^0}_{n,r}\equiv (n^{-1}\sum_{i=1}^n\,\abs{m_i^0}^r)^{1/r}=O_c(1)$.
\item \label{ass:A3} There exists a Lipschitz $F_0\colon\R\to\R$ such that taking $\bar{\gamma}\sim\pi$, we have $\E\bigl(F_0(\bar{\gamma})^2\bigr)\leq\tau_1^2$ and $\ipr{m^0}{\phi(\gamma)}_n=n^{-1}\sum_{i=1}^n f_0(h_i^0,\gamma_i)\,\phi(\gamma_i)\cvc\E\big(F_0(\bar{\gamma})\phi(\bar{\gamma})\bigr)$ for all Lipschitz $\phi\colon\R\to\R$.
\end{enumerate}

\unparskip
\ref{ass:A1} holds if for each $n$, the entries of $\gamma\equiv\gamma(n)$ are drawn independently from a distribution $\pi$ on $\R$ with a finite $r^{th}$ moment. In general, $\pi$ can be thought of as a `limiting prior distribution' in statistical applications.~\ref{ass:A2} includes a boundedness assumption on the empirical $r^{th}$ moment of $m^0\equiv m^0(n)$. Both~\ref{ass:A1} and~\ref{ass:A2} are less stringent and more natural than analogous conditions on $(2r-2)^{th}$ moments in the existing literature on AMP; see Remark~\ref{rem:AMPm0}, which also discusses~\ref{ass:A3}. 

Given $\pi\in\mathcal{P}_1(r)$ from~\ref{ass:A1} and $\tau_1\in (0,\infty)$ from~\ref{ass:A2}, the state evolution parameters $(\tau_k^2:k\in\N)$ are defined inductively by
\begin{equation}
\label{eq:statevolsym}
\tau_{k+1}^2:=\E\bigl(f_k(G_k,\bar{\gamma})^2\bigr),
\end{equation}
where $G_k\sim N(0,\tau_k^2)$ and $\bar{\gamma}\sim\pi$ are independent. 
Since the functions $f_k$ are Lipschitz and $\E(\bar{\gamma}^2)^{1/2}\leq\E(\abs{\bar{\gamma}}^r)^{1/r}<\infty$ under~\ref{ass:A1}, it follows by induction that $\tau_k^2\in [0,\infty)$ for all~$k$. 

We will make two further mild regularity assumptions. Suppose henceforth that if $r>2$, then

\unparskip
\begin{enumerate}[label=(A4)]
\item \label{ass:A4} 
$\pi\bigl(\{y\in\R:x\mapsto f_k(x,y)\text{ is non-constant}\}\bigr)>0$ for each $k\in\N$.
\end{enumerate}

\unparskip
This is a `non-degeneracy' condition that ensures that $\tau_k^2>0$ for all $k\in\N$; see also Lemma~\ref{lem:posdef} below.

\unparskip
\begin{enumerate}[label=(A5)]
\item \label{ass:A5} For each $k\in\N$, the set $D_k$ of discontinuities of $f_k'$ satisfies $(\lambda\otimes\pi)(D_k)=0$, where $\lambda$ denotes Lebesgue measure on $\R$.
\end{enumerate}

\unparskip
This guarantees the existence of a deterministic limit for $b_k\equiv b_k(n)$ in~\eqref{eq:AMPsym} as $n\to\infty$ for each~$k$ (see Remark~\ref{rem:bk} below), and is satisfied by the functions $f_k$ that are typically used in statistical applications, such as those based on soft-thresholding functions $\ST_t\colon u\mapsto \sgn(u)(\abs{u}-t)_+$ for $t>0$. See Section~\ref{sec:AMPrem} for some technical remarks on~\ref{ass:A1}--\ref{ass:A5}, which can be skipped on a first reading.

We are now ready to state our first master theorem, which is a substantial result in random matrix theory. As mentioned in (i) above, this reveals in particular that the asymptotic distributional behaviour of the AMP iterates is governed by the scalar recursion~\eqref{eq:statevolsym}.
\begin{theorem}
\label{thm:AMPmaster}
Suppose that~\emph{\ref{ass:A0}}--\emph{\ref{ass:A5}} hold for a sequence of symmetric AMP recursions~\eqref{eq:AMPsym} indexed by $n\in\N$. Then for each $k\in\N$, we have $d_r\bigl(\nu_n(h^k,\gamma),N(0,\tau_k^2)\otimes\pi\bigr)\cvc 0$ as $n\to\infty$, or equivalently
\begin{equation}
\label{eq:AMPmaster}
\widetilde{d}_r\bigl(\nu_n(h^k,\gamma),N(0,\tau_k^2)\otimes\pi\bigr)=\sup_{\psi\in\PL_2(r,1)}\;\biggl|\frac{1}{n}\sum_{i=1}^n\psi(h_i^k,\gamma_i)-\E\bigl(\psi(G_k,\bar{\gamma})\bigr)\biggr|\cvc 0\quad\text{as }n\to\infty,
\end{equation}
where $G_k\sim N(0,\tau_k^2)$ and $\bar{\gamma}\sim\pi$ are independent.
\end{theorem}

\unparskip
In the AMP literature, this conclusion is usually stated as
\begin{equation}
\label{eq:AMPmaster0}
\frac{1}{n}\sum_{i=1}^n\psi(h_i^k,\gamma_i)\cvas\E\bigl(\psi(G_k,\bar{\gamma})\bigr)\;\;\text{as }n\to\infty,\;\text{for every }\psi\in\PL_2(r).
\end{equation}
In fact, $\cvas$ can be strengthened to $\cvc$, and the resulting version of~\eqref{eq:AMPmaster0} is equivalent to~\eqref{eq:AMPmaster};
in other words, it can be upgraded automatically to a convergence statement that holds \emph{uniformly} over the class $\PL_2(r,1)$ of pseudo-Lipschitz test functions. See Remarks~\ref{rem:AMPconv} and~\ref{rem:AMPunif} for further details.

To gain some insight into the form of the recursion~\eqref{eq:AMPsym} and its asymptotic characterisation in Theorem~\ref{thm:AMPmaster}, suppose for simplicity that $\gamma\equiv\gamma(n)=0\in\R^n$ for all $n$, and first consider $k=1$. Since $m^0\equiv m^0(n)$ is independent of $W\equiv W(n)$ for each $n$ by~\ref{ass:A0}, it follows that
$h^1\equiv h^1(n)=Wm^0$ is conditionally Gaussian given $m^0$. In fact, conditional on $m^0$,
\[h^1\;\;\text{and}\;\;h^{1,0}:=\norm{m^0}_n\tilde{Z}+\tilde{\zeta}m^0=\tau_1\tilde{Z}+\Delta^1\;\;\;\text{are identically distributed for each }n,\]
where $\tilde{Z}\sim N_n(0,I_n)$ is independent of $\tilde{\zeta}\sim N(0,1/n)$, and where $\Delta^1:=(\norm{m^0}_n-\tau_1)\tilde{Z}+\tilde{\zeta}m^0$; see Lemma~\ref{lem:Wu},~\eqref{eq:condzero} and~\eqref{eq:sum10}. 

By~\ref{ass:A2}, $\norm{m^0}_n\cvc\tau_1$ and $\norm{m^0}_{n,r}=O_c(1)$ as $n\to\infty$, from which it follows (by the triangle inequality for $\norm{{\cdot}}_{n,r}$) that $\norm{\Delta^1}_{n,r}\equiv (n^{-1}\sum_{i=1}^n\,\abs{\Delta_i^1}^r)^{1/r}\cvc 0$; see $\mathcal{H}_1(a)$ at the start of Section~\ref{sec:masterproofs}. This means that $\Delta^1$ has asymptotically vanishing influence on the empirical distribution of the entries of $h^{1,0}\eqd h^1\in\R^n$ as $n\to\infty$, while the empirical distribution of the entries of $\tau_1\tilde{Z}$ converges completely in $d_r$ to $N(0,\tau_1^2)$ (essentially by the strong law of large numbers, or the concentration inequality in Lemma~\ref{lem:pseudoconc}). This yields the conclusion of Theorem~\ref{thm:AMPmaster} for $h^1$, and also implies that 
\begin{align*}
\norm{m^1}_n^2=\norm{f_1(h^1,0)}_n^2=\frac{1}{n}\sum_{i=1}^n f_1(h_i^1,0)^2&\cvc\E\bigl(f_1(G_1,0)^2\bigr)=\tau_2^2,\\
\norm{m^1}_{n,r}^r=\norm{f_1(h^1,0)}_{n,r}^r=\frac{1}{n}\sum_{i=1}^n \,\abs{f_1(h_i^1,0)}^r&\cvc\E\bigl(\abs{f_1(G_1,0)}^r\bigr)<\infty,
\end{align*}
by the state evolution recursion~\eqref{eq:statevolsym} and the fact that $f_1$ is Lipschitz, whence $f_1^2,\abs{f_1}^r\in\PL_2(r)$; see Corollary~\ref{cor:Wr}(b). Continuing inductively in this vein, we conclude that for each fixed $k\in\N$, the Gaussian distribution $N(0,\tau_k^2)$ in Theorem~\ref{thm:AMPmaster} is the $d_r$ limit of the empirical distribution of the entries of $\breve{h}^k\equiv\breve{h}^k(n)\in\R^n$ in the `toy' recursion
\begin{equation}
\label{eq:AMPtoy}
\breve{h}^1:=\breve{W}^0m^0,\qquad \breve{m}^k:=f_k(\breve{h}^{k},\gamma),\qquad\breve{h}^{k+1}:=\breve{W}^k\breve{m}^k\qquad\text{for }k\in\N,
\end{equation}
where each $\breve{W}^k\equiv\breve{W}^k(n)\sim\GOE(n)$ is independent of $m^0,\gamma\,$($=0$ here) and $\breve{W}^0,\dotsc,\breve{W}^{k-1}$, and hence of $\breve{m}^k$. 

On the other hand, observe that in the original recursion~\eqref{eq:AMPsym}, the \emph{same} $\GOE(n)$ matrix $W\equiv W(n)$ appears in every iteration, so $W$ and $m^k$ are not in general independent for $k\in\N$, and in fact $Wm^k$ is not asymptotically Gaussian in the above sense. To compensate for this, the Onsager correction $-b_k m^{k-1}$ is designed specifically as a debiasing term to ensure that $h^{k+1}=Wm^k-b_k m^{k-1}$ has the same limiting behaviour as $\breve{h}^{k+1}$ in~\eqref{eq:AMPtoy} above.
Indeed, an important technical step in the proof of Theorem~\ref{thm:AMPmaster} is to characterise the conditional distribution of $Wm^k$ given $m^0,\gamma$ and the previous iterates $h^1,\dotsc,h^k$ (Proposition~\ref{prop:AMPconddist}), 
and then show that the `non-Gaussian components' thereof are asymptotically cancelled out by the Onsager term.

This ingenious conditioning technique was first developed by~\citet{Bol14} and~\citet{BM11}, and later used extensively in the analysis of various other AMP iterations in which $W$ is drawn from a rotationally invariant matrix ensemble. For example,~\citet{BMN20} introduced a `Long AMP' recursion in which each iterate $h^{k+1}$ is defined more explicitly in terms of the Gaussian part of the conditional distribution of $Wf_k(h^k,\gamma)$. For the symmetric AMP recursion~\eqref{eq:AMPsym}, the relevant results on conditional distributions are stated in Section~\ref{sec:AMPcond}, where we discuss the subtleties in their derivation, and then rigorously proved in Section~\ref{sec:AMPcondproofs}.

We give a technical summary of the proof of Theorem~\ref{thm:AMPmaster} in Section~\ref{sec:inductive}, where the key result is Proposition~\ref{prop:AMPproof}, and defer the formal arguments to Section~\ref{sec:masterproofs}. The proof proceeds by induction on $k\in\N$ and actually establishes a stronger conclusion (Theorem~\ref{thm:masterext} below) that implies Theorem~\ref{thm:AMPmaster}: in particular, for fixed $k\in\N$, the joint empirical distribution of the components of $h^1,\dotsc,h^k\in\R^n$ converges completely in $d_r$ to a Gaussian limit $N_k(0,\bar{\mathrm{T}}^{[k]})$ as $n\to\infty$. 

The sequence $(\bar{\mathrm{T}}^{[k]}\in\R^{k\times k}:k\in\N)$ of covariance matrices is defined recursively as an extension of the state evolution~\eqref{eq:statevolsym}. First, let $G_1\sim N(0,\tau_1^2)$ and $\bar{\mathrm{T}}_{1,1}:=\tau_1^2$, so that $\bar{\mathrm{T}}^{[1]}\equiv\Var(G_1)=\bar{\mathrm{T}}_{1,1}\geq 0$. For a general $k\geq 2$, suppose inductively that we have already defined a non-negative definite $\bar{\mathrm{T}}^{[k-1]}\in\R^{(k-1)\times (k-1)}$ with entries $\bar{\mathrm{T}}_{ij}^{[k-1]}=\bar{\mathrm{T}}_{i,j}$ for $1\leq i,j\leq k-1$, and then let
\begin{equation}
\label{eq:Sigmabar}
\bar{\mathrm{T}}_{k,\ell}=\bar{\mathrm{T}}_{\ell,k}:=
\begin{cases}
\,\E\bigl(F_0(\bar{\gamma})\cdot f_{k-1}(G_{k-1},\bar{\gamma})\bigr)\quad&\text{for }\ell=1\\
\,\E\bigl(f_{\ell-1}(G_{\ell-1},\bar{\gamma})\cdot f_{k-1}(G_{k-1},\bar{\gamma})\bigr)\quad&\text{for }\ell=2,\dotsc,k,
\end{cases}
\end{equation}
where $F_0$ is as in~\ref{ass:A3} and $\bar{\gamma}\sim\pi$ is independent of $(G_1,\dotsc,G_{k-1})\sim N_{k-1}(0,\bar{\mathrm{T}}^{[k-1]})$. Define $\bar{\mathrm{T}}^{[k]}$ to be the $k\times k$ matrix with entries $\bar{\mathrm{T}}_{ij}^{[k]}=\bar{\mathrm{T}}_{i,j}$ for $1\leq i,j\leq k$, so that $\bar{\mathrm{T}}^{[k-1]}$ is the top-left principal $(k-1)\times (k-1)$ submatrix of $\bar{\mathrm{T}}^{[k]}$. For every $a\equiv (a_1,\dotsc,a_k)\in\R^k$, we have
\begin{equation}
\label{eq:posdef}
a^\top\bar{\mathrm{T}}^{[k]}a=\E\bigl\{\bigl(a_1 F_0(\bar{\gamma})+\textstyle\sum_{\ell=2}^k a_\ell f_{\ell-1}(G_{\ell-1},\bar{\gamma})\bigr)^2\bigr\}+a_1^2\bigl\{\tau_1^2-\E\bigl(F_0(\bar{\gamma})^2\bigr)\bigr\}\geq 0
\end{equation}
since $\E\bigl(F_0(\bar{\gamma})^2\bigr)\leq\tau_1^2$ by~\ref{ass:A3}, so $\bar{\mathrm{T}}^{[k]}\in\R^{k\times k}$ is non-negative definite. In fact, we have the following:
\begin{lemma}
\label{lem:posdef}
Under~\emph{\ref{ass:A4}}, $\bar{\mathrm{T}}^{[k]}\in\R^{k\times k}$ is positive definite and hence invertible for every $k\in\N$.
\end{lemma}

\unparskip
The proof of this fact is given in Section~\ref{sec:addproofs}. By induction, we have $\tau_k^2=\E\bigl(f_{k-1}(G_{k-1},\bar{\gamma})^2\bigr)=\bar{\mathrm{T}}_{k,k}>0$ for all $k\in\N$, so~\eqref{eq:Sigmabar} does indeed extend~\eqref{eq:statevolsym}. Our second master theorem is the following:
\begin{theorem}
\label{thm:masterext}
Under the hypotheses of Theorem~\ref{thm:AMPmaster}, $d_r\bigl(\nu_n(h^1,\dotsc,h^k,\gamma),N_k(0,\bar{\mathrm{T}}^{[k]})\otimes\pi\bigr)\cvc 0$ for each fixed $k\in\N$ as $n\to\infty$, or equivalently
\begin{equation}
\label{eq:masterext}
\widetilde{d}_r\bigl(\nu_n(h^1,\dotsc,h^k,\gamma),N_k(0,\bar{\mathrm{T}}^{[k]})\otimes\pi\bigr)=\sup_{\psi\in\PL_{k+1}(r,1)}\;\biggl|\frac{1}{n}\sum_{i=1}^n\psi(h_i^1,\dotsc,h_i^k,\gamma_i)-\E\bigl(\psi(G_1,\dotsc,G_k,\bar{\gamma})\bigr)\biggr|\cvc 0
\end{equation}
as $n\to\infty$, where $(G_1,\dotsc,G_k)\sim N_k(0,\bar{\mathrm{T}}^{[k]})$ and $\bar{\gamma}\sim\pi$ are independent.
\end{theorem}
\begin{remark}
\label{rem:bk}
The precise form of the Onsager coefficient $b_k$ in~\eqref{eq:AMPsym} is essentially due to Stein's lemma; see~\eqref{eq:covid2} and Proposition~\ref{prop:AMPproof}(g) below. The latter shows that under~\ref{ass:A5}, \[b_k(n)=\abr{f_k'(h^k,\gamma)}_n\cvc\E\bigl(f_k'(\bar{G}_k,\bar{\gamma})\bigr)=:\bar{b}_k\]
for each $k$ as $n\to\infty$. The conclusions of Theorems~\ref{thm:AMPmaster} and~\ref{thm:masterext} remain valid if we replace $b_k\equiv b_k(n)$ with $\bar{b}_k$ in the recursion~\eqref{eq:AMPsym} for all $k,n$, in which case~\ref{ass:A5} is no longer needed. 
\end{remark}

\unparskip
For $1\leq j,\ell\leq k$, since $\psi\colon (x_1,\dotsc,x_k,y)\mapsto x_jx_\ell$ lies in $\PL_{k+1}(2)\subseteq\PL_{k+1}(r)$,~\eqref{eq:masterext} implies that $\ipr{h^j}{h^\ell}_n\cvc\E(G_jG_\ell)=\bar{\mathrm{T}}_{j,\ell}$. Thus, the limiting covariance structure of $h^1,\dotsc,h^k$ is given by $\mathrm{T}^{[k]}$, which in general is not a diagonal matrix. By contrast, while $\breve{h}^k$ in the toy recursion~\eqref{eq:AMPtoy} has the same asymptotics as $h^k$ as $n\to\infty$, it turns out that $\breve{h}^1,\dotsc,\breve{h}^k$ are asymptotically independent, in the sense that the $d_r$ limit of the joint empirical distribution of their components is a centred Gaussian with covariance $\diag(\tau_1^2,\dotsc,\tau_k^2)$.

\hfparskip
\subsection{Asymmetric AMP}
\label{sec:AMPnonsym}
For $n,p\in\N$, the abstract asymmetric AMP recursion~\eqref{eq:AMPnonsym} below is based on a matrix $W\in\R^{n\times p}$, two vectors $\beta\in\R^p$ and $\gamma\in\R^n$ of auxiliary information and two sequences $(g_k,f_{k+1}:k\in\N_0)$ of Lipschitz functions $g_k,f_{k+1}\colon\R^2\to\R$. Given $q^{-1}:=0\in\R^n$, $b_0\in\R$ and $m^0\in\R^p$, we inductively define 
\begin{equation}
\label{eq:AMPnonsym}
\begin{alignedat}{3}
e^k&:=Wm^k-b_kq^{k-1},\qquad&q^k&:=g_k(e^k,\gamma),\qquad&c_k&:=
n^{-1}\textstyle\sum_{i=1}^n g_k'(e_i^k,\gamma_i),\\ 
h^{k+1}&:=W^\top q^k-c_k m^k,\qquad&m^{k+1}&:=f_{k+1}(h^{k+1},\beta),\qquad&b_{k+1}&:=
n^{-1}\textstyle\sum_{j=1}^p f_{k+1}'(h_j^{k+1},\beta_j)
\end{alignedat}
\end{equation}
for $k\in\N_0$. Here, $g_k',f_{k+1}'\colon\R^2\to\R$ are bounded, Borel measurable functions that agree with the partial derivatives of $g_k,f_{k+1}$ respectively with respect to their first arguments, wherever the latter are defined. 

A master theorem for~\eqref{eq:AMPnonsym} is stated below as Theorem~\ref{thm:AMPnonsym}, whose hypotheses and conclusions are similar to those of Theorems~\ref{thm:AMPmaster} and~\ref{thm:masterext} for the symmetric iteration~\eqref{eq:AMPsym}. Consider a sequence of recursions~\eqref{eq:AMPnonsym} indexed by $n\in\N$ and $p\equiv p_n$, for which $n/p\to\delta\in (0,\infty)$ as $n\to\infty$. In this asymptotic regime, suppose that there exist $r\in [2,\infty)$ and $\sigma_0\in (0,\infty)$ for which the following analogues of~\ref{ass:A0}--\ref{ass:A5} hold:

\unparskip
\begin{enumerate}[label=(B\arabic*)]
\setcounter{enumi}{-1}
\item \label{ass:B0} For each $n$, the matrix $W\equiv W(n)$ has entries $W_{ij}\iid N(0,1/n)$ for $1\leq i\leq n$ and $1\leq j\leq p$, and is independent of $(m^0,\beta,\gamma)\equiv\bigl(m^0(n),\beta(n),\gamma(n)\bigr)$.
\item \label{ass:B1} There exist probability distributions $\pi_{\bar{\beta}},\pi_{\bar{\gamma}}\in\mathcal{P}_1(r)$ such that writing $\nu_p(\beta)$ and $\nu_n(\gamma)$ for the empirical distributions of the components of $\beta\in\R^p$ and $\gamma\in\R^n$ respectively, we have $d_r\bigl(\nu_p(\beta),\pi_{\bar{\beta}}\bigr)\cvc 0$ and $d_r\bigl(\nu_n(\gamma),\pi_{\bar{\gamma}}\bigr)\cvc 0$.
\item \label{ass:B2} $\sqrt{p/n}\,\norm{m^0}_p\equiv(n^{-1}\sum_{j=1}^p\,\abs{m_j^0}^2)^{1/2}\cvc\sigma_0$ and $\norm{m^0}_{p,r}\equiv(p^{-1}\sum_{j=1}^p\,\abs{m_j^0}^r)^{1/r}=O_c(1)$.
\item There exists a Lipschitz $F_0\colon\R\to\R$ such that taking $\bar{\beta}\sim\pi_{\bar{\beta}}$, we have $\E\bigl(F_0(\bar{\beta})^2\bigr)\leq\sigma_0^2$ and $\ipr{m^0}{\phi(\beta)}_p=p^{-1}\sum_{j=1}^p f_0(h_j^0,\beta_j)\,\phi(\beta_j)\cvc\E\big(F_0(\bar{\beta})\phi(\bar{\beta})\bigr)$ for all Lipschitz $\phi\colon\R\to\R$.
\item \label{ass:B4} For each $k\in\N_0$, we have $\pi_{\bar{\gamma}}\bigl(\{y\in\R:x\mapsto g_k(x,y)\text{ is non-constant}\}\bigr)>0$ and\\
$\pi_{\bar{\beta}}\bigl(\{y\in\R:x\mapsto f_{k+1}(x,y)\text{ is non-constant}\}\bigr)>0$.
\item \label{ass:B5} For each $k\in\N_0$, writing $D_k,C_{k+1}$ for the sets of discontinuities of $g_k',f_{k+1}'$ respectively, we have $(\lambda\otimes\pi_{\bar{\gamma}})(D_k)=(\lambda\otimes\pi_{\bar{\beta}})(C_{k+1})=0$, where $\lambda$ denotes Lebesgue measure on $\R$.
\end{enumerate}

\unparskip
\textbf{State evolution}: With $\sigma_0>0$ as above, inductively define
\begin{equation}
\label{eq:statevolns}
\tau_{k+1}^2:=\E\bigl(g_k(G_k^\sigma,\bar{\gamma})^2\bigr)\quad\text{and}\quad
\sigma_{k+1}^2:=\inv{\delta}\,\E\bigl(f_{k+1}(G_{k+1}^\tau,\bar{\beta})^2\bigr)
\end{equation}
for $k\in\N_0$, where we take $G_k^\sigma\sim N(0,\sigma_k^2)$ to be independent of $\bar{\beta}\sim\pi_{\bar{\beta}}$, and $G_{k+1}^\tau\sim N(0,\tau_{k+1}^2)$ to be independent of $\bar{\gamma}\sim\pi_{\bar{\gamma}}$. 

\textbf{Limiting covariance structure}: Let $\bar{\Sigma}^{[1]}\equiv\bar{\Sigma}_{0,0}:=\sigma_0^2$ and $\bar{\mathrm{T}}^{[1]}\equiv\bar{\mathrm{T}}_{1,1}:=\tau_1^2$, and for a general $k\in\N$, suppose inductively that we have already defined non-negative definite matrices $\bar{\Sigma}^{[k]},\bar{\mathrm{T}}^{[k]}\in\R^{k\times k}$ with entries $\bar{\Sigma}_{ij}^{[k]}=\bar{\Sigma}_{i-1,j-1}$ and $\bar{\mathrm{T}}_{ij}^{[k]}=\bar{\mathrm{T}}_{i,j}$ for $1\leq i,j\leq k$. Then let
\begin{equation}
\label{eq:Sigmabarns}
\bar{\Sigma}_{k,\ell}=\bar{\Sigma}_{\ell,k}:=
\begin{cases}
\,\inv{\delta}\,\E\bigl(F_0(\bar{\beta})\cdot f_k(G_k^\tau,\bar{\beta})\bigr)\quad&\text{for }\ell=0\\
\,\inv{\delta}\,\E\bigl(f_\ell(G_\ell^\tau,\bar{\beta})\cdot f_k(G_k^\tau,\bar{\beta})\bigr)\quad&\text{for }\ell=1,\dotsc,k,
\end{cases}
\end{equation}
where $(G_1^\tau,\dotsc,G_k^\tau)\sim N_k(0,\bar{\mathrm{T}}^{[k]})$ is independent of $\bar{\gamma}\sim\pi_{\bar{\gamma}}$, and define $\bar{\Sigma}^{[k+1]}\in\R^{(k+1)\times (k+1)}$ by $\bar{\Sigma}_{ij}^{[k+1]}:=\bar{\Sigma}_{i-1,j-1}$ for $1\leq i,j\leq k+1$. As in~\eqref{eq:posdef}, it is easily verified that $\bar{\Sigma}^{[k+1]}$ is non-negative definite. In addition, let
\begin{equation}
\label{eq:Taubarns}
\bar{\mathrm{T}}_{k+1,\ell}=\bar{\mathrm{T}}_{\ell,k+1}:=\,\E\bigl(g_{\ell-1}(G_{\ell-1}^\sigma,\bar{\gamma})\cdot g_k(G_k^\sigma,\bar{\gamma})\bigr)\hspace{1cm}\!\;\text{for }\ell=1,\dotsc,k+1,
\end{equation}
where $(G_0^\sigma,\dotsc,G_k^\sigma)\sim N_{k+1}(0,\bar{\Sigma}^{[k+1]})$ is independent of $\bar{\beta}\sim\pi_{\bar{\beta}}$, and define $\bar{\mathrm{T}}_{ij}^{[k+1]}:=\bar{\mathrm{T}}_{i,j}$ for $1\leq i,j\leq k+1$, so that the resulting matrix $\bar{\mathrm{T}}^{[k+1]}\in\R^{(k+1)\times (k+1)}$ is again non-negative definite. Under~\ref{ass:B4}, it can be shown as in Lemma~\ref{lem:posdef} that $\bar{\Sigma}^{[k]},\bar{\mathrm{T}}^{[k]}$ are positive definite for all $k\in\N$, and also that~\eqref{eq:Sigmabarns}--\eqref{eq:Taubarns} extends~\eqref{eq:statevolns}, with $\sigma_{k-1}^2=\bar{\Sigma}_{k-1,k-1}>0$ and $\tau_k^2=\bar{\mathrm{T}}_{k,k}>0$ for all $k$.
\begin{theorem}
\label{thm:AMPnonsym}
Suppose that~\emph{\ref{ass:B0}}--\emph{\ref{ass:B5}} hold for a sequence of asymmetric AMP recursions~\eqref{eq:AMPnonsym} indexed by $n\in\N$ and $p\equiv p_n$ with $n/p\to\delta\in (0,\infty)$.
Then for each fixed $k \in\N_0$, we have
\begin{equation}
\label{eq:AMPnonsym1}
\begin{split}
\widetilde{d}_r\bigl(\nu_n(e^k,\gamma),N(0,\sigma_k^2)\otimes\pi_{\bar{\gamma}}\bigr)&=\sup_{\psi\in\PL_2(r,1)}\;\biggl|\frac{1}{n}\sum_{i=1}^n\psi(e_i^k,\gamma_i)-\E\bigl(\psi(G_k^\sigma,\bar{\gamma})\bigr)\biggr|\cvc 0,\\
\widetilde{d}_r\bigl(\nu_p(h^{k+1},\beta),N(0,\tau_{k+1}^2)\otimes\pi_{\bar{\beta}}\bigr)&=\sup_{\psi\in\PL_2(r,1)}\;\biggl|\frac{1}{p}\sum_{j=1}^p\psi(h_j^{k+1},\beta_j)-\E\bigl(\psi(G_{k+1}^\tau,\bar{\beta})\bigr)\biggr|\cvc 0,
\end{split}
\end{equation}
\begin{equation}
\label{eq:AMPnonsym2}
\begin{split}
&\widetilde{d}_r\bigl(\nu_n(e^0,\dotsc,e^k,\beta),N_{k+1}(0,\bar{\Sigma}^{[k+1]})\otimes\pi_{\bar{\beta}}\bigr)\\
&\hspace{2cm}=\sup_{\psi\in\PL_{k+2}(r,1)}\;\biggl|\frac{1}{n}\sum_{i=1}^n\psi(e_i^0,\dotsc,e_i^k,\gamma_i)-\E\bigl(\psi(G_0^\sigma,\dotsc,G_k^\sigma,\bar{\gamma})\bigr)\biggr|\cvc 0,\\[1ex]
&\widetilde{d}_r\bigl(\nu_p(h^1,\dotsc,h^{k+1},\beta),N_{k+1}(0,\bar{\mathrm{T}}^{[k+1]})\otimes\pi_{\bar{\beta}}\bigr)\\
&\hspace{2cm}=\sup_{\psi\in\PL_{k+2}(r,1)}\;\biggl|\frac{1}{p}\sum_{j=1}^p\psi(h_j^1,\dotsc,h_j^{k+1},\beta_j)-\E\bigl(\psi(G_1^\tau,\dotsc,G_{k+1}^\tau,\bar{\beta})\bigr)\biggr|\cvc 0
\end{split}
\end{equation}
as $n\to\infty$. Equivalent statements hold with $d_r$ in place of $\widetilde{d}_r$.
\end{theorem}

\unparskip
Together with the master theorems in Section~\ref{sec:AMPsym}, Theorem~\ref{thm:AMPnonsym} can be generalised to abstract AMP recursions with matrix-valued iterates; see Section~\ref{sec:AMPmatrix}.

Similarly to the discussion after Theorem~\ref{thm:AMPmaster}, one can argue that for each $k\in\N_0$, the Gaussian distributions $N(0,\sigma_k^2)$ and $N(0,\tau_{k+1}^2)$ in~\eqref{eq:AMPnonsym1} are the $d_r$ limits of the empirical distributions of the entries of $\breve{e}^k\in\R^n$ and $\breve{h}^{k+1}\in\R^p$ respectively in the toy recursion 
\begin{equation}
\label{eq:AMPtoyns}
\breve{e}^0:=\tilde{W}^0 m^0,\qquad\breve{h}^{k+1}:=\breve{W}^k g_k(\breve{e}^k,\gamma),\qquad\breve{e}^{k+1}:=\tilde{W}^{k+1}f_k(\breve{h}^{k+1},\beta)\qquad\text{for }k\in\N_0
\end{equation}
as $n,p\to\infty$ with $n/p\to\delta$. Here, each iteration features a new matrix with i.i.d.\ $N(0,1/n)$ entries that is independent of everything thus far.
In the original abstract iteration~\eqref{eq:AMPnonsym}, where the same Gaussian matrix $W$ is used throughout, the Onsager correction terms $-b_k q^{k-1}$ and $-c_k m^k$ are designed to ensure that $e^k\in\R^n$ and $h^{k+1}\in\R^p$ have the same limiting behaviour as $\breve{e}^k$ and $\breve{h}^{k+1}$ respectively. 
We note however that the limiting joint empirical distributions in~\eqref{eq:AMPnonsym2} are in general different from those in~\eqref{eq:AMPtoyns}.

One way to establish Theorem~\ref{thm:AMPnonsym} is to analyse the asymmetric recursion~\eqref{eq:AMPnonsym} directly, by adapting the techniques and arguments from the proof of Theorem~\ref{thm:masterext} for the symmetric iteration~\eqref{eq:AMPsym}. An important first step is to obtain an analogue of Proposition~\ref{prop:AMPconddist} that characterises the conditional distribution of each of the iterates in~\eqref{eq:AMPnonsym}, given the inputs $m^0,\beta,\gamma$ and all the previous iterates. This then sets up an inductive proof along the lines of Proposition~\ref{prop:AMPproof}~\citep{BM11}. \citet{RV18} established a finite-sample version of Theorem~\ref{thm:AMPnonsym} under finite-sample analogues of its hypotheses (see Remark~\ref{rem:finitesample}).

There is an alternative derivation of Theorem~\ref{thm:AMPnonsym} that proceeds by first embedding~\eqref{eq:AMPnonsym} within a suitable symmetric recursion
(featuring a $\GOE(n+p)$ matrix), whose output at iteration $k\in\N_0$ contains $h^\ell$ when $k=2\ell$ and $e^\ell$ when $k=2\ell+1$~\citep{JM13,BMN20}. The construction of this augmented recursion is based on a slightly more general version of the original symmetric iteration~\eqref{eq:AMPsym} that offers the additional flexibility to apply (two) different Lipschitz functions to different components of each AMP iterate.

\section{Low-rank matrix estimation}
\label{Sec:LowRank}
\subsection{An AMP algorithm for estimating a symmetric rank-one matrix}
\label{sec:rankone}
In this subsection, we will motivate and analyse an AMP algorithm for 
reconstructing a symmetric rank-one matrix based on an observation
\begin{equation}
\label{eq:symspike}
A\equiv A(n)=\frac{\lambda}{n}vv^\top+W\in\R^{n\times n}
\end{equation}
for some $n\in\N$, where $\lambda>0$ is a deterministic scalar, $v\equiv v(n)\in\R^n$ is the signal (or `spike') that we wish to estimate, and $W\equiv W(n)\sim\GOE(n)$ is a noise matrix. The asymptotic setting of interest to us here is one where $\norm{v}_{n}\equiv n^{-1/2}\,\norm{v}$ converges to 1 as $n\to\infty$; see~\eqref{eq:vconv} below. 

A natural estimator of $v$ is a principal eigenvector $\hat{\varphi}\equiv\varphi^1(A)\in\R^n$ (with $\norm{\hat{\varphi}}_n=1$) corresponding to the largest eigenvalue $\lambda_1(A)$ of the observation matrix $A$. A cornerstone of the spectral theory of such `deformed' GOE matrices is the so-called `BBP' phase transition. This was first established in the seminal paper of~\citet{BBP05} and later explored in greater generality by~\citet{BS06},~\citet{FP07},~\citet{CDF09} and~\citet{BN11}, among many others. See~\citet{JP18} for an accessible summary of this line of work, which reveals that in the limiting regime where $\norm{v}_n$ converges to 1, the eigenstructure of $A\equiv A(n)$ for large $n$ exhibits two different types of qualitative behaviour depending on whether $\lambda\leq 1$ or $\lambda>1$. In particular, when $n\to\infty$, it follows from the concentration results in~\citet[Theorems~2.7 and~6.3]{KY13} that
\begin{equation}
\label{eq:princeigen}
\lambda_1(A)\cvc
\begin{cases}
\lambda+\lambda^{-1}>2\quad&\text{if }\lambda>1\\
2\quad&\text{if }\lambda\in (0,1],
\end{cases}
\qquad\quad\frac{\abs{\ipr{\hat{\varphi}}{v}}}{\norm{\hat{\varphi}}\,\norm{v}}\cvc
\begin{cases}
\sqrt{1-\lambda^{-2}}\quad&\text{if }\lambda>1\\
0\quad&\text{if }\lambda\in (0,1];
\end{cases}
\end{equation}
see also~\citet[Theorem~3.1]{Pen12} for the former and Corollary~\ref{prop:power} below for the latter. 

In the `supercritical' phase when $\lambda>1$, the effect of the spike $v$ can be seen in the limiting expressions above: with high probability, $\hat{\varphi}$ is at least partially aligned with~$v$ (although it does not estimate $v$ consistently) and $\lambda_1(A)$ is an outlier that is separated from the `bulk' of the spectrum of $A$. Indeed, the remaining eigenvalues of $A$ are asymptotically distributed according to the Wigner semicircle law on $[-2,2]$, and it can be shown that the second-largest eigenvalue $\lambda_2(A)$ of $A\equiv A(n)$ satisfies $\lambda_2(A)\cvc 2$ as $n\to\infty$, so the limiting spectral gap $\lambda_1(A)-\lambda_2(A)$ is strictly positive. 

On the other hand, in the `subcritical' phase when $\lambda\leq 1$, the noise matrix $W$ obscures the signal in~\eqref{eq:symspike} to such an extent that $\hat{\varphi}$ is asymptotically uninformative as an estimator of $v$, as evidenced by the asymptotic orthogonality in~\eqref{eq:princeigen}, and $\lambda_1(A)$ remains attached to the bulk of the eigenvalues of $A$. In this low signal-to-noise regime, the limits for $\lambda_1(A)$ and $\hat{\varphi}$ in~\eqref{eq:princeigen} are the same as for the leading eigenvalue and eigenvector of $W$ respectively.

A further limitation of the classical spectral estimator $\hat{\varphi}$ is that it is unable to exploit any additional information about the structure of $v$ that may be relevant for inference. For example, in some matrix estimation problems such as hidden clique detection and non-negative or sparse principal component analysis, there are natural constraints that force $v$ to be non-negative or sparse, or to lie in some finite set such as $\{0,1\}^n$ \citep{alon1998clique,ZHT06,VuLei2013,deshpande2015clique,montanari2016non}. 
A Bayesian approach to modelling a structured signal $v$ is to assume that its components are drawn from some suitable prior distribution that is fully or partially known. However, for general priors, a practical issue is the lack of efficient (i.e.\ polynomial-time) algorithms for computing or accurately approximating the Bayes estimator of $v$ with respect to quadratic loss, namely the posterior mean $\E(v\,|\,A)$.

We will now present a generic (and computationally feasible) AMP procedure~\eqref{eq:AMPmatsym} for estimating~$v$~\citep[cf.][]{Deshpande2014,deshpande2016asymptotic,MV21}, and obtain an exact characterisation of its asymptotic performance in terms of a state evolution recursion (Theorem~\ref{thm:AMPlowsym} and Corollary~\ref{cor:AMPlowsym}). Guided by these theoretical guarantees, we will explain in Sections~\ref{sec:spectral} and~\ref{sec:lowrankgk} how the inputs to the algorithm can be specialised further to take advantage of different types of prior information, and thereby produce estimators that outperform $\hat{\varphi}$ in terms of asymptotic mean squared error.

Let $(g_k)_{k=0}^\infty$ be a sequence of Lipschitz functions on $\R$ with corresponding weak derivatives $g_k'$. Given $\hat{v}^{-1}\equiv\hat{v}^{-1}(n):=0\in\R^n$ and an initialiser $v^0\equiv v^0(n)\in\R^n$ for some $n\in\N$, we recursively define $v^k\equiv v^k(n)\in\R^n$, $b_k\equiv b_k(n)\in\R$ and $\hat{v}^{k+1}\equiv \hat{v}^{k+1}(n)\in\R^n$ by
\begin{equation}
\label{eq:AMPmatsym}
\hat{v}^k:=g_k(v^k),\qquad b_k:=\abr{g_k'(v^k)}_n=\frac{1}{n}\sum_{i=1}^n g_k'(u_i^k),\qquad v^{k+1}:=A\hat{v}^k-b_k\hat{v}^{k-1}
\end{equation}
for $k\in\N_0$. This has a very similar form to the abstract recursion~\eqref{eq:AMPsym} that we studied in Section~\ref{sec:AMPsym}, the main difference being that~\eqref{eq:AMPmatsym} is a valid algorithm with the data matrix $A\equiv A(n)$ in place of the unobserved noise matrix $W\equiv W(n)$.

As mentioned in the Introduction, we can view~\eqref{eq:AMPmatsym} as a generalised power iteration, in which the additional Onsager correction term $-b_k \hat{v}^{k-1}$ is crucial for ensuring that the iterates $v^k$ have the desired asymptotic distributional properties. In fact, we will see in Section~\ref{sec:lowrankgk} that for a specific choice of linear functions $g_k$ given by~\eqref{eq:gklinear}, the corresponding recursion~\eqref{eq:AMPmatsym} is asymptotically equivalent to a standard power iteration that converges to the principal eigenvector $\hat{\varphi}$ of $A$.

To set up our asymptotic framework, consider a sequence of recursions~\eqref{eq:AMPmatsym} indexed by $n\in\N$, for which the following conditions hold:

\unparskip
\begin{enumerate}[label=(M\arabic*)]
\setcounter{enumi}{-1}
\item \label{ass:S0} The noise matrix $W\equiv W(n)\sim\GOE(n)$ in~\eqref{eq:symspike} is independent of $(\hat{v}^0,v)\equiv\bigl(\hat{v}^0(n),v(n)\bigr)$ for each $n$.
\item \label{ass:S1} There exist $\mu_0,\sigma_0\in\R$ and independent random variables $U,V$ with $\E(U^2)=\E(V^2)=1$, such that \[\sup_{\psi\in\PL_2(2,1)}\;\biggl|\frac{1}{n}\sum_{i=1}^n\psi(v_i^0,v_i)-\E\bigl\{\psi\bigl(\mu_0 V+\sigma_0 U,V\bigr)\bigr\}\biggr|\cvc 0.\]
In other words, writing $\bar{\mu}^0$ for the distribution of $(\mu_0 V+\sigma_0 U,V)$, and $\nu_n(v^0,v)$ for the joint empirical distribution of the components of $v^0,v\in\R^n$ for $n\in\N$, we have \[\widetilde{d}_2\bigl(\nu_n(v^0,v),\bar{\mu}^0\bigr)\cvc 0\quad\text{or equivalently}\quad d_2\bigl(\nu_n(v^0,v),\bar{\mu}^0\bigr)\cvc 0.\]
\item \label{ass:S2} For each $k\in\N$, the function $g_k'\colon\R\to\R$ is continuous Lebesgue almost everywhere, i.e.\ the set of discontinuities of $g_k'$ has Lebesgue measure 0.
\end{enumerate}

\unparskip
Henceforth, we will write $\pi$ for the distribution of $V$, which can be viewed as the `limiting prior distribution' of the components of the signal $v\equiv v(n)$. Note that while $v$ is only identifiable up to a sign in the original spiked model~\eqref{eq:symspike}, knowledge of $\pi$ may help us to distinguish $v$ from $-v$ in the limit $n\to\infty$, for example if $\pi$ has non-zero mean. By considering the $\PL_2(2)$ functions $(x,y)\mapsto y^2$, $(x,y)\mapsto xy$ and $(x,y)\mapsto x^2$, we deduce from~\ref{ass:S1} that
\begin{align}
\label{eq:vconv}
\norm{v}_n^2=\frac{1}{n}\sum_{i=1}^n v_i^2&\cvc\E(V^2)=1,\\
\label{eq:musigma1}
\lambda\ipr{\hat{v}^0}{v}_n\cvc\lambda\E\bigl(Vg_0(\mu_0 V+\sigma_0 U)\bigr)=:\mu_1\quad&\text{and}\quad\norm{\hat{v}^0}_n^2\cvc\E\bigl(g_0(\mu_0 V+\sigma_0 U)^2\bigr)=:\sigma_1^2.
\end{align}
\textbf{State evolution}: Starting with $\mu_1\in\R$ and $\sigma_1\in [0,\infty)$, we inductively define state evolution parameters $\mu_k\in\R$ and $\sigma_k\in [0,\infty)$ for $k\in\N$ by
\begin{equation}
\label{eq:statevolsymat}
\mu_{k+1}:=\lambda\E\bigl(Vg_k(\mu_kV+\sigma_k G)\bigr),\qquad\sigma_{k+1}^2:=\E\bigl(g_k(\mu_kV+\sigma_k G)^2\bigr),
\end{equation}
where $V\sim\pi$ and $G\sim N(0,1)$ are independent. Note that since each $g_k$ is Lipschitz and $\E(V^2)=\E(G^2)=1$, we indeed have $\mu_k\in\R$ and $\sigma_k\in [0,\infty)$ for all $k$ by induction; we will see below that these represent the \emph{effective signal strength} and \emph{effective noise level} respectively at iteration $k$.

\textbf{Limiting covariance structure}: We now extend~\eqref{eq:statevolsymat} by specifying the covariance matrices of the limiting Gaussian distributions in Theorem~\ref{thm:AMPlowsym} below. Let $\bar{\Sigma}^{[1]}=\bar{\Sigma}_{1,1}:=\sigma_1^2\geq 0$. For a general $k\geq 2$, suppose inductively that we have already defined a non-negative definite $\bar{\Sigma}^{[k-1]}\in\R^{(k-1)\times (k-1)}$ with entries $\bar{\Sigma}_{ij}^{[k-1]}=\bar{\Sigma}_{i,j}$ for $1\leq i,j\leq k-1$, and then let
\begin{equation}
\label{eq:Sigmabarmat}
\bar{\Sigma}_{k,\ell}:=
\begin{cases}
\,\E\bigl(g_0(\mu_0 V+\sigma_0 U)\cdot g_{k-1}(\mu_{k-1}V+\sigma_{k-1}G_{k-1})\bigr)\quad&\text{for }\ell=1\\
\,\E\bigl(g_{\ell-1}(\mu_{\ell-1}V+\sigma_{\ell-1}G_{\ell-1})\cdot g_{k-1}(\mu_{k-1}V+\sigma_{k-1}G_{k-1})\bigr)\quad&\text{for }\ell=2,\dotsc,k,
\end{cases}
\end{equation}
for $1\leq\ell\leq k$, where $(\sigma_1 G_1,\dotsc,\sigma_{k-1} G_{k-1})\sim N_{k-1}(0,\bar{\Sigma}^{[k-1]})$ is independent of $(U,V)$ from~\ref{ass:S1}. Let $\bar{\Sigma}^{[k]}$ be the $k\times k$ matrix with entries $\bar{\Sigma}_{ij}^{[k]}:=\bar{\Sigma}_{i,j}$ for $1\leq i,j\leq k$, so that $\bar{\Sigma}^{[k-1]}$ is the top-left principal $(k-1)\times (k-1)$ submatrix of $\bar{\Sigma}^{[k]}$. It can be verified as in~\eqref{eq:posdef} that $\bar{\Sigma}^{[k]}$ is non-negative definite. By induction, $\sigma_k^2=\E\bigl(g_{k-1}(\mu_{k-1}V+\sigma_{k-1}G)^2\bigr)=\bar{\Sigma}_{k,k}$ for all $k\in\N$, so~\eqref{eq:Sigmabarmat} does indeed extend~\eqref{eq:statevolsymat}.

We are now ready to state the main result of this subsection, which for each $k\in\N$ establishes the 2-Wasserstein ($d_2$) limit of the joint empirical distributions of the components of $v^0,v^1,\dotsc,v^k,v\in\R^n$ as $n\to\infty$.
\begin{theorem}
\label{thm:AMPlowsym}
Suppose that~\emph{\ref{ass:S0}}--\emph{\ref{ass:S2}} hold for a sequence of AMP iterations~\eqref{eq:AMPmatsym}, where for each $n\in\N$, the symmetric matrix $A\equiv A(n)$ is generated according to the spiked model~\eqref{eq:symspike} for some fixed $\lambda>0$ that does not depend on $n$. Then for each $k\in\N$, we have
\begin{equation}
\label{eq:AMPlowsym}
\sup_{\psi\in\PL_{k+2}(2,1)}\;\biggl|\frac{1}{n}\sum_{i=1}^n\psi(v_i^0,v_i^1\dotsc,v_i^k,v_i)-\E\bigl(\psi(\mu_0 V+\sigma_0 U,\mu_1 V+\sigma_1 G_1,\dotsc,\mu_k V+\sigma_k G_k,V)\bigr)\biggr|\cvc 0
\end{equation}
as $n\to\infty$, where $(\sigma_1 G_1,\dotsc,\sigma_k G_k)\sim N_k(0,\bar{\Sigma}^{[k]})$ is independent of $(U,V)$ from~\emph{\ref{ass:S1}}. In other words, writing $\breve{\nu}^k$ for the distribution of $(\mu_0 V+\sigma_0 U,\mu_1 V+\sigma_1 G_1,\dotsc,\mu_k V+\sigma_k G_k,V)$, we have
\[\widetilde{d}_2\bigl(\nu_n(v^0,v^1,\dotsc,v^k,v),\breve{\nu}^k\bigr)\cvc 0\quad\text{or equivalently}\quad d_2\bigl(\nu_n(v^0,v^1,\dotsc,v^k,v),\breve{\nu}^k\bigr)\cvc 0\quad\text{as } n\to\infty.\]
\end{theorem}

\unparskip
Before discussing Theorem~\ref{thm:AMPlowsym} and its proof, we note that as an immediate consequence of~\eqref{eq:AMPlowsym}, Corollary~\ref{cor:AMPlowsym} below yields an exact expression for the asymptotic deviation of $\hat{v}^k=g_k(v^k)$ from $v$ with respect to any pseudo-Lipschitz loss function of order 2. In particular, the asymptotic mean squared error and empirical correlation in~\eqref{eq:asymmse} and~\eqref{eq:asymcorr} respectively depend only on $\lambda$ and the state evolution parameters $\mu_{k+1},\sigma_{k+1}$.
\begin{corollary}
\label{cor:AMPlowsym}
In the setting of Theorem~\ref{thm:AMPlowsym}, fix $k\in\N$ and let $n\to\infty$. Then taking $G_k\sim N(0,1)$ to be independent of $V\sim\pi$, we have
\begin{equation}
\label{eq:asymdev}
\frac{1}{n}\sum_{i=1}^n\psi(\hat{v}_i^k,v_i)\cvc\E\bigl\{\psi\bigl(g_k(\mu_k V+\sigma_k G_k),V\bigr)\bigr\}
\end{equation}
for all $\psi\in\PL_2(2)$. Consequently,
\begin{align}
\label{eq:asymmse}
\norm{\hat{v}^k-v}_n^2&\cvc\E\bigl\{\bigl(g_k(\mu_k V+\sigma_k G_k)-V\bigr)^2\bigr\}=\sigma_{k+1}^2-\frac{2\mu_{k+1}}{\lambda}+1\\[1.5ex]
\label{eq:asymcorr}
\text{and}\qquad\frac{\abs{\ipr{\hat{v}^k}{v}_n}}{\norm{\hat{v}^k}_n\norm{v}_n}&\cvc\frac{\bigl|\E\bigl(Vg_k(\mu_kV+\sigma_k G_k)\bigr)\bigr|}{\sqrt{\E\bigl(g_k(\mu_k V+\sigma_k G_k)^2\bigr)}}=\frac{\abs{\mu_{k+1}}}{\lambda\sigma_{k+1}}.
\end{align}
\end{corollary}
\begin{remark}
\label{rem:musigmaest}
Observe that $\norm{v^k}_n^2\cvc\E\bigl((\mu_k V+\sigma_k G_k)^2\bigr)=\mu_k^2+\sigma_k^2$ and $\norm{\hat{v}^{k-1}}_n^2=\norm{g_{k-1}(v^{k-1})}_n\cvc\sigma_k^2$ for all $k\in\N$, so $\norm{v^k}_n^2-\norm{\hat{v}^{k-1}}_n^2$ and $\norm{\hat{v}^{k-1}}_n^2$ are strongly consistent estimators of $\mu_k^2$ and $\sigma_k^2$ respectively.
\end{remark}

\unparskip
\textbf{Interpretation}: Through the state evolution recursion~\eqref{eq:statevolsymat}, Corollary~\ref{cor:AMPlowsym} establishes a precise correspondence between the asymptotic behaviour of the AMP iterates $\bigl(\hat{v}^k\equiv \hat{v}^k(n):n\in\N\bigr)$ and a univariate deconvolution problem, where we estimate $V$ by $g_k(\mu_k V+\sigma_k G_k)$ when given a single noisy observation $\mu_k V+\sigma_k G_k$. In this context, 
the quantity $\rho_k:=(\mu_k/\sigma_k)^2$ can be interpreted as an \emph{effective signal-to-noise ratio}, which arises naturally in~\eqref{eq:asymcorr} above. Returning to the spiked model~\eqref{eq:symspike}, we can think of $\hat{v}^k\equiv \hat{v}^k(n)=g_k(v^k)$ as an estimate of $v\equiv v(n)$ based on an `\emph{effective observation}' $v^k\equiv v^k(n)$ whose components have approximately the same empirical distribution as those of $\mu_k v+\sigma_k\xi$ when $n$ is large, where $\xi\equiv\xi(n)\sim N_n(0,I_n)$ is independent of $v$. 

Theorem~\ref{thm:AMPlowsym} and Corollary~\ref{cor:AMPlowsym} can be rigorously proved by means of an instructive application of the master theorems for the abstract symmetric AMP iteration~\eqref{eq:AMPsym} in Section~\ref{sec:AMPsym}. In the next few paragraphs (which can be skipped on a first reading), we will outline the key arguments in the setting of Corollary~\ref{cor:AMPlowsym}; a full proof of the more general Theorem~\ref{thm:AMPlowsym} can be found in Section~\ref{sec:LowRankproofs}. 

In summary, we begin by rewriting the AMP algorithm~\eqref{eq:AMPmatsym} in terms of the `noise' components $\breve{u}^k\equiv\breve{u}^k(n):=v^k-\mu_k v$ of the effective observations $v^k$, and aim to show that the corresponding noise variables in the limiting univariate problem are indeed Gaussian (and independent of $V$), with mean~0 and variance $\sigma_k^2$ given by~\eqref{eq:statevolsymat}. To this end, it can be seen that the resulting recursion~\eqref{eq:matsymdelta} below for $(\breve{u}^k:k\in\N)$ is very similar to an iteration
of the abstract form~\eqref{eq:AMPsym}, whose exact asymptotics are given by Theorems~\ref{thm:AMPmaster} and~\ref{thm:masterext}. In addition to these main workhorse results, some additional technical arguments are needed to take care of a `correction term' in~\eqref{eq:matsymdelta} below with asymptotically vanishing influence. 

The conclusion is that for each $k$, the joint empirical distribution $\nu_n(\breve{u}^k,v)$ of the entries of $\breve{u}^k(n)=v^k(n)-\mu_k v(n)$ and $v\equiv v(n)$ converges completely in $\widetilde{d}_2$ to the distribution of $(\sigma_k G_k,V)$ as $n\to\infty$. Equivalently, $\nu_n(v^k,v)$ converges completely in $\widetilde{d}_2$ to the distribution of $(\mu_k V+\sigma_k G_k,V)$ as $n\to\infty$, whence the conclusion~\eqref{eq:asymdev} of Corollary~\ref{cor:AMPlowsym} follows straightforwardly.

\unparskip
\begin{proof}[Proof sketch for Corollary~\ref{cor:AMPlowsym}]
\label{sketch:rankone}
More precisely, under the spiked model~\eqref{eq:symspike}, $A$ is the sum of independent signal and noise matrices $\lambda vv^\top/n$ and $W$ respectively, so~\eqref{eq:AMPmatsym} becomes $v^{k+1}\equiv v^{k+1}(n)=\lambda\ipr{\hat{v}^k}{v}_n v+Wg_k(v^k)-b_k g_{k-1}(v^{k-1})$ for $k,n\in\N$. Rearranging this and defining \[\delta_k\equiv\delta_k(n):=\lambda\ipr{\hat{v}^{k-1}}{v}_n-\mu_k\]
for all $k$ and $n$, we see that $\breve{u}^k\equiv\breve{u}^k(n)=v^k(n)-\mu_k v(n)$ satisfies
\begin{equation}
\label{eq:matsymdelta}
\breve{u}^1=W\hat{v}^0+\delta_1 v,\qquad\breve{u}^{k+1}=Wg_k(\breve{u}^k+\mu_k v)-b_kg_{k-1}(\breve{u}^{k-1}+\mu_{k-1}v)+\delta_{k+1}v\quad\text{for }k\in\N,
\end{equation}
where $b_k\equiv b_k(n)=\abr{g_k'(v^k)}_n=\abr{g_k'(\breve{u}^k+\mu_k v)}_n$. Setting $u^1:=W\hat{v}^0$ and dropping the final $\delta_{k+1}v$ term from the right hand side of~\eqref{eq:matsymdelta}, we obtain a related recursion
\begin{equation}
\label{eq:matsymabs}
u^{k+1}\equiv u^{k+1}(n):=Wg_k(u^k+\mu_k v)-\tilde{b}_kg_{k-1}(u^{k-1}+\mu_{k-1}v)\quad\text{for }k\in\N,
\end{equation}
where $\tilde{b}_k\equiv\tilde{b}_k(n):=\abr{g_k'(u^k+\mu_k v)}_n$. This is an instance of~\eqref{eq:AMPsym} with $f_k,f_k'\colon\R^2\to\R$ given by $f_k(x,y)=g_k(x+\mu_k y)$ and $f_k'(x,y)=g_k'(x+\mu_k y)$ for $x,y\in\R$. Under~\ref{ass:S0}--\ref{ass:S2}, it is straightforward to verify that~\ref{ass:A0}--\ref{ass:A5} are satisfied with
\begin{equation}
\label{eq:statevolrankone}
\tau_1^2=\clim_{n\to\infty}\norm{\hat{v}^0}_n^2=\sigma_1^2\quad\text{and}\quad \tau_{k+1}^2=\E\bigl(f_k(\sigma_k G_k,V)^2\bigr)=\E\bigl(g_k(\mu_k V+\sigma_k G_k)^2\bigr)=\sigma_{k+1}^2
\end{equation}
for all $k\in\N$ (by induction), in view of the state evolution recursion for $(\sigma_k:k\in\N)$ in~\eqref{eq:musigma1}--\eqref{eq:statevolsymat}. It follows from Theorem~\ref{thm:AMPmaster} that for each $k$ in~\eqref{eq:matsymabs}, the joint empirical distribution $\nu_n(u^k,v)$ converges completely in $d_2$ to the distribution $N(0,\sigma_k^2)\otimes\pi$ of $(\sigma_k G_k,V)$ as $n\to\infty$. 

It now remains to establish that the $\delta_{k+1}v$ term in~\eqref{eq:matsymdelta} has asymptotically negligible effect, in the sense that the iterates in~\eqref{eq:matsymdelta} remain close to those for~\eqref{eq:matsymabs} and hence have the same limiting distributions. Specifically, it can be shown by induction on $k\in\N$ that $\delta_k(n)\cvc 0$, that $\norm{\breve{u}^k-u^k}_n\cvc 0$ and hence that $\widetilde{d}_2\bigl(\nu_n(\breve{u}^k,v),N(0,\sigma_k^2)\otimes\pi\bigr)\cvc 0$ as $n\to\infty$ for each fixed $k$. The arguments involved are fairly routine, and are spelled out in detail in Section~\ref{sec:LowRankproofs}. We mention here that the first part of the inductive step reveals the origins of the state evolution recursion for $(\mu_k:k\in\N)$ in~\eqref{eq:musigma1}--\eqref{eq:statevolsymat}: it follows from the inductive hypothesis $\widetilde{d}_2\bigl(\nu_n(\breve{u}^k,v),N(0,\sigma_k^2)\otimes\pi\bigr)\cvc 0$ that
\[\lambda\ipr{\hat{v}^k}{v}_n=\frac{\lambda}{n}\sum_{i=1}^n v_i g_k(\breve{u}_i^k+\mu_k v_i)\cvc\lambda\E\bigl(Vg_k(\mu_k V+\sigma_k G)\bigr)=\mu_{k+1}\]
as $n\to\infty$, so indeed $\delta_{k+1}(n)\cvc 0$ as $n\to\infty$.
\end{proof}

\unparskip
We conclude this subsection by noting that for a given sequence of (random) spikes $v\equiv v(n)$ satisfying~\ref{ass:S1}, the quality of the estimates $\hat{v}^k\equiv \hat{v}^k(n)$ clearly depends on the vectors $v^0\equiv v^0(n)$ that are used to initiate the AMP iterations, as well as the sequence of Lipschitz functions $g_k\colon\R\to\R$. In the next two subsections, we will describe how these inputs to~\eqref{eq:AMPmatsym} can be suitably chosen to achieve good estimation performance, based on the information that we have about the distribution of $V$.

\hfparskip
\subsection{Spectral initialisation}
\label{sec:spectral}
In the context of the spiked model~\eqref{eq:symspike}, it is helpful to think of AMP as a method by which we can potentially improve a `pilot' estimator $\hat{v}^0\equiv \hat{v}^0(n)=g_0(v^0)$ of $v\equiv v(n)$, in the sense that we may be able to increase the asymptotic empirical correlation in~\eqref{eq:asymcorr} (i.e.\ the effective signal-to-noise ratio) by repeatedly iterating~\eqref{eq:AMPmatsym}. To this end, a minimum requirement is that we obtain effective signal-to-noise ratios $\rho_{k+1}$ that are strictly positive, since the corresponding estimates $\hat{v}^k\equiv \hat{v}^k(n)$ ought to be at least partially aligned with $v$ in the limit $n\to\infty$. 

When $\E(V)\neq 0$, we will see in Section~\ref{sec:lowrankgk} that if the functions $g_k$ are chosen appropriately, then it suffices to take $v^0\equiv v^0(n)=c\mathbf{1}_n\equiv(c,\dotsc,c)\in\R^n$ for each $n$, where $c\in\R$ is fixed. However, this does not work when $\E(V)=0$: in this case, $\mu_0=\clim_{n\to\infty}\ipr{c\mathbf{1}_n}{v}_n=0$ in~\ref{ass:S1}, and for any choice of $(g_k)$, the state evolution recursion~\eqref{eq:statevolsymat} then yields $\mu_k=0$ and $\rho_k=(\mu_k/\sigma_k)^2=0$ for all $k\in\N$ (since $V$ and $G_k$ are independent). For each $k$, it follows from~\eqref{eq:asymcorr} that $\ipr{\hat{v}^k}{v}_n\cvc 0$ as $n\to\infty$, so $\hat{v}^k\equiv\hat{v}^k(n)$ is asymptotically uninformative as an estimator of $v\equiv v(n)$.

Thus, when $\E(V)=0$, we require $\mu_0\neq 0$ and pilot estimators that have non-zero asymptotic empirical correlation with $v$. 
For $n\in\N$, consider initialising the AMP algorithm~\eqref{eq:AMPmatsym} with $v^0=c\hat{\varphi}$ for some $c\neq 0$, where $\hat{\varphi}\equiv\hat{\varphi}(n)$ is a normalised principal eigenvector of $A\equiv A(n)$ with $\norm{\hat{\varphi}}_n=1$. This is almost surely well-defined up to its sign, and yields an initial estimate with the desired property precisely when $\lambda>1$; indeed, recall from~\eqref{eq:princeigen} that $\abs{\ipr{\hat{\varphi}}{v}_n}/\norm{v}_n\cvc\sqrt{1-\lambda^{-2}}>0$ for such $\lambda$. Using the orthogonal invariance of $W\sim\GOE(n)$, Proposition~\ref{prop:power} below extends this convergence result to show that $\bigl\{\bigl(\hat{\varphi}(n),v(n)\bigr):n\in\N\bigr\}$ satisfies condition~\ref{ass:S1} with $\mu_0=\sqrt{1-\lambda^{-2}}$, $\sigma_0=1/\lambda$ and $U\sim N(0,1)$, provided that $\ipr{\hat{\varphi}}{v}_n\geq 0$ for all $n$; see Remark~\ref{rem:spectralsign} below for further discussion of this final issue.
\begin{proposition}
\label{prop:power}
Suppose that $V\sim\pi$ satisfies $\E(V^2)=1$ and $d_2\bigl(\nu_n(v),\pi\bigr)\cvc 0$ as $n\to\infty$, where $\nu_n(v)=\inv{n}\sum_{i=1}^n\delta_{v_i}$ denotes the empirical distribution of $v\equiv v(n)$ for $n\in\N$. If $\lambda>1$ in~\eqref{eq:symspike}, and each $\hat{\varphi}\equiv\hat{\varphi}(n)$ is a principal eigenvector of $A\equiv A(n)$ whose direction is chosen so that $\ipr{\hat{\varphi}}{v}_n\geq 0$ for all $n$, then
\[\sup_{\psi\in\PL_2(2,1)}\;\biggl|\frac{1}{n}\sum_{i=1}^n\psi(\hat{\varphi}_i,v_i)-\E\bigl\{\psi\bigl(\sqrt{1-\lambda^{-2}}\,V+\lambda^{-1}G_0,V\bigr)\bigr\}\biggr|\cvc 0\]
as $n\to\infty$, where $G_0\sim N(0,1)$ is independent of $V$.
\end{proposition}

\unparskip
For proofs of more general results of this type for finite-rank perturbations of GOE matrices, see~\citet[Lemma~C.1 and Corollary~C.3]{MV21}.

In the subsequent asymptotic analysis of the AMP algorithm~\eqref{eq:AMPmatsym} with spectral initialisation, an additional technical challenge stems from the fact that $\hat{\varphi}\equiv\hat{\varphi}(n)$ is not independent of the noise matrix $W\equiv W(n)$ for any $n$. This means that condition~\ref{ass:S0} does not hold in general, so the theory from Section~\ref{sec:rankone} is not directly applicable in this setting. Nevertheless,~\citet{MV21} established the following to recover the desired conclusion for this particular initialisation.
\begin{theorem}
\label{thm:spectral}
Suppose that $\lambda>1$ in the spiked model~\eqref{eq:symspike}, and that 
the hypotheses of Proposition~\ref{prop:power} are satisfied for a sequence of AMP algorithms~\eqref{eq:AMPmatsym} initialised with $v^0\equiv v^0(n)=c\,\hat{\varphi}(n)$ and $\hat{v}^{-1}\equiv\hat{v}^{-1}(n)=\lambda^{-1}c\,\hat{\varphi}(n)$ for each $n\in\N$, where $\ipr{\hat{\varphi}}{v}_n\geq 0$ and $c\neq 0$ is fixed. Starting with $\mu_0=c\sqrt{1-\lambda^{-2}}$, $\sigma_0=c/\lambda$ and $U\sim N(0,1)$ in~\eqref{eq:musigma1}, define the state evolution parameters $\mu_k,\sigma_k,\bar{\Sigma}^{[k]}$ for $k\in\N$ according to~\eqref{eq:statevolsymat}--\eqref{eq:Sigmabarmat}. Then under~\emph{\ref{ass:S2}}, the conclusions of Theorem~\ref{thm:AMPlowsym} and Corollary~\ref{cor:AMPlowsym} remain valid.
\end{theorem}

\unparskip
To circumvent the difficulty mentioned above, Theorem~\ref{thm:spectral} can be proved by first applying the existing AMP machinery to a suitably modified version of the iteration~\eqref{eq:AMPmatsym} for which~\ref{ass:S0} is satisfied, and then showing that this has the same asymptotics as the original procedure with spectral initialisation. In the spiked model~\eqref{eq:symspike} where the signal matrix has rank 1, one approach along these lines is to design a more tractable two-stage iteration, in which the input to~\eqref{eq:AMPmatsym} in the second phase is the output of a surrogate power method that approximates $v^0=c\,\hat{\varphi}$. 
This `artificial' first phase takes the form of an AMP iteration with specially chosen linear threshold functions (see~\eqref{eq:gklinear} in Section~\ref{sec:lowrankgk}) and a (non-spectral) initialiser that is independent of $W$. 
The success of this strategy relies on the fact that the spectral gap $\lambda_1(A)-\lambda_2(A)$ of $A\equiv A(n)$ has a strictly positive limit as $n\to\infty$ when $\lambda>1$, as mentioned at the start of Section~\ref{sec:rankone}. For further details of applications of this proof technique in the GAMP setting of Section~\ref{Sec:GAMP}, see~\citet{MTM20} and~\citet{MM20}.

We refer the reader to~\citet[Appendix~A]{MV21} for a different proof of Theorem~\ref{thm:spectral} that extends more readily to a wider class of AMP algorithms for general low-rank matrix estimation (see Section~\ref{sec:lowrankgeneral}). This involves studying a variant of~\eqref{eq:AMPmatsym} in which $A\equiv A(n)$ is replaced with
\[\tilde{A}\equiv\tilde{A}(n)=\frac{\lambda_1(A)}{n}\hat{\varphi}\hat{\varphi}^\top+\hat{P}^\perp\biggl(\frac{\lambda}{n}vv^\top+\tilde{W}\biggr)\hat{P}^\perp\]
for each $n$, where $\lambda_1(A)$ is the maximal eigenvalue of $A$, the matrix $\hat{P}:=I-\hat{\varphi}\hat{\varphi}^\top/n$ represents the projection onto the orthogonal complement of $\hat{\varphi}$, and (crucially) $\tilde{W}\sim\GOE(n)$ is independent of $W$ and $v$. To relate the simplified iteration based on $\tilde{A}$ to the original AMP procedure, an important technical step is to show that the conditional distributions of $A$ and $\tilde{A}$ given $\bigl(\hat{\varphi},\lambda_1(A)\bigr)$ are close in total variation distance when $n$ is large.
\begin{remark}
\label{rem:spectralsign}
In an estimation context where each $v\equiv v(n)$ is unknown, it is sometimes not possible to consistently determine the sign of the leading eigenvector of $A\equiv A(n)$ that should be used as a spectral initialiser, to ensure that it has non-negative asymptotic empirical correlation with $v$. 
For example, this is the case if the limiting prior distribution $\pi$ is symmetric, i.e.\ $V\eqd -V$. On the event of probability 1 where $A$ has a unique maximal eigenvalue, suppose that one of the two possible directions for the corresponding eigenvector is chosen uniformly at random when carrying out spectral initialisation. In other words, let $v^0\equiv v^0(n)=\epsilon\hat{\varphi}$ for each $n$, where $\ipr{\hat{\varphi}}{v}_n\geq 0$ and $\epsilon\equiv\epsilon(n)$ is a Rademacher random variable that is independent of everything else. With this choice of $v^0$, there are two different state evolution trajectories that can arise: for $\epsilon'\in\{-1,1\}$, let $\mu_0(\epsilon'):=\epsilon'\sqrt{1-\lambda^{-2}}$ and $\sigma_0,U$ be as above, and define $\mu_k(\epsilon'),\sigma_k(\epsilon'),\bar{\Sigma}^{[k]}(\epsilon')$ for $k\in\N$ as per~\eqref{eq:statevolsymat}--\eqref{eq:Sigmabarmat}. For the resulting iterates $v^k\equiv v^k(n)$, Theorem~\ref{thm:spectral} implies that as $n\to\infty$, we have
\[\sup_{\psi\in\PL_2(2,1)}\;\biggl|\frac{1}{n}\sum_{i=1}^n\psi(v_i^k,v_i)-\E\bigl\{\psi\bigl(\mu_k(\epsilon)V+\sigma_k(\epsilon)G_k,V\bigr)\!\bigm|\!\epsilon\bigr\}\biggr|\cvc 0\]
for each $k\in\N_0$, as well as appropriate analogues of~\eqref{eq:AMPlowsym}--\eqref{eq:asymcorr}. Since $\E\bigl\{\psi\bigl(\mu_k(\epsilon)V+\sigma_k(\epsilon)G_k,V\bigr)\!\bigm|\!\epsilon\bigr\}$ is random for each $\psi$, the empirical distribution of the components of $v^k(n)$ may not converge (completely in $d_2$) to a deterministic limit as $n\to\infty$, unlike in earlier results.
Instead, we see that for large $n$, the behaviour of the AMP iterates is characterised by a state evolution recursion with a random initial condition $\mu_0(\epsilon)$ that depends on $v^0\equiv v^0(n)$ through the (unknown) sign $\epsilon\equiv\epsilon(n)$.
\end{remark}

\umparskip
\subsection{Choosing the functions \texorpdfstring{$g_k$}{g\_k}}
\label{sec:lowrankgk}
Recall that our goal is to specialise the general AMP algorithm~\eqref{eq:AMPmatsym} to produce estimates $\hat{v}^k=g_k(v^k)$ of $v$ that exploit full or partial knowledge of the limiting prior distribution $\pi$ from~\ref{ass:S1}. Corollary~\ref{cor:AMPlowsym} suggests that we should aim to choose a sequence of Lipschitz `denoising' functions $g_k\colon\R\to\R$ for which each $g_k(\mu_k V+\sigma_k G_k)$ performs well as an estimator of $V\sim\pi$ in the limiting univariate problem, where $G_k\sim N(0,1)$ is independent of $V$ for $k\in\N$. More precisely, it would be desirable to ensure that the effective signal-to-noise ratio $\rho_k=(\mu_k/\sigma_k)^2$ is large for each $k$, since~\eqref{eq:asymcorr} tells us that the asymptotic empirical correlation between $\hat{v}^k$ and $v$ is given by $\sqrt{\rho_{k+1}}/\lambda$. In fact, the implication of Lemma~\ref{lem:snrmax} below is that achieving a high effective signal-to-noise ratio ought to be our first priority, even when the ultimate objective is for $\hat{v}^k=g_k(v^k)$ to have low asymptotic estimation error $\E\bigl\{\psi\bigl(g_k(\mu_k V+\sigma_k G_k),V\bigr)\bigr\}$ with respect to some specific loss function $\psi\in\PL_2(2)$. 
\begin{lemma}
\label{lem:snrmax}
Let $G\sim N(0,1)$ be independent of $V\sim\pi$. Then for any Borel measurable loss function $\psi\colon\R^2\to [0,\infty)$, \[\rho\mapsto\inf_g\E\bigl\{\psi\bigl(g(\sqrt{\rho}V+G),V\bigr)\bigr\}=:R_{\pi,\psi}(\rho)\]
is non-increasing on $[0,\infty)$, where the infimum is over all Borel measurable functions $g\colon\R\to\R$. This infimum is attained for all $\rho\in [0,\infty)$ if for example $\psi(x,y)=\Psi(x-y)$ for some convex function $\Psi$ with $\Psi(u)\to\infty$ as $\abs{u}\to\infty$. 
\end{lemma}

\unparskip
The intuition behind this result is straightforward: to minimise $\E\bigl\{\psi\bigl(g(\sqrt{\rho}V+G),V\bigr)\bigr\}$ jointly over $\rho$ (belonging to a given range) and all measurable $g$, we should always begin by taking the largest possible $\rho$ (i.e.\ the least noisy $\sqrt{\rho}V+G$) before subsequently optimising over $g$. A formal proof of Lemma~\ref{lem:snrmax} is deferred to Section~\ref{sec:LowRankproofs}. The arguments therein show also that the first assertion of the lemma remains valid if the infimum is instead taken only over Lipschitz functions (which are more relevant to the setting of AMP). 

Note that for (known) $\mu,\sigma\in\R$ with $(\mu/\sigma)^2=\rho$, the quantity $R_{\pi,\psi}(\rho)$ is the $\pi$-Bayes risk with respect to $\psi$ in a Bayesian mean estimation problem where we place a prior $\pi$ on $V$ and observe $Y=\mu V+\sigma G$ (as in the paragraph above), i.e.\ $Y\,|\,V\sim N(\mu V,\sigma^2)$. If there exists a Borel measurable $g^*\colon\R\to\R$ that attains the infimum in the definition of $R_{\pi,\psi}(\rho)$, then $g^*(Y)$ is a $\pi$-Bayes estimator of~$V$ (with respect to $\psi$) based on $Y$.

\textbf{Bayes-AMP}: Suppose first that for some $k\in\N_0$, we are given the distribution $\pi$ of $V$ and the state evolution parameters $\mu_k,\sigma_k$ (which depend on $\mu_0,\sigma_0$ in~\ref{ass:S1} as well as the functions $g_0,g_1,\dotsc,g_{k-1}$). For convenience, when $k=0$, we write $G_0$ for the random variable $U$ from~\ref{ass:S1}, and assume that its distribution is also known. Let $g_k^*\colon\R\to\R$ be any measurable function with
\begin{equation}
\label{eq:gkbayes}
g_k^*(\mu_k V+\sigma_k G_k)=\E(V\,|\,\mu_k V+\sigma_k G_k),
\end{equation}
which in principle can be computed based on $Y_k:=\mu_k V+\sigma_k G_k$. In particular, for $k\in\N$, we have $G_k\sim N(0,1)$, in which case if $\sigma_k> 0$, then $Y_k$ has a smooth (real analytic), strictly positive Lebesgue density on $\R$ given by $p_k(y):=\int_\R\phi_{\sigma_k}(y-\mu_k x)\,d\pi(x)$, where $\phi_{\sigma_k}$ is the density of a $N(0,\sigma_k^2)$ random variable. Then by \emph{Tweedie's formula}~\citep{Rob56,Efr11}, we can take
\begin{equation}
\label{eq:tweedie}
g_k^*(y)=\frac{y+\sigma_k^2\,(\log p_k)'(y)\Ind_{\{\sigma_k\neq 0\}}}{\mu_k}\Ind_{\{\mu_k\neq 0\}}=\frac{y+\sigma_k^2\,(p_k'/p_k)(y)\Ind_{\{\sigma_k\neq 0\}}}{\mu_k}\Ind_{\{\mu_k\neq 0\}}\quad\text{for }y\in\R.
\end{equation}
For example, if $\pi$ is the uniform distribution on $\{-1,1\}$ and $\sigma_k\neq 0$, then $g_k^*(y)=\tanh(\mu_k y/\sigma_k^2)$ for $y\in\R$. 

In the AMP literature, $g_k^*$ is referred to as the `Bayes optimal' choice of threshold function in~\eqref{eq:AMPmatsym}, since the posterior mean $g_k^*(Y_k)=\E(V\,|\,Y_k)$ is the Bayes estimator of $V$ based on $Y_k$ with respect to quadratic loss (often known as the \emph{minimum mean squared error} (MMSE) estimator). Indeed, by the characterisation of $\E(V\,|\,Y_k)$ as an orthogonal projection,
\begin{equation}
\label{eq:ortho}
\E\bigl\{\bigl(V-g(Y_k)\bigr)^2\bigr\}=\E\bigl\{\bigl(V-g_k^*(Y_k)\bigr)^2\bigr\}+\E\{(g_k^*-g)^2(Y_k)\}\geq\E\bigl\{\bigl(V-g_k^*(Y_k)\bigr)^2\bigr\}
\end{equation}
for all measurable $g\colon\R\to\R$. In addition,
\begin{equation}
\label{eq:snrCS}
\frac{\E\bigl(Vg(Y_k)\bigr)^2}{\E\bigl(g(Y_k)^2\bigr)}=\frac{\E\bigl(g_k^*(Y_k)\,g(Y_k)\bigr)^2}{\E\bigl(g(Y_k)^2\bigr)}\leq\E\bigl(g_k^*(Y_k)^2\bigr)
\end{equation}
by the Cauchy--Schwarz inequality, with equality if $g$ is a (non-zero) scalar multiple of $g_k^*$. Thus, for given $\mu_k,\sigma_k$, the function $g_k^*$ simultaneously minimises the asymptotic mean squared error in~\eqref{eq:asymmse} and maximises the asymptotic empirical correlation (i.e.\ the effective signal-to-noise ratio $\rho_{k+1}^*$) in~\eqref{eq:asymcorr} over all measurable $g_k\colon\R\to\R$. 

The following result is a slight extension of~\citet[Remark~2.3]{MV21} (with a different, simpler proof given in Section~\ref{sec:LowRankproofs}) that provides sufficient conditions on $\pi$ under which $g_k^*$ is Lipschitz and satisfies~\ref{ass:S2}.
\begin{lemma}
\label{lem:tweedie}
Suppose either that $V$ has a log-concave density, or that there exist independent random variables $U_0,V_0$ such that $U_0$ is Gaussian, $V_0$ is compactly supported and $V\eqd U_0+V_0$. Then for $\mu_k,\sigma_k\neq 0$, the function $g_k^*$ in~\eqref{eq:tweedie} is smooth and Lipschitz on $\R$.
\end{lemma}

\unparskip
Assuming now that we have complete knowledge of the distributions of $U,V$ as well as $\lambda>0$ in~\eqref{eq:symspike} and $\mu_0,\sigma_0$ from~\ref{ass:S1}, we can construct a \emph{`Bayes-AMP' algorithm} of the form~\eqref{eq:AMPmatsym} by recursively defining $(g_k^*:k\in\N)$ and state evolution sequences $(\mu_k^*,\sigma_k^*:k\in\N)$ in accordance with~(\ref{eq:gkbayes},\,\ref{eq:tweedie}) and~\eqref{eq:statevolsymat} respectively. We will write $v^{k,\mathrm{B}}\equiv v^{k,\mathrm{B}}(n)$ for the resulting Bayes-AMP iterates (i.e.\ effective observations) and $\hat{v}^{k,\mathrm{B}}\equiv\hat{v}^{k,\mathrm{B}}(n):=g_k^*(v^{k,\mathrm{B}})$ for the Bayes-AMP estimates of $v\equiv v(n)$. 

For each $k\in\N$, we have $\mu_{k+1}^*=\lambda\E\bigl(Vg_k^*(Y_k)\bigr)=\lambda\E\bigl(g_k^*(Y_k)^2\bigr)=\lambda(\sigma_{k+1}^*)^2$ by~\eqref{eq:snrCS}, and since $\E(V^2)=1$ by~\ref{ass:S1}, the effective signal-to-noise ratios in Bayes-AMP satisfy
\begin{equation}
\label{eq:rhok}
\rho_{k+1}^*:=(\mu_{k+1}^*/\sigma_{k+1}^*)^2=\lambda^2(\sigma_{k+1}^*)^2=\lambda^2\,\E\bigl(g_k^*(Y_k)^2\bigr)=\lambda^2\bigl(1-\E\bigl\{\bigl(V-g_k^*(Y_k)\bigr)^2\bigr\}\bigr).
\end{equation}
Thus, the state evolution recursion~\eqref{eq:statevolsymat} for Bayes-AMP can be compactly written as
\begin{equation}
\label{eq:gammak}
\rho_0^*:=(\mu_0/\sigma_0)^2,\quad\rho_{k+1}^*:=\lambda^2\bigl(1-\mmse_k(\rho_k^*)\bigr)\quad\text{for }k\in\N_0,
\end{equation}
where for $\rho\in [0,\infty)$ we denote by
\[\mmse_k(\rho):=\E\bigl\{\bigl(V-\E(V\,|\,\sqrt{\rho}V+G_k)\bigr)^2\bigr\}\]
the \emph{minimum mean squared error} (i.e.\ the Bayes risk with respect to squared error loss $\psi_2\colon (x,y)\mapsto (x-y)^2$) for the problem of reconstructing $V$ based on the corrupted observation $\sqrt{\rho}V+G_k$. For $k\in\N$, we have $G_k\sim N(0,1)$, in which case we simply write $\mmse(\rho)$ for $\mmse_k(\rho)=R_{\pi,\psi_2}(\rho)$. For concreteness, we set $\mmse(\infty)=0$, which is consistent with the fact that $\mmse(\rho)\to 0$ as $\rho\to\infty$.

At each iteration $k\in\N$, it turns out that $\rho_{k+1}^*$ is the highest effective signal-to-noise ratio that can be achieved with any choice of functions $(g_k)$ in the generic AMP procedure~\eqref{eq:AMPmatsym}. 
\begin{corollary}
\label{cor:snrbayes}
Consider any sequence of AMP iterations $\bigl(v^k\equiv v^k(n):k,n\in\N\bigr)$ of the form~\eqref{eq:AMPmatsym} for which the hypotheses of Theorem~\ref{thm:AMPlowsym}  or~\ref{thm:spectral} are satisfied with $V\sim\pi$ and suitable $\mu_0,\sigma_0$. Let $(\mu_k,\sigma_k:k\in\N_0)$ and $\bigl(\rho_k=(\mu_k/\sigma_k)^2:k\in\N_0\bigr)$ be the associated sequences of state evolution parameters and effective signal-to-noise ratios respectively. Define $(\rho_k^*:k\in\N_0)$ as in~\eqref{eq:rhok}. Then for each $k\in\N_0$ and any $\psi\in\PL_2(2)$, the estimates $\hat{v}^k\equiv\hat{v}^k(n)=g_k(v^k)$ satisfy
\begin{align}
\label{eq:snrbayes}
\frac{\abs{\ipr{\hat{v}^k}{v}_n}}{\norm{\hat{v}^k}_n\norm{v}_n}&\cvc\frac{\sqrt{\rho_{k+1}}}{\lambda}\leq\frac{\sqrt{\rho_{k+1}^*}}{\lambda}\\
\label{eq:bayesdev}
\text{and}\;\;\quad\frac{1}{n}\sum_{i=1}^n\psi(\hat{v}_i^k,v_i)&\cvc\E\bigl\{\psi\bigl(g_k(\mu_k V+\sigma_k G_k),V\bigr)\bigr\}\geq R_{\pi,\psi}(\rho_k^*)\quad\text{as }n\to\infty.
\end{align}
\end{corollary}

\unparskip
This follows from~\eqref{eq:asymcorr} and~\eqref{eq:snrCS} above, as well as Lemma~\ref{lem:snrmax}, which implies in particular that $\rho\mapsto\mmse(\rho)$ is decreasing on $[0,\infty)$. See Section~\ref{sec:LowRankproofs} for a full justification of Corollary~\ref{cor:snrbayes}. 

Under the conditions of Lemma~\ref{lem:tweedie} above, the Bayes optimal functions $g_k^*$ are Lipschitz and satisfy~\ref{ass:S2}. We can then apply the general results in Sections~\ref{sec:rankone} and~\ref{sec:spectral} to obtain the exact asymptotics for Bayes-AMP, for which it follows that~\eqref{eq:snrbayes} holds with equality. In other words, at every iteration, the Bayes-AMP estimate $\hat{v}^{k,\mathrm{B}}=g_k^*(v^{k,\mathrm{B}})$ achieves the optimal asymptotic empirical correlation among all AMP algorithms that are covered by the theory above. Moreover, with the initialisations in (i) and (ii) below, Theorem~\ref{thm:bayesamp} shows that Bayes-AMP achieves the objective set out at the start of Section~\ref{sec:spectral}, namely that $\hat{v}^{k+1,\mathrm{B}}$ is a strict improvement on $\hat{v}^{k,\mathrm{B}}$ in terms of its asymptotic squared error and empirical correlation (i.e.\ the effective signal-to-noise ratio $\rho_{k+1}^*$) for each $k$. 
This means that for large $k$ and $n$, the performance of $\hat{v}^{k,\mathrm{B}}\equiv\hat{v}^{k,\mathrm{B}}(n)$ is approximately characterised by a fixed point of the recursion in~\eqref{eq:gammak} to which $(\rho_k^*)$ converges monotonically; see Figure~\ref{fig:bayesamp}.
\begin{theorem}
\label{thm:bayesamp}
Let $\bigl(v^{k,\mathrm{B}}\equiv v^{k,\mathrm{B}}(n):k,n\in\N\bigr)$ be a sequence of Bayes-AMP iterations that satisfies either (i) or (ii) below.

\unparskip
\begin{enumerate}[label=(\roman*)]
\item \emph{(Non-spectral initialisation)} $v^0\equiv v^0(n)=c\mathbf{1}_n$ for each $n$, where $c\in\R$ is fixed, and the hypotheses of Theorem~\ref{thm:AMPlowsym} are satisfied with $\E(V)\neq 0$, in which case $\mu_0=0$, $\sigma_0=c$ and $\rho_0^*=0$.
\item \emph{(Spectral initialisation)} $v^0\equiv v^0(n)=c\,\hat{\varphi}(n)$ for each $n$, where $\ipr{\hat{\varphi}}{v}_n\geq 0$ and $c\neq 0$ is fixed, and the hypotheses of Theorem~\ref{thm:spectral} are satisfied with $\lambda>1$, in which case $\mu_0=c\sqrt{1-\lambda^{-2}}$, $\sigma_0=c/\lambda$ and $\rho_0^*=\lambda^2-1$.
\end{enumerate}

\unparskip
Suppose that $V\sim\pi$ satisfies one of the conditions of Lemma~\ref{lem:tweedie}. Then we have the following:

\unparskip
\begin{enumerate}[label=(\alph*)]
\item The sequence $(\rho_k^*:k\in\N_0)$ of effective signal-to-noise ratios defined through~\eqref{eq:rhok} is strictly increasing, and converges to the \emph{smallest} strictly positive fixed point of $\rho=\lambda^2\bigl(1-\mmse(\rho)\bigr)$, which we denote by $\rho_\AMP^*\equiv\rho_\AMP^*(\lambda)\in (0,\lambda^2]$.
\item For $k\in\N$ and a (convex, non-negative) loss function $\psi\in\PL_2(2)$, suppose that $g_{k,\psi}^*\colon\R\to\R$ is Lipschitz and attains the infimum in the definition of $R_{\pi,\psi}(\rho_k^*)$. Then the estimates $\hat{v}^{k,\psi}\equiv\hat{v}^{k,\psi}(n):=g_{k,\psi}^*(v^{k,\mathrm{B}})$ satisfy~\eqref{eq:bayesdev} with equality, i.e.\ $n^{-1}\sum_{i=1}^n\psi(\hat{v}_i^{k,\psi},v_i)\cvc R_{\pi,\psi}(\rho_k^*)$ as $n\to\infty$, and $R_{\pi,\psi}(\rho_k^*)\geq R_{\pi,\psi}(\rho_{k+1}^*)$.

\item The Bayes-AMP estimates $\hat{v}^{k,\mathrm{B}}=g_k^*(v^{k,\mathrm{B}})$ satisfy
\begin{align} 
\label{eq:AMPesterr}
\clim_{n\to\infty}\norm{\hat{v}^{k,\mathrm{B}}-v}_n^2=1-\frac{\rho_{k+1}^*}{\lambda^2}&\searrow 1-\frac{\rho_\AMP^*(\lambda)}{\lambda^2}\\
\label{eq:AMPcorr}
\text{and}\quad\;\;\clim_{n\to\infty}\frac{\ipr{\hat{v}^{k,\mathrm{B}}}{v}_n}{\norm{\hat{v}^{k,\mathrm{B}}}_n\norm{v}_n}=\frac{\sqrt{\rho_{k+1}^*}}{\lambda}&\nearrow\frac{\sqrt{\rho_\AMP^*(\lambda)}}{\lambda}
\qquad\text{as }k\to\infty.
\end{align}
\end{enumerate}

\unparskip
\end{theorem}
\pdfsuppresswarningpagegroup=1
\begin{figure}[htbp!]
\begin{center}
\includegraphics[height=0.18\textheight]{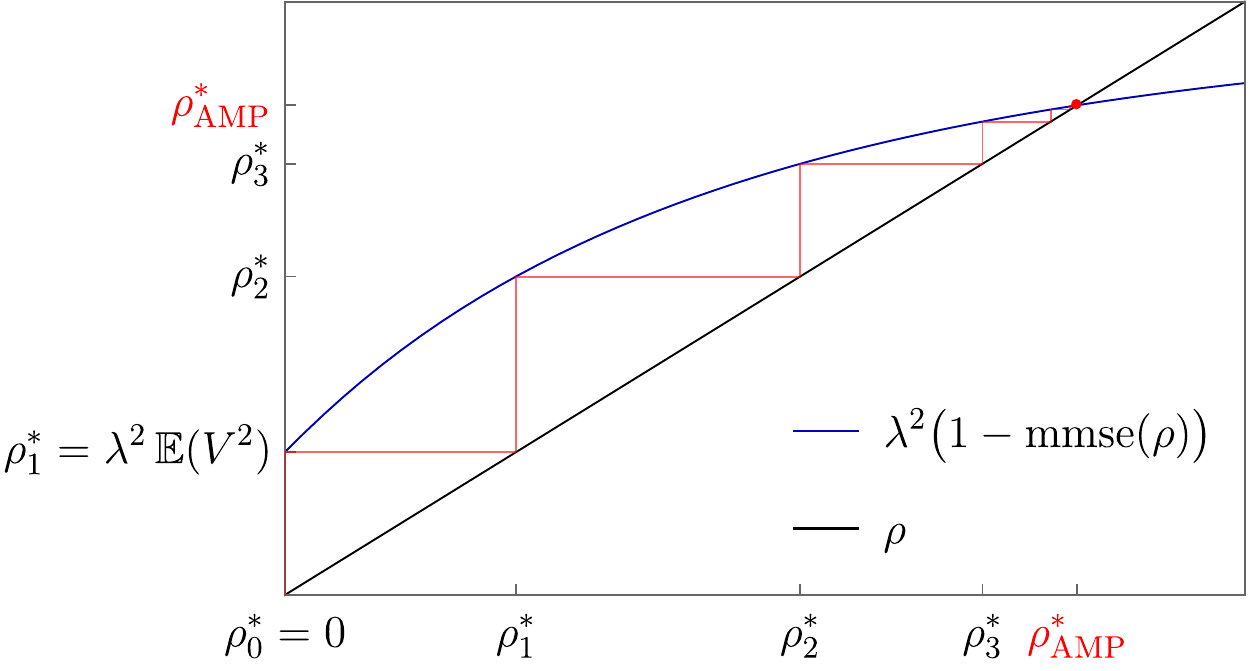}
\hspace{10pt}
\includegraphics[height=0.18\textheight]{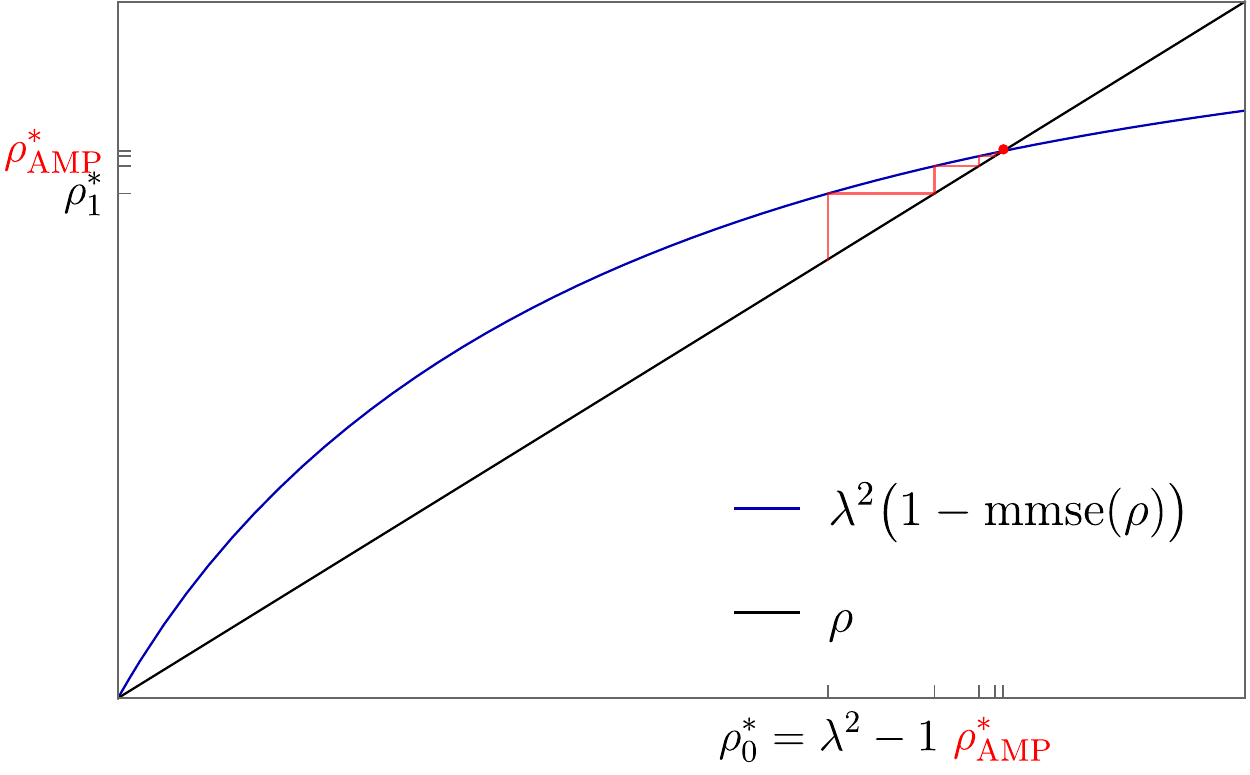}
\end{center}
\caption{\label{fig:bayesamp}`Cobweb diagrams' illustrating the conclusion of Theorem~\ref{thm:bayesamp}(a) that $\rho_k^*\nearrow\rho_\AMP^*$ as $k\to\infty$, under (i) and (ii) respectively with $\lambda=1.7$; note that $1-\mmse(0)=\E(V)^2$:\\[1ex]\emph{Left, non-spectral initialisation}: $V\sim\pi=\frac{3}{4}\delta_0+\frac{1}{4}\delta_2$, with $\E(V)=1/2\neq 0$ and $\E(V^2)=1$: convergence to $\rho_\AMP^*$ occurs when $\rho_0^*=0$.\\[1ex]\emph{Right, spectral initialisation}: $V\sim\pi=\frac{1}{2}\delta_{-1}+\frac{1}{2}\delta_1$, with $\E(V)=0$ and $\E(V^2)=1$: convergence to $\rho_\AMP^*$ occurs only if $\rho_0^*>0$.}
\end{figure}

To understand the implications of (b) above, suppose that we wish to use AMP to obtain estimates $\tilde{v}\equiv\tilde{v}(n)$ of $v\equiv v(n)$ with small (asymptotic) $\ell_1$ estimation error $n^{-1}\sum_{i=1}^n\,\abs{\tilde{v}_i-v_i}$. In view of~\eqref{eq:bayesdev} and Theorem~\ref{thm:bayesamp}(b), with $\psi$ taken to be absolute error loss $\psi_1\colon (x,y)\mapsto\abs{x-y}$, we should first run Bayes-AMP to obtain the highest possible effective signal-to-noise ratio $\rho_k^*$ at every iteration. Then for each $k\in\N$, we should consider $g_{k,\psi_1}^*\colon\R\to\R$ for which $g_{k,\psi_1}^*(y)$ is a median of the conditional (i.e.\ posterior) distribution of $V$ given $\mu_k V+\sigma_k G=y$ for (Lebesgue almost) every $y\in\R$. If we can find a Lipschitz $g_{k,\psi_1}^*$ with this property, then $\hat{v}^{k,\psi_1}:=g_{k,\psi_1}^*(v^{k,\mathrm{B}})$ attains the lowest possible limiting mean absolute error $R_{\pi,\psi_1}(\rho_k^*)=\inf_g\E\bigl\{\bigl|V-g\bigl(\sqrt{\rho_k^*}\,V+G\bigr)\bigr|\bigr\}$, among all estimators obtained from the $k^{th}$ iteration of some AMP algorithm of the form~\eqref{eq:AMPmatsym}. In cases where there is no suitable Lipschitz $g_{k,\psi_1}^*$, for example when $V$ has a discrete distribution, one possible modification of the approach above would be to replace $g_{k,\psi_1}^*$ with a Lipschitz approximation when constructing the estimator, in the hope that the resulting asymptotic $\ell_1$ error is close to $R_{\pi,\psi_1}(\rho_k^*)$.

As for Theorem~\ref{thm:bayesamp}(c), one can compare the asymptotic mean squared error~\eqref{eq:AMPesterr} and empirical correlation~\eqref{eq:AMPcorr} achieved by Bayes-AMP with the corresponding Bayes optimal quantities (i.e.\ the best possible limiting values that can be attained by \emph{any} estimator). In a spiked model~\eqref{eq:symspike} where the entries of $v\equiv v(n)$ are i.i.d.\ with distribution $\pi$, closed-form asymptotic expressions for the Bayes estimator $\E(v\,|\,A)$ were rigorously established by~\cite{barbier2016mutual} and~\citet{lelarge2019fundamental}. It turns out that the Bayes optimal performance is characterised by a fixed point $\rho_{\mathrm{B}}^*$ of $\rho=\lambda^2\bigl(1-\mmse(\rho)\bigr)$ that maximises a specific free-energy functional; see~\citet[Section~2.4]{MV21} for further details. Thus, we can precisely characterise the performance gap between Bayes-AMP and Bayes optimal estimation for symmetric rank-one matrix estimation. In particular, when the equation $\rho=\lambda^2\bigl(1-\mmse(\rho)\bigr)$ has a unique positive solution (as is the case for the $U\{-1,1\}$ prior in Figure~\ref{fig:amse}), Bayes-AMP achieves the Bayes optimal performance. Furthermore, in cases where $\rho_\AMP^*\neq\rho_{\mathrm{B}}^*$ (i.e.\ AMP is not Bayes optimal), there is currently no known polynomial-time algorithm that is superior to Bayes-AMP in terms of the limiting effective signal-to-noise ratio in~\eqref{eq:AMPcorr}.
\begin{remark}
Suppose that the limiting prior distribution $\pi$ is symmetric, i.e.\ $V\eqd -V$, in which case $v$ and $-v$ are asymptotically indistinguishable. Then $\E(V)=0$, and as mentioned in Remark~\ref{rem:spectralsign}, it is not possible to consistently choose the sign of the spectral initialiser in a data-driven way, so as to ensure that $\ipr{\hat{\varphi}}{v}_n\geq 0$ for each $n$. Nevertheless, the two possible state evolution trajectories for Bayes-AMP (with spectral initialisation) are easily seen to be identical up to the sign of each $\mu_k$, so the limits in~\eqref{eq:AMPesterr} and~\eqref{eq:AMPcorr} remain valid for \[\min_{\epsilon\in\{-1,1\}}\norm{\hat{v}^{k,\mathrm{B}}-\epsilon v}_n^2\quad\;\;\text{and}\quad\;\;\min_{\epsilon\in\{-1,1\}}\frac{\ipr{\hat{v}^{k,\mathrm{B}}}{\epsilon v}_n}{\norm{\hat{v}^{k,\mathrm{B}}}_n\norm{v}_n}\quad\;\;\text{respectively.}\]
\end{remark}
\begin{remark}
If the limiting prior distribution $\pi$ is known but some or all of $\lambda,\mu_0,\sigma_0$ are not, then starting with $\hat{v}^0\equiv\hat{v}^0(n)$ for some~$n$, we can construct an `\emph{empirical Bayes-AMP algorithm}' based on estimates of $\mu_k,\sigma_k$ for each $k$. Specifically, recalling Remark~\ref{rem:musigmaest} and proceeding inductively, we can use~\eqref{eq:tweedie} to define $\hat{g}_k^*$ based on \[\hat{\mu}_k:=\bigl(\norm{v^k}_n^2-\norm{\hat{v}^{k-1}}_n^2\bigr)^{1/2}\quad\text{and}\quad\hat{\sigma}_k:=\norm{\hat{v}^{k-1}}_n,\]
and then obtain $\hat{v}^k=\hat{g}_k^*(v^k)$ and $v^{k+1}$ via~\eqref{eq:AMPmatsym} for each~$k$. Alternatively, since $\mu_k^*=\lambda(\sigma_k^*)^2\geq 0$ in Bayes-AMP, we could instead take $\hat{\mu}_k=\hat{\lambda}\hat{\sigma}_k^2=\hat{\lambda}\norm{\hat{v}^{k-1}}_n^2$, where \[\hat{\lambda}:=\frac{\lambda_1(A)+\sqrt{\lambda_1(A)^2-4}}{2}\]
is a strongly consistent estimator of $\lambda$ by~\eqref{eq:princeigen}. Yet another approach is to first define $(\hat{\rho}_k)$ recursively by $\hat{\rho}_1:=\hat{\mu}_1^2/\hat{\sigma}_1^2$ and $\hat{\rho}_{k+1}:=\hat{\lambda}^2\bigl(1-\mmse(\hat{\rho}_k)\bigr)$ for each $k$. In view of~\eqref{eq:rhok}, we can then estimate $\mu_k,\sigma_k$ by $\hat{\rho}_{k+1}/\hat{\lambda}$ and $\hat{\rho}_{k+1}/\hat{\lambda}^2$ respectively, and use these to define $\hat{g}_k^*$ and hence $\hat{v}^k,v^{k+1}$ for each $k$ as above. The theoretical guarantees in Theorems~\ref{thm:AMPlowsym} and~\ref{thm:spectral} extend fairly straightforwardly to empirical Bayes-AMP; see~\citet[Lemma~G.1]{MV21}.
\end{remark}
\unparskip
\textbf{Sparse signal recovery}: To give another example where the AMP procedure~\eqref{eq:AMPmatsym} can be specialised appropriately, suppose that the exact distribution $\pi$ of $V$ is not known, but that for some fixed $c\in (0,1)$ and every $n\in\N$, the spike $v\equiv v(n)$ is known to have at most $sn$ non-zero entries. This implies that $\pi$ satisfies $\pi(\{0\})\geq 1-s$. In line with the classical theory on denoising sparse vectors~\citep[Section~9.3]{DJ94a,DJ98,Mon12}, we can take $(g_k)_{k\in\N_0}$ to be a sequence of soft-thresholding functions
\[
g_k(y)=\ST_{t_k}(y):=\sgn(y)(\abs{y}-t_k)_+,
\]
so that the AMP algorithm~\eqref{eq:AMPmatsym} becomes
\begin{align}
\label{eq:AMPsofthr}
\hat{v}^k=\ST_{t_k}(v^k),\qquad b_k=\frac{1}{n}\sum_{i=1}^n\Ind_{\{\hat{v}_i^k\neq 0\}}\equiv\frac{\norm{\hat{v}^k}_0}{n},\qquad v^{k+1}=A\hat{v}^k-b_k\hat{v}^{k-1}\qquad\text{for }k\in\N_0.
\end{align}
When each of the thresholds $t_k\in (0,\infty)$ is suitably chosen in terms of $\lambda$ and the state evolution parameter $\sigma_k$ (or consistent estimators thereof),~\citet{MV21} establish lower bounds on the effective signal-to-noise ratios $\rho_k=(\mu_k/\sigma_k)^2$ that hold uniformly over the class of distributions $\pi$ with $\pi(\{0\})\geq 1-\delta$. In conjunction with Corollary~\ref{cor:AMPlowsym}, this analysis leads to a theoretical guarantee on the performance of~\eqref{eq:AMPsofthr} for any sequence of $n\delta$-sparse spikes $v\equiv v(n)$ satisfying~\ref{ass:S1}; see Proposition~2.1 in the aforementioned paper.

We also mention that~\citet{BMR20} recently established statistical and computational limits for sparse signal recovery in an asymptotic regime where the expected number of non-zero entries of $v\equiv v(n)$ is a sublinear function of $n$. Specifically, for each $n$, the entries of $v$ are drawn independently from a prior $\pi_n$ with $\pi_n(\{0\})\geq 1-s_n$ and $s_n\to 0$ as $n\to\infty$. In this setting, the analysis makes use of finite-sample versions of the AMP master theorems (see Remark~\ref{rem:finitesample}).

In summary, the state evolution characterisation of the AMP algorithm~\eqref{eq:AMPmatsym} allows us to choose the functions $g_k$ in a principled way, depending on the prior information available about the signal $v$.  A poor choice of $(g_k)$ will lead to low effective signal-to-noise ratios $\rho_k=(\mu_k/\sigma_k)^2$, but the asymptotic convergence results~\eqref{eq:AMPlowsym}--\eqref{eq:asymcorr} will continue to hold provided that the hypotheses of Theorem~\ref{thm:AMPlowsym} or~\ref{thm:spectral} are satisfied.

A key strength of the AMP framework is that it gives us the flexibility to choose non-linear functions $g_k$, such as the soft-thresholding functions above. Note that the MMSE denoising functions $g_k^*$ in~\eqref{eq:gkbayes} are non-linear except in special cases (such as when $V$ is Gaussian). 
Nevertheless, iterations with linear~$g_k$ can sometimes be useful as a theoretical device for obtaining distributional information about spectral estimators, as the following example shows; see also~\citet{MTM20} and~\citet{MM20}.

\textbf{Connection with the power method}: Suppose that we initialise~\eqref{eq:AMPmatsym} with $\hat{v}^0\equiv \hat{v}^0(n):=\mu_0 v+\xi$, where $\mu_0\neq 0$ and $\xi\equiv\xi(n)\sim N_n(0,I_n)$, and define
\begin{equation}
\label{eq:gklinear}
\beta_k:=\sqrt{1+\mu_k^2},\qquad g_k(x):=\frac{x}{\beta_k}\quad\;\;\text{for }\,x\in\R,\qquad\mu_{k+1}:=\frac{\lambda}{\sqrt{1+\mu_k^{-2}}}\quad\text{for }k\in\N_0.
\end{equation}
These functions $g_k$ are constructed in a such a way that the corresponding state evolution formula~\eqref{eq:statevolsymat} yields $\sigma_k^2=1$ for every $k$, and parameters $\mu_k$ that coincide exactly with those defined in~\eqref{eq:gklinear}. Observe now that the AMP iteration~\eqref{eq:AMPmatsym} corresponding to~\eqref{eq:gklinear} yields $\bigl(\hat{v}^k\equiv\hat{v}^k(n):k\in\N\bigr)$ satisfying $\beta_0\hat{v}^1=v^1=A\hat{v}^0$ and
\begin{equation}
\label{eq:AMPpower}
\rbr{\beta_k+\frac{1}{\beta_{k-1}}}\hat{v}^{k+1}-\frac{1}{\beta_{k-1}}(\hat{v}^{k+1}-\hat{v}^{k-1})=A\hat{v}^k\quad\text{for }k\in\N.
\end{equation}
The key steps in the theoretical analysis of~\eqref{eq:AMPpower} can be summarised as follows:

\unparskip
\begin{enumerate}[label=(\Roman*)]
\item When $\lambda>1$, some elementary analysis (e.g.\ based on the contraction mapping theorem) shows that $\sqrt{\lambda^2-1}$ is a stable fixed point of the deterministic recursion for $(\mu_k)$ in~\eqref{eq:gklinear}, and hence that $\beta_k\to\lambda$ as $k\to\infty$.
\item Using Theorem~\ref{thm:AMPlowsym} and the covariance matrix defined in~\eqref{eq:Sigmabarmat}, we can obtain the $d_2$ limit of the joint empirical distribution of the components of $\hat{v}^{k+1}\equiv\hat{v}^{k+1}(n)$ and $\hat{v}^{k-1}\equiv \hat{v}^{k-1}(n)$ as $n\to\infty$; in particular, $\norm{\hat{v}^{k+1}}_n\cvc 1$. It then follows from (I) and routine arguments that $\lim_{k\to\infty}\clim_{n\to\infty}\norm{\hat{v}^{k+1}-\hat{v}^{k-1}}_n=0$. In other words, $\norm{\hat{v}^{k+1}-\hat{v}^{k-1}}_n$ converges completely to some deterministic limit $\ell_k$ as $n\to\infty$ for each fixed $k$, and $\ell_k\to 0$ as $k\to\infty$.
\item Thus, writing~\eqref{eq:AMPpower} in the form $A\hat{v}^k=(\lambda+\lambda^{-1})\hat{v}^{k+1}+\vartheta_k$ for $k,n\in\N$, where \[\vartheta_k\equiv\vartheta_k(n):=\biggl\{\biggl(\beta_k+\frac{1}{\beta_{k-1}}\biggr)-\biggl(\lambda+\frac{1}{\lambda}\biggr)\biggr\}\,\hat{v}^{k+1}-\frac{1}{\beta_{k-1}}(\hat{v}^{k+1}-\hat{v}^{k-1}),\]
we deduce from (I) and (II) that $\lim_{k\to\infty}\clim_{n\to\infty}\norm{\vartheta_k}_n=0$.
\end{enumerate}

\unparskip
Using these ingredients and the fact that the limiting spectral gap of $A\equiv A(n)$ is strictly positive when $\lambda>1$, it can be established that \[\lim_{k\to\infty}\clim_{n\to\infty}\,\frac{\abs{\ipr{\hat{v}^k}{\hat{\varphi}}_n}}{\norm{\hat{v}^k}_n}=1.\]
This shows that the specific instance~\eqref{eq:AMPpower} of the AMP iteration is asymptotically equivalent to the well-known power method for approximating $\hat{\varphi}$, although the dependence of $\hat{v}^0$ on the unknown $v$ means that we cannot use~\eqref{eq:AMPpower} as an algorithm in practice. Nevertheless, this asymptotic equivalence ensures that we can apply Theorem~\ref{thm:AMPlowsym} to obtain the $d_2$ convergence result in Proposition~\ref{prop:power} for the joint empirical distribution of the components of $\hat{\varphi}$ and the signal $v$.

\hfparskip
\subsection{Confidence intervals and \texorpdfstring{$p$}{p}-values}
\label{sec:lowrankconf}
As a consequence of Theorem~\ref{thm:AMPlowsym}, recall from the discussion after Corollary~\ref{cor:AMPlowsym} that for fixed $k$ and large $n$, the AMP iterates (i.e.\ effective observations) $v^k\equiv v^k(n)$ in the generic procedure~\eqref{eq:AMPmatsym} have the property that
$\{(v_i^k-\mu_k v_i)/\sigma_k:1\leq i\leq n\}$ behaves approximately like an i.i.d.\ sample of size $n$ from the $N(0,1)$ distribution. Thus, for a given $\alpha\in [0,1]$, we would expect roughly $n(1-\alpha)$ of these components to have absolute value at most $z_{\alpha/2}:=\Phi^{-1}(1-\alpha/2)$, where $\Phi^{-1}$ denotes the quantile function of the $N(0,1)$ distribution. Using this observation, we will now outline briefly how to construct confidence intervals for the entries of $v\equiv v(n)$, as well as associated $p$-values. By Remark~\ref{rem:musigmaest}, the (possibly unknown) state evolution parameters $\mu_k,\sigma_k$ can be estimated consistently by $\hat{\mu}_k\equiv\hat{\mu}_k(n)=\bigl(\norm{v^k}_n^2-\norm{v^{k-1}}_n^2\bigr)^{1/2}$ and $\hat{\sigma}_k\equiv\hat{\sigma}_k(n)=\norm{v^{k-1}}_n$ respectively for each $k\in\N$, so we define
\begin{equation}
\label{eq:lowrankconf}
\hat{J}_i^k(n,\alpha):=\biggl[\frac{v_i^k-z_{\alpha/2}\,\hat{\sigma}_k}{\hat{\mu}_k},\,\frac{v_i^k+z_{\alpha/2}\,\hat{\sigma}_k}{\hat{\mu}_k}\biggr]\quad\text{and}\quad p_i^k\equiv p_i^k(n)=2\,\biggl\{1-\Phi\biggl(\frac{|v_i^k|}{\hat{\sigma}_k}\biggr)\biggr\}
\end{equation}
for $k,n\in\N$, $1\leq i\leq n$ and $\alpha\in [0,1]$. \citet[Corollary~3.1]{MV21} showed that for fixed $k\in\N$ and $\alpha\in [0,1]$, the confidence intervals $\hat{J}_1^k(n,\alpha),\dotsc,\hat{J}_n^k(n,\alpha)$ have asymptotic mean coverage level $1-\alpha$; specifically,
\[\clim_{n\to\infty}\frac{1}{n}\sum_{i=1}^n\Ind_{\{v_i(n)\in\hat{J}_i^k(n,\alpha)\}}=1-\alpha=\lim_{n\to\infty}\frac{1}{n}\sum_{i=1}^n\Pr\bigl(v_i(n)\in\hat{J}_i^k(n,\alpha)\bigr).\]
The first limit above can be established by considering Lipschitz approximations to indicator functions of intervals and appealing to either Theorem~\ref{thm:AMPlowsym} or~\ref{thm:spectral} (for non-spectral and spectral initialisations respectively). The dominated convergence theorem can then be applied to deduce the second equality from the first. Note that for fixed $k,\alpha$, the asymptotic width of each $\hat{J}_i^k(n,\alpha)$ is $2z_{\alpha/2}/\rho_k$, which is minimised when the empirical Bayes-AMP iterates are used to construct these intervals.

In addition, suppose that the proportion of non-zero entries in the spike $v\equiv v(n)$ tends to $\delta\in (0,1)$ as $n\to\infty$. Then the result cited above asserts that the $p$-values defined in~\eqref{eq:lowrankconf} are asymptotically valid for the nulls $\mathcal{N}_n:=\{1\leq i\leq n:v_i\equiv v_i(n)=0\}$ in the following sense: for any sequence of indices $\bigl(i_0(n)\in\mathcal{N}_n:n\in\N\bigr)$ and all fixed $k\in\N$ and $\alpha\in [0,1]$, we have $\lim_{n\to\infty}\Pr(p_{i_0(n)}^k\leq\alpha)=\alpha$. 

\hfparskip
\subsection{AMP for more general low-rank matrix estimation problems}
\label{sec:lowrankgeneral}

\textbf{Estimation of a rectangular rank-one matrix}: Let $A\in\R^{n \times p}$ be an observation matrix given by
\begin{equation}
\label{eq:rectspike}
A\equiv A(n)=\frac{\lambda}{n}uv^\top+W',
\end{equation} 
where $W'$ is a Gaussian noise matrix with $W_{ij}'\iid N(0,1/n)$ for $1\leq i\leq n$ and $1\leq j\leq p$, and seek to estimate one or both of the unknown vectors $u\in\R^n$ and $v\in\R^p$.	

An important example of this observation scheme is a spiked covariance model \citep{JohnstoneICM,johnstoneLu09consistency} where $a_1,\dotsc,a_n\iid N_p(0,\Sigma)$  with $\Sigma:=(\lambda vv^\top+I_p)/n\in\R^{p\times p}$. In this case, the matrix $A\in\R^{n\times p}$ with rows $a_1,\dotsc,a_n$ is of the form~\eqref{eq:rectspike} with $u\sim N_n(0,I_n)$.

By analogy with the symmetric case in Section~\ref{sec:rankone}, an AMP algorithm for the model~\eqref{eq:rectspike} can be obtained by replacing the Gaussian matrix $W$ in the abstract asymmetric AMP iteration~\eqref{eq:AMPnonsym} with the data matrix $A$~\citep{rangan2012iterative,Deshpande2014}. For $k\in\N_0$ and generic sequences of Lipschitz functions $(f_k)_{k=0}^\infty$ and $(g_k)_{k=0}^\infty$ satisfying~\ref{ass:S2}, the corresponding AMP procedure takes the form
\begin{equation}
\label{eq:AMPmatrect} 
\begin{alignedat}{2}
u^k&:=Af_k(v^k)-b_k g_{k-1}(u^{k-1}),\qquad&c_k&:=n^{-1}\textstyle\sum_{i=1}^n g_k'(u_i^k),\\
v^{k+1}&:=A^\top g_k(u^k)-c_k f_k(v^k),\qquad&b_{k+1}&:=n^{-1}\textstyle\sum_{i=1}^p f_{k+1}'(v_i^{k+1})\\
\end{alignedat}
\end{equation}
for $k\in\N_0$. Based on an appropriate state evolution recursion, analogues of Theorems~\ref{thm:AMPlowsym} and~\ref{thm:spectral} can be formulated for~\eqref{eq:AMPmatrect} with non-spectral and spectral initialisations respectively. These results apply to an asymptotic regime where $n,p\to\infty$ with $n/p\to\delta$ for some $\delta\in (0,1)$, and where a version of~\ref{ass:S1} holds (with $\norm{u}_n,\norm{v}_p\cvc 1$ and the empirical distributions of the components of $u$ and $v$ converging completely in $d_2$ to suitable limits). A suitable spectral initialiser for~\eqref{eq:AMPmatrect} is $v^0=\hat{\varphi}^R$, a principal right singular vector of $A$ with $\norm{\hat{\varphi}^R}_p=1$~\citep[Section~4]{MV21}. The associated spectral threshold is at $\sqrt{\delta}$: if $n/p\to\delta$ and $\lambda>1/\sqrt{\delta}$, then the limiting empirical correlation 
$\abs{\ipr{\hat{\varphi}^R}{v}_n}/\norm{v}_n$ is strictly positive~\citep{paul2007asymptotics,BaiSilverstein}.

\textbf{Estimation of rank-$s$ matrices for $s>1$}: The general rank-$s$ spiked models take the form
\begin{equation}
\label{eq:genspike}
A=\sum_{j=1}^s\frac{\lambda_j}{n}v_j v_j^\top+W\quad\text{(symmetric)};\qquad 
A=\sum_{j=1}^s\frac{\lambda_j}{n}u_j v_j^\top+W'\quad\text{(asymmetric)},
\end{equation}
where $\lambda_1\geq\cdots\geq\lambda_s$ and the noise matrices $W,W'$ are as in~\eqref{eq:AMPmatsym} and~\eqref{eq:AMPmatrect} respectively.~\citet{PSC14a,PSC14b},~\citet{kabashima2016phase},~\citet{lesieur2017constrained} and~\citet{MV21} proposed generalisations of the AMP algorithms~\eqref{eq:AMPmatsym} and~\eqref{eq:AMPmatrect} for estimating $u_1,\dotsc,u_s,v_1,\dotsc,v_s$ and hence the signal matrices in~\eqref{eq:genspike}. For $s>1$, the main difference with the rank-one case is that the iterates in these procedures are matrices rather than vectors. When the initialiser is a matrix consisting of eigenvectors corresponding to the $s$ largest eigenvalues of~$A$, a rigorous state evolution result was obtained by~\citet[Section~6]{MV21}.
Additional complications arise in degenerate cases where $\lambda_1,\dotsc,\lambda_s$ are not all distinct.

\textbf{Universality}: As mentioned in the Introduction, the theoretical framework for AMP was originally built around Gaussian random matrices, but the conclusions of Theorems~\ref{thm:AMPlowsym} and~\ref{thm:spectral} (as well as the master theorems in Section~\ref{Sec:Master}) have now been extended to encompass more general random matrix ensembles. In so-called `spiked Wigner' models of the form~\eqref{eq:symspike}, the symmetric noise matrices $W\equiv W(n)$ have independent upper-triangular entries $(W_{ij}:1\leq i\leq j\leq n)$ that are uniformly subexponential across $n\in\N$ with $\E(W_{ij})=0$ and $\Var(W_{ij})=(1+\delta_{ij})/n$.
It was previously known that the eigenstructure of the corresponding observation matrix $A$ undergoes the BBP phase transition described in Section~\ref{sec:rankone} at the same spectral threshold $\lambda=1$ as for `spiked GOE' matrices; see for instance~\citet{AGZ10},~\citet{KY13} and~\citet{PWBM18}. Recently,~\citet[Examples~2.1 and 2.2]{CL20} used the method of Slepian interpolation to prove that in AMP algorithms of the form~\eqref{eq:AMPmatsym} based on matrices $A$ from rank-one spiked Wigner models, the iterates have the same asymptotics as in the original Gaussian setting, with or without spectral initialisation.

In a different direction,~\citet[Section~3]{Fan20} developed a more general class of AMP procedures for symmetric and rectangular rank-one spiked models~(\ref{eq:symspike},\,\ref{eq:rectspike}) in which the noise matrices are orthogonally invariant. In the symmetric case, this means that $W\equiv W(n)$ satisfies $W\eqd Q^\top WQ$ for all deterministic orthogonal $Q\in\R^{n\times n}$, and it can be shown that the only such $W$ with independent, mean-zero upper-triangular entries are scalar multiples of GOE matrices~\citep[e.g.][]{Meh04}. For other orthogonally invariant $W$,~\citet{Fan20} showed that
the original symmetric AMP algorithm~\eqref{eq:AMPmatsym} can be modified to allow each $v^{k+1}$ to depend on all of the previous iterates via
\begin{equation}
\label{eq:AMPFan}
\hat{v}^k=g_k(v^k)\quad\text{and}\quad v^{k+1}=A\hat{v}^k-\sum_{j=1}^k b_{kj}\hat{v}^{j-1}\quad\text{for }k\in\N_0,
\end{equation}
in such a way that the joint empirical distributions have well-defined Wasserstein limits. To achieve this, the technical crux is to design suitable Onsager coefficients $b_{k1},\dotsc,b_{kk}$ that depend on the limiting spectral distribution of $W$ (when it exists) through its moments and free cumulants, which also appear in the resulting state evolution recursion.
Asymptotic convergence results similar in spirit to Theorems~\ref{thm:AMPlowsym} and~\ref{thm:bayesamp} can then be established for iterations of the form~\eqref{eq:AMPFan} and their Bayes-AMP versions. As in Section~\ref{sec:lowrankgk}, it turns out that for large $k$, these Bayes-AMP estimates $\hat{v}^k$ of the spike $v$ can substantially improve on the spectral estimator (namely a leading eigenvector of the observation matrix $A$) in terms of asymptotic mean squared error.

\section{GAMP for generalised linear models}
\label{Sec:GAMP}
In this section, we give a unified treatment of a class of AMP algorithms for models of the following generic form: suppose that we generate a design matrix $X\in\R^{n\times p}$ with rows $x_1,\dotsc,x_n\in\R^p$, and observe $y\equiv(y_1,\dotsc,y_n)\in\R^n$ satisfying
\begin{equation}
\label{eq:glm1}
y_i=h(x_i^\top\beta,\varepsilon_i)\quad\text{for }i=1,\dotsc,n,
\end{equation}
where $\beta\equiv(\beta_1,\dotsc,\beta_p)$ is the target of inference, $\varepsilon\equiv(\varepsilon_1,\dotsc,\varepsilon_n)$ is a vector of noise variables and $h\colon\R^2\to\R$ is a known function. We will focus on the random design setting where $x_1,\dotsc,x_n\iid N_p(0,I_p/n)$, which is a common assumption in high-dimensional statistics and compressed sensing. Frequently, $\varepsilon_1,\dotsc,\varepsilon_n$ are assumed to be independent of each other and of $X$, in which case~\eqref{eq:glm1} becomes
\begin{equation}
\label{eq:glm2}
y_i\,|\,x_i\sim Q_i(\cdot\,|\,x_i^\top\beta),
\end{equation}
where $Q_i(\cdot\,|\,z)$ denotes the distribution of $h(z,\varepsilon_i)$ for a fixed $z\in\R$ and $1\leq i\leq n$. 
In statistics,~\eqref{eq:glm2} is traditionally referred to as a \emph{generalised linear model} (GLM) for $(x_1,y_1),\dotsc,(x_n,y_n)$ if the conditional distributions of $y_i$ given $x_i$ have densities of \emph{exponential dispersion family form}~\citep{PS97}
\begin{equation}
\label{eq:glm3}
u\mapsto a(\sigma_i^2,u)\exp\biggl\{\frac{u\Theta(\mu_i)-K(\Theta(\mu_i))}{\sigma_i^2}\biggr\}
\end{equation}
with respect to either Lebesgue measure on $\R$ or counting measure on $\Q$. In~\eqref{eq:glm3}, the mean parameter $\mu_i\in\mathcal{M}\subseteq\R$ is related to $x_i$ via $\mu_i=\eta^{-1}(x_i^\top\beta)$ for some strictly increasing, twice differentiable \emph{link function} $\eta$, and $\sigma_i\in\mathcal{D}\subseteq (0,\infty)$ is the \emph{dispersion parameter}, while $a,K,\Theta$ are fixed functions with $K''>0$ on $\R$ and $\Theta=(K')^{-1}$. The GLM framework encompasses a broad class of parametric models, including the standard linear model, phase retrieval (where $y_i=(x_i^\top\beta)^2+\varepsilon_i$ for $1\leq i\leq n$), and logistic, binomial and Poisson regression~\citep[e.g.][]{MN89,Agr15}. Sometimes, `GLM' is used as an umbrella term to describe more general models of the form~(\ref{eq:glm1},\,\ref{eq:glm2}).

Likelihood-based inference for $\beta$ in~(\ref{eq:glm2},\,\ref{eq:glm3}) is justified by classical asymptotic theory when $p$ is fixed and $n\to\infty$, or when $p$ grows sufficiently slowly with $n$~\citep{Por84,Por85,Por88}. 
However, in modern high-dimensional regimes where $n,p\to\infty$ and the aspect ratio $n/p$ of the design matrix $X$ is bounded, different tools are needed to construct and analyse estimators of $\beta$, and it is in this context that we introduce the GAMP paradigm below.

\hfparskip
\subsection{Master theorem for GAMP}
\label{sec:GAMPmaster}
The \emph{generalised AMP \emph{(GAMP)} algorithm} proposed by~\citet{RanganGAMP} iteratively produces estimates $\hat{\beta}^k,\theta^k$ of $\beta\in\R^p$ and $\theta:=X\beta\in\R^n$ respectively in~\eqref{eq:glm1}, via update steps of the following form: given $\hat{r}^{-1}:=0\in\R^n$, $b_0\in\R$ and an initialiser $\hat{\beta}^0\in\R^p$, recursively define
\begin{equation}
\label{eq:GAMP}
\begin{alignedat}{3}
\theta^k&:=X\hat{\beta}^k-b_k\hat{r}^{k-1},\qquad&\hat{r}^k&:=g_k(\theta^k,y),\qquad&c_k&:=n^{-1}\textstyle\sum_{i=1}^n g_k'(\theta_i^k,y_i),\\
\beta^{k+1}&:=X^\top\hat{r}^k-c_k\hat{\beta}^k,\qquad&\hat{\beta}^{k+1}&:=f_{k+1}(\beta^{k+1}),\qquad&b_{k+1}&:=n^{-1}\textstyle\sum_{j=1}^p f_{k+1}'(\beta_j^{k+1}),
\end{alignedat}
\end{equation}
for $k\in\N_0$. Here, $g_k\colon\R^2\to\R$ and $f_{k+1}\colon\R\to\R$ are Lipschitz in their first argument, and $g_k'\colon\R^2\to\R$, $f_{k+1}'\colon\R\to\R$ agrees with the partial derivatives of $g_k,f_{k+1}$ respectively with respect to their first arguments, wherever the latter are defined. As in previous sections, these functions are understood to act componentwise on their vector arguments in~\eqref{eq:GAMP}. The goal of Section~\ref{Sec:GAMP} is to develop the theory and applications of GAMP, whose statistical utility can be summarised in the following key points:

\unparskip
\begin{enumerate}[label=(\roman*)]
\item \emph{Exact asymptotic characterisation via state evolution}: The Onsager correction terms $-b_k\hat{r}^{k-1}$, $-c_k\hat{\beta}^k$ are designed to ensure that in a high-dimensional limiting regime where $n,p\to\infty$ with $n/p\to\delta\in (0,\infty)$, the empirical distributions of the entries of the iterates in~\eqref{eq:GAMP} converge to well-defined Wasserstein limits. These asymptotic distributions are characterised by the state evolution recursion~\eqref{eq:GAMPstatevol1}--\eqref{eq:GAMPstatevol2} below. Consequently, for each fixed $k\in\N_0$, the entries of $\hat{\beta}^{k+1}\in\R^p$ have approximately the same empirical distribution as those of $f_{k+1}(\mu_k\beta+\sigma_k\xi)$ when $p$ is large; here, $\beta\in\R^p$ is the unknown signal, $\xi\sim N_p(0,I_p)$ is an independent noise vector, $\mu_k,\sigma_k$ are the effective signal strength and noise level respectively, and $f_{k+1}$ can be viewed as a denoising function. This result facilitates a targeted approach to inference for structured signals $\beta$, whereby informed choices of $(f_k,g_k:k\in\N_0)$ can be made to accommodate different types of prior information (Section~\ref{sec:GAMPbeta}). 
\item \emph{Link to convex optimisation problems}: For suitable choices of $f_k,g_k$, the GAMP recursion~\eqref{eq:GAMP} can be interpreted as an alternating minimisation procedure for solving a convex optimisation problem of the form~\eqref{eq:constropt}, and the fixed points of this iteration are minimisers of the convex objective function (Proposition~\ref{prop:GAMPopt} in Section~\ref{sec:GAMPopt}). Together with the state evolution description of~\eqref{eq:GAMP}, this forms the basis of a \emph{systematic} approach to deriving exact performance guarantees for the Lasso and other (penalised or unpenalised) M-estimators in high-dimensional GLMs (Sections~\ref{sec:lasso}--\ref{sec:logistic}).
\end{enumerate}

\unparskip
In this subsection, we address point (i) above and formally state a `master theorem' for GAMP (Theorem~\ref{thm:GAMPmaster}). Consider a sequence of recursions~\eqref{eq:GAMP} indexed by $n\in\N$ and $p\equiv p_n$, where $n/p\to\delta\in (0,\infty)$ as $n\to\infty$, and assume that

\unparskip
\begin{enumerate}[label=(G0)]
\item \label{ass:G0} For each $n$, the design matrix $X\equiv X(n)\in\R^{n\times p}$ has i.i.d.\ $N(0,1/n)$ entries and is independent of $\bigl(\hat{\beta}^0(n),\beta(n),\varepsilon(n)\bigr)\in\R^p\times\R^p\times\R^n$.
\end{enumerate}

\unparskip
At first sight, it would appear that the GAMP algorithm~\eqref{eq:GAMP} is an instance of the abstract asymmetric AMP recursion~\eqref{eq:AMPnonsym}, but in models~\eqref{eq:glm1} where~\ref{ass:G0} holds, the crucial difference in the probabilistic structure is that the observation vector $y\equiv y(n)\in\R^n$ is in general not independent of $X\equiv X(n)$. This means that condition~\ref{ass:B0} does not hold with $\gamma=y$, so the original master theorem for asymmetric AMP (Theorem~\ref{thm:AMPnonsym}) cannot be directly applied in this setting, and in fact does not give the correct limiting distributions for~\eqref{eq:GAMP}. 

Instead, Theorem~\ref{thm:GAMPmaster} below is derived from a general state evolution result for \emph{matrix-valued} AMP iterations (Section~\ref{sec:AMPmatrix}), under suitable analogues of~\ref{ass:B1}--\ref{ass:B5} on the inputs to the GAMP recursions~\eqref{eq:GAMP} as $n\to\infty$ and $n/p\to\delta$: for some $r\in [2,\infty)$, suppose that

\unparskip
\begin{enumerate}[label=(G\arabic*)]
\item \label{ass:G1} There exist random variables $\bar{\beta}\sim \pi_{\bar{\beta}}$ and $\bar{\varepsilon}\sim P_{\bar{\varepsilon}}$ with $\E(\bar{\beta}^2)>0$ and $\E(\abs{\bar{\beta}}^r),\E(\abs{\bar{\varepsilon}}^r)<\infty$, such that writing $\nu_p(\beta)$ and $\nu_n(\varepsilon)$ for the empirical distributions of the components of $\beta\equiv\beta(p)$ and $\varepsilon\equiv\varepsilon(n)$ respectively, we have $d_r\bigl(\nu_p(\beta),\pi_{\bar{\beta}}\bigr)\cvc 0$ and $d_r\bigl(\nu_n(\varepsilon),P_{\bar{\varepsilon}}\bigr)\cvc 0$.
\item \label{ass:G2} $\norm{\hat{\beta}^0}_{p,r}=O_c(1)$ and there exists a non-negative definite $\Sigma_0\in\R^{2\times 2}$ such that $\breve{\beta}^0:=(\beta\;\:\hat{\beta}^0)\in\R^{p\times 2}$ satisfies
\[\frac{1}{n}(\breve{\beta}^0)^\top\breve{\beta}^0=\frac{1}{n}\begin{pmatrix}\beta^\top\beta&\beta^\top\hat{\beta}^0\\(\hat{\beta}^0)^\top\beta&(\hat{\beta}^0)^\top\hat{\beta}^0\end{pmatrix}
\cvc\Sigma_0.\]
\item \label{ass:G3} There exists a Lipschitz $F_0\colon\R\to\R$ such that $\ipr{\hat{\beta}^0}{\phi(\beta)}_p\cvc\E\big(F_0(\bar{\beta})\phi(\bar{\beta})\bigr)$ and $\E\bigl(F_0(\bar{\beta})^2\bigr)\leq(\Sigma_0)_{22}$ for all Lipschitz $\phi\colon\R\to\R$.
\item \label{ass:G4} For each $k\in\N_0$, the function $f_{k+1}$ is non-constant on $\R$, and $\tilde{g}_k\colon (z,u,v)\mapsto g_k(u,h(z,v))$ is Lipschitz on $\R^3$ with $P_{\bar{\varepsilon}}\bigl(\{v:(z,u)\mapsto\tilde{g}_k(z,u,v)\text{ is non-constant}\}\bigr)>0$.
\end{enumerate}

\unparskip
We remark here that while~\ref{ass:G2} is in general a stronger requirement than~\ref{ass:B2}, both~\ref{ass:G2} and~\ref{ass:G3} are implied by~\ref{ass:G1} if for some fixed $c\in\R$ we have $\hat{\beta}^0\equiv\hat{\beta}^0(n)=c\mathbf{1}_p$ for all $n$.
As in Section~\ref{Sec:LowRank}, constraints on $\beta\equiv\beta(n)$ such as sparsity or entrywise non-negativity will be reflected in the form of the `limiting prior distribution' $\pi_{\bar{\beta}}$. Note that
\begin{equation}
\label{eq:kappa}
(\Sigma_0)_{11}=\clim_{n\to\infty}\,\biggl(\frac{p}{n}\cdot\frac{\norm{\beta}^2}{p}\biggr)=\frac{\E(\bar{\beta}^2)}{\delta}>0
\end{equation}
by~\ref{ass:G1} and~\ref{ass:G2}. Also, the condition on $\varepsilon\equiv\varepsilon(n)$ in~\ref{ass:G1} is satisfied if $\varepsilon_1,\dotsc,\varepsilon_n\iid P_{\bar{\varepsilon}}$ for each $n$. 

\textbf{State evolution}: With $\Sigma_0$ as in~\ref{ass:G2}, the state evolution parameters $\bigl(\mu_k\in\R,\,\sigma_k\in [0,\infty),\,\Sigma_k\in\R^{2\times 2}:k\in\N\bigr)$ are recursively defined by
\begin{equation}
\label{eq:GAMPstatevol1}
\mu_{k+1}:=\E\bigl(\partial_z\tilde{g}_k(Z,Z_k,\bar{\varepsilon})\bigr),\qquad\quad
\sigma_{k+1}^2:=\E\bigl(\tilde{g}_k(Z,Z_k,\bar{\varepsilon})^2\bigr)=\E\bigl(g_k(Z_k,Y)^2\bigr),
\end{equation}
\begin{equation}
\label{eq:GAMPstatevol2}
\Sigma_{k+1}:=\frac{1}{\delta}\begin{pmatrix}\E(\bar{\beta}^2)&\E\{\bar{\beta}f_{k+1}(\mu_{k+1}\bar{\beta}+\sigma_{k+1} G_{k+1})\}\\\E\{\bar{\beta}f_{k+1}(\mu_{k+1}\bar{\beta}+\sigma_{k+1} G_{k+1})\}&\E\{f_{k+1}(\mu_{k+1}\bar{\beta} + \sigma_{k+1}G_{k+1})^2\}\end{pmatrix}
\end{equation}
for $k\in\N_0$, where we take $(Z,Z_k)\sim N_2(0,\Sigma_k)$ to be independent of $\bar{\varepsilon}\sim P_{\bar{\varepsilon}}$, define $Y:=h(Z,\bar{\varepsilon})$, and take $G_{k+1}\sim N(0,1)$ to be independent of $\bar{\beta}\sim \pi_{\bar{\beta}}$. Under~\ref{ass:G4}, it can be shown as in Lemma~\ref{lem:posdef} that if $\sigma_1>0$, then  $\sigma_k>0$ and $\Sigma_k$ is positive definite for all $k\in\N$. In~\eqref{eq:GAMPstatevol1}, $\partial_z\tilde{g}_k$ denotes the partial derivative of $\tilde{g}_k$ with respect to its first argument; observe that by~\ref{ass:G4}, $z\mapsto\tilde{g}_k(z,u,v)$ is Lipschitz and hence differentiable almost everywhere for all $(u,v)\in\R^2$, so $\mu_{k+1}$ is well-defined. 

Stein's lemma (Lemma~\ref{lem:steinsigma}) can be used to derive some alternative expressions for $\mu_{k+1}$ that will be useful later on; see~\citet[Proposition~3.1]{MM20} or Section~\ref{sec:GAMPproofs} for the proof of the following lemma.
\begin{lemma}
\label{lem:GAMPstein}
For each $k\in\N$, letting $\tilde{G}_k\sim N(0,1)$ be independent of $(Z,\bar{\varepsilon})$, we have $(Z,Z_k,\bar{\varepsilon})\eqd (Z,\mu_{Z,k}Z+\sigma_{Z,k}\tilde{G}_k,\bar{\varepsilon})$, where 
\begin{equation}
\label{eq:zkmusigma}
\begin{split}
\mu_{Z,k}&:=\frac{\E\bigl(\bar{\beta}f_k(\mu_k\bar{\beta}+\sigma_k \tilde{G}_k)\bigr)}{\E(\bar{\beta}^2)}=\frac{\Sigma_{21}}{\Sigma_{11}},\\
\sigma_{Z,k}^2&:=\frac{\E(\bar{\beta}^2)\,\E\bigl(f_k(\mu_k\bar{\beta}+\sigma_k\tilde{G}_k)^2\bigr)-\E\bigl(\bar{\beta}f_k(\mu_k\bar{\beta}+\sigma_k\tilde{G}_k)\bigr)^2}{\delta\,\E(\bar{\beta}^2)}=\Sigma_{22}-\frac{\Sigma_{12}^2}{\Sigma_{11}},
\end{split}
\end{equation}
with $\Sigma\equiv\Sigma_k$. Thus, $\mu_{k+1}=\E\bigl(\partial_z\tilde{g}_k(Z,\mu_{Z,k}Z+\sigma_{Z,k}\tilde{G}_k,\bar{\varepsilon})\bigr)$ and $\sigma_{k+1}^2=\E\bigl(\tilde{g}_k(Z,\mu_{Z,k}Z+\sigma_{Z,k}\tilde{G}_k,\bar{\varepsilon})^2\bigr)$. Moreover,
\begin{equation}
\label{eq:GAMPstein}
\mu_{k+1}=\frac{\delta}{\E(\bar{\beta}^2)}\E\bigl(Zg_k(Z_k,Y)\bigr)-\mu_{Z,k}\,\E\bigl(g_k'(Z_k,Y)\bigr)=\E\biggl(\frac{\E(Z\,|\,Z_k,Y)-\E(Z\,|\,Z_k)}{\Var(Z\,|\,Z_k)}\,g_k(Z_k,Y)\biggr).
\end{equation}
\end{lemma}
\unparskip
Before stating the main result of this subsection, we make an further regularity assumption that is similar to~\ref{ass:B5}.

\unparskip
\begin{enumerate}[label=(G5)]
\item \label{ass:G5} For each $k\in\N_0$, writing $D_k\subseteq\R^2$ for the set of discontinuities of $g_k'$, we have $\Pr\bigl((Z_k,Y)\in D_k\bigr)=0$, and $f_{k+1}'$ is continuous Lebesgue almost everywhere.
\end{enumerate}

\unparskip
\begin{theorem}
\label{thm:GAMPmaster}
Suppose that~\emph{\ref{ass:G0}}--\emph{\ref{ass:G5}} hold for a sequence of GAMP recursions~\eqref{eq:GAMP} indexed by $n$ and $p\equiv p_n$, with $n/p\to\delta\in (0,\infty)$ and $\sigma_1>0$. Then for each $k\in\N_0$, we have
\begin{align}
\label{eq:GAMPbeta}
&\sup_{\psi\in\PL_2(r,1)}\;\biggl|\frac{1}{p}\sum_{j=1}^p\psi(\beta_j^{k+1},\beta_j)-\E\bigl(\psi(\mu_{k+1}\bar{\beta}+\sigma_{k+1}G_{k+1},\bar{\beta})\bigr)\biggr|\cvc 0,\\
\label{eq:GAMPtheta}
&\sup_{\psi\in\PL_3(r,1)}\;\biggl|\frac{1}{n}\sum_{i=1}^n\psi(\theta_i^k,\theta_i,\varepsilon_i)-\E\bigl(\psi(\mu_{Z,k}Z+\sigma_{Z,k}\tilde{G}_k,Z,\bar{\varepsilon})\bigr)\biggr|\cvc 0
\end{align}
as $n,p\to\infty$ with $n/p\to\delta$, where $\theta_i\equiv\theta_i(n)=x_i^\top\beta$ for $n\in\N$ and $1\leq i\leq n$. 
\end{theorem}

\unparskip
Writing $\nu_p(\beta^k,\beta)$ for the joint empirical distribution of the components of $\beta^k,\beta\in\R^p$, and $\breve{\nu}^k$ for the distribution of $(\mu_k\bar{\beta}+\sigma_k G_k,\bar{\beta})$, we can express the conclusion of~\eqref{eq:GAMPbeta} as
\[\widetilde{d}_r\bigl(\nu_p(\beta^k,\beta),\breve{\nu}^k\bigr)\cvc 0,\quad\text{or equivalently}\quad d_r\bigl(\nu_p(\beta^k,\beta),\breve{\nu}^k\bigr)\cvc 0\quad\text{ as }n\to\infty.\]
Likewise,~\eqref{eq:GAMPtheta} says that the joint empirical distribution $\nu_n(\theta^k,\theta,\varepsilon)$ converges completely in $d_r$ to the distribution of $(\mu_{Z,k}Z+\sigma_{Z,k}\tilde{G}_k,Z,\bar{\varepsilon})\eqd (Z_k,Z,\bar{\varepsilon})$. 

\textbf{Interpretation}: Informally, when $p$ is large, the components of $\beta^k$ have approximately the same empirical distribution as those of $\mu_k\beta+\sigma_k\xi$, where $\xi\sim N_p(0,I_p)$ is independent of $\beta\in\R^p$. By analogy with the limiting univariate problem of estimating $\bar{\beta}\sim\pi_{\bar{\beta}}$ based on a corrupted observation $\mu_k\bar{\beta}+\sigma_k G_k$, we can regard $\beta^k$ as an \emph{effective observation} and $\rho_k:=(\mu_k/\sigma_k)$ as an \emph{effective signal-to-noise ratio}; recall the discussion after Corollary~\ref{cor:AMPlowsym}. 
\begin{remark}
\label{rem:GAMPbk}
Similarly to Remark~\ref{rem:bk}, it turns out that in the setting of Theorem~\ref{thm:GAMPmaster}, condition~\ref{ass:G5} ensures that
\begin{equation}
\label{eq:GAMPbk}
\begin{split}
c_k=\frac{1}{n}\sum_{i=1}^n g_k'(\theta_i^k,y_i)&\cvc\E\bigl(g_k'(Z_k,Y)\bigr)=:\bar{c}_k,\\ b_{k+1}=\frac{1}{n}\sum_{j=1}^p f_{k+1}'(\beta_j^{k+1})&\cvc\frac{\E\bigl(f_{k+1}'(\mu_{k+1}\bar{\beta}+\sigma_{k+1}G_{k+1})\bigr)}{\delta}=:\bar{b}_{k+1}
\end{split}
\end{equation}
as $n,p\to\infty$ with $n/p\to\delta$, for each $k\in\N_0$. In fact, the theorem holds under~\ref{ass:G0}--\ref{ass:G4} if $b_k,c_k$ are replaced with $\bar{b}_k,\bar{c}_k$ respectively in~\eqref{eq:GAMP}, in which case~\ref{ass:G5} is not needed.
\end{remark}

\unparskip
By defining an augmented state evolution that specifies the covariance structure of the limiting Gaussians $G_1,G_2,\dotsc$ and $\tilde{G}_1,\tilde{G}_2,\dotsc$, we can establish the $d_r$ limits of the joint empirical distributions $\nu_p(\beta^1,\dotsc,\beta^k,\beta)$ and $\nu_n(\theta^0,\dotsc,\theta^k,\theta)$, similarly to~\eqref{eq:Sigmabarmat} and Theorem~\ref{thm:AMPlowsym}. For simplicity of presentation, we do not state this stronger conclusion. Its proof is identical in most respects to that of Theorem~\ref{thm:GAMPmaster}, which we now summarise.

\unparskip	
\begin{proof}[Proof (sketch) of Theorem~\ref{thm:GAMPmaster}]
As mentioned previously, the overall objective is to handle the dependence of $y$ on $X$ (through $\theta=X\beta$) in~\eqref{eq:GAMP}, and show that the `noise' components $\tilde{\beta}^k\equiv\tilde{\beta}^k(n):=\beta^k-\mu_k\beta$ of the effective observations is approximately Gaussian (and independent of $\beta$) for large $n$. To this end, consider rewriting the second update step as
\begin{equation}
\label{eq:GAMP1}
\tilde{\beta}^{k+1}\equiv\beta^{k+1}-\mu_{k+1}\beta=X^\top\tilde{g}_k(\theta,\theta^k,\varepsilon)-\Bigl(\beta\;\;\;f_k(\tilde{\beta}^k+\mu_k\beta)\Bigr)\begin{pmatrix}\mu_{k+1}\\\abr{D_2\tilde{g}_k(\theta,\theta^k,\varepsilon)}_n\end{pmatrix}.
\end{equation}
Here, $\tilde{g}_k(\theta,\theta^k,\varepsilon)=g_k(\theta^k,y)=\hat{r}^k$ (applying $\tilde{g}_k$ componentwise), $f_k(\tilde{\beta}^k+\mu_k\beta)=f_k(\beta^k)=\hat{\beta}^k$ and $D_2\tilde{g}_k(z,u,v):=g_k'(u,h(z,v))$ agrees with the partial derivative of $\tilde{g}_k$ with respect to its second argument, wherever the latter is defined. A useful feature of~\eqref{eq:GAMP1} is that unlike $y$, the noise vector $\varepsilon$ is independent of $X$ by~\ref{ass:G0}. Since both $\theta$ and $\theta^k$ depend on $X$, this suggests treating $\tilde{\theta}^k:=(\theta\;\:\theta^k)\in\R^{n\times 2}$ as a single entity, and rewriting the first update step in~\eqref{eq:GAMP} as
\begin{equation}
\label{eq:GAMP2}
\tilde{\theta}^k\equiv(\theta\;\:\theta^k)=X\,\Bigl(\beta\;\;\;f_k(\tilde{\beta}^k+\mu_k\beta)\Bigr)-\tilde{g}_{k-1}(\theta,\theta^{k-1},\varepsilon)\,\biggl(0\quad\frac{p}{n}\abr{f_k'(\tilde{\beta}^k+\mu_k\beta)}_n\biggr),
\end{equation}
where $\frac{p}{n}\abr{f_k'(\tilde{\beta}^k+\mu_k\beta)}_n=\frac{p}{n}\abr{f_k'(\beta^k)}_n=b_k$. In doing so, we have recast~\eqref{eq:GAMP} as a matrix-valued AMP iteration~\eqref{eq:GAMP1}--\eqref{eq:GAMP2} that is no longer a valid algorithm for practical purposes, but is more amenable to theoretical analysis. Indeed, its asymptotics can be derived by applying a master theorem for abstract recursions~\eqref{eq:AMPmatrix} of this type; see Section~\ref{sec:AMPmatrix}. The significance of the definition of $\mu_{k+1}=\E\bigl(\partial_z\tilde{g}_k(Z,Z_k,\bar{\varepsilon})\bigr)\equiv\E\bigl(D_1\tilde{g}_k(Z,Z_k,\bar{\varepsilon})\bigr)$ in~\eqref{eq:GAMPstatevol1} is that the final term in~\eqref{eq:GAMP1} is a non-linear correction based on the derivative (gradient) of $\tilde{g}_k$. The final term in~\eqref{eq:GAMP2} has a similar interpretation as a multivariate analogue of the original $b_k$ in~\eqref{eq:GAMP}, and together these ensure that the limiting empirical distributions of the iterates in~\eqref{eq:GAMP1}--\eqref{eq:GAMP2} are indeed Gaussian.
\end{proof}

\umparskip
\subsection{Choosing the functions \texorpdfstring{$f_k,g_k$}{f\_k,g\_k}, and inference for \texorpdfstring{$\beta$}{beta}}
\label{sec:GAMPbeta}
\textbf{Asymptotic estimation error}: Since the functions $f_k$ in~\eqref{eq:GAMP} are Lipschitz by assumption, it follows as in Corollary~\ref{cor:AMPlowsym} that in the setting of Theorem~\ref{thm:GAMPmaster} above, the asymptotic estimation error of $\hat{\beta}^k$ with respect to any loss function $\psi\in\PL_2(r)$ is given by
\begin{equation}
\label{eq:GAMPesterr}
\frac{1}{p}\sum_{j=1}^p\psi(\hat{\beta}_j^k,\beta_j)\cvc\E\bigl\{\psi\bigl(f_k(\mu_k\bar{\beta}+\sigma_k G_k),\bar{\beta}\bigr)\bigr\}
\end{equation}
for each $k\in\N$, as $n,p\to\infty$ with $n/p\to\delta$. In particular, taking $\psi(x,y)=\abs{x-y}^q$ for $q\in [1,r]$, we obtain the asymptotic normalised $\ell_q$ error $\clim_{p\to\infty}p^{-1}\norm{\hat{\beta}^k-\beta}_q=\E\bigl\{\bigl(f_k(\mu_k\bar{\beta}+\sigma_k G_k)-\bar{\beta}\bigr)^q\bigr\}^{1/q}$. 

\textbf{Bayes-GAMP}: If the limiting prior distribution $\pi_{\bar{\beta}}$, the limiting noise distribution $P_{\bar{\varepsilon}}$ 
and the initial $\Sigma_0\in\R^{2\times 2}$ are known, then guided by Lemma~\ref{lem:snrmax}, we can proceed as in Section~\ref{sec:lowrankgk} and choose $f_k,g_k$ in~\eqref{eq:GAMP} so as to maximise the effective signal-to-noise ratios $\rho_k=(\mu_k/\sigma_k)^2$ and $\rho_{Z,k}:=(\mu_{Z,k}/\sigma_{Z,k})^2$ for each $k$. 

Specifically, given the matrix $\Sigma\equiv\Sigma_k\in\R^{2\times 2}$ in~\eqref{eq:GAMPstatevol2} for some $k\in\N_0$, we can obtain $\mu_{Z,k},\sigma_{Z,k}$ from~\eqref{eq:zkmusigma}; conversely, given $\mu_{Z,k},\sigma_{Z,k}$, we can recover $\Sigma$ since $\Sigma_{11}=\delta^{-1}\E(\bar{\beta}^2)$ is known, and~\eqref{eq:zkmusigma} yields $\Sigma_{21}=\Sigma_{11}\mu_{Z,k}$ and $\Sigma_{22}=\sigma_{Z,k}^2+\Sigma_{11}\mu_{Z,k}^2$. Now take $(Z,Z_k)\sim N_2(0,\Sigma_k)$ to be independent of $\bar{\varepsilon}\sim P_{\bar{\varepsilon}}$, and let $Y=h(Z,\bar{\varepsilon})$, so that $Y$ and $Z_k$ are conditionally independent given $Z$. Based on the joint distribution of $(Z,Z_k,Y)$, let $g_k^*\colon\R^2\to\R$ be a measurable function satisfying
\begin{equation}
\label{eq:GAMPgk}
g_k^*(Z_k,Y)=\frac{\E(Z\,|\,Z_k,Y)-\E(Z\,|\,Z_k)}{\Var(Z\,|\,Z_k)},
\end{equation}
where $\E(Z\,|\,Z_k)=m_k Z_k$ with
\[m_k:=\frac{\Sigma_{21}}{\Sigma_{22}}
=\frac{\mu_{Z,k}}{\sigma_{Z,k}^2}\Var(Z\,|\,Z_k),\quad\;\;\Var(Z\,|\,Z_k)=\Sigma_{11}-\frac{\Sigma_{21}^2}{\Sigma_{22}}=\frac{\Sigma_{11}\sigma_{Z,k}^2}{\sigma_{Z,k}^2+\Sigma_{11}\mu_{Z,k}^2}=\biggl(\frac{\E(\bar{\beta}^2)}{\delta}+\rho_{Z,k}\biggr)^{-1}.\]
Then by~\eqref{eq:GAMPstatevol1},~\eqref{eq:GAMPstein} and the Cauchy--Schwarz inequality, we have
\[\rho_{k+1}=\frac{\mu_{k+1}^2}{\sigma_{k+1}^2}=\frac{\E\bigl(g_k^*(Z_k,Y)\,g_k(Z_k,Y)\bigr)^2}{\E\bigl(g_k(Z_k,Y)^2\bigr)}\leq\E\bigl(g_k^*(Z_k,Y)^2\bigr),\]
with equality when $g_k$ is a (non-zero) scalar multiple of $g_k^*$.

Now given $\mu_k,\sigma_k$ for some $k\in\N$, we wish to find $f_k\colon\R\to\R$ such that defining $\Sigma\equiv\Sigma_k$ as in~\eqref{eq:zkmusigma}, the quantity
\[\rho_{Z,k}=\frac{\mu_{Z,k}^2}{\sigma_{Z,k}^2}=\frac{\Sigma_{21}^2\Sigma_{11}^{-2}}{\Sigma_{22}-\Sigma_{21}^2\Sigma_{11}^{-1}}=\biggl(\frac{\Sigma_{22}}{\Sigma_{21}^2}\,\Sigma_{11}^2-\Sigma_{11}\biggr)^{-1}\]
is as large as possible. Since $\Sigma_{11}=\delta^{-1}\E(\bar{\beta}^2)$ is fixed, this amounts to maximising
\[\frac{\Sigma_{21}^2}{\Sigma_{22}}=\frac{\E\bigl(\bar{\beta}f_k(\mu_k\bar{\beta}+\sigma_k \tilde{G}_k)\bigr)^2}{\E\bigl(f_k(\mu_k\bar{\beta}+\sigma_k\tilde{G}_k)^2\bigr)}.\]
Again by the Cauchy--Schwarz inequality (see~\eqref{eq:snrCS} in Section~\ref{sec:lowrankgk}), this can be done by taking $f_k$ to be any (non-zero) scalar multiple of $f_k^*$ satisfying
\begin{equation}
\label{eq:GAMPfk}
f_k^*(\mu_k\bar{\beta}+\sigma_k\tilde{G}_k)=\E(\bar{\beta}\,|\,\mu_k\bar{\beta}+\sigma_k\tilde{G}_k),
\end{equation}
in which case $\Sigma_{21}=\Sigma_{22}<\Sigma_{11}$. An exact expression for $f_k^*$ is given by Tweedie's formula~\eqref{eq:tweedie}, and if $\bar{\beta}\sim \pi_{\bar{\beta}}$ satisfies the conditions of Lemma~\ref{lem:tweedie}, then $f_k^*$ is Lipschitz. As we saw in~\eqref{eq:ortho}, the choice $f_k=f_k^*$ also minimises the asymptotic mean squared error $\E\bigl\{\bigl(f_k(\mu_k\bar{\beta}+\sigma_k G_k)-\bar{\beta}\bigr)^2\bigr\}$, for given $\mu_k,\sigma_k$; in other words, $f_k^*$ is the Bayes optimal (i.e.\ MMSE) denoising function.

By recursively defining $g_k=g_k^*$ (or some scalar multiple thereof) and $f_{k+1}=f_{k+1}^*$ for $k\in\N$ using~\eqref{eq:GAMPgk} and~\eqref{eq:GAMPfk}, together with corresponding sequences $(\mu_k^*,\sigma_k^*,\mu_{Z,k}^*,\sigma_{Z,k}^*:k\in\N)$ of state evolution parameters through~\eqref{eq:GAMPstatevol1}--\eqref{eq:GAMPstatevol2}, we obtain a \emph{Bayes-GAMP algorithm} of the form~\eqref{eq:GAMP}. A version of this was originally derived by~\citet[Section~IV-B]{RanganGAMP} as an approximation to a sum-product loopy belief propagation algorithm. The limiting empirical distributions for the Bayes-GAMP iterates can be obtained from Theorem~\ref{thm:GAMPmaster}, provided that the functions $f_k^*$ and $\tilde{g}_k^*\colon (z,u,v)\mapsto g_k^*(u,h(z,v))$ are all Lipschitz and~\ref{ass:G0}--\ref{ass:G5} are satisfied.

Even when $\pi_{\bar{\beta}}$ is not completely known, it can still be possible to tailor the choices of $f_k,g_k$ to wider classes of limiting prior distributions that induce certain types of structure in the signals $\beta$. For instance, if we are told that $\beta\in\R^p$ has at most $sp$ non-zero entries for some $s\in (0,1)$ and every $p\equiv p_n$, then as in Section~\ref{sec:lowrankgk}, we can take each $f_k$ to be a soft-thresholding function $\mathrm{S}_{t_k}\colon u\mapsto\sgn(u)(\abs{u}-t_k)_+$ for some $t_k>0$. Using an AMP recursion~\eqref{eq:AMPlinear} of this form (for the linear model in Section~\ref{sec:AMPlinear}) with appropriately chosen thresholds $t_k$,~\citet{BM12} derived exact high-dimensional asymptotics for the Lasso estimator; see Section~\ref{sec:lasso}.

\textbf{Spectral initialisation}: Under the conditions of Theorem~\ref{thm:GAMPmaster}, it follows from~\eqref{eq:GAMPesterr} and Lemma~\ref{lem:GAMPstein} that for each $k\in\N$, the estimates $\hat{\beta}^k$ in the generic GAMP procedure~\eqref{eq:GAMP} satisfy $\ipr{\hat{\beta}^k}{\beta}_p=p^{-1}\sum_{j=1}^p\hat{\beta}_j^k\beta_j\cvc\E\bigl(\bar{\beta}f_k(\mu_k\bar{\beta}+\sigma_k G_k)\bigr)=\mu_{Z,k}\,\E(\bar{\beta}^2)$ as $n,p\to\infty$ with $n/p\to\delta$. To ensure that $\mu_{Z,k}\neq 0$ for some $k$, and hence that the corresponding $\hat{\beta}^k$ has non-zero asymptotic empirical correlation with the signal $\beta$, it is sometimes necessary to start with pilot estimators $\hat{\beta}^0\in\R^p$ that themselves have the property that 
$\clim_{p\to\infty}\ipr{\hat{\beta}^0}{\beta}_p\neq 0$. Indeed, suppose that the limiting random variables in the state evolution recursion~\eqref{eq:GAMPstatevol1}--\eqref{eq:GAMPstatevol2} are such that
\begin{equation}
\label{eq:GAMPspectral}
\E(\bar{\beta})=0\quad\text{and}\quad\E(Z\,|\,Y)=0\quad\text{almost surely},
\end{equation}
where the latter condition is equivalent to (3.13) in~\citet{MM20}. Now given estimates $\hat{\beta}^0\in\R^p$ for which $\clim_{n\to\infty}\ipr{\hat{\beta}^0}{\beta}_p=\delta(\Sigma_0)_{21}=0$, we see from Lemma~\ref{lem:GAMPstein} that $\mu_{Z,0}=0$ and $Z_0$ is independent of $(Z,Y	)$, whence $g_0^*(Z_0,Y)=\E(Z\,|\,Y)/\Var(Z)=0$ almost surely in~\eqref{eq:GAMPgk} and $\mu_1=\E\bigl(g_0^*(Z_0,Y)\,g_0(Z_0,Y)\bigr)=0$ by~\eqref{eq:GAMPstein}. 
This means that $\mu_{Z,1}\,\E(\bar{\beta}^2)=\delta(\Sigma_1)_{21}=\E\bigl(\bar{\beta}f_1(\sigma_1G_1)\bigr)=0$ by the independence of $\bar{\beta}$ and $G_1$. Continuing inductively, we conclude that $\mu_k=\mu_{Z,k}=0$ for all $k\in\N$, irrespective of the choices of $g_k,f_{k+1}$ for $k\in\N_0$, so $\hat{\beta}^k$ is asymptotically uninformative as an estimator of $\beta\in\R^p$ for every $k\in\N_0$.

Thus, while there are some GLMs (such as the linear model in Section~\ref{sec:AMPlinear}) in which it suffices to take $\hat{\beta}^0=c\mathbf{1}_p$ for some fixed $c\in\R$, a different initialiser is required when~\eqref{eq:GAMPspectral} holds. We note that the second condition therein is satisfied in the phase retrieval model, where $y_i=h(x_i^\top\beta,\varepsilon_i)=(x_i^\top\beta)^2+\varepsilon_i$ for $1\leq i\leq n$, and more generally in all non-identifiable models of the form~\eqref{eq:glm1} where $h(z,w)=h(-z,w)$ for all $z,w$ (and hence $Q(\cdot\,|\,z)=Q(\cdot\,|\,{-z})$ in~\eqref{eq:glm2} for all $z$).
Indeed, for such functions $h$, we have $\E(Z\,|\,Y)=\E\bigl(Z\,|\,h(Z,\bar{\varepsilon})\bigr)=-\E\bigl(Z\,|\,h(Z,\bar{\varepsilon})\bigr)$ and hence $\E(Z\,|\,Y)=0$ almost surely.

\citet{MM20} established a version of Theorem~\ref{thm:GAMPmaster} for GAMP algorithms in which $\hat{\beta}^0$ is taken to be a leading eigenvector of $X^\top DX\in\R^{p\times p}$, where $D=\diag\bigl(g(y_1),\dotsc,g(y_n)\bigr)\in\R^{n\times n}$ for some $g\colon\R\to\R$. Since this spectral initialiser is correlated with the random design matrix $X$, condition~\ref{ass:G0} for the original Theorem~\ref{thm:GAMPmaster} does not hold in general. As mentioned in Section~\ref{sec:spectral}, the authors overcome this obstacle by analysing a two-phase artificial GAMP iteration in which the first stage effectively approximates $\hat{\beta}^0$ by the power method. 

\textbf{Confidence intervals and $p$-values}: For fixed $k$ and large $n$, Theorem~\ref{thm:GAMPmaster} tells us that $\{(\beta_i^k-\mu_k\beta_i)/\sigma_k:1\leq i\leq n\}$ behaves approximately like an i.i.d.\ sample of size $n$ from the $N(0,1)$ distribution. Thus, to carry out inference for $\beta$, we can proceed similarly as in Section~\ref{sec:lowrankconf}, to which we refer the reader for further details. We mention here that if the state evolution parameters $\mu_k,\sigma_k$ are unknown, then they can be estimated consistently by $\hat{\mu}_k:=\bigl(\norm{\beta^k}_p^2-\norm{\hat{r}^{k-1}}_n^2\bigr)^{1/2}/\E(\bar{\beta}^2)^{1/2}$ and $\hat{\sigma}_k:=\norm{\hat{r}^{k-1}}_n$ provided that $\E(\bar{\beta}^2)>0$ is known. Indeed, by~\eqref{eq:GAMPbeta} and~\eqref{eq:GAMPtheta} respectively, \begin{align*}
\norm{\beta^k}_p^2&\cvc\E\bigl((\mu_k\bar{\beta}+\sigma_k G_k)^2\bigr)=\E(\bar{\beta}^2)\mu_k^2+\sigma_k^2,\\ \norm{\hat{r}^{k-1}}_n^2&=\norm{g_{k-1}(\theta^{k-1},y)}_n^2=\norm{\tilde{g}_{k-1}(\theta,\theta^{k-1},\varepsilon)}\cvc\E\bigl(\tilde{g}_{k-1}(Z,Z_k,\bar{\varepsilon})^2\bigr)=\E\bigl(g_{k-1}(Z_k,Y)^2\bigr)=\sigma_k^2
\end{align*} 
for each $k\in\N$ as $n,p\to\infty$ with $n/p\to\delta$.

\hfparskip
\subsection{AMP for the linear model}
\label{sec:AMPlinear}
Much of the early work on AMP~\citep[e.g.][]{donohoMM2009,BM11,BM12,krzakala2012} was centred around the standard linear model
\[y=X\beta+\varepsilon,\]
where $\varepsilon_1,\dotsc,\varepsilon_n\iid P_{\bar{\varepsilon}}$ have second moment $\sigma^2>0$ and a finite $r^{th}$ moment for some $r\in [2,\infty)$ (or more generally where the empirical distribution $\nu_n(\varepsilon)=n^{-1}\sum_{i=1}^n\delta_{\varepsilon_i}$ converges completely in $d_r$ to $P_{\bar{\varepsilon}}$ as $n\to\infty$). This is a special case of the model~\eqref{eq:glm1} with $h(z,v)=z+v$.

Given $\hat{r}^{-1}=0\in\R^n$, $b_0\in\R$ and an initial estimator $\hat{\beta}^0\in\R^p$, the original AMP algorithm of~\citet{donohoMM2009} and~\citet{BM11} can be recovered by setting $g_k(u,v):=v-u$ for $u,v\in\R$ in the GAMP recursion~\eqref{eq:GAMP}, so that $c_k=\abr{g_k'(\theta^k,y)}_n=-1$ and
\begin{equation}
\label{eq:AMPlinear}
\hat{r}^k=y-X\hat{\beta}^k+b_k\hat{r}^{k-1},\qquad\hat{\beta}^{k+1}=f_{k+1}(X^\top\hat{r}^k+\hat{\beta}^k),\qquad b_{k+1}=\frac{1}{n}\sum_{j=1}^p f_{k+1}'(\beta_j^{k+1})
\end{equation}
for $k\in\N_0$. Here, $\hat{r}^k=g_k(\theta^k,y)=y-\theta^k=y-X\hat{\beta}^k+b_k\hat{r}^{k-1}$ is a `corrected' residual at iteration $k$, and $\beta^{k+1}=X^\top\hat{r}^k+\hat{\beta}^k$ is the effective observation. 

\textbf{State evolution}: The GAMP state evolution equations~\eqref{eq:GAMPstatevol1}--\eqref{eq:GAMPstatevol2} simplify to the recursion
\begin{equation}
\label{eq:statevolinear}
\mu_k\equiv 1,\qquad\sigma_1^2=\sigma^2+\E\bigl((Z-Z_0)^2\bigr),\qquad\sigma_{k+1}^2=\sigma^2+\frac{1}{\delta}\,\E\bigl\{\bigl(\bar{\beta}-f_k(\bar{\beta}+\sigma_k G_k)\bigr)^2\bigr\}
\end{equation}
for $k\in\N$, where $(Z,Z_0)\sim N_2(0,\Sigma_0)$, and $\bar{\beta}\sim \pi_{\bar{\beta}}$ is independent of $G_k\sim N(0,1)$. Note that by~\ref{ass:G2}, $\sigma_1^2=\sigma^2+\clim_{n\to\infty}n^{-1}\norm{\beta-\hat{\beta}^0}^2$, and that if the pilot estimate of $\beta\in\R^p$ is taken to be $\hat{\beta}^0=0\in\R^p$ for each $p\equiv p_n$, then $Z_0\equiv 0$ and $\sigma_1^2=\sigma^2+\E(Z^2)=\sigma^2+\delta^{-1}\E(\bar{\beta}^2)$. 

\textbf{Asymptotic estimation error}: Under~\ref{ass:G0}--\ref{ass:G5} with $r\in [2,\infty)$, the main result of~\citet[Theorem~1]{BM11} on the asymptotic performance of the estimators $\hat{\beta}^k$ in~\eqref{eq:AMPlinear} can be stated as
\begin{equation}
\label{thm:AMPlinear}
\sup_{\psi\in\PL_2(r,1)}\;\biggl|\frac{1}{p}\sum_{j=1}^p\psi(\hat{\beta}_j^k,\beta_j)-\E\bigl\{\psi\bigl(f_k(\bar{\beta}+\sigma_k G_k),\bar{\beta}\bigr)\bigr\}\biggr|\cvc 0\quad\text{as }n,p\to\infty\text{ with }n/p\to\delta,
\end{equation}
for each $k\in\N$, where $\sigma_k$ is as in~\eqref{eq:statevolinear}. This can be obtained as a special case of Theorem~\ref{thm:GAMPmaster} and~\eqref{eq:GAMPesterr}. Alternatively,~\eqref{thm:AMPlinear} can be established via a direct reduction to an abstract asymmetric AMP recursion of the type in Section~\ref{sec:AMPnonsym}; see~\citet[Section~3.3]{BM11}. This involves writing~\eqref{eq:AMPlinear} in terms of $e^k:=\varepsilon-\hat{r}^k$ and $h^{k+1}:=\beta^{k+1}-\beta$, which turn out to be the asymptotically Gaussian `noise' components of $\hat{r}^k$ and $\beta^{k+1}$ respectively. 

Originally, the $d_r$ convergence result~\eqref{thm:AMPlinear} was derived under a stronger version of~\ref{ass:G1} that assumed $d_{2r-2}$ convergence to limiting distributions $\pi_{\bar{\beta}},P_{\bar{\varepsilon}}$ with finite $(2r-2)^{th}$ moments. In~\ref{ass:G1}, we relax this to a more natural $d_r$ condition under which the conclusion still holds; see the first part of Remark~\ref{rem:AMPm0}. We also mention that under suitable finite-sample analogues of the conditions above, a complementary finite-sample version of~\eqref{thm:AMPlinear} was established by~\citet{RV18} in the case $r=2$; see Remark~\ref{rem:finitesample}. 

\textbf{Link to Bayes-GAMP}: If the limiting prior distribution $\pi_{\bar{\beta}}$ is known, then to minimise the effective noise variance $\sigma_{k+1}^2$, we can take $f_k$ in~\eqref{eq:AMPlinear} to be the Bayes optimal $f_k^*$ from~\eqref{eq:GAMPfk}. In general, $g_k\colon (u,v)\mapsto v-u$ does not coincide with $g_k^*$ in~\eqref{eq:GAMPgk}. However, when $P_{\bar{\varepsilon}}=N(0,\sigma^2)$ with $\sigma^2>0$, $\hat{\beta}^0\equiv\hat{\beta}^0(n)=0$ for every $n$ and $f_k=f_k^*$ for each $k\in\N$, it turns out that~\eqref{eq:AMPlinear} is an instance of a Bayes-GAMP procedure (with $g_k\propto g_k^*$) that maximises the effective signal-to-noise ratios $\rho_k=(\mu_k/\sigma_k)^2$ and $\rho_{Z,k}=(\mu_{Z,k}/\sigma_{Z,k})^2$ at each iteration. Indeed, in this special case, it can be verified by direct computation that 
\[g_k^*(u,v)=c_k\biggl(\frac{\Sigma_{21}}{\Sigma_{22}}u-v\biggr)=c_k(u-v)=-c_kg_k(u,v)\]
for each $k\in\N_0$, where $\Sigma\equiv\Sigma_k\in\R^{2\times 2}$ is as in~\eqref{eq:GAMPstatevol2}, with $\Sigma_{21}=\Sigma_{22}=0<\Sigma_{11}$ when $k=0$ and $\Sigma_{21}=\delta^{-1}\E(\bar{\beta}f_k^*(\mu_k\bar{\beta}+\sigma_k G_k)\bigr)=\delta^{-1}\E(f_k^*(\mu_k\bar{\beta}+\sigma_k G_k)^2\bigr)=\Sigma_{22}<\Sigma_{11}=\delta^{-1}\E(\bar{\beta}^2)$ by~\eqref{eq:GAMPfk} when $k\in\N$, and
\[c_k
=-\frac{\Sigma_{11}-\Sigma_{22}}{\Sigma_{11}-\Sigma_{22}+\sigma^2}<0\]
is deterministic. Here, $\delta(\Sigma_{11}-\Sigma_{22})=\E\bigl\{\bigl(\bar{\beta}-\E(\bar{\beta}\,|\,\mu_k\bar{\beta}+\sigma_k G_k)\bigr)^2\bigr\}$ is the minimum mean squared error for the problem of estimating $\bar{\beta}$ based on $\mu_k\bar{\beta}+\sigma_k G_k$.

\hfparskip
\subsection{GAMP algorithms for convex optimisation}
\label{sec:GAMPopt}
Given $y\in\R^n$ and $X\in\R^{n\times p}$ with rows $x_1,\dotsc,x_n$, many statistical estimators of $\beta$ in~\eqref{eq:glm1} are defined as minimisers of objective functions of the form $\tilde{\beta}\mapsto\mathcal{C}(\tilde{\beta};X,y):=\sum_{i=1}^n\ell(x_i^\top\tilde{\beta},y_i)+\sum_{j=1}^p J(\tilde{\beta}_j)$, or equivalently as solutions to constrained optimisation problems of the form
\begin{equation}
\label{eq:constropt}
\text{minimise}\quad\sum_{i=1}^n\ell(\tilde{\theta}_i,y_i)+\sum_{j=1}^p J(\tilde{\beta}_j)\quad\text{over }(\tilde{\beta},\tilde{\theta})\in\R^p\times\R^n\text{ with }\tilde{\theta}=X\tilde{\beta}, 
\end{equation}
where $\ell\colon\R^2\to\R$ is a loss function and $J\colon\R\to\R$ is a penalty function. In particular, consider a GLM of the form~\eqref{eq:glm2} in which $y_i\,|\,(x_i,\beta)\sim q(\cdot\,|\,x_i^\top\beta)$ for $1\leq i\leq n$, where $q(\cdot\,|\,z)$ is a Lebesgue density on $\R$ for each $z\in\R$. Then the maximum likelihood estimators (MLEs) of $\beta$ and $\theta=X\beta$ are given by
\[(\hat{\beta}^{\mathrm{MLE}},\hat{\theta}^{\mathrm{MLE}}):=\argmin_{\substack{(\tilde{\beta},\tilde{\theta})\in\R^p\times\R^n\\ \tilde{\theta}=X\tilde{\beta}}}\;\sum_{i=1}^n -\log q(y_i\,|\,\tilde{\theta}_i).\]
If in addition $\beta_1,\dotsc,\beta_p\iid p_{\bar{\beta}}$ for some prior density $p_{\bar{\beta}}$, then the maximum a posteriori (MAP) estimates of $\beta$ and $\theta$ are
\[(\hat{\beta}^{\mathrm{MAP}},\hat{\theta}^{\mathrm{MAP}}):=\argmin_{\substack{(\tilde{\beta},\tilde{\theta})\in\R^p\times\R^n\\ \tilde{\theta}=X\tilde{\beta}}}\;\Biggl(\sum_{i=1}^n -\log q(y_i\,|\,\tilde{\theta}_i)+\sum_{j=1}^p -\log p_{\bar{\beta}}(\tilde{\beta}_j)\Biggr).\]
Assuming henceforth that $\ell$ and $J$ are convex in their first arguments, we will now design a GAMP iteration~\eqref{eq:GAMPopt} whose fixed points are solutions to the associated optimisation problem~\eqref{eq:constropt}; see Proposition~\ref{prop:GAMPopt} below. By exploiting this connection and applying the GAMP theory from Section~\ref{sec:GAMPmaster}, we will explain later how to obtain a statistical payoff in the form of exact high-dimensional asymptotics for estimators defined by~\eqref{eq:constropt}.

To begin the construction, fix two sequences of deterministic scalars $\bar{b}_k>0$ and $\bar{c}_k<0$ for $k\in\N_0$. These will later be assigned appropriate values in~\eqref{eq:GAMPopt} below, but for the time being, we will treat them as generic constants. For $k\in\N_0$, define $\bar{g}_k,g_k\colon\R^2\to\R$ and $f_{k+1}\colon\R\to\R$ by
\begin{align}
\label{eq:gkbar}
\bar{g}_k(u,v)&:=\argmin_{z\in\R}\,\Bigl\{\ell(z,v)+\frac{1}{2\bar{b}_k}(z-u)^2\Bigr\},\qquad g_k(u,v):=\frac{\bar{g}_k(u,v)-u}{\bar{b}_k},\\
\label{eq:fkbar}
f_{k+1}(w)&:=\argmin_{z\in\R}\,\biggl\{J(z)-\frac{\bar{c}_k}{2}\biggl(z+\frac{w}{\bar{c}_k}\biggr)^2\biggr\}.
\end{align}
Note that since $\ell$ and $J$ are assumed to be convex in their first arguments, $\bar{g}_k(u,v)$ and $f_{k+1}(w)$ are well-defined as unique minima of strongly convex functions. The pertinence of this specific choice of $g_k,f_{k+1}$ will become apparent through Proposition~\ref{prop:GAMPopt} below and its proof. At this point, it is helpful to recall that for a convex function $\mathrm{M}\colon\R\to\R$ and $\eta>0$, the associated \emph{proximal operator} $\prox_{\eta\mathrm{M}}\colon\R\to\R$ is given by
\begin{equation}
\label{eq:prox}
\prox_{\eta\mathrm{M}}(z):=\argmin_{t\in\R}\,\Bigl\{\eta\mathrm{M}(t)+\frac{1}{2}(t-z)^2\Bigr\},
\end{equation}
and moreover that $\prox_{\eta\mathrm{M}}$ is always non-decreasing and 1-Lipschitz~\citep[cf.][Sections~2.3 and~3.1]{PB13}. We see that $\bar{g}_k(u,v)=\prox_{\bar{b}_k\ell(\cdot,v)}(u)$ and $f_{k+1}(w)=\prox_{-J/\bar{c}_k}(-w/\bar{c}_k)$ for $u,v,w\in\R$, so
$\bar{g}_k,g_k,f_{k+1}$ are all Lipschitz with constants 1, $\bar{b}_k^{-1}$ and $\abs{\bar{c}_k}^{-1}$ respectively, and hence weakly differentiable with respect to their first arguments. Writing $\bar{g}_k',g_k',f_{k+1}'$ for the corresponding weak derivatives, we have
\begin{equation}
\label{eq:fgkderiv}
f_{k+1}'(w)\geq 0,\qquad\bar{g}_k'(u,v)\leq 1\quad\text{and hence}\quad g_k'(u,v)\leq 0
\end{equation}
for all $u,v,w$. If in addition $\ell$ and $J$ are twice continuously differentiable, then $J'(f_{k+1}(w))-\bigl(\bar{c}_kf_{k+1}(w)+w\bigr)=0$ for each $w$, so it follows from the implicit function theorem that
\begin{equation}
\label{eq:proxderiv}
f_{k+1}'(w)=\bigl(J''(f_{k+1}(w))-\bar{c}_k\bigr)^{-1}\quad\text{and similarly}\quad\bar{g}_k'(u,v)=\bigl(\bar{b}_k\ell''(f_{k+1}(w))+1\bigr)^{-1}
\end{equation}
for all $u,v,w$, where $\ell''$ denotes the second partial derivative of $\ell$ with respect to its first argument.

We will now define a GAMP recursion of the form~\eqref{eq:GAMP} as a precursor to the iteration~\eqref{eq:GAMPopt} that will subsequently be used to analyse the statistical properties of the solutions to the optimisation problem~\eqref{eq:constropt}.
Given $\hat{s}^{-1}:=0\in\R^n$, a fixed $b_0>0$ and an initialiser $\hat{\beta}^0\in\R^p$, inductively define
\begin{alignat}{3}
\theta^k&:=X\hat{\beta}^k-b_k\hat{s}^{k-1},\qquad&\hat{\theta}^k&:=\bar{g}_k(\theta^k,y),\qquad&c_k&:=n^{-1}\textstyle\sum_{i=1}^n g_k'(\theta_i^k,y_i),\qquad\hat{s}^k:=g_k(\theta^k,y),\notag
\\
\label{eq:GAMPopt1}
\beta^{k+1}&:=X^\top\hat{s}^k-c_k\hat{\beta}^k,\qquad&\hat{\beta}^{k+1}&:=f_{k+1}(\beta^{k+1}),\qquad&b_{k+1}&:=n^{-1}\textstyle\sum_{j=1}^p f_{k+1}'(\beta_j^{k+1})
\end{alignat}
for $k\in\N_0$. Note that $\hat{s}^k=(\hat{\theta}^k-\theta^k)/\bar{b}_k$, and that if $\ell$ and $J$ are convex and twice continuously differentiable with respect to their first arguments, then~\eqref{eq:proxderiv} yields
\begin{align*}
c_k=\frac{1}{\bar{b}_k}\biggl(\frac{1}{n}\sum_{i=1}^n\bar{g}_k'(\theta_i^k,y_i)-1\biggr)=-\frac{1}{n}\sum_{i=1}^n\frac{\ell''(\hat{\theta}_i^k,y_i)}{\bar{b}_k\ell''(\hat{\theta}_i^k,y_i)+1},\qquad b_{k+1}
=\frac{1}{n}\sum_{j=1}^p\frac{1}{J''(\hat{\beta}_j^{k+1})-\bar{c}_k}.
\end{align*}
If the hypotheses of Theorem~\ref{thm:GAMPmaster} are satisfied by a sequence of recursions~\eqref{eq:GAMPopt1}, then the limiting empirical distributions of the iterates therein are characterised by the associated state evolution parameters $(\mu_k,\sigma_k,\Sigma_k:k\in\N)$ defined through~\eqref{eq:GAMPstatevol1}--\eqref{eq:GAMPstatevol2}. Moreover, with $(Z,Z_k)\sim N_2(0,\Sigma_k)$ and $Y=h(Z,\bar{\varepsilon})$ as in Lemma~\ref{lem:GAMPstein} for each fixed $k$, recall from~\eqref{eq:GAMPbk} that $b_k\cvc\delta^{-1}\,\E\bigl(f_k'(\mu_k\bar{\beta}+\sigma_k G_k)\bigr)$ and $c_k\cvc\E\bigl(g_k'(Z_k,Y)\bigr)$ as $n,p\to\infty$ with $n/p\to\delta\in (0,\infty)$. 

Based on this observation, we will define $\bar{b}_k$ and $\bar{c}_k$ above to coincide with these limiting values, and substitute these deterministic quantities for the random $b_k,c_k$ in~\eqref{eq:GAMPopt1} to obtain the following modified recursion. As before, we start with $\hat{s}^{-1}:=0\in\R^n$, $\bar{b}_0>0$, $\hat{\beta}^0\in\R^p$, as well as a positive definite $\Sigma_0\in\R^{2\times 2}$ as in~\ref{ass:G2}. Given $\hat{\beta}^k,\hat{s}^{k-1}$ and $\bar{b}_k,\Sigma_k$ for a general $k\in\N_0$, we inductively define $\bar{g}_k,g_k$ as in~\eqref{eq:gkbar}, along with
\begin{alignat}{3}
\theta^k&:=X\hat{\beta}^k-\bar{b}_k\hat{s}^{k-1},\qquad&\hat{\theta}^k&:=\bar{g}_k(\theta^k,y),\qquad&\bar{c}_k&:=\E\bigl(g_k'(Z_k,Y)\bigr),\qquad\hat{s}^k:=g_k(\theta^k,y),\notag
\\
\label{eq:GAMPopt}
\beta^{k+1}&:=X^\top\hat{s}^k-\bar{c}_k\hat{\beta}^k,\qquad&\hat{\beta}^{k+1}&:=f_{k+1}(\beta^{k+1}),\qquad&\bar{b}_{k+1}&:=\delta^{-1}\,\E\bigl(f_{k+1}'(\mu_{k+1}\bar{\beta}+\sigma_{k+1}G_{k+1})\bigr).
\end{alignat}
In~\eqref{eq:GAMPopt}, we take $(Z,Z_k)\sim N_2(0,\Sigma_k)$ and $Y=h(Z,\bar{\varepsilon})$ as above, and define the state evolution parameters $\mu_{k+1},\sigma_{k+1}$ as in~\eqref{eq:GAMPstatevol1} based on $g_k$, while using $\bar{c}_k$ and~\eqref{eq:fkbar} to specify $f_{k+1}$. Finally, define $\Sigma_{k+1}$ in terms of $f_{k+1},\mu_{k+1},\sigma_{k+1}$ according to~\eqref{eq:GAMPstatevol2}. We emphasise that the functions $\bar{g}_k,f_{k+1}$ are indeed well-defined through~\eqref{eq:gkbar}--\eqref{eq:fkbar} for all $k$ since $\bar{b}_k>0>\bar{c}_k$ by~\eqref{eq:fgkderiv} and the fact that $\prox_{\mathrm{M}}$ is non-constant for any convex $\mathrm{M}\colon\R\to\R$.

The iteration~\eqref{eq:GAMPopt} has two important features that make it a useful theoretical tool.
First, Remark~\ref{rem:GAMPbk} ensures that its iterates are characterised by the state evolution parameters $(\mu_k,\sigma_k,\Sigma_k:k\in\N)$ under the hypotheses of Theorem~\ref{thm:GAMPmaster}. In addition, the following result highlights the significance of~\eqref{eq:GAMPopt} as an optimisation procedure for the original constrained problem~\eqref{eq:constropt}.

\begin{proposition}[{\citealp[Theorem~1]{rangan2016fixed}}]
\label{prop:GAMPopt}
In~\eqref{eq:constropt}, suppose that $\ell$ and $J$ are convex in their first arguments, and define the associated Lagrangian by
\begin{equation}
\label{eq:lagrangian}
L(\tilde{\beta},\tilde{\theta},s):=\sum_{i=1}^n\ell(\tilde{\theta}_i,y_i)+\sum_{j=1}^p J(\tilde{\beta}_j)+s^\top(\tilde{\theta}-X\tilde{\beta})
\end{equation}
for $\tilde{\beta}\in\R^p$ and $\tilde{\theta},s\in\R^n$. Then the iterates in~\eqref{eq:GAMPopt} satisfy
\begin{align}
\label{eq:Lcond1}
\hat{\beta}^{k+1}&=\argmin_{\tilde{\beta}\in\R^p}\,\Bigl\{L(\tilde{\beta},\hat{\theta}^k,\hat{s}^k)-\frac{\bar{c}_k}{2}\norm{\tilde{\beta}-\hat{\beta}^k}^2\Bigr\},\\
\label{eq:Lcond2}
\hat{\theta}^{k+1}&=\argmin_{\tilde{\theta}\in\R^n}\,\Bigl\{L(\hat{\beta}^{k+1},\tilde{\theta},\hat{s}^k)+\frac{1}{2\bar{b}_{k+1}}\norm{\tilde{\theta}-X\hat{\beta}^{k+1}}^2\Bigr\},\\
\label{eq:Lcond3}
\hat{s}^{k+1}&=\hat{s}^k+\frac{(\hat{\theta}^{k+1}-X\hat{\beta}^{k+1})}{\bar{b}_{k+1}}.
\end{align}
for $k\in\N_0$. Moreover, if $(\beta^*,\theta^*,\hat{\beta}^*,\hat{\theta}^*,\hat{s}^*)$ is a fixed point of~\eqref{eq:GAMPopt}, 
then $(\hat{\beta}^*,\hat{\theta}^*)$ is a solution to the optimisation problem~\eqref{eq:constropt}, i.e.\ $\hat{\beta}^*\in\argmin_{\tilde{\beta}\in\R^p}\mathcal{C}(\tilde{\beta};X,y)$.
\end{proposition}

\unparskip
In fact, the proof we give in Section~\ref{sec:GAMPproofs} reveals that Proposition~\ref{prop:GAMPopt} holds for any choice of deterministic scalars $\bar{b}_k>0$ and $\bar{c}_k<0$ in the first column of~\eqref{eq:GAMPopt}, provided that these are also used to define $\bar{g}_k,f_{k+1}$. The characterisation in~\eqref{eq:Lcond1}--\eqref{eq:Lcond3} shows that the GAMP algorithm~\eqref{eq:GAMPopt} is closely related to (but not completely identical to) a `linearised' Alternating Direction Method of Multipliers (ADMM) procedure~\citep[Section~4.4.2]{PB13} for optimising~\eqref{eq:constropt}. 
Alternating algorithms of this type are particularly well-suited to handling objective functions of the form~\eqref{eq:lagrangian} since each minimisation step involves only one of $J$ and $\ell$ (while~\eqref{eq:Lcond3} is a dual update step). The forms of the quadratic penalties in~\eqref{eq:Lcond1}--\eqref{eq:Lcond2} ensure that the `augmented Lagrangians' therein are separable, and hence can be minimised separately in each coordinate of $\tilde{\beta}$ or $\tilde{\theta}$. 
This is why $\hat{\beta}^{k+1},\hat{\theta}^{k+1}$ are obtained from $\beta^{k+1},\theta^{k+1}$ by componentwise applications of $f_{k+1},\bar{g}_{k+1}$ respectively, whose expressions in~\eqref{eq:gkbar}--\eqref{eq:fkbar} emerge naturally from~\eqref{eq:Lcond1}--\eqref{eq:Lcond2}. See also~\citet{boyd2011distributed} for an accessible introduction to ADMM, and~\citet{rangan2016fixed} for further details on the connection between GAMP and conventional convex optimisation algorithms.

Based on Proposition~\ref{prop:GAMPopt} and the reasoning above, we might expect the high-dimensional limiting behaviour of the estimators $\hat{\beta}^*\in\argmin_{\tilde{\beta}\in\R^p}\mathcal{C}(\tilde{\beta};X,y)$ to be governed by some fixed point of the state evolution for~\eqref{eq:GAMPopt} (if it exists). To prove this, we might hope to be able to establish convergence of both the GAMP iteration~\eqref{eq:GAMPopt} and its state evolution to their respective fixed points (in the sense of~\eqref{eq:GAMPconv} below). We conclude this subsection by setting out a general strategy along these lines. 
In Sections~\ref{sec:lasso}--\ref{sec:logistic}, we will go on to demonstrate that it unifies existing derivations of high-dimensional asymptotic results for the Lasso, and M-estimators in the linear model and logistic regression model.

\textbf{Step 1}: For given $\ell$ and $J$ (and fixed $n$ and $p\equiv p_n$), find a fixed point of~\eqref{eq:GAMPopt} \emph{together with its state evolution}, satisfying
\begin{equation}
\label{eq:GAMPfixed}
\begin{alignedat}{3}
\theta^*&:=X\hat{\beta}^*-\bar{b}_*\hat{s}^*,\qquad&\hat{\theta}^*&:=\bar{g}_*(\theta^*,y),\qquad&\bar{c}_*&:=\E\bigl(g_*'(Z_*,Y)\bigr),\qquad\hat{s}^*:=g_*(\theta^*,y),
\\
\beta^*&:=X^\top\hat{s}^*-\bar{c}_*\hat{\beta}^*,\qquad&\hat{\beta}^*&:=f_*(\beta^*),\qquad&\bar{b}_*&:=\delta^{-1}\,\E\bigl(f_*'(\mu_*\bar{\beta}+\sigma_*G_*)\bigr).
\end{alignedat}
\end{equation}
Here, $f_*,\bar{g}_*,g_*$ are defined in terms of $\bar{b}_*>0$, $\bar{c}_*<0$ as in~\eqref{eq:gkbar}--\eqref{eq:fkbar}, with $(Z,Z_*)\sim N_2(0,\Sigma_*)$ and $Y=h(Z,\bar{\varepsilon})$, while $G_*\sim N(0,1)$ is independent of $\bar{\beta}\sim \pi_{\bar{\beta}}$ and $\mu_*,\sigma_*,\Sigma_*,f_*,g_*$ satisfy~\eqref{eq:GAMPstatevol1}--\eqref{eq:GAMPstatevol2}. In each of the subsequent examples, the system~\eqref{eq:GAMPfixed} reduces to
a smaller set of (non-linear) equations.
The existence and uniqueness of a state evolution fixed point usually needs to be verified on a case-by-case basis, and may depend on the values of parameters such as the limiting sampling ratio $\delta$ and the asymptotic signal strength $\E(\beta^2)/\delta$ (the variance of $Z$ above).

\textbf{Step 2}: If Step 1 yields suitable $f_*,\bar{g}_*,g_*,\bar{b}_*,\bar{c}_*$, then consider the following 
`stationary' version of~\eqref{eq:GAMPopt} for each $n$ and $p\equiv p_n$:
\begin{equation}
\label{eq:GAMPstat}
\begin{alignedat}{3}
\theta^k&:=X\hat{\beta}^k-\bar{b}_*\hat{s}^{k-1},\qquad&\hat{\theta}^k&:=\bar{g}_*(\theta^k,y),\qquad\hat{s}^k:=g_*(\theta^k,y),
\\
\beta^{k+1}&:=X^\top\hat{s}^k-\bar{c}_*\hat{\beta}^k,\qquad&\hat{\beta}^{k+1}&:=f_*(\beta^{k+1}).
\end{alignedat}
\end{equation}
Henceforth, we will use~\eqref{eq:GAMPstat} as a theoretical device rather than as a practical algorithm, which gives us the flexibility to initialise it with $\hat{s}^{-1}=0\in\R^n$ and an `oracle' $\hat{\beta}^0=f_*(\mu_*\beta+\sigma_*\xi)\in\R^p$, where $\xi\sim N_p(0,I_p)$ is independent of the signal $\beta\in\R^p$. 
This is a convenient choice because it ensures that $\Sigma_0=\Sigma_*$ and hence that the state evolution for~\eqref{eq:GAMPstat} is stationary, i.e.\ $\mu_k=\mu_*$, $\sigma_k=\sigma_*$ and $\Sigma_k=\Sigma_*$ for all $k\in\N$. In addition, as $n,p\to\infty$ with $n/p\to\delta$ under~\ref{ass:G1}, the $d_2$ limit of the empirical distribution of the entries of $\hat{\beta}^0$ is the distribution of $f_*(\mu_*\bar{\beta}+\sigma_*G_*)$ by construction, and under the hypotheses of Theorem~\ref{thm:GAMPmaster}, this is also true of $\hat{\beta}^k$ for each \emph{fixed} $k\in\N$ by Remark~\ref{rem:GAMPbk}. The remaining technical challenge to establish the same distributional limit for the fixed point $\hat{\beta}^*$, which solves the optimisation problem~\eqref{eq:constropt} by Proposition~\ref{prop:GAMPopt}. 

\textbf{Step 3}: Show that the estimates $\hat{\beta}^k$ in~\eqref{eq:GAMPstat} converge to $\hat{\beta}^*\in\argmin_{\tilde{\beta}\in\R^p}\mathcal{C}(\tilde{\beta};X,y)$ in the sense that
\begin{equation}
\label{eq:GAMPconv}
\lim_{k\to\infty}\clim_{p\to\infty}\frac{\norm{\hat{\beta}^k-\hat{\beta}^*}^2}{p}=0.
\end{equation}
In the examples in Sections~\ref{sec:lasso}--\ref{sec:logistic}, this is achieved by first establishing a `Cauchy property'
\[\clim_{p\to\infty}\frac{\norm{\hat{\beta}^{k+1}-\hat{\beta}^k}^2}{p}=0,\qquad\clim_{n\to\infty}\frac{\norm{\hat{s}^{k+1}-\hat{s}^k}}{n}^2=0\]
for each $k$ (using the limiting covariance structure mentioned after Theorem~\ref{thm:GAMPmaster}), and then proving that for large $k$ and $p$, the original convex cost function $\tilde{\beta}\mapsto\mathcal{C}(\tilde{\beta};X,y)$ is approximately minimised by $\hat{\beta}^k$ in the following sense: if $\hat{\gamma}^k\in\R^p$ belongs to the subgradient of $\mathcal{C}(\cdot\,;X,y)$ at $\hat{\beta}^k$ for $k\in\N$ and $p\equiv p_n$, then 
\begin{equation}
\label{eq:GAMPconv1}
\lim_{k\to\infty}\clim_{p\to\infty}\frac{\norm{\hat{\gamma}^k}^2}{p}=0.
\end{equation}
If $\mathcal{C}(\cdot\,;X,y)$ is strongly convex (on a subset of its domain that contains $\hat{\beta}^k,\hat{\beta}^*$) with high probability, then the desired conclusion~\eqref{eq:GAMPconv} follows readily from~\eqref{eq:GAMPconv1} and the basic inequality $\mathcal{C}(\hat{\beta}^*;X,y)\leq\mathcal{C}(\hat{\beta}^k;X,y)$; see (98)--(100) in~\citet{donohoMest16}. Otherwise (as in the case of the Lasso in Section~\ref{sec:lasso}), further work must be done to show that in a random design setting, it is vanishingly unlikely that $\norm{\hat{\gamma}^k}_p$ is small but $\norm{\hat{\beta}^k-\hat{\beta}^*}_p$ is large~\citep[cf.][Theorem~1.8 and Lemma~3.1]{BM12}.

\hfparskip
\subsection{AMP for the Lasso}
\label{sec:lasso}
In high-dimensional linear models $y=X\beta+\varepsilon$, the Lasso~\citep{Tib96} is a popular method for obtaining sparse estimates of $\beta\in\R^p$ via $\ell_1$-penalised least squares. Given $X\in\R^{n\times p}$, $y\in\R^n$ and a regularisation parameter $\lambda>0$, the Lasso estimator is defined by
\begin{equation}
\label{eq:lasso}
\hat{\beta}_\lambda^{\mathrm L}\in\argmin_{\tilde{\beta}\in\R^p}\,\biggl\{\frac{1}{2}\norm{y-X\tilde{\beta}}^2+\lambda\norm{\tilde{\beta}}_1\biggr\}.
\end{equation}
In the random design setting of~\ref{ass:G0} and~\ref{ass:G1},~\citet{BM12} derived an exact expression~\eqref{eq:lassoest} for the asymptotic estimation error of $\hat{\beta}_\lambda^{\mathrm{L}}$ as $n,p\to\infty$ with $n/p\to\delta\in (0,\infty)$. By following the GAMP recipe in Section~\ref{sec:GAMPopt}, we will show how to design and calibrate an AMP iteration that is central to the proof of their main result (Theorem~\ref{thm:lasso} below).

To begin with, note that $\hat{\beta}_\lambda^{\mathrm L}$ solves a convex optimisation problem of the form~\eqref{eq:constropt} with $\ell\colon(u,v)\mapsto (u-v)^2/2$ and $J\colon x\mapsto\lambda\abs{x}$. For $k\in\N_0$, the corresponding $\bar{g}_k,g_k,f_{k+1}$ in~\eqref{eq:gkbar}--\eqref{eq:fkbar} are given by
\begin{equation}
\label{eq:lassofgk}
\bar{g}_k(u,v)=\frac{u+\bar{b}_kv}{1+\bar{b}_k},\qquad g_k(u,v)=\frac{v-u}{1+\bar{b}_k},\qquad f_{k+1}(w)=-\ST_{\lambda/\bar{c}_k}\biggl(-\frac{w}{\bar{c}_k}\biggr)=-\frac{\ST_\lambda(w)}{\bar{c}_k},
\end{equation}
where as in Section~\ref{sec:lowrankgk}, we denote by $\ST_t$ the soft-thresholding function $w\mapsto\sgn(w)(\abs{w}-t)_+$ for $t>0$. Given $\hat{r}^{-1}=0\in\R^n$, $\tilde{b}_0\equiv\bar{b}_0>0$ and $\hat{\beta}^0\in\R^p$, the resulting GAMP algorithm~\eqref{eq:GAMPopt} can be succinctly written as
\begin{equation}
\label{eq:GAMPlasso}
\hat{r}^k=y-X\hat{\beta}^k+\tilde{b}_k\hat{r}^{k-1},\qquad\hat{\beta}^{k+1}=\ST_{t_{k+1}}\bigl(X^\top\hat{r}^k+\hat{\beta}^k\bigr)\qquad\text{for }k\in\N_0,
\end{equation}
where $\hat{r}^k:=y-\hat{\theta}^k=(1+\bar{b}_k)\hat{s}^k$. Observe that~\eqref{eq:GAMPlasso} is (asymptotically equivalent to) an instance of the AMP recursion~\eqref{eq:AMPlinear} in Section~\ref{sec:AMPlinear} for the linear model, 
whose state evolution formula is given by~\eqref{eq:statevolinear}, with $\mu_k=1$ for all $k$. By~\eqref{eq:GAMPopt} and~\eqref{eq:lassofgk}, the deterministic scalars $\tilde{b}_k:=\bar{b}_k/(1+\bar{b}_{k-1})>0$ and $t_{k+1}:=\lambda(1+\bar{b}_k)=-\lambda/\bar{c}_k>0$ in~\eqref{eq:GAMPlasso} are related to each other and the state evolution parameters $\sigma_k^2$ via
\begin{alignat}{2}
\sigma_1^2=\sigma^2+\E\bigl((Z-Z_0)^2\bigr),\quad t_1=\lambda(1+\tilde{b}_0),\qquad&&\tilde{b}_k&=\frac{\E\bigl(\ST_{t_k}'(\bar{\beta}+\sigma_kG_k)\bigr)}{\delta}=\frac{\Pr\bigl(\abs{\bar{\beta}+\sigma_kG_k}>t_k\bigr)}{\delta},\notag\\ 
\label{eq:statevolasso}
\sigma_{k+1}^2=\sigma^2+\frac{\E\bigl\{\bigl(\bar{\beta}-\ST_{t_k}(\bar{\beta}+\sigma_k G_k)\bigr)^2\bigr\}}{\delta},\qquad&&t_{k+1}&=\lambda+\tilde{b}_kt_k=\lambda+\frac{t_k\,\Pr\bigl(\abs{\bar{\beta}+\sigma_kG_k}>t_k\bigr)}{\delta}
\end{alignat}
for $k\in\N$, where $\bar{\beta}\sim \pi_{\bar{\beta}}$ and $G_k\sim N(0,1)$ are independent, $(Z,Z_0)\sim N_2(0,\Sigma_0)$, and $\sigma^2>0$ is the second moment of $P_{\bar{\varepsilon}}$. 

Proceeding as in \textbf{Step 1} in Section~\ref{sec:GAMPopt}, we now seek a fixed point $(\hat{r}^*,\hat{\beta}^*,\tilde{b}_*,\sigma_*,t_*>0)$ of~\eqref{eq:GAMPlasso}--\eqref{eq:statevolasso} satisfying
\begin{alignat}{2}
\label{eq:lassofixed1}
\hat{r}^*&=y-X\hat{\beta}^*+\tilde{b}_*\hat{r}^*,\qquad\tilde{b}_*=\frac{t_*-\lambda}{t_*},\qquad&\hat{\beta}^*&=\ST_{t_*}\bigl(X^\top\hat{r}^*+\hat{\beta}^*\bigr),\\
\label{eq:lassofixed2}
\sigma_*^2&=\sigma^2+\frac{\E\bigl\{\bigl(\bar{\beta}-\ST_{t_*}(\bar{\beta}+\sigma_*G_*)\bigr)^2\bigr\}}{\delta},\qquad& t_*&=\lambda\,\biggl(1-\frac{\Pr(\abs{\bar{\beta}+\sigma_*G_*}>t_*)}{\delta}\biggr)^{-1},
\end{alignat}
where $\bar{\beta}\sim \pi_{\bar{\beta}}$ and $G_*\sim N(0,1)$ are independent. Noting that the condition~\eqref{eq:lassofixed1} simplifies to $\hat{\beta}^*=\ST_{t_*}\bigl(\hat{\beta}^*+t_*\lambda^{-1}X^\top(y-X\hat{\beta}^*)\bigr)$, we can either apply Proposition~\ref{prop:GAMPopt} or verify the Karush--Kuhn--Tucker (KKT) conditions directly to deduce that $\hat{\beta}^*$ is a Lasso solution satisfying~\eqref{eq:lasso}. 

The next task is to show that for any $\lambda>0$ in~\eqref{eq:lasso} and $\delta,\sigma>0$, there exist unique solutions $\sigma_*\equiv\sigma_*(\lambda,\delta,\sigma)>0$ and $t_*\equiv t_*(\lambda,\delta,\sigma)$ to the non-linear equations in~\eqref{eq:lassofixed2}. To this end,~\citet[Proposition~1.3]{BM12} first verified that for fixed $\alpha>0$, there is a unique $\tilde{\sigma}_\alpha\equiv\tilde{\sigma}_\alpha(\delta,\sigma)$ satisfying
\[\tilde{\sigma}_\alpha^2=\sigma^2+\frac{\E\bigl\{\bigl(\bar{\beta}-\ST_{\alpha\tilde{\sigma}_\alpha}(\bar{\beta}+\tilde{\sigma}_\alpha G_*)\bigr)^2\bigr\}}{\delta}\]
provided that 
\begin{equation}
\label{eq:upsilonalpha}
\upsilon(\alpha):=(1+\alpha^2)\Phi(-\alpha)-\alpha\phi(\alpha)<\frac{\delta}{2},
\end{equation}
where $\phi$ and $\Phi$ denote the standard Gaussian density and distribution functions respectively. Since $\upsilon\colon\R\to\R$ is a strictly decreasing continuous function with range $(0,\infty)$,~\eqref{eq:upsilonalpha} holds for all positive $\alpha>\upsilon^{-1}(\delta/2)$. 
In addition, some elementary calculus shows that for some $\alpha_0\equiv\alpha_0(\delta,\sigma)\geq\upsilon^{-1}(\delta/2)$, the map 
\[\alpha\mapsto\Lambda_{\delta,\sigma}(\alpha):=\alpha\tilde{\sigma}_\alpha\biggl(1-\frac{\Pr(\abs{\bar{\beta}+\tilde{\sigma}_\alpha G_*}>\alpha\tilde{\sigma}_\alpha)}{\delta}\biggr)\]
is a continuous bijection from $(\alpha_0,\infty)$ to $(0,\infty)$~\citep[Proposition~1.4 and Corollary~1.7]{BM12}, so that for any $\lambda>0$, there is a unique $\alpha_*\equiv\alpha_*(\lambda,\delta,\sigma)>\alpha_0$ such that $\lambda=\Lambda_{\delta,\sigma}(\alpha_*)$. It follows from this that $\sigma_*=\tilde{\sigma}_{\alpha_*}$ and $t_*=\alpha_*\sigma_*$ are the unique solutions to~\eqref{eq:lassofixed2}.

For $n\in\N$ and $p\equiv p_n$, the resulting `stationary' AMP iteration~\eqref{eq:GAMPstat} in \textbf{Step 2} in Section~\ref{sec:GAMPopt} takes the form
\begin{equation}
\label{eq:GAMPlassostat}
\hat{r}^k=y-X\hat{\beta}^k+\tilde{b}_*\hat{r}^{k-1},\qquad\hat{\beta}^{k+1}=\ST_{t_*}\bigl(X^\top\hat{r}^k+\hat{\beta}^k\bigr)\qquad\text{for }k\in\N_0,
\end{equation}
where $\hat{r}^{-1}=0\in\R^n$, $\tilde{b}_*=\delta^{-1}\,\Pr(\abs{\bar{\beta}+\sigma_*G_*}>t_*)$, and $\hat{\beta}^0=\ST_{t_*}(\beta+\sigma_*\xi)\in\R^p$ is an oracle initialiser with $\xi\sim N_p(0,I_p)$ taken to be independent of the signal $\beta\in\R^p$. Under the hypotheses of Theorem~\ref{thm:GAMPmaster}, it follows from Remark~\ref{rem:GAMPbk} and~\eqref{eq:GAMPesterr} that for each fixed $k\in\N_0$, the empirical distribution of the entries of $\hat{\beta}^k\equiv\hat{\beta}^k(n)$ converges completely in $d_2$ to the distribution of $\ST_{t_*}(\bar{\beta}+\sigma_*G_*)$ as $n,p\to\infty$ with $n/p\to\delta$. 

Theorem~\ref{thm:lasso} below asserts that the same asymptotic conclusion holds for the fixed point $\hat{\beta}^*$ of~\eqref{eq:GAMPlassostat}, which is a Lasso solution by virtue of~\eqref{eq:lassofixed1}. The additional technical challenge in its proof is to show that the AMP iterates $\hat{\beta}^k$ in~\eqref{eq:GAMPlassostat} actually converge to a fixed point in the sense of~\eqref{eq:GAMPconv}, when we take $n,p\to\infty$ followed by $k\to\infty$~\citep[Theorem~1.8]{BM12}.\footnote{\citet{BM12} originally established this result for a AMP recursion~\eqref{eq:GAMPlasso} initialised with $\hat{\beta}^0=0$, in which the thresholds are defined instead by $t_{k+1}=\alpha_*\sigma_{k+1}$ in~\eqref{eq:statevolasso} with $\alpha_*=\alpha_*(\lambda,\delta,\sigma)$ as above, and the state evolution sequence $(\sigma_k)$ is non-constant but converges to $\sigma_*$. Their analysis yields the same conclusion for~\eqref{eq:GAMPlassostat}, and also shows that~\eqref{eq:lassoactive} holds even though $\psi\colon (u,v)\mapsto\Ind_{\{u\neq 0\}}$ is discontinuous.} This constitutes \textbf{Step 3} in Section~\ref{sec:GAMPopt}, and as mentioned there, the arguments involved turn out to be highly non-trivial in this case because the Lasso objective function in~\eqref{eq:lasso} is not strongly convex.
\begin{theorem}[{\citealp[Theorem~1.5]{BM12}}]
\label{thm:lasso}
Consider a sequence of linear models $y=X\beta+\varepsilon$ satisfying~\emph{\ref{ass:G0}} and \emph{\ref{ass:G1}} for $r=2$ as $n,p\to\infty$ with $n/p\to\delta\in (0,\infty)$. Suppose that the limiting prior distribution $\pi_{\bar{\beta}}$ satisfies $\pi_{\bar{\beta}}(\{0\})>0$, so that an asymptotically non-vanishing proportion of the entries of $\beta\in\R^p$ are equal to 0. For $\lambda>0$, let $\hat{\beta}^{\mathrm L}\equiv\hat{\beta}_\lambda^{\mathrm L}\in\R^p$ be a Lasso estimator~\eqref{eq:lasso} for each $p\equiv p_n$, and let $\sigma_*\equiv\sigma_*(\lambda,\delta,\sigma)>0$ and $t_*\equiv t_*(\lambda,\delta,\sigma)>0$ be the unique solutions to~\eqref{eq:lassofixed2}. Then
\begin{equation}
\label{eq:lassoest}
\sup_{\psi\in\PL_2(2,1)}\;\biggl|\frac{1}{p}\sum_{j=1}^p\psi(\hat{\beta}_j^{\mathrm L},\beta_j)-\E\bigl\{\psi\bigl(\ST_{t_*}(\bar{\beta}+\sigma_*G),\bar{\beta}\bigr)\bigr\}\biggr|\cvc 0 
\end{equation}
as $n,p\to\infty$ with $n/p\to\delta$, where $\bar{\beta}\sim \pi_{\bar{\beta}}$ is independent of $G\sim N(0,1)$. In particular, the asymptotic mean squared error of the Lasso estimator is given by
\begin{equation}
\label{eq:lassoamse}
\clim_{p\to\infty}\frac{\norm{\hat{\beta}_\lambda^{\mathrm L}-\beta}^2}{p}=\E\bigl\{\bigl(\bar{\beta}-\ST_{t_*}(\bar{\beta}+\sigma_*G)\bigr)^2\bigr\}=\delta(\sigma_*^2-\sigma^2).
\end{equation}
\end{theorem}

\unparskip
We emphasise once again the complex, non-linear dependence of $\sigma_*$ in~\eqref{eq:lassoest} on the asymptotic sparsity level $\pi_{\bar{\beta}}(\{0\})$ and $\lambda,\delta,\sigma>0$ through~\eqref{eq:lassofixed2}, and also the fact the asymptotic guarantees of Theorem~\ref{thm:lasso} hold for a \emph{fixed} value of the regularisation parameter $\lambda>0$.~\citet{MMB18} showed that the asymptotic mean squared error of $\hat{\beta}_\lambda^{\mathrm{L}}$ in~\eqref{eq:lassoamse} is a quasi-convex function of $\lambda$ (i.e.\ decreasing on $(0,\lambda^*]$ and increasing on $[\lambda^*,\infty)$ for some $\lambda^*>0$), and moreover that 
\begin{equation}
\label{eq:lassoactive}
\clim_{p\to\infty}\frac{\norm{\hat{\beta}_\lambda^{\mathrm{L}}}_0}{p}\equiv\clim_{p\to\infty}\frac{1}{p}\sum_{j=1}^p\Ind_{\{\hat{\beta}_j^{\mathrm{L}}\neq 0\}}=\Pr\bigl(\ST_{t_*}(\bar{\beta}+\sigma_*G_*)\neq 0\bigr)=\Pr(\abs{\bar{\beta}+\sigma_*G_*}>t_*)=\delta\tilde{b}_*(\lambda,\delta,\sigma)
\end{equation}
is a decreasing function of $\lambda$, as might be intuitively expected.

When the Lasso is used to perform variable selection (possibly with an adaptive choice of $\lambda$),~\citet{su2017false} established a tradeoff between the false discovery proportion and false negative proportion along the regularisation path $\lambda\mapsto\hat{\beta}_\lambda^{\mathrm{L}}$ in the high-dimensional asymptotic regime above. To this end, by extending the results of~\citet{BM12}, they proved that these two quantities converge \emph{uniformly} to deterministic limits over $\lambda\in [\lambda_{\min},\lambda_{\max}]$, for any $0<\lambda_{\min}<\lambda_{\max}$.
\begin{remark}
The SLOPE estimator~\citep{bogdan2015slope,su2016slope,bellec2018slope} is a generalisation of the Lasso that solves a regularised least squares problem in which the penalty is a \emph{sorted} $\ell_1$-norm: for $\lambda_1\geq\lambda_2 \geq \cdots\geq\lambda_p\geq 0$, define
\begin{equation}
\label{eq:slope}
\hat{\beta}^{\mathrm{SLOPE}}(\lambda_1,\dotsc,\lambda_p)\in\argmin_{\tilde{\beta}\in \R^p}\,\biggl\{\frac{1}{2}\norm{y-X\tilde{\beta}}^2+\sum_{j=1}^p\lambda_j\abs{\tilde{\beta}}_{(j)}\biggr\},
\end{equation}
where $\abs{\tilde{\beta}}_{(1)}\geq\abs{\tilde{\beta}}_{(2)}\geq\dotsc\geq\abs{\tilde{\beta}}_{(p)}$ are the absolute values of the entries of $\tilde{\beta}$ arranged in decreasing order. This is a convex optimisation problem that produces sparse solutions like the Lasso, but offers more flexibility due to the choices available for $\lambda_1,\dotsc,\lambda_p$. For example, SLOPE can be used to control the false discovery rate in variable selection via a judicious choice of these regularisation parameters. Note however that when the $\lambda_j$ are distinct, the optimisation problem~\eqref{eq:slope} is not of the form~\eqref{eq:constropt} since the SLOPE penalty is not an additively separable function of the components of $\tilde{\beta}$. Consequently, the GAMP construction~\eqref{eq:GAMPopt} in Section~\ref{sec:GAMPopt} is not applicable to this setting.

Nevertheless,~\citet{bu2019slope} show that an appropriately tuned AMP algorithm converges to the SLOPE solution in the sense of~\eqref{eq:GAMPconv}, under assumptions similar to those for Theorem~\ref{thm:lasso}. This AMP iteration for SLOPE is somewhat similar to that for the Lasso, the main difference being that the soft-thresholding function in~\eqref{eq:GAMPlasso} is replaced by the proximal operator associated with the SLOPE penalty. This proximal operator is non-separable (i.e.\ does not act componentwise on its vector input), which is why the analysis is based on master theorems recently obtained by~\cite{BMN20} for AMP recursions with non-separable denoising functions.
\end{remark}

\umparskip
\subsection{AMP for M-estimation in the linear model}
\label{sec:Mestlin}
Consider again the linear model $y=X\beta+\varepsilon$ from Section~\ref{sec:AMPlinear}, and define an M-estimator of $\beta\in\R^p$ by
\begin{equation}
\label{eq:betaM}
\hat{\beta}^{\mathrm{M}}\in\argmin_{\tilde{\beta}\in\R^p}\sum_{i=1}^n\,\mathrm{M}(y_i-x_i^\top\tilde{\beta})
\end{equation}
for some convex $\mathrm{M}\colon\R\to\R$ that is bounded below. The existence of $\hat{\beta}^{\mathrm{M}}$ is guaranteed if for example $\mathrm{M}$ is strongly convex. If $\varepsilon_1,\dotsc,\varepsilon_n\iid f_{\bar{\varepsilon}}$ for some known (strictly positive log-concave) density $f_{\bar{\varepsilon}}$, then taking $\mathrm{M}=-\log f_{\bar{\varepsilon}}$ in~\eqref{eq:betaM} yields a maximum likelihood estimator of $\beta$; see~\citet[Section~3]{DSS11} for a maximum likelihood approach to estimating $\beta$ when $f_{\bar{\varepsilon}}$ is unknown. Other popular choices of~$\mathrm{M}$ include squared error loss $w\mapsto w^2$, Huber loss $w\mapsto w^2\Ind_{\{\abs{w}\leq B\}}+(2\abs{w}-B)B\Ind_{\{\abs{w}>B\}}$ (for robust regression) with $B>0$, and quantile loss $w\mapsto\tau w-\Ind_{\{w<0\}}$ (for quantile regression) with $\tau\in(0,1)$. In a classical setting where the dimension $p$ is fixed, $x_1,\dotsc,x_n\iid P_X$ on $\R^p$ and $\varepsilon_1,\dotsc,\varepsilon_n\iid P_{\bar{\varepsilon}}$ on $\R$ for all $n$,~\citet{Hub64,Hub73} proved that
\begin{equation}
\label{eq:huber}
\sqrt{n}(\hat{\beta}^{\mathrm{M}}-\beta)\cvd N_p(0,\Sigma^{\mathrm{M}})\quad\text{as }n\to\infty,\quad\text{with}\quad\Sigma^{\mathrm{M}}:=\frac{\int_\R(\mathrm{M}')^2\,dP_{\bar{\varepsilon}}}{\bigl(\int_\R\mathrm{M}''\,dP_{\bar{\varepsilon}}\bigr)^2}\,\biggl(\int_{\R^p}xx^\top\,dP_X\biggr)^{-1},
\end{equation}
under appropriate regularity conditions on $\mathrm{M}$ and the \emph{score function} $\mathrm{S}:=\mathrm{M'}$; see also~\citet{HR09} and~\citet[Example~5.28]{vdV98}. When $\bar{\varepsilon}\sim P_{\bar{\varepsilon}}$ has a differentiable density $f_{\bar{\varepsilon}}$, it follows from the Cauchy--Schwarz inequality that the variance functional 
\begin{equation}
\label{eq:hubervar}
V(\mathrm{S};\bar{\varepsilon}):=\frac{\E\bigl(\mathrm{S}(\bar{\varepsilon})^2\bigr)}{\E\bigl(\mathrm{S}'(\bar{\varepsilon})\bigr)^2}=\frac{\int_\R(\mathrm{M}')^2\,dP_{\bar{\varepsilon}}}{\bigl(\int_\R\mathrm{M}''\,dP_{\bar{\varepsilon}}\bigr)^2}
\end{equation}
that appears in~\eqref{eq:huber} is bounded below by the Fisher information $I(P_{\bar{\varepsilon}}):=\int_\R\,(f_{\bar{\varepsilon}}'/f_{\bar{\varepsilon}})^2\,dP_{\bar{\varepsilon}}$, with equality when $\mathrm{M}=-\log f_{\bar{\varepsilon}}$ (in which case the maximum likelihood estimator $\hat{\beta}^{\mathrm{M}}$ is asymptotically efficient). 

In contrast to~\eqref{eq:huber},~\citet{donohoMest16} showed that the M-estimator $\hat{\beta}^{\mathrm{M}}$ suffers from variance inflation (and cannot be asymptotically efficient) in high-dimensional regimes where $n,p\to\infty$ with $n/p\to\delta\in (1,\infty)$. AMP machinery plays a pivotal role in the analysis that leads to their main result (stated as Theorem~\ref{thm:Mest} below), and as in Section~\ref{sec:lasso}, we will now present the main steps within the context of the GAMP framework of Sections~\ref{sec:GAMPmaster} and~\ref{sec:GAMPopt}.

Observing that the convex optimisation problem in~\eqref{eq:betaM} is an instance of~\eqref{eq:constropt} with $\ell\colon (u,v)\mapsto\mathrm{M}(u-v)$ and $J\equiv 0$ (i.e.\ no penalty term), we first write down an associated GAMP algorithm~\eqref{eq:GAMPMest} based on the general construction in Section~\ref{sec:GAMPopt}. For $\eta>0$, define a `smoothed' version of $\eta\mathrm{M}$ by
\[\mathrm{M}_\eta(z):=\min_{t\in\R}\,\Bigl\{\eta\mathrm{M}(t)+\frac{1}{2}(t-z)^2\Bigr\}\]
for $z\in\R$. (The function $\eta^{-1}\mathrm{M}_\eta$ is called a \emph{Moreau envelope} of $\mathrm{M}$.) We note here that $\prox_{\eta\mathrm{M}}(z)$ in~\eqref{eq:prox} is the unique $t$ that achieves this minimum for each $z\in\R$, and also that $\mathrm{M}_\eta$ is convex and differentiable with $\mathrm{S}_\eta(z):=(\mathrm{M}_\eta)'(z)=z-\prox_{\eta\mathrm{M}}(z)$ for all $z$; see for example~\citet[Theorem~31.5]{Rock97} and~\citet[Section~3.2]{PB13}. Moreover, $\mathrm{S}_\eta$ is non-decreasing and 1-Lipschitz~\citep[cf.][Sections~2.3 and~3.1]{PB13}. 

For $k\in\N_0$, the functions $\bar{g}_k,g_k,f_{k+1}$ in~\eqref{eq:gkbar}--\eqref{eq:fkbar} are given by
\[\bar{g}_k(u,v)=v-\prox_{\bar{b}_k\mathrm{M}}(v-u)=u+\mathrm{S}_{\bar{b}_k}(v-u),\qquad g_k(u,v)=\frac{\mathrm{S}_{\bar{b}_k}(v-u)}{\bar{b}_k},\qquad f_{k+1}(w)=-\frac{w}{\bar{c}_k}.\]
Given $\bar{b}_0>0$, $\hat{\beta}^0\in\R^p$ and $\hat{r}^0:=y-X\hat{\beta}^0$, we now write the GAMP algorithm~\eqref{eq:GAMPopt} in terms of $\hat{r}^k=y-\theta^k$ and $\hat{\beta}^{k+1}=f_{k+1}(\beta^{k+1})=-\beta^{k+1}/\bar{c}_k$, and obtain the recursion
\begin{equation}
\label{eq:GAMPMest}
\hat{\beta}^{k+1}=\frac{\delta\bar{b}_{k+1}}{\bar{b}_k}X^\top\mathrm{S}_{\bar{b}_k}(\hat{r}^k)+\hat{\beta}^k,\qquad\hat{r}^{k+1}=y-X\hat{\beta}^{k+1}+\frac{\bar{b}_{k+1}}{\bar{b}_k}\,\mathrm{S}_{\bar{b}_k}(\hat{r}^k),\qquad
\end{equation} 
for $k\in\N_0$, where $\bar{b}_{k+1}=-1/(\delta\bar{c}_k)$. Moreover, expressing the state evolution recursion~\eqref{eq:GAMPstatevol1}--\eqref{eq:GAMPstatevol2} for~\eqref{eq:GAMPopt} in terms of $\tilde{\mu}_k:=\delta\bar{b}_k\mu_k$, $\tilde{\sigma}_k:=\delta\bar{b}_k\sigma_k$ and $\tau_k:=\E\bigl((Z-Z_k)^2\bigr)^{1/2}$ with $(Z,Z_k)\sim N(0,\Sigma_k)$, we have $\tau_0=\E\bigl((Z-Z_0)^2\bigr)^{1/2}$ and
\begin{equation}
\label{eq:statevolMest}
\begin{alignedat}{2}
\tilde{\mu}_{k+1}&=\delta\,\E\bigl(\mathrm{S}_{\bar{b}_k}'(\bar{\varepsilon}+\tau_k G_k)\bigr),\qquad&\tilde{\sigma}_{k+1}^2&=\delta^2\,\E\bigl(\mathrm{S}_{\bar{b}_k}(\bar{\varepsilon}+\tau_k G_k)^2\bigr),\\
\bar{b}_{k+1}&=-\frac{1}{\delta\bar{c}_k}=\frac{\bar{b}_k}{\delta\,\E\bigl(\mathrm{S}_{\bar{b}_k}'(\bar{\varepsilon}+\tau_k G_k)\bigr)},\qquad&\tau_{k+1}^2&=\frac{\E(\bar{\beta}^2)(\tilde{\mu}_{k+1}-1)^2+\tilde{\sigma}_{k+1}^2}{\delta}
\end{alignedat}
\end{equation}
for $k\in\N_0$, where $\bar{\varepsilon}\sim P_{\bar{\varepsilon}}$ is independent of $G_k\sim N(0,1)$.

Turning now to \textbf{Step 1} in Section~\ref{sec:GAMPopt}, we seek a fixed point $(\hat{r}^*,\hat{\beta}^*,\bar{b}_*>0,\tilde{\mu}_*,\tilde{\sigma}_*,\tau_*)$ of~\eqref{eq:GAMPMest}--\eqref{eq:statevolMest} satisfying
\begin{alignat}{2}
\label{eq:Mestfixed}
0&=\delta X^\top\mathrm{S}_{\bar{b}_*}(\hat{r}^*),\qquad&\hat{r}^*&=y-X\hat{\beta}^*+\mathrm{S}_{\bar{b}_*}(\hat{r}^*),\\
\label{eq:statevolMest1}
\tilde{\mu}_*&=\delta\,\E\bigl(\mathrm{S}_{\bar{b}_*}'(\bar{\varepsilon}+\tau_*G_*)\bigr)=1,&\qquad\tau_*^2&=\delta\,\E\bigl(\mathrm{S}_{\bar{b}_*}(\bar{\varepsilon}+\tau_*G_*)^2\bigr),\qquad\tilde{\sigma}_*=\sqrt{\delta}\tau_*,
\end{alignat}
where $\bar{\varepsilon}\sim P_{\bar{\varepsilon}}$ is independent of $G_*\sim N(0,1)$, and $\mu_*=\tilde{\mu}_*/(\delta\bar{b}_*)$ and $\sigma_*=\tilde{\sigma}_*/(\delta\bar{b}_*)$ are fixed points of the original state evolution equation~\eqref{eq:GAMPstatevol1}. By Proposition~\ref{prop:GAMPopt}, $\hat{\beta}^*$ solves the M-estimation problem in~\eqref{eq:betaM}. Assuming that
\begin{equation}
\label{eq:Mestsmooth}
\mathrm{M}\text{ is continuously differentiable and }\mathrm{S}=\mathrm{M}'\text{ is absolutely continuous with }\sup_{w\in\R}\mathrm{S}'(w)<\infty,
\end{equation}
\citet[Lemma~6.5]{donohoMest16} showed that for any $\tau>0$, the map $b\mapsto\E\bigl(\mathrm{S}_b'(\bar{\varepsilon}+\tau G_*)\bigr)=:F_\tau(b)$ is continuous on $(0,\infty)$ with $\lim_{b\to 0}F_\tau(b)=0$ and $\lim_{b\to\infty}F_\tau(b)=1$, and hence that there exists $b\equiv b_\tau>0$ satisfying $\E\bigl(\mathrm{S}_b'(\bar{\varepsilon}+\tau G_*)\bigr)=\delta^{-1}$ for $\delta\in (1,\infty)$. 
Using this, they deduced that under~\eqref{eq:Mestsmooth}, there exists a unique solution $(\tau_*,\bar{b}_*)$  to~\eqref{eq:statevolMest1} for any such $\delta$~\citep[Corollary~4.4]{donohoMest16}. 

The functions $f_*,g_*$ in \textbf{Step 2} in Section~\ref{sec:GAMPopt} are given by $f_*\colon w\mapsto\delta\bar{b}_*w$ and $g_*\colon (u,v)\mapsto\mathrm{S}_{\bar{b}_*}(v-u)/\bar{b}_*$, so for $n\in\N$ and $p\equiv p_n$, the `stationary' AMP iteration~\eqref{eq:GAMPstat} can be written as
\begin{equation}
\label{eq:GAMPMeststat}
\hat{\beta}^{k+1}=X^\top\mathrm{S}_{\bar{b}_*}(\hat{r}^k)+\hat{\beta}^k,\qquad\hat{r}^{k+1}=y-X\hat{\beta}^{k+1}+\mathrm{S}_{\bar{b}_*}(\hat{r}^k)\qquad\text{for }k\in\N_0.
\end{equation}
Here, $\hat{r}^0=y-X\hat{\beta}^0$ and $\hat{\beta}^0=\beta+\tilde{\sigma}_*\xi=f_*(\mu_*\beta+\sigma_*\xi)\in\R^p$, where $\xi\sim N_p(0,I_p)$ is independent of the signal $\beta\in\R^p$. 
This choice of oracle initialiser ensures that the corresponding state evolution sequence is stationary with $\tau_k=\tau_*$ for all $k\in\N_0$. Then under the conditions~\ref{ass:G0}--\ref{ass:G5} of Theorem~\ref{thm:GAMPmaster} with $r=2$, it follows from Remark~\ref{rem:GAMPbk} and~\eqref{eq:GAMPbeta} that for each fixed $k\in\N$, the empirical distributions of the components of $\hat{r}^k-\varepsilon\in\R^n$ and $\hat{\beta}^k-\beta\in\R^p$ converge completely in $d_2$ to $N(0,\tau_*^2)$ and $N(0,\tilde{\sigma}_*^2)=N(0,\delta\tau_*^2)$ respectively as $n,p\to\infty$ with $n/p\to\delta\in (1,\infty)$. 

We remark that this result can in fact be derived by directly transforming~\eqref{eq:GAMPMeststat} into an abstract asymmetric AMP iteration of the form~\eqref{eq:AMPnonsym}. Note in particular that since $h(z,v)=z+v$ in~\eqref{eq:glm1} for the linear model and $\mathrm{S}_\eta$ is 1-Lipschitz for all $\eta>0$, the function $\tilde{g}_k=\tilde{g}_*\colon (z,u,v)\mapsto\mathrm{S}_{\bar{b}_*}(z+v-u)/\bar{b}_*$ in~\ref{ass:G4} is indeed Lipschitz.

As in Section~\ref{sec:lasso}, the remaining ingredient (\textbf{Step 3} in Section~\ref{sec:GAMPopt}) is to show that the iterates $\hat{\beta}^k$ in~\eqref{eq:GAMPMeststat} converge in the sense of~\eqref{eq:GAMPconv} to some $\hat{\beta}^*$ satisfying~\eqref{eq:Mestfixed}, which is an M-estimator by Proposition~\ref{prop:GAMPopt}. Under~\eqref{eq:Mestsmooth} and the additional assumption that $\mathrm{M}$ is strongly convex, i.e.\ $\inf_{w\in\R}\mathrm{S}'(w)>0$, the conclusion of~\citet[Theorem~4.1]{donohoMest16} is indeed that
\begin{equation}
\label{eq:Mestconv2}
\lim_{k\to\infty}\clim_{p\to\infty}\frac{\norm{\hat{\beta}^k-\hat{\beta}^*}^2}{p}=0.
\end{equation}
Together with the state evolution characterisation of the iterates in~\eqref{eq:GAMPMeststat}, this leads to the following characterisation of the asymptotic performance of the M-estimator. 
\begin{theorem}[{\citealp[Theorem~4.2]{donohoMest16}}]
\label{thm:Mest}
Consider a sequence of linear models $y=X\beta+\varepsilon$ satisfying~\emph{\ref{ass:G0}} and~\emph{{\ref{ass:G1}}}, with $n/p\to\delta\in (1,\infty)$ as $n,p\to\infty$. Assume that the loss function $\mathrm{M}$ is continuously differentiable, and that the score function $\mathrm{S}=\mathrm{M}'$ is absolutely continuous with $0<\inf_{w\in\R}\mathrm{S}'(w)\leq\sup_{w\in\R}\mathrm{S}'(w)<\infty$. Let $(\tau_*,\bar{b}_*)$ be the unique fixed point of~\eqref{eq:statevolMest1}. Then
\begin{equation}
\label{eq:Mest}
\sup_{\psi\in\PL_2(2,1)}\;\biggl|\frac{1}{p}\sum_{j=1}^p\psi(\hat{\beta}_j^{\mathrm{M}}-\beta_j,\beta_j)-\E\bigl(\psi(\sqrt{\delta}\tau_*G,\bar{\beta})\bigr)\biggr|\cvc 0
\end{equation}
as $n,p\to\infty$ with $n/p\to\delta$, where $G\sim N(0,1)$. In particular, the asymptotic mean squared error of $\hat{\beta}^{\mathrm{M}}$ is given by
\begin{equation}
\label{eq:Mestmse}
\clim_{p\to\infty}\frac{\norm{\hat{\beta}^{\mathrm{M}}-\beta}^2}{p}=V(\mathrm{S}_{\bar{b}_*};\bar{\varepsilon}+\tau_*G)=\frac{\E\bigl(\mathrm{S}_{\bar{b}_*}(\bar{\varepsilon}+\tau_*G)^2\bigr)}{\E\bigl(\mathrm{S}_{\bar{b}_*}'(\bar{\varepsilon}+\tau_*G)\bigr)^2}=\frac{\tau^2_*/\delta}{1/\delta^{2}}=\delta\tau_*^2.
\end{equation}
\end{theorem}

\unparskip
Under condition~\ref{ass:G1} on the signal vectors $\beta\in\R^p$, Theorem~\ref{thm:Mest} provides the limiting joint empirical distribution of the entries of $\hat{\beta}^{\mathrm{M}},\beta\in\R^p$. It turns out that even in the absence of~\ref{ass:G1}, we have
\[\sup_{\psi\in\PL_1(2,1)}\;\biggl|\frac{1}{p}\sum_{j=1}^p\psi(\hat{\beta}_j^{\mathrm{M}}-\beta_j)-\E\bigl(\psi(\sqrt{\delta}\tau_*G)\bigr)\biggr|\cvc 0,\]
as evidenced by the fact that $\bar{\beta}$ does not appear in the state evolution recursion~\eqref{eq:statevolMest1}. Comparing the variance functional $V(\mathrm{S}_{\bar{b}_*};\bar{\varepsilon}+\tau_*G)$ in~\eqref{eq:Mestmse} with that in the classical setting, namely $V(\mathrm{S};\bar{\varepsilon})$ in~\eqref{eq:hubervar}, we emphasise the following points of difference. First, the asymptotic variance in the high-dimensional setting depends on $\mathrm{S}_{\bar{b}_*}=\mathrm{M}_{\bar{b}_*}'$, the score function of a regularised version of $\mathrm{M}$ (rather than $\mathrm{M}$ itself). In addition, the `effective noise' in the high-dimensional regime is $\bar{\varepsilon}+\tau_* G$, rather than $\bar{\varepsilon}$. In fact, 
\begin{equation}
\label{eq:Mesthidim}
V(\mathrm{S}_{\bar{b}_*};\bar{\varepsilon}+\tau_*G)\geq\frac{1}{1-\delta^{-1}}\cdot\frac{1}{I(P_{\bar{\varepsilon}})}
\end{equation}
by~\citet[Corollary~4.3]{donohoMest16}, where $I(P_{\varepsilon})^{-1}$ is the classical lower bound. This shows that the M-estimator is inefficient in high dimensions, particularly so when $\delta$ is close to 1.

We also mention that~\citet[Theorem~2.2]{DM15} extended the conclusion~\eqref{eq:Mest} to M-estimators defined with respect to the Huber loss function, which is not strongly convex on $\R$ and hence is not covered by Theorem~\ref{thm:Mest}.~\citet[Section~6]{donohoMest16} noted an interesting connection between the Lasso and Huber M-estimators, as a special case ($J\colon w\mapsto\lambda\abs{w}$) of a duality relationship between the following optimisation problems: 

\unparskip
\begin{enumerate}[label=(\roman*)]
\item The regularised least squares problem
\[\text{minimise}\quad\frac{1}{2}\norm{\breve{y}-\breve{X}\breve{\beta}}^2+\sum_{j=1}^p J(\breve{\beta}_j)\quad\text{over }\breve{\beta}\in\R^n\]
based on $\breve{X}\in\R^{(n-p)\times n}$ and $\breve{y}\in\R^{n-p}$, with convex penalty $J\colon\R\to\R$;
\item The unpenalised M-estimation problem~\eqref{eq:betaM} based on $X\in\R^{n\times p}$, $y\in\R^n$ satisfying 
$\breve{X}X=0$ and $\breve{y}=\breve{X}y$, with convex loss function $\mathrm{M}\colon w\mapsto\min_{z\in\R}\,\{J(z)+(z-w)^2/2\}$.
\end{enumerate}

\umparskip
\subsection{GAMP for logistic regression}
\label{sec:logistic}
To further illustrate the generality and utility of the GAMP framework, we will now demonstrate how it can be applied to a popular non-linear GLM, namely the logistic regression model with canonical logit link. Suppose that we observe $(x_1,y_1),\dotsc,(x_n,y_n)\in\R^p\times\{0,1\}$ with
\begin{equation}
\label{eq:logistic}
\Pr(y_i=1\,|\,x_i^\top\beta)=\frac{e^{x_i^\top\beta}}{1+e^{x_i^\top\beta}}=\zeta'(x_i^\top\beta),\quad\text{where}\;\;\zeta(z):=\log(1+e^z)
\end{equation}
for $1\leq i\leq n$. Equivalently, we may view this as an instance of the model~\eqref{eq:glm1} with $\varepsilon_1,\dotsc,\varepsilon_n\iid U[0,1]$ and $h(z,v)=\Ind_{\{v\leq\zeta'(z)\}}$, so that $y_i=h(x_i^\top\beta,\varepsilon_i)=\Ind_{\{\varepsilon_i\leq\zeta'(x_i^\top\beta)\}}$ for each $i$, and seek to estimate $\beta\in\R^p$ by maximum likelihood via
\begin{equation}
\label{eq:logmle}
\hat{\beta}^{\mathrm{MLE}}\in\argmin_{\tilde{\beta}\in\R^p}\sum_{i=1}^n\,\bigl\{\zeta(x_i^\top\tilde{\beta})-y_i x_i^\top\tilde{\beta}\bigr\},
\end{equation}
where the objective function in~\eqref{eq:logmle} is the negative log-likelihood.~\citet{AA84} showed that this MLE exists if and only if $\mathcal{X}_0:=\{x_i:1\leq i\leq n,\,y_i=0\}$ and $\mathcal{X}_1:=\{x_i:1\leq i\leq n,\,y_i=1\}$ are not (strongly) linearly separable, i.e.\ for any $\tilde{\beta}\neq 0$, there either exists $x_{i_0}\in\mathcal{X}_0$ with $x_{i_0}^\top\tilde{\beta}>0$ or $x_{i_1}\in\mathcal{X}_1$ with $x_{i_1}^\top\tilde{\beta}<0$. In the random design setting of~\ref{ass:G0} where $x_1,\dotsc,x_n\iid N_p(0,I_p/n)$ for each $n$ and $p\equiv p_n$,~\citet{candes2020} established a sharp phase transition for the existence of $\hat{\beta}^{\mathrm{MLE}}$. Specifically, they proved that there exists a decreasing function $s_{\mathrm{MLE}}\colon (0,\infty)\to [0,\infty)$ with the following property: if the signals $\beta\in\R^p$ are such that $n^{-1/2}\norm{\beta}\cvc\kappa\in (0,\infty)$ as $n,p\to\infty$ with $n/p\to\delta\in (1,\infty)$, then $\hat{\beta}^{\mathrm{MLE}}$ exists with probability tending to 0 if $\kappa>s_{\mathrm{MLE}}(1/\delta)$, and exists with probability tending to 1 if $\kappa<s_{\mathrm{MLE}}(1/\delta)$. 

Henceforth, we will restrict attention to the latter regime, and use the GAMP formalism in Sections~\ref{sec:GAMPmaster} and~\ref{sec:GAMPopt} to explain how to derive a result of~\citet{SC19a,SC19b} on the high-dimensional asymptotics of $\hat{\beta}^{\mathrm{MLE}}$, which is formally stated as Theorem~\ref{thm:logistic} below. Recall from~\eqref{eq:kappa} that for a sequence of logistic regression models~\eqref{eq:logistic} satisfying~\ref{ass:G1}, the asymptotic signal strength $\kappa^2=\clim_{n\to\infty}\norm{\beta}^2/n$ is equal to $\E(\bar{\beta}^2)/\delta$. 
Noting that $\hat{\beta}^{\mathrm{MLE}}$ in~\eqref{eq:logmle} solves a convex optimisation problem of the form~\eqref{eq:constropt} with $J\equiv 0$ and $\ell(u,v)=\zeta(u)-vu$, we see that the functions $\bar{g}_k,g_k,f_{k+1}$ in~\eqref{eq:gkbar}--\eqref{eq:fkbar} are given by
\begin{align}
\bar{g}_k(u,v)&=\prox_{\bar{b}_k\zeta}(u+\bar{b}_k v)=u+\bar{b}_kv-\bar{b}_k\zeta'\bigl(\prox_{\bar{b}_k\zeta}(u+\bar{b}_k v)\bigr),\notag\\
\label{eq:logfgkbar} 
g_k(u,v)&=v-\zeta'\bigl(\prox_{\bar{b}_k\zeta}(u+\bar{b}_k v)\bigr),\qquad f_{k+1}(w)=-\frac{w}{\bar{c}_k}
\end{align}
for $k\in\N_0$, since $b\zeta'(\prox_{b\zeta}(u))+\prox_{b\zeta}(u)-u=0$ by the definition of $\prox_{b\zeta}$ in~\eqref{eq:prox} for $b>0$. 

Given $\bar{b}_0>0$, $\hat{\beta}^0\in\R^p$ and $\theta^0:=X\hat{\beta}^0$, the GAMP recursion~\eqref{eq:GAMPopt} therefore takes the form
\begin{equation}
\label{eq:GAMPlog1}
\begin{split}
\hat{\beta}^{k+1}&=\delta\bar{b}_{k+1}X^\top\big\{y-\zeta'\bigl(\prox_{\bar{b}_k\zeta}(\theta^k+\bar{b}_k y)\bigr)\bigr\}+\frac{\bar{b}_{k+1}}{\bar{b}_k}\hat{\beta}^k,\\
\theta^{k+1}&=X\hat{\beta}^{k+1}-\bar{b}_{k+1}\bigl\{y-\zeta'\bigl(\prox_{\bar{b}_k\zeta}(\theta^k+\bar{b}_k y)\bigr)\bigr\}
\end{split}
\end{equation}
for $k\in\N_0$, where $\bar{b}_{k+1}=-1/(\delta\bar{c}_k)$. Using~\eqref{eq:GAMPstein} from Lemma~\ref{lem:GAMPstein}, as well as~\eqref{eq:proxderiv}, we now write the corresponding state evolution recursion~\eqref{eq:GAMPstatevol1}--\eqref{eq:GAMPstatevol2} for~\eqref{eq:GAMPlog1} in terms of $\tilde{\mu}_k:=\delta\bar{b}_k\mu_k$ and $\tilde{\sigma}_k:=\delta\bar{b}_k\sigma_k$. This yields
\begin{align} \bar{b}_{k+1}&=\frac{\bar{b}_k}{\delta}\biggl(1-\E\biggl\{\frac{1}{1+\bar{b}_k\zeta''\bigl(\prox_{\bar{b}_k\zeta}(Z_k+\bar{b}_kY)\bigr)}\biggr\}\biggr)^{-1},\notag\\
\label{eq:statevolog}
\tilde{\mu}_{k+1}&=\frac{\delta^2\bar{b}_{k+1}}{\E(\bar{\beta}^2)}\,\E\Bigl(Z \Bigl\{Y-\zeta'\bigl(\prox_{\bar{b}_k\zeta}(Z_k+\bar{b}_kY)\bigr)\Bigr\}\Bigr)+\tilde{\mu}_k,\\
\tilde{\sigma}_{k+1}^2&=\delta^2\bar{b}_{k+1}^2\,\E\Bigl(\Bigl\{Y-\zeta'\bigl(\prox_{\bar{b}_k\zeta}(Z_k+\bar{b}_kY)\bigr)\Bigr\}^2\Bigr)\notag
\end{align}
for $k\in\N_0$, where given independent $Z\sim N(0,\E(\bar{\beta}^2)/\delta)$, $\tilde{G}_k\sim N(0,1)$ and $\bar{\varepsilon}\sim P_{\bar{\varepsilon}}$, we set \[Y=h(Z,\bar{\varepsilon})=\Ind_{\{\bar{\varepsilon}\leq\zeta'(Z)\}},\qquad Z_k=\mu_{Z,k}Z+\sigma_{Z,k}\tilde{G}_k=\tilde{\mu}_k Z+\delta^{-1/2}\,\tilde{\sigma}_k\tilde{G}_k\]
in view of~\eqref{eq:logistic},~\eqref{eq:zkmusigma} and the definition of $f_{k+1}$ in~\eqref{eq:logfgkbar}.~\citet[Section~3.1]{SC19b} showed that~\eqref{eq:statevolog} is equivalent to the original state evolution recursion they defined in~\citet[Section~4.1]{SC19a}.

In accordance with \textbf{Step 1} in Section~\ref{sec:GAMPopt}, we seek a fixed point $(\hat{\beta}^*,\theta^*,\tilde{\mu}_*,\tilde{\sigma}_*,\bar{b}_*>0)$ of~\eqref{eq:GAMPlog1}--\eqref{eq:statevolog} satisfying
\begin{alignat}{2}
\label{eq:GAMPlogfp}
\theta^*&=X\hat{\beta}^*-\bar{b}_*\bigl\{y-\zeta'\bigl(\prox_{\bar{b}_*\zeta}(\theta^*+\bar{b}_*y)\bigr)\bigr\},\quad\;\;&0&=X^\top\big\{y-\zeta'\bigl(\prox_{\bar{b}_*\zeta}(\theta^*+\bar{b}_*y)\bigr)\bigr\},\\
\label{eq:logfp1}
\tilde{\sigma}^2_*&=\delta^2\bar{b}_*^2\,\E\Bigl(\Bigl\{Y-\zeta'\bigl(\prox_{\bar{b}_*\zeta}(Z_*+\bar{b}_*Y)\bigr)\Bigr\}^2\Bigr),\quad\;\;&0&=\E\Bigl(Z \Bigl\{Y-\zeta'\bigl(\prox_{\bar{b}_*\zeta}(Z_*+\bar{b}_*Y)\bigr)\Bigr\}\Bigr),\\
\label{eq:logfp2}
1-\frac{1}{\delta}&=\E\,\biggl\{\frac{1}{1+\bar{b}_*\zeta''\bigl(\prox_{\bar{b}_*\zeta}(Z_*+\bar{b}_*Y)\bigr)}\biggr\},&&
\end{alignat}
where $Z_*:=\tilde{\mu}_*Z+\delta^{-1/2}\,\tilde{\sigma}_*\tilde{G}_*$ with $Z\sim N(0,\E(\bar{\beta}^2)/\delta)$ independent of $\tilde{G}_*\sim N(0,1)$. It turns out that there exists a unique solution $(\tilde{\mu}_*,\tilde{\sigma}_*,\bar{b}_*>0)$ to~\eqref{eq:logfp1}--\eqref{eq:logfp2} precisely when $\E(\bar{\beta}^2)/\delta\equiv\kappa^2<s_{\mathrm{MLE}}(1/\delta)^2$~\citep[Lemma~7 and Remark~1]{SC19b}, in which case $\hat{\beta}^{\mathrm{MLE}}$ exists with probability tending to 1. By Proposition~\ref{prop:GAMPopt}, $\hat{\beta}^*$ in~\eqref{eq:GAMPlogfp} is an MLE for $\beta$ in the logistic regression model. 

Proceeding as in \textbf{Step 2} in Section~\ref{sec:GAMPopt}, we can use the fixed points in~\eqref{eq:GAMPlogfp}--\eqref{eq:logfp2} to construct a stationary version of~\eqref{eq:GAMPlog1} based on $f_*\colon w\mapsto\delta\bar{b}_*w$ and $g_*\colon (u,v)\mapsto v-\zeta'\bigl(\prox_{\bar{b}_*\zeta}(u+\bar{b}_*v)\bigr)$. For each $n\in\N$ and $p\equiv p_n$, let $\hat{\beta}^0:=\tilde{\mu}_*\beta+\tilde{\sigma}_*\xi=f_*(\mu_*\beta+\sigma_*\xi)\in\R^p$ be an oracle initialiser with $\xi\sim N_p(0,I_p)$ taken to be independent of the signal $\beta\in\R^p$. Then setting $\theta^0=X\hat{\beta}^0$, we inductively define
\begin{equation}
\label{eq:GAMPlog2}
\hat{\beta}^{k+1}=\delta\bar{b}_*X^\top\big\{y-\zeta'\bigl(\prox_{\bar{b}_*\zeta}(\theta^k+\bar{b}_*y)\bigr)\bigr\}+\hat{\beta}^k,\qquad\theta^{k+1}=X\hat{\beta}^k-\bar{b}_*\bigl\{y-\zeta'\bigl(\prox_{\bar{b}_*\zeta}(\theta^k+\bar{b}_*y)\bigr)\bigr\}\\
\end{equation}
for $k\in\N_0$. By the choice of $\hat{\beta}^0$ above, the associated state evolution recursion~\eqref{eq:statevolog} is stationary, i.e.\ $\tilde{\mu}_k=\tilde{\mu}_*$ and $\tilde{\sigma}_k=\tilde{\sigma}_*$ for all $k\in\N_0$. Consequently, under the hypotheses of Theorem~\ref{thm:GAMPmaster} with $r=2$, it follows from Remark~\ref{rem:GAMPbk} that for each fixed $k\in\N$, the joint empirical distribution of the entries of $\hat{\beta}^k,\beta\in\R^p$ converges completely in $d_2$ to the distribution of $(\tilde{\mu}_*\bar{\beta}+\tilde{\sigma}_*G,\bar{\beta})$ as $n,p\to\infty$ with $n/p\to\delta$, where $\bar{\beta}\sim \pi_{\bar{\beta}}$ is independent of $G\sim N(0,1)$. 
On a technical note, we remark that the function 
\[\tilde{g}_*\colon (z,u,v)\mapsto g_*(u,h(z,v))=\Ind_{\{v\leq\zeta'(z)\}}-\zeta'\Bigl(\prox_{\bar{b}_*\zeta}\bigl(u+\bar{b}_*\Ind_{\{v\leq\zeta'(z)\}}\bigr)\Bigr)\]
in~\ref{ass:G4} is not Lipschitz since $h\colon(z,v)\mapsto\Ind_{\{v\leq\zeta'(z)\}}$ is not continuous, so an additional approximation argument is needed to formally justify the application of Theorem~\ref{thm:GAMPmaster}.

Finally, we discuss \textbf{Step 3} in Section~\ref{sec:GAMPopt}, whose aim is to show that the iterates in~\eqref{eq:GAMPlog2} converge in the sense of~\eqref{eq:GAMPconv} to a fixed point $\hat{\beta}^*\equiv\hat{\beta}^{\mathrm{MLE}}$ satisfying~\eqref{eq:GAMPlogfp}. This is the content of~\citet[Theorem~7]{SC19b}, and follows from similar arguments to those used by~\citet{donohoMest16} to prove~\eqref{eq:Mestconv2} for the M-estimators in Section~\ref{sec:Mestlin}. An additional technical obstacle in this setting is that $\zeta\colon z\mapsto\log(1+e^z)$ and hence the negative log-likelihood function in~\eqref{eq:logmle} are strongly convex on compact sets but not on the entirety of their domains. One way to address this issue is to show that $\hat{\beta}^k,\hat{\beta}^{\mathrm{MLE}}$ are contained in some sufficiently large Euclidean ball with overwhelming probability. Indeed, it follows from the state evolution characterisation of~\eqref{eq:GAMPlog2} that $\norm{\hat{\beta}^k}^2/p=O_c(1)$ for each fixed $k$; in addition,~\citet[Theorem~4]{SC19b} established the boundedness property $\norm{\hat{\beta}^{\mathrm{MLE}}}^2/p=O_c(1)$ in the regime $\kappa<s_{\mathrm{MLE}}(1/\delta)$ where $\hat{\beta}^{\mathrm{MLE}}$ exists with probability tending to 1. 
\begin{theorem}[{\citealp[Theorem~2]{SC19a}}]
\label{thm:logistic}
Consider a sequence of logistic regression models~\eqref{eq:logistic} satisfying~\emph{\ref{ass:G0}} and~\emph{\ref{ass:G1}} for $r=2$ as $n,p\to\infty$ with $n/p\to\delta\in (1,\infty)$. Assume that $\E(\bar{\beta}^2)/\delta\equiv\kappa^2<s_{\mathrm{MLE}}(1/\delta)^2$, so that~\eqref{eq:logmle} defines a maximum likelihood estimator $\hat{\beta}^{\mathrm{MLE}}$ with probability tending to 1, and there exist $\tilde{\mu}_*,\tilde{\sigma}_*,\bar{b}_*$ satisfying~\eqref{eq:logfp1}--\eqref{eq:logfp2}. Then
\[\sup_{\psi\in\PL_2(2,1)}\;\biggl|\frac{1}{p}\sum_{j=1}^p\psi\bigl(\hat{\beta}_j^{\mathrm{MLE}}-\tilde{\mu}_*\beta_j, \beta_j\bigr)-\E\bigl(\psi(\tilde{\sigma}_*G,\bar{\beta})\bigr)\biggr|\cvc 0\]
as $n,p\to\infty$ with $n/p\to\delta$, where $G\sim N(0,1)$ is independent of $\bar{\beta}\sim \pi_{\bar{\beta}}$. In particular,
\[\frac{1}{p}\sum_{j=1}^p\bigl(\hat{\beta}_j^{\mathrm{MLE}}-\tilde{\mu}_*\beta_j\bigr)\cvc 0,\quad\frac{1}{p}\sum_{j=1}^p\bigl(\hat{\beta}_j^{\mathrm{MLE}}-\tilde{\mu}_*\beta_j\bigr)^2\cvc\tilde{\sigma}_*^2,\quad\frac{\norm{\hat{\beta}^{\mathrm{MLE}}-\beta}^2}{p}\cvc (\tilde{\mu}_*-1)^2\,\E(\bar{\beta}^2)+\tilde{\sigma}_*^2.\]
\end{theorem}

\unparskip
Thus, for large $p$, the components of $\hat{\beta}^{\mathrm{MLE}}\in\R^p$ have approximately the same empirical distribution as those of $\tilde{\mu}_*\beta+\tilde{\sigma}_*\xi$ (the oracle initialiser $\hat{\beta}^0$ in~\eqref{eq:GAMPlog2} above), so we can interpret $\tilde{\mu}_*$ as an asymptotic bias factor and $\tilde{\sigma}_*^2$ as a limiting variance.~\citet{SC19a} observe empirically that 
when $n,p\to\infty$ with $\delta\in (1,\infty)$, both the limiting bias and variance are larger than they would be in classical settings where $p$ is fixed or grows sufficiently slowly with $n$ (in which case $\hat{\beta}^{\mathrm{MLE}}$ would be asymptotically unbiased (with $\tilde{\mu}_*=1$) and asymptotically efficient as $n\to\infty$). Their Figure~7 illustrates that this high-dimensional phenomenon becomes increasingly pronounced when either $\delta$ is reduced or $\kappa$ is enlarged; 
in fact, when $\kappa$ approaches the critical value $s_{\mathrm{MLE}}(1/\delta)$ for the existence of $\hat{\beta}^{\mathrm{MLE}}$, the value of $\tilde{\mu}_*$ diverges to infinity, as does the ratio between $\tilde{\sigma}_*$ and the Cram\'er--Rao lower bound.

It is instructive to compare the high-dimensional asymptotic performance of $\hat{\beta}^{\mathrm{MLE}}$ in the logistic model with that of the M-estimator~\eqref{eq:betaM} in the linear model. Note that while both estimators exhibit variance inflation (as quantified by Theorems~\ref{thm:Mest} and~\ref{thm:logistic}), only the former suffers from bias inflation. Indeed, in the linear model, the AMP state evolution recursion~\eqref{eq:statevolMest1} yields $\mu_k=1$ for all $k$, and hence $\mu_*=1$ (implicitly) in Theorem~\ref{thm:Mest} for the M-estimator; see also~\eqref{eq:statevolinear} in Section~\ref{sec:AMPlinear}.

\section{Conclusions}
\label{Sec:Conclusions}
With the abstract AMP recursions in Section~\ref{Sec:Master} as our starting point, we have shown how to design and analyse AMP algorithms for estimating structured signals, both in low-rank spiked models with Gaussian noise matrices and in GLMs with Gaussian design matrices. In high-dimensional asymptotic regimes where the matrix dimensions scale proportionally to each other, we have illustrated how to apply the abstract master theorems to derive precise state evolution characterisations of AMP estimation performance, which we have stated as complete convergence guarantees.

In Section~\ref{Sec:GAMP}, we have presented a general recipe that uses AMP systematically to obtain exact expressions for the asymptotic error of penalised and unpenalised M-estimators in GLMs with Gaussian design matrices. An alternative approach to deriving such guarantees is via Gaussian comparison inequalities and the convex Gaussian min-max theorem (CGMT); see for instance~\citet{TOH15,TAH18},~\citet{MM18} and~\citet{LS20} for applications of these techniques to regularised M-estimators, the Lasso and boosting respectively. 

Remaining within the realm of Gaussian matrices, we mention the results in this paper can be extended to AMP recursions with (i) non-separable denoising functions that do not act componentwise on their vector arguments, and can therefore take advantage of correlation between entries of the signal~\citep{MRB19,BMN20}; (ii) matrices with independent entries and a blockwise variance structure~\citep{JM13}. With a carefully chosen variance structure (`spatial coupling'), AMP has been shown to achieve the information-theoretic limit for compressed sensing~\citep{donohoJM_SC2013}.

In the setting of AMP for asymmetric matrices in Section~\ref{sec:AMPnonsym}, the results of Theorem~\ref{thm:AMPnonsym} can be generalised to matrices with i.i.d.\ sub-Gaussian entries with mild additional assumptions~\citep{bayati2015universality,CL20}. It is likely that the proof strategies in these papers can be developed further to extend other theoretical results (such as Theorem~\ref{thm:GAMPmaster} for GAMP) to these more general random matrix ensembles.

When the data matrix does not have i.i.d.\ Gaussian entries, AMP is not guaranteed to converge, and in fact can even diverge in sometimes pathological ways; see~\citet{rangan2019convergence} for a discussion of this issue.
For this reason, a number of other AMP-based algorithms have been introduced that allow for this assumption to be weakened in various ways, such as Vector AMP (VAMP)~\citep{ranganVAMP19}, orthogonal AMP (OAMP)~\citep{ma2017,Tak19} and other generalisations of AMP for rotationally invariant matrices~\citep{OCW16,Fan20}.

AMP has also been used to obtain lower bounds on the limiting estimation error of a broad class of general first-order methods such as gradient descent and mirror descent~\citep{CMW20}. An active area of current research is to determine whether AMP outperforms all other polynomial-time algorithms in low-rank matrix estimation and GLMs. In these settings, the statistical-computational gap has been precisely characterised in terms of the critical points of a `potential function'~\citep{lelarge2019fundamental,barbier2019optimal}. As mentioned in Section~\ref{sec:lowrankgk}, the performance of both Bayes-AMP and the Bayes optimal estimator correspond to (possibly different) critical points of this function, and when the potential function has a single critical point, Bayes-AMP achieves Bayes optimal performance. This connection suggests that AMP will play an important role in understanding statistical-computational gaps in a wider statistical context. 

\clearpage
\section{Appendix: proofs and technical remarks}
\label{Sec:Proofs}

In addition to the definitions in Section~\ref{sec:notation}, we introduce the following notation. The Moore--Penrose pseudoinverse of a matrix $A\in\R^{k\times\ell}$ will be denoted by $A^+\in\R^{\ell\times k}$. This satisfies $A^+=(A^\top A)^+A^\top$~\citep[e.g.][Proposition~3.2]{BH12}, and if $k=\ell$ and $A$ is invertible, then $A^+=A^{-1}$. For non-negative, real-valued functions $f,g$, we write $f \lesssim g$ if there exists a universal constant $C > 0$ such that $f \leq Cg$; more generally, given parameters $\alpha_1,\dotsc,\alpha_N$, we write $f\lesssim_{\alpha_1,\dotsc,\alpha_N}\!g$ if there exists $C\equiv C_{\alpha_1,\dotsc,\alpha_N}>0$, depending only on $\alpha_1,\dotsc,\alpha_N$, such that $f\leq Cg$.  

\hfparskip
\subsection{Technical remarks on the master theorems in Section~\ref{sec:AMPsym}}
\label{sec:AMPrem}
In this subsection, we will make some general observations that unify Theorems~\ref{thm:AMPmaster} and~\ref{thm:masterext} with other master theorems in the AMP literature~\citep[e.g.][]{Bol14,BM11,JM13}. There are a number of respects in which our results are presented differently and/or in slightly greater generality, and we discuss each of these in turn.
\begin{remark}[\emph{Complete convergence}]
\label{rem:AMPconv}
In Section~\ref{sec:inductive}, we will also establish the following variants of Theorem~\ref{thm:AMPmaster}, neither of which implies the other (or the original theorem): for a sequence of symmetric AMP recursions~\eqref{eq:AMPsym} satisfying~\ref{ass:A0},~\ref{ass:A4} and~\ref{ass:A5}, and an associated sequence of state evolution parameters $(\tau_k^2:k\in\N)$ as in~\eqref{eq:statevolsym}, the following hold for each fixed $k\in\N$ as $n\to\infty$:

\unparskip
\begin{enumerate}[label=(\alph*)]
\item Suppose that~\ref{ass:A1}--\ref{ass:A3} hold with $\cvp$ and $O_p(1)$ in place of $\cvc$ and $O_c(1)$ respectively. Then $d_r\bigl(\nu_n(h^k,\gamma),N(0,\tau_k^2)\otimes\pi\bigr)\cvp 0$, or equivalently $\widetilde{d}_r\bigl(\nu_n(h^k,\gamma),N(0,\tau_k^2)\otimes\pi\bigr)\cvp 0$.
\item Suppose instead that~\ref{ass:A1}--\ref{ass:A3} hold with $\cvas$ and $O_{a.s.}(1)$ in place of $\cvc$ and $O_c(1)$ respectively, and moreover that $\bigl(W(n):n\in\N\bigr)$ is independent of $\bigl(m^0(n),\gamma(n):n\in\N\bigr)$. Then $d_r\bigl(\nu_n(h^k,\gamma),N(0,\tau_k^2)\otimes \pi\bigr)\cvas 0$, or equivalently $\widetilde{d}_r\bigl(\nu_n(h^k,\gamma),N(0,\tau_k^2)\otimes \pi\bigr)\cvas 0$.
\end{enumerate}

\unparskip
Stronger versions of these statements can be formulated as analogues of Theorem~\ref{thm:masterext}. We now explain why we have stated our AMP master theorems (and all subsequent asymptotic results in the paper)
in terms of complete convergence.

\unparskip
\begin{itemize}[leftmargin=0.4cm]
\item Complete convergence is stronger than almost sure convergence and convergence in probability, so the conclusions of Theorems~\ref{thm:AMPmaster} and~\ref{thm:masterext} provide stronger convergence guarantees than (a) and (b).
\item In view of Remark~\ref{rem:compdist}, neither the conditions~\ref{ass:A0}--\ref{ass:A3} nor their analogues in (a) impose any restrictions on the dependence structure across $n\in\N$ of the random triples $\bigl(m^0(n),\gamma(n),W(n)\bigr)$ that generate the AMP iterates. 
By contrast, the additional assumption in (b) is somewhat unnatural from a statistical point of view, except perhaps when $\bigl(m^0(n),\gamma(n):n\in\N\bigr)$ is taken to be deterministic sequence that satisfies the other conditions in (b). Note however that this special case is covered by Theorems~\ref{thm:AMPmaster} and~\ref{thm:masterext}, which yield stronger conclusions than (b), as mentioned above.
\item The method of proof of Theorems~\ref{thm:AMPmaster} and~\ref{thm:masterext} (via Proposition~\ref{prop:AMPproof}) is well-suited to complete convergence and convergence in probability, but appears not to be able to handle almost sure convergence directly; it is not clear whether (b) holds in general if we only assume~\ref{ass:A0} rather than the stronger independence condition above. The reason for this is that in many of the key technical arguments,
the convergence of some random sequence $(X_n)$ of interest is established by first identifying a more tractable sequence $(Y_n)$ such that $Y_n\eqd X_n$ for all $n$. To show that $X_n\cvc x$ for some deterministic $x$, or that $X_n=O_c(1)$, it suffices to prove that $Y_n\cvc x$ or $Y_n=O_c(1)$ respectively in view of Definition~\ref{def:compconv} of complete convergence. Similarly, $Y_n\cvp x$ implies that $X_n\cvp x$, and $Y_n=O_p(1)$ implies that $X_n=O_p(1)$. However, if $Y_n\cvas x$, then it does not necessarily follow that $X_n\cvas x$, and if $Y_n=O_{a.s.}(1)$, then it need not be the case that $X_n=O_{a.s.}(1)$.
\end{itemize}

\unparskip
\end{remark}
\begin{remark}[\emph{Uniformity over $\PL_D(r,1)$ and the link between pseudo-Lipschitz functions and Wasserstein convergence}]
\label{rem:AMPunif}
Many asymptotic convergence results for AMP iterations are stated in the form
\begin{equation}
\label{eq:PLconv}
\frac{1}{n}\sum_{i=1}^n\psi(X_{ni}^k)\leadsto\E\bigl(\psi(\bar{X}^k)\bigr)\in\R\;\;\text{as }n\to\infty,\;\text{for every }\psi\in\PL_D(r),
\end{equation}
where $r\in [2,\infty)$, $\leadsto$ denotes one of the three modes of stochastic convergence discussed in Remark~\ref{rem:AMPconv}, the random vectors $\bar{X}^k,X_{ni}^k$ take values in $\R^D$ for some fixed $D\in\N$, and $k\in\N$ is a fixed iteration number; usually, each $X_{ni}^k$ depends on the $i^{th}$ coordinates of vector quantities in the first $k$ iterations of an AMP recursion indexed by $n$. Recalling the definition~\eqref{eq:Wrtilde} of $\widetilde{d}_r$, we deduce from Corollary~\ref{cor:Wr} that any conclusion of the form~\eqref{eq:PLconv} can be automatically upgraded to a uniform statement
\begin{equation}
\label{eq:PLconvunif}
\widetilde{d}_r(\mu_n^k,\bar{\mu}^k)=\sup_{\psi\in\PL_D(r,1)}\;\biggl|\frac{1}{n}\sum_{i=1}^n\psi(X_{ni}^k)-\E\bigl(\psi(\bar{X}^k)\bigr)\biggr|\leadsto 0\;\;\text{as }n\to\infty
\end{equation}
featuring the same mode of convergence $\leadsto$ as in~\eqref{eq:PLconv}, where we write $\mu_n^k$ for the empirical distribution of $X_{n1}^k,\dotsc,X_{nn}^k$ on $\R^D$, and $\bar{\mu}^k$ for the distribution of the limiting random vector $\bar{X}^k$. Furthermore, by Corollary~\ref{cor:Wr}, both~\eqref{eq:PLconv} and~\eqref{eq:PLconvunif} are equivalent to the assertion that $d_r(\mu_n^k,\bar{\mu}^k)\leadsto 0$. In essence, this is because $d_r,\widetilde{d}_r$ are equivalent metrics, in the sense that they generate the same topology on the space $\mathcal{P}_D(r)$ of probability distributions on $\R^D$ with a finite $r^{th}$ moment; see Theorem~\ref{thm:Wr} and Remark~\ref{rem:Wrmetric}. 

On a technical note, the measurability of the random quantities $\widetilde{d}_r(\mu_n^k,\bar{\mu}^k)$ and $d_r(\mu_n^k,\bar{\mu}^k)$ is guaranteed by analytic considerations; it is shown in Proposition~\ref{prop:Wrcount} that the supremum in~\eqref{eq:PLconvunif} can instead be taken over a deterministic countable subset $T'\subseteq\PL_D(r)$ of bounded Lipschitz functions.
\end{remark}
\begin{remark}[\emph{Finite-sample analysis}]
\label{rem:finitesample}
To complement and refine some of the asymptotic conclusions of the type~\eqref{eq:PLconv} for general AMP procedures, the relevant proof techniques have been adapted to establish concentration inequalities for quantities of the form $n^{-1}\sum_{i=1}^n\psi(X_{ni}^k)-\E\bigl(\psi(\bar{X}^k)\bigr)$ for $k,n\in\N$ and fixed arbitrary $\psi\in\PL_D(r,1)$, under suitable assumptions. 
For $r=2$, such finite-sample guarantees were obtained for asymmetric recursions by~\citet{RV18} and for symmetric recursions by~\citet{BMR20}. Their conclusions can be generalised to $r>2$ with the aid of Lemma~\ref{lem:pseudoconc}, a general concentration result for sums of pseudo-Lipschitz functions of independent Gaussian random variables. It would be interesting to see whether the above results can be extended to derive a stronger finite-sample analogue of Theorem~\ref{thm:AMPmaster} in the form of a concentration inequality for $\widetilde{d}_r\bigl(\nu_n(h^k,\gamma),N(0,\tau_k^2)\otimes\pi\bigr)=\sup_{\psi\in\PL_2(r,1)}\,\bigl|n^{-1}\sum_{i=1}^n\psi(h_i^k,\gamma_i)-\E\bigl(\psi(G_k,\bar{\gamma})\bigr)\bigr|$ or $d_r\bigl(\nu_n(h^k,\gamma),N(0,\tau_k^2)\otimes\pi\bigr)$ for $k,n\in\N$.
\end{remark}
\begin{remark}[\emph{Conditions~\emph{\ref{ass:A2}} and~\emph{\ref{ass:A3}}}]
\label{rem:AMPm0}
For $r\geq 2$, conclusions of the form~\eqref{eq:PLconv} have previously been derived for general AMP iterations under a boundedness assumption on the $(2r-2)^{th}$ moments of the empirical distributions $\nu_n(m^0)$ for $n\in\N$. In~\ref{ass:A2}, we relax this to a boundedness condition $\norm{m^0}_{n,r}=O_c(1)$ on the empirical $r^{th}$ moments, which is more natural and in line with what one would expect for a $d_r$ convergence result. To accommodate this weaker assumption, we apply H\"older's inequality rather than the Cauchy--Schwarz inequality in Lemma~\ref{lem:pseudoavg}, which is used in a key estimate in the proof of Proposition~\ref{prop:AMPproof}(c) below; see~\eqref{eq:H1cTn2} and~\eqref{eq:HkcT2}. By making similar alterations to the statements and proofs of other AMP results, it ought to be possible to avoid any mention of $(2r-2)^{th}$ empirical moments.

The primary purpose of~\ref{ass:A3} is to ensure that the asymptotic dependence between different iterates $h^j,h^\ell$ (as measured by the inner product $\ipr{h^j}{h^\ell}_n$ between them) has a deterministic limiting expression, namely $\bar{\mathrm{T}}_{j,\ell}$ as defined in~\eqref{eq:Sigmabar}; see also Proposition~\ref{prop:AMPproof}(d,\,e,\,f). The existence of the limiting covariance structure captured by~\eqref{eq:Sigmabar} is crucial to the success of the proof strategy for Theorems~\ref{thm:AMPmaster} and~\ref{thm:masterext}; in fact, its existence is a \emph{necessary} condition for the more general conclusion in Theorem~\ref{thm:masterext}, as can be seen by taking $\psi(x_1,\dotsc,x_k):=x_jx_\ell$ therein for $1\leq j,\ell\leq k$. 
\end{remark}
\begin{remark}
\label{rem:A4'}
Since $\pi\in\mathcal{P}_1(r)$ by~\ref{ass:A1}, recall from Section~\ref{sec:notation} that if $\bar{\gamma}\sim\pi$, then $\E\bigl(\psi(\bar{\gamma})\bigr)=\int_{\R}\psi\,d\pi<\infty$ for all $\psi\in\PL_1(r)$, the set of all pseudo-Lipschitz functions on $\R$ of order $r$. Thus, in~\ref{ass:A3}, given Lipschitz functions $F_0,\phi$ on $\R$, Lemma~\ref{lem:pseudoprod} ensures that $x\mapsto F_0(x)\phi(x)$ lies in $\PL_1(2)\subseteq\PL_1(r)$ since $r\geq 2$, so $\E\big(F_0(\bar{\gamma})\phi(\bar{\gamma})\bigr)$ is finite.

It can be shown by fairly routine arguments that the following condition implies the first condition in~\ref{ass:A2} as well as~\ref{ass:A3}; see Section~\ref{sec:addproofs} for a full justification.

\unparskip
\begin{enumerate}[label=(A1$^+$),leftmargin=1.2cm]
\item \label{ass:A1+} There exists a Lipschitz function $\tilde{f}_0\colon\R^2\to\R$ and a probability distribution $\tilde{\nu}^0\in\mathcal{P}_1(2)$ such that writing $\mu^0$ for the distribution of $\bigl(\tilde{f}_0(\bar{\eta},\bar{\gamma}),\bar{\gamma}\bigr)$ when $\bar{\eta}\sim\tilde{\nu}^0$ and $\bar{\gamma}\sim\pi$ are independent, we have $d_2\bigl(\nu_n(m^0,\gamma),\mu^0\bigr)\cvc 0$.
\end{enumerate}

\unparskip
In applications,~\ref{ass:A1+} can be more convenient to verify than~\ref{ass:A3}. Note that if $d_2\bigl(\nu_n(h^0,\gamma),\tilde{\nu}^0\otimes\pi\bigr)\cvc 0$ with $\tilde{\nu}^0$ as above, then~\ref{ass:A1+} holds with $\tilde{f}_0=f_0$.
\end{remark}
\begin{remark}
At least when $r=2$, the master theorems in Section~\ref{Sec:Master} can be extended to abstract recursions for which the non-degeneracy condition~\ref{ass:A4} does not hold and the limiting covariance matrices need not be positive definite. These degenerate cases can be handled by first perturbing the Lipschitz functions $f_k$ and then applying a continuity argument that has some similarities with the proof of Theorem~\ref{thm:AMPlowsym} in Section~\ref{sec:LowRankproofs}; see~\citet[Section~4.2.1]{JM13} and~\citet[Section~5.4]{BMN20} for further details. 

An important fact in the proof is that $\norm{W}_{2\to 2}:=\sup_{u\neq 0}\norm{Wu}_2/\norm{u}_2=O_c(1)$ for $W\sim\GOE(n)$ as $n\to\infty$~\citep[e.g.][]{AGZ10,KY13}. For $r\in [1,\infty]$, we mention here that $\norm{W}_{r\to r}:=\sup_{u\neq 0}\norm{Wu}_r/\norm{u}_r=O_c(1)$ \emph{if and only if} $r=2$; this can be seen by taking $u=e_1$ and appealing to Lemma~\ref{lem:Wu} when $r\in [1,2)$, and then noting that $\norm{W}_{r\to r}=\norm{W}_{r'\to r'}$ when $1/r+1/r'=1$.
\end{remark}
\begin{remark}
\ref{ass:A5} is a non-vacuous albeit very mild condition. For any Lipschitz $f\colon\R^2\to\R$, the partial derivative $\frac{\partial f}{\partial x}$ is bounded on its domain of definition, which is a Borel set of full Lebesgue measure. Nevertheless, there are examples of Lipschitz $f\colon\R^2\to\R$ for which $\frac{\partial f}{\partial x}$ cannot be extended to a function on $\R^2$ that is continuous $(\lambda\otimes\pi)$-almost everywhere (see Remark~\ref{rem:lipderiv}). That said, it is inconceivable that such pathological choices of $f_k$ would be made in any practical AMP procedure, where the functions $f_k'$ usually have the property that $\{x\in\R:(x,y)\in D_k\}$ is finite for every $y\in\R$, and hence satisfy~\ref{ass:A5}.
\end{remark}

\umparskip
\subsection{Conditional distributions for symmetric AMP}
\label{sec:AMPcond}
In this subsection, we fix $n\in\N$, and in most places, we suppress the dependence on $n$ of all quantities such as $W\equiv W(n)$ and $h^k\equiv h^k(n)$. When we refer to \emph{orthonormal} sets, it is implicit that the constituent vectors have unit Euclidean norm, i.e.\ that the underlying inner product is $\ipr{\cdot\,}{\cdot}$, not $\ipr{\cdot\,}{\cdot}_n$. All statements concerning conditional distributions can be understood formally in terms of the rigorous definition of regular conditional probability, as outlined in Section~\ref{sec:conddists}. The proofs of the results below are given in Section~\ref{sec:AMPcondproofs}.

In the setting of Section~\ref{sec:AMPsym}, define the $n\times k$ matrices 
\[H_k\equiv H_k(n):=\bigl(h^1\;\cdots\;h^k\bigr),\quad M_k\equiv M_k(n):=\bigl(m^0\;m^1\;\cdots\;m^{k-1}\bigr),\quad Y_k\equiv Y_k(n):=\bigl(y^0\;y^1\;\cdots\;y^{k-1}\bigr),\]
where $y^j\equiv y^j(n):=Wm^j=h^{j+1}+b_j m^{j-1}$ for $j=0,1,\dotsc,k-1$. For convenience, we also define $M_0(n)=Y_0(n):=0\in\R^n$. Then the symmetric AMP recursion~\eqref{eq:AMPsym} can be rewritten as $WM_k=Y_k$ for $k\in\N$. 

For each $0\leq k\leq n-1$, let $P_k:=M_kM_k^+=M_k(M_k^\top M_k)^+ M_k^\top$ and $P_k^\perp:=I_n-P_k$ be the $n\times n$ matrices representing the orthogonal projections onto $\Img(M_k):=\Span\{m^j:0\leq j\leq k-1\}$ and $V_k:=\Img(M_k)^\perp$ respectively, and define $r_k:=\rank(M_k)=\dim\Img(M_k)$. Let $\barpp{m}{k}:=P_k^\perp m^k$ for $0\leq k\leq n-1$, so that the span of $\barpp{m}{k}$ is the orthogonal complement of $V_{k+1}$ within $V_k$. Furthermore, define $\mathscr{S}_{-1}:=\{\emptyset,\Omega\}$ to be the trivial $\sigma$-algebra, and for $k\in\N_0$, let \[\mathscr{S}_k:=\sigma(\gamma,m^0,h^j:1\leq j\leq k).\]
Then since $b_k,m^k$ are measurable functions of $h^k$ and $\gamma$, we see from~\eqref{eq:AMPsym} that 
\begin{equation}
\label{eq:Sk}
\mathscr{S}_k=\sigma(\gamma,m^0,y^j:0\leq j\leq k-1),
\end{equation}
and that $m^0,\dotsc,m^k$ and $r_0,\dotsc,r_{k+1}$ are $\mathscr{S}_k$-measurable for each $-1\leq k\leq n-1$. (It is not true in general that $\Pr(r_k=k)=1$ for all $1\leq k\leq n-1$, even in recursions~\eqref{eq:AMPsym} with non-pathological $f_k$.)

Our first task is to establish an important fact (Proposition~\ref{prop:condgoe}) that will be used to derive the (regular) conditional distributions of $W$ and $h^{k+1}$ given $\mathscr{S}_k$ in Proposition~\ref{prop:AMPconddist} below, for each fixed $k\in\{0,1,\dotsc,n-1\}$. We will use the symbol `$\eqdcond{\mathscr{S}_k}$' to indicate (almost-sure) equality of conditional distributions given $\mathscr{S}_k$, a notion that is defined formally in Section~\ref{sec:conddists}.
\begin{proposition}
\label{prop:condgoe}
Fix $0\leq k\leq n-1$ and suppose as in~\emph{\ref{ass:A0}} that $W\sim\GOE(n)$ is independent of $(m^0,\gamma)$. If $\tilde{U}_k$ is any $\mathscr{S}_{k-1}$-measurable $n\times(n-r_k)$ matrix whose columns form an orthonormal basis of $V_k$, then given $\mathscr{S}_{k-1}$, the matrix $\tilde{U}_k^\top W\tilde{U}_k$ has conditional distribution $\GOE(n-r_k)$ and is conditionally independent of $\mathscr{S}_k$. Consequently, $\tilde{U}_k^\top W\tilde{U}_k$ has conditional distribution $\GOE(n-r_k)$ given $\mathscr{S}_k$, and if $\tilde{W}\sim\GOE(n)$ is independent of $\mathscr{S}_k$, then $\tilde{U}_k^\top W\tilde{U}_k\eqdcond{\mathscr{S}_k}\tilde{U}_k^\top\tilde{W}\tilde{U}_k$.
\end{proposition}
\begin{remark}
\label{rem:condgoesp}
Consider the important special case where $\mathbb{P}(r_k=k)=1$.  Then under the hypotheses of the proposition, $\tilde{U}_k^\top W\tilde{U}_k\sim \GOE(n-k)$ is conditionally independent of $\mathscr{S}_k$ given $\mathscr{S}_{k-1}$, and is independent of $\mathscr{S}_k$.
\end{remark}
\begin{remark}
\label{rem:gs}
To explicitly construct a (random) $\tilde{U}_k$ with the above properties, consider applying the Gram--Schmidt procedure to $m^0,\dotsc,m^{k-1},e_1,\dotsc,e_n\in\R^n$ (in that order) and retaining only the non-zero vectors in the output (which are all normalised to have unit Euclidean length). This yields an $\mathscr{S}_{k-1}$-measurable orthonormal basis $\tilde{m}^1,\dotsc,\tilde{m}^n$ of $\R^n$, where $\tilde{m}^1,\dotsc,\tilde{m}^{r_k}$ are obtained from $m^0,\dotsc,m^{k-1}$ and therefore span $\Img(M_k)$, while $\tilde{m}^{r_k+1},\dotsc,\tilde{m}^n$ span $V_k=\Img(M_k)^\perp$. Thus, we can take $\tilde{U}_k=\bigl(\tilde{m}^{r_k+1}\;\cdots\;\tilde{m}^n\bigr)$. 
\end{remark}

\unparskip
The main result of this subsection is Proposition~\ref{prop:AMPconddist} below, which plays a crucial role in the inductive proof of the AMP master theorems given in Sections~\ref{sec:inductive} and~\ref{sec:masterproofs}. For each $k\in\{1,\dotsc,n-1\}$, let
\begin{equation}
\label{eq:alphak}
\alpha^k\equiv\alpha^k(n)\equiv (\alpha_1^k,\dotsc,\alpha_k^k):=M_k^+ m^k=(M_k^\top M_k)^+ M_k^\top m^k\in\R^k
\end{equation}
be a vector of projection coefficients satisfying $P_km^k=M_k\alpha^k=\sum_{\ell=1}^k\alpha_\ell^k\,m^{\ell-1}$. When $M_k$ has full rank (i.e.\ when $r_k=k$), note that $\alpha^k=(M_k^\top M_k)^{-1}M_k^\top m^k$ is the unique vector with this property. In addition, let $B_1:=(0,0)\in\R^2$
and $B_k:=\diag(b_0,\dotsc,b_{k-1})\in\R^{k\times k}$ for $k\in\{2,\dotsc,n\}$, so that $Y_k=H_k+(0\;M_{k-1})B_k$ for all $k\in\{1,\dotsc,n\}$.
\begin{proposition}
\label{prop:AMPconddist}
For $n\in\N$, consider a symmetric AMP recursion~\eqref{eq:AMPsym} for which\emph{~\ref{ass:A0}} holds. For $k\in\{0,1,\dotsc,n-1\}$, let both $\tilde{W}^k\equiv\tilde{W}^k(n)\sim\GOE(n)$ and $(\tilde{Z}^{k+1},\tilde{\zeta}^{k+1})\equiv\bigl(\tilde{Z}^{k+1}(n),\tilde{\zeta}^{k+1}(n)\bigr)\sim N_n(0,I_n)\otimes N(0,1/n)$ be independent of $\mathscr{S}_k$. Then
\begin{equation}
\label{eq:condzero}
W\eqdcond{\mathscr{S}_0}\tilde{W}^0\quad\text{and}\quad h^1\eqdcond{\mathscr{S}_0}\norm{m^0}_n\tilde{Z}^1+\tilde{\zeta}^1 m^0=:h^{1,0},
\end{equation}
and for each $k\in\{1,\dotsc,n-1\}$, we have
\begin{align}
\label{eq:Wcond}
W&\eqdcond{\mathscr{S}_k}WP_k+(WP_k)^\top P_k^\perp+P_k^\perp\tilde{W}^kP_k^\perp=Y_kM_k^++(Y_kM_k^+)^\top P_k^\perp+P_k^\perp\tilde{W}^kP_k^\perp\\
h^{k+1}&\eqdcond{\mathscr{S}_k}H_k\alpha^k+P_k^\perp(\tilde{W}^k\barpp{m}{k})+\bigl\{(M_k^+)^\top H_k^\top\barpp{m}{k}-b_km^{k-1}+(0\;M_{k-1})B_k\alpha^k\bigr\}\notag\\
\label{eq:hcond}
&\eqdcond{\mathscr{S}_k}\sum_{\ell=1}^k\alpha_\ell^k\,h^\ell+\norm{\barpp{m}{k}}_n(P_k^\perp\tilde{Z}^{k+1})+\tilde{\zeta}^{k+1}\barpp{m}{k}+M_k(M_k^\top M_k)^+\biggl(v^{k,k}-\sum_{\ell=1}^k\alpha_\ell^k\,v^{k,\ell-1}\biggr)\\
&\phantom{\restr{=}{}}=:h^{k+1,k},\notag
\end{align}
where $v^{k,\ell}\equiv v^{k,\ell}(n):=H_k^\top m^\ell-b_\ell\,M_k^\top m^{\ell-1}\in\R^k$ for $\ell\in\{0,\dotsc,k\}$.
\end{proposition}

\unparskip
The crux of the proof of Proposition~\ref{prop:AMPconddist} is to establish~\eqref{eq:Wcond}, which characterises the conditional distribution of $W$ given $\mathscr{S}_k$. It is intuitively helpful to think of this as being obtained by conditioning~$W$ on the `linear constraints' $Wm^0=y^0,\dotsc,Wm^{k-1}=y^{k-1}$. However, since $m^1,\dotsc,m^{k-1}$ are random and depend on $W$, this heuristic argument is not sufficient on its own to constitute a formal proof of Proposition~\ref{prop:AMPconddist}. For the benefit of readers interested in the technicalities, we give a more detailed explanation below.

Observe that for fixed $k\in\N$ and deterministic $y,a^0,a^1,\dotsc,a^k\in\R^n$, the event $\Omega_{y,a^0,\dotsc,a^k}:=\{\gamma=y,m^0=a^0,h^1=a^1,\dotsc,h^k=a^k\}$ can be expressed as 
\begin{equation}
\label{eq:eventequiv}
\Omega_{y,a^0,\dotsc,a^k}=\{\gamma=y,m^0=a^0,Wt^j=z^j\text{ for all }0\leq j\leq k-1\}=\{\gamma=y,m^0=a^0,WT_k=Z_k\},
\end{equation}
where $t^j := f_j(a^j,y)\in\R^n$ and $z^j := a^{j+1} + \abr{f_j'(a^j,y)}_n\,f_{j-1}(a^{j-1},y)$ for $0\leq j\leq k-1$, and $T_k:=(t^0\;t^1\;\cdots\;t^{k-1})$ and $Z_k:=(z^0\;z^1\;\cdots\;z^{k-1})$ are fixed $n\times k$ matrices. Now for $W\sim\GOE(n)$ and any fixed $T\in\R^{n\times k}$ of rank $p$, we can derive the conditional distribution of $W$ given $WT$ by writing
\begin{equation}
\label{eq:Wcondfix}
W=WP+\bigl(P+P^\perp\bigr)^\top WP^\perp=WP+(WP)^\top P^\perp+P^\perp WP^\perp,
\end{equation}
where $P:=TT^+$ and $P^\perp:=I_n-TT^+$ represent the orthogonal projections onto $\Img(T)$ and $\Img(T)^\perp$ respectively. The first two terms on the right hand side of~\eqref{eq:Wcondfix} are measurable functions of $WP=(WT)T^+$ (and hence $WT$), while the third term $P^\perp WP^\perp$ is independent of $WT$. Thus, $\E(W\,|\,WT)=WP+(WP)^\top P^\perp$. 
Moreover, we can write $P^\perp=\tilde{U}\tilde{U}^\top$, where the columns of $\tilde{U}$ form an orthonormal basis for $\Img(T)^\perp$, so that $P^\perp WP^\perp=\tilde{U}(\tilde{U}^\top W\tilde{U})\tilde{U}^\top$, and $\tilde{U}^\top W\tilde{U}\sim\GOE(n-p)$ is independent of $WT$. For $Z\in\R^{n\times k}$, this enables us to interpret `the conditional distribution of $W$ given $WT=Z$' as the distribution of 
\[ZT^+ +(ZT^+)^\top P^\perp+\tilde{U}^\top\tilde{W}\tilde{U},\]
where $\tilde{W}\sim\GOE(n-p)$. We denote this distribution by $\mathcal{L}_Z(T)$.

In view of~\eqref{eq:eventequiv} and the assumption that $W$ is independent of $(m^0,\gamma)$ in~\ref{ass:A0}, it is then tempting to argue heuristically that
\begin{align*}
W\,|\,`\{\gamma=y,m^0=a^0,h^1=a^1,\dotsc,h^k=a^k\}\text{'}&\eqd W\,|\,`\{\gamma=y,m^0=a^0,WT_k=Z_k\}\text{'}\\
&\eqd W\,|\,`\{WT_k=Z_k\}\text{'}\sim\mathcal{L}_{Z_k}(T_k),
\end{align*}
and conclude on this basis that $W$ has (regular) conditional distribution $\mathcal{L}_\omega\equiv\mathcal{L}_{Y_k(\omega)}\bigl(M_k(\omega)\bigr)$ given $\mathscr{S}_k=\sigma(\gamma,m^0,h^1,\ldots,h^k)$, noting that $M_k(\omega)=T_k$ and $Y_k(\omega)=Z_k$ for $\omega\in\Omega_{y,a^0,\dotsc,a^k}$. However, this line of reasoning appears to involve conditioning explicitly on an event of potentially zero probability, and is not formally justified by the above argument; cf.~the Borel paradox \citep[][pp.~350--351]{Dud02} for the associated hazards. 

As mentioned above, the issue is that $M_k$ is random and is in general not independent of $W$, whereas the distributional claims in the previous paragraph relied on the fact that $T$ was fixed. Nevertheless, the key point is that the randomness of $M_k$ and its dependence on $W$ turn out not to cause irreconcilable difficulties, due to the conditional independence established in Proposition~\ref{prop:condgoe}. It follows from this result that $\E(W\,|\,\mathscr{S}_k)=WP_k+(WP_k)^\top P_k^\perp$, so the conditional distributional equality~\eqref{eq:Wcond} in Proposition~\ref{prop:AMPconddist} and the decomposition~\eqref{eq:Wcondproof} in its proof are the appropriate analogues of~\eqref{eq:Wcondfix}. 

\hfparskip
\subsection{Proofs of results in Section~\ref{sec:AMPcond}}
\label{sec:AMPcondproofs}
A key ingredient in the proof of Proposition~\ref{prop:condgoe} is Lemma~\ref{lem:orthgoe} below, which extends the orthogonal invariance property of the $\GOE(n)$ distribution. Given a finite collection of disjoint measurable spaces $(\mathscr{X}_1,\mathcal{A}_1),\dotsc,(\mathscr{X}_m,\mathcal{A}_m)$, we equip the disjoint union $\bigsqcup_{\,k=1}^{\,m}\mathscr{X}_k$ with the $\sigma$-algebra $\bigl\{\bigsqcup_{\,k=1}^{\,m}A_k:A_k\in\mathcal{A}_k\text{ for all }k\bigr\}$.  
\begin{lemma}
\label{lem:orthgoe}
Let $\mathcal{G}\subseteq\mathcal{F}$ be a sub-$\sigma$-algebra and let $X\colon (\Omega,\mathcal{F},\Pr)\to\bigsqcup_{\,k=1}^{\,n}\R^{k\times k}$ be a measurable function. Suppose that there is a partition of $\Omega$ into disjoint events $\Omega_1,\dotsc,\Omega_m\in\mathcal{G}$ such that for each $k=1,\dotsc,m$, the map $X$ takes values in $\R^{n_k\times n_k}$ on $\Omega_k$ and has conditional distribution $\GOE(n_k)$ given $\mathcal{G}$ on $\Omega_k$, for some (deterministic) $n_k\in\{1,\dotsc,n\}$. Moreover, let $Q=(Q_1\;Q_2)\colon (\Omega,\mathcal{F},\Pr)\to\bigsqcup_{\,k=1}^{\,n}\mathbb{O}_k$ be a $\mathcal{G}$-measurable function such that on each event $\Omega_k$, the map $Q$ takes values in $\mathbb{O}_{n_k}$, and $Q_1,Q_2$ have $\ell_k$ and $n_k-\ell_k$ columns respectively, for some (deterministic) $\ell_k\in\{1,\dotsc,n_k-1\}$. Then, given $\mathcal{G}$, we have the following:

\unparskip
\begin{enumerate}[label=(\alph*)]
\item $Q^\top XQ$ has conditional distribution $\GOE(n_k)$ on $\Omega_k$ for every $k=1,\ldots,m$;
\item $Q_2^\top XQ_2$ has conditional distribution $\GOE(n_k-\ell_k)$ on $\Omega_k$ for every $k=1,\ldots,m$;
\item $Q^\top XQ_1$ and $Q_2^\top XQ_2$ are conditionally independent.
\end{enumerate}
\end{lemma}
\begin{remark}
Note that if $n_1=\cdots=n_m=n$, then under the first condition of the lemma, it follows from Remark~\ref{rem:indepcond} that $X$ has unconditional distribution $\GOE(n)$ and is independent of $\mathcal{G}$. Thus, in the instructive special case where $m=1$ and $1\leq\ell_1<n=n_1$, the result above simplifies to the following: suppose that $X\sim\GOE(n)$, and is independent of $\mathcal{G}$, and moreover that $Q=(Q_1\;Q_2)\colon (\Omega,\mathcal{F},\Pr)\to\mathbb{O}_n$ is a $\mathcal{G}$-measurable map such that $Q_1,Q_2$ have $\ell_1$ and $n-\ell_1$ columns respectively. Then

\unparskip
\begin{enumerate}[label=(\alph*)]
\item $Q^\top XQ \sim \GOE(n)$ and is independent of $\mathcal{G}$;
\item $Q_2^\top XQ_2 \sim \GOE(n-\ell_1)$ and is independent of $\mathcal{G}$;
\item $Q^\top XQ_1$ and $Q_2^\top XQ_2$ are independent, and also conditionally independent given $\mathcal{G}$.
\end{enumerate}
\end{remark}

\deparskip
\begin{proof}[Proof of Lemma~\ref{lem:orthgoe}]
\noindent (a) For $\ell=1,\ldots,n$, let $\mathcal{A}_\ell$ and $\mathcal{B}_\ell$ be the Borel $\sigma$-algebras on $\mathscr{X}_\ell:=\R^{\ell\times\ell}$ and $\mathscr{Y}_\ell:=\mathbb{O}_\ell$ respectively. Define $\phi_\ell\colon \mathscr{X}_\ell\times\mathscr{Y}_\ell\rightarrow\mathscr{X}_\ell$ by $\phi_\ell(M,J) := J^\top MJ$. In the notation of Lemma~\ref{lem:conddist}(b), the orthogonal invariance property of $\GOE(\ell)$ can be restated as $\GOE(\ell) = \GOE(\ell)\circ (\phi_\ell\circ \iota_J)^{-1}$ for every $J \in \mathbb{O}_\ell$. Thus, observing that $\phi_{n_k}(X,Q) = Q^\top X Q$ on $\Omega_k$, and applying Lemma~\ref{lem:conddist}(b) to $\phi_{n_k}$, we see that $Q^\top X Q$ has conditional distribution $\GOE(n_k)$ given $\mathcal{G}$ on $\Omega_k$, as required.

(b) For $k=1,\dotsc,m$, let $\psi_k\colon\mathscr{X}_{n_k}\rightarrow\mathscr{X}_{n_k-\ell_k}$ denote the map that extracts the lower-right $(n_k-\ell_k)\times (n_k-\ell_k)$ block of entries of an $n_k\times n_k$ matrix. Then $\psi_k(W)\sim\GOE(n_k-\ell_k)$ whenever $W\sim\GOE(n_k)$, so $\GOE(n_k-\ell_k) = \GOE(n_k)\circ\psi_k^{-1}=\GOE(n_k)\circ (\phi_{n_k}\circ \iota_J)^{-1} \circ \psi_k^{-1}$ for every $J\in\mathbb{O}_{n_k}$.  We can therefore apply Lemma~\ref{lem:conddist}(b) to $\psi_k\circ\phi_{n_k}$ to conclude that $Q_2^\top XQ_2$ has conditional distribution $\GOE(n_k-\ell_k)$ given $\mathcal{G}$ on $\Omega_k$.

(c) For $\omega\in\Omega$, let $P_\omega$, $Q_\omega$ and $R_\omega$ respectively denote the conditional distributions of $Q^\top XQ_1$, $Q_2^\top XQ_2$ and $(Q^\top XQ_1,Q_2^\top XQ_2)$ given $\mathcal{G}$. For $k=1,\ldots,m$, let $\tilde{\psi}_k\colon\mathscr{X}_{n_k}\rightarrow \R^{n_k\times \ell_k}$ denote the map that extracts the first $\ell_k$ columns of a $n_k\times n_k$ matrix. Now define $\Psi_k\colon\mathscr{X}_{n_k}\rightarrow \R^{n_k\times \ell_k}\times\mathscr{X}_{n_k-\ell_k}$ by $\Psi_k(M) := \bigl(\tilde{\psi}_k(M),\psi_k(M)\bigr)$. Then $\tilde{\psi}_k(W)$ and $\psi_k(W)$ are independent whenever $W\sim\GOE(n_k)$, so $\GOE(n_k)\circ \Psi_k^{-1} = \bigl(\GOE(n_k)\circ\tilde{\psi}_k^{-1}\bigr) \otimes \bigl(\GOE(n_k) \circ \psi_k^{-1}\bigr)$.  Since $(Q^\top XQ_1,Q_2^\top XQ_2) = (\Psi_k\circ\phi_k)(X,Q)$ on $\Omega_k$, we may apply Lemma~\ref{lem:conddist}(b) to $\tilde{\psi}_k\circ \phi_{n_k}$, $\psi_k\circ \phi_{n_k}$ and $\Psi_k \circ \phi_{n_k}$ to deduce that $R_\omega = P_\omega \otimes Q_\omega$ for all $\omega\in\Omega_k$.  Since $k\in\{1,\ldots,m\}$ was arbitrary, we conclude that $R_\omega = P_\omega \otimes Q_\omega$ for all $\omega\in\Omega=\bigsqcup_{\,k=1}^{\,m}\Omega_k$, which together with Lemma~\ref{lem:condind}(b) implies that $Q^\top XQ_1$ and $Q_2^\top XQ_2$ are conditionally independent given $\mathcal{G}$.
\end{proof}

\deparskip
\begin{proof}[Proof of Proposition~\ref{prop:condgoe}]
We argue by induction on $k\in\{0,1,\dotsc,n-1\}$. The case $k=0$ is trivial since $W\sim\GOE(n)$ and is independent of $(m^0,\gamma)$ by assumption. 
For a general $1\leq k\leq n-1$ (when $n\geq 2$), let $\tilde{U}_{k-1}$ be any $\mathscr{S}_{k-2}$-measurable $n\times (n-r_{k-1})$ matrix whose columns form an orthonormal basis of $V_{k-1}$, and fix an arbitrary $\mathscr{S}_{k-1}$-measurable $n\times (n-r_k)$ matrix $\tilde{U}_k$ whose columns form an orthonormal basis of $V_k$. Moreover, let $E\equiv E_{k-1}$ be the event $\{\barpp{m}{k-1}\neq 0\}=\{r_k=r_{k-1}+1\}\in\mathscr{S}_{k-1}$, and note that $n-r_{k-1}\geq n-k+1\geq 2$.

Next, define an $\mathscr{S}_{k-1}$-measurable $n\times (n-r_{k-1})$ matrix $\check{U}$ by setting $\check{U}:=(\barpp{m}{k-1}\;\tilde{U}_k)$ on $E$ and $\check{U}:=\tilde{U}_k$ on $E^c=\{\barpp{m}{k-1}=0\}$. Letting $\breve{U}$ be the $\mathscr{S}_{k-1}$-measurable $n\times (n-r_{k-1}-1)$ matrix obtained by removing the first column $\check{U}e_1$ of $\check{U}$, we therefore have $\tilde{U}_k=\breve{U}$ on $E$ and $\tilde{U}_k=\check{U}$ on $E^c$. Now $\check{U},\tilde{U}_{k-1}$ have orthonormal columns that span $V_{k-1}$, so $Q:=\tilde{U}_{k-1}^\top\check{U}$ is an $\mathscr{S}_{k-1}$-measurable orthogonal $(n-r_{k-1})\times (n-r_{k-1})$ matrix such that $\check{U}^\top W\check{U}=Q^\top\bigl(\tilde{U}_{k-1}^\top W\tilde{U}_{k-1}\bigr)Q$. 
By the inductive hypothesis, $\tilde{U}_{k-1}^\top W\tilde{U}_{k-1}$ has conditional distribution $\GOE(n-r_{k-1})$ given $\mathscr{S}_{k-1}$, so it follows from parts (a) and (b) respectively of Lemma~\ref{lem:orthgoe} (with $\ell\equiv 1$ and $N=n-r_{k-1}\geq 2$) that $\check{U}^\top W\check{U}$ and $\breve{U}^\top W\breve{U}$ have conditional distributions $\GOE(n-r_{k-1})$ and $\GOE(n-r_{k-1}-1)$ respectively given $\mathscr{S}_{k-1}$. Since $\tilde{U}_k^\top W\tilde{U}_k=\check{U}^\top W\check{U}$ on $E^c=\{r_k=r_{k-1}\} \in\mathscr{S}_{k-1}$ and $\tilde{U}_k^\top W\tilde{U}_k=\breve{U}^\top W\breve{U}$ on $E=\{r_k=r_{k-1}+1\}\in\mathscr{S}_{k-1}$, we deduce from Lemma~\ref{lem:conddist}(a) that $\tilde{U}_k^\top W\tilde{U}_k$ has conditional distribution $\GOE(n-r_k)$ given $\mathscr{S}_{k-1}$, as required.

In addition, it holds trivially that 0 and $\check{U}^\top W\check{U}$ are conditionally independent given $\mathscr{S}_{k-1}$, and Lemma~\ref{lem:orthgoe}(c) implies that $\check{U}^\top W(\check{U}e_1) = Q^\top(\tilde{U}_{k-1}^\top W\tilde{U}_{k-1}\bigr)Qe_1$ and $\breve{U}^\top W\breve{U}$ are also conditionally independent given $\mathscr{S}_{k-1}$. Since $W\barpp{m}{k-1}=0$ on $E^c$ and $\check{U}e_1=\barpp{m}{k-1}$ on $E$, an application of Lemma~\ref{lem:condind}(a) shows that $\check{U}^\top W\barpp{m}{k-1}$ (and hence $\sigma(\mathscr{S}_{k-1},\check{U}^\top W\barpp{m}{k-1})$ by Lemma~\ref{lem:condindequiv}) is conditionally independent of $\tilde{U}_k^\top W\tilde{U}_k$ given $\mathscr{S}_{k-1}$. Moreover, $WP_{k-1}=WM_{k-1}M_{k-1}^+=Y_{k-1}M_{k-1}^+$, $m^{k-1}$, $\barpp{m}{k-1}$, $\check{U}$ and $P_{k-1}^\perp=\check{U}\check{U}^\top$ are $\mathscr{S}_{k-1}$-measurable, so
\begin{align}
y^{k-1}=W\bigl(P_{k-1}+P_{k-1}^\perp\bigr)m^{k-1}&=(WP_{k-1})m^{k-1}+\bigl(P_{k-1}+P_{k-1}^\perp\bigr)^\top W\barpp{m}{k-1}\notag\\
\label{eq:yk-1}
&=(WP_{k-1})m^{k-1}+(WP_{k-1})^\top\barpp{m}{k-1}+\check{U}\bigl(\check{U}^\top W\barpp{m}{k-1}\bigr)
\end{align}
is measurable with respect to $\sigma(\mathscr{S}_{k-1},\check{U}^\top W\barpp{m}{k-1})$. Thus, given $\mathscr{S}_{k-1}$, we conclude that $\tilde{U}_k^\top W\tilde{U}_k$ is conditionally independent of $y^{k-1}$, and hence conditionally independent of $\mathscr{S}_k=\sigma(\mathscr{S}_{k-1},y^{k-1})$ by Lemma~\ref{lem:condindequiv} and~\eqref{eq:Sk}. Therefore, since $\tilde{U}_k^\top W\tilde{U}_k$ has conditional distribution $\GOE(n-r_k)$ given $\mathscr{S}_{k-1}$, it also has conditional distribution $\GOE(n-r_k)$ given $\sigma(\mathscr{S}_{k-1},\mathscr{S}_k)=\mathscr{S}_k$. 

Finally, it remains to show that if $\tilde{W}\sim\GOE(n)$ is independent of $\mathscr{S}_k$, then $\tilde{U}_k^\top\tilde{W}\tilde{U}_k$ also has conditional distribution $\GOE(n-r_k)$ given $\mathscr{S}_k$. To see this, let $\tilde{m}^1,\dotsc,\tilde{m}^{r_k}$ be an $\mathscr{S}_k$-measurable orthonormal basis of $\Img(M_k)$, obtained for example by applying the Gram--Schmidt procedure to $m^0,\dotsc,m^{k-1}$, as in Remark~\ref{rem:gs} above. Then taking $Q_1=\bigl(\tilde{m}^1\;\cdots\;\tilde{m}^{r_k}\bigr)$ and $Q_2=\tilde{U}_k$, we see that $Q=(Q_1\;Q_2)$ satisfies the hypotheses of Lemma~\ref{lem:orthgoe} with $\Omega_j=\{r_k=j\}\in\mathscr{S}_k$ and $\ell_j=j\leq n=n_j$ for $j=1,\dotsc,k$. The desired conclusion now follows directly from Lemma~\ref{lem:orthgoe}(b), and this completes the inductive step.
\end{proof}

\unparskip
As mentioned above, the proof of Proposition~\ref{prop:AMPconddist} relies crucially on the final assertion in Proposition~\ref{prop:condgoe}. To obtain the conditional distributional equalities in~\eqref{eq:condzero} and~\eqref{eq:hcond}, we will also apply the following elementary fact.
\begin{lemma}
\label{lem:Wu}
If $W\sim\GOE(n)$ and $u\in\R^n$ is fixed, then $Wu\eqd\norm{u}_n Z+\zeta u$, where $Z\sim N_n(0,I_n)$ and $\zeta\sim N(0,1/n)$ are independent.
\end{lemma}

\deparskip
\begin{proof}[Proof of Lemma~\ref{lem:Wu}]
The result holds trivially when $u=0$, and is also true when $u=e_1$ since $We_1\sim N_n\bigl(0,\diag(2/n,1/n,\dotsc,1/n)\bigr)$. For a general $u\in\R^n\setminus\{0\}$, let $Q\in\R^{n\times n}$ be an orthogonal matrix with $Qe_1=u/\norm{u}$, so that $Q^\top u=\norm{u}e_1$. Then
\[Wu\eqd QWQ^\top u=Q(We_1)\norm{u}\eqd Q(\norm{e_1}_n Z+\zeta e_1)\norm{u}\eqd\norm{u}_n Z+\zeta u,\]
as required, where we have used the orthogonal invariance of $W\sim\GOE(n)$, the result for $e_1$ and the orthogonal invariance of $Z\sim N_n(0,I_n)$ respectively to obtain the distributional equalities above.
\end{proof}

\deparskip
\begin{proof}[Proof of Proposition~\ref{prop:AMPconddist}]
We start by proving~\eqref{eq:Wcond} for every $k\in\{0,1,\dotsc,n-1\}$. Let $\tilde{U}_k$ be any $\mathscr{S}_{k-1}$-measurable $n\times(n-r_k)$ matrix whose columns form an orthonormal basis of $V_k$; see Remark~\ref{rem:gs} for a specific construction of $\tilde{U}_k$. Similarly to~\eqref{eq:yk-1} in the proof of Proposition~\ref{prop:condgoe}, we can write
\begin{align}
W=WP_k&+\bigl(P_k+P_k^\perp\bigr)^\top WP_k^\perp\notag\\
=WP_k&+(WP_k)^\top P_k^\perp+P_k^\perp WP_k^\perp\notag\\
\label{eq:Wcondproof}
=WP_k&+(WP_k)^\top P_k^\perp+\tilde{U}_k\bigl(\tilde{U}_k^\top W\tilde{U}_k\bigr)\tilde{U}_k^\top\\
\eqdcond{\mathscr{S}_k}WP_k&+(WP_k)^\top P_k^\perp+\tilde{U}_k\bigl(\tilde{U}_k^\top\tilde{W}^k\tilde{U}_k\bigr)\tilde{U}_k^\top=WP_k+(WP_k)^\top P_k^\perp+P_k^\perp\tilde{W}^kP_k^\perp,\notag
\end{align}
where $\tilde{W}^k\sim\GOE(n)$ is independent of $\mathscr{S}_k$. To justify the key distributional equality after~\eqref{eq:Wcondproof}, we can apply Lemma~\ref{lem:conddist}(c); indeed, note that $WP_k=Y_kM_k^+$, $\tilde{U}_k$ and $P_k^\perp=\tilde{U}_k\tilde{U}_k^\top$ are $\mathscr{S}_k$-measurable, and that $\tilde{U}_k^\top W\tilde{U}_k\eqdcond{\mathscr{S}_k}\tilde{U}_k^\top\tilde{W}^k\tilde{U}_k$ by the final assertion of Proposition~\ref{prop:condgoe}. By replacing $WP_k$ with $Y_kM_k^+$ in the display above, we obtain~\eqref{eq:Wcond} for every $k\in\{0,1,\dotsc,n-1\}$, as desired. Since $I-P_0=P_0^\perp=I_n$, this specialises to $W\eqdcond{\mathscr{S}_0}\tilde{W}^0$ when $k=0$, which is the first part of~\eqref{eq:condzero}.

Using~\eqref{eq:Wcond}, we now derive the conditional distribution of $h^{k+1}$ given $\mathscr{S}_k$ for $k\in\{0,1,\dotsc,n-1\}$. When $k=0$, we have $h^1=Wm^0$, so the associated identity in~\eqref{eq:condzero} follows directly from the first part of~\eqref{eq:condzero}, Lemma~\ref{lem:Wu} and Lemma~\ref{lem:conddist}(c). Turning now to~\eqref{eq:hcond} with $k\geq 1$, we have $h^{k+1}=Wm^k-b_km^{k-1}$, where $b_k,m^{k-1}$ are $\mathscr{S}_k$-measurable, so we can deduce from~\eqref{eq:Wcond} and Lemma~\ref{lem:conddist}(c) that
\begin{align}
h^{k+1}&\eqdcond{\mathscr{S}_k}Y_kM_k^+m^k+(Y_kM_k^+)^\top P_k^\perp m^k+P_k^\perp\tilde{W}^kP_k^\perp m^k-b_km^{k-1}\notag\\
&\phantom{\restr{=}{}}=Y_k\alpha^k+(Y_kM_k^+)^\top\barpp{m}{k}+P_k^\perp(\tilde{W}^k\barpp{m}{k})-b_km^{k-1}\notag\\
\label{eq:hcondproof}
&\phantom{\restr{=}{}}=H_k\alpha^k+(0\;M_{k-1})B_k\alpha^k+(H_kM_k^+)^\top\barpp{m}{k}+P_k^\perp(\tilde{W}^k\barpp{m}{k})-b_km^{k-1}.
\end{align}
Indeed, to obtain the final equality above, observe that $Y_k=H_k+(0\;M_{k-1})B_k$ and $B_k^\top (0\;M_{k-1})^\top\barpp{m}{k}=B_k^\top (0\;M_{k-1})^\top P_k^\perp m^k=0$ in view of the fact that $P_k^\perp M_{k-1}=0$. Since $\barpp{m}{k}$ is $\mathscr{S}_k$-measurable and $\tilde{W}^k\sim\GOE(n)$ is independent of $\mathscr{S}_k$ (and therefore has conditional distribution $\GOE(n)$ given $\mathscr{S}_k$), it follows from Lemmas~\ref{lem:Wu} and~\ref{lem:conddist}(b) that $\tilde{W}^k\barpp{m}{k}\eqdcond{\mathscr{S}_k}\norm{\barpp{m}{k}}_n\tilde{Z}^{k+1}+\tilde{\zeta}^{k+1}\barpp{m}{k}$. Now since $P_k^\perp$ and all the other summands in~\eqref{eq:hcondproof} are $\mathscr{S}_k$-measurable, a further application of Lemma~\ref{lem:conddist}(c) shows that the random variable in~\eqref{eq:hcondproof} and
\begin{align}
&H_k\alpha^k+P_k^\perp\bigl(\norm{\barpp{m}{k}}_n\tilde{Z}^{k+1}+\tilde{\zeta}^{k+1}\barpp{m}{k}\bigr)+(H_kM_k^+)^\top\barpp{m}{k}-\bigl\{b_km^{k-1}-(0\;M_{k-1})B_k\alpha^k\bigr\}\notag\\
\label{eq:hcondid}
&=\sum_{\ell=1}^k\alpha_\ell^k\,h^\ell+P_k^\perp\bigl(\norm{\barpp{m}{k}}_n\tilde{Z}^{k+1}+\tilde{\zeta}^{k+1}\barpp{m}{k}\bigr)+(H_kM_k^+)^\top\barpp{m}{k}-\biggl(b_k m^{k-1}-\sum_{\ell=1}^k \alpha_\ell^k\,b_{\ell-1}m^{\ell-2}\biggr)
\end{align}
are identically distributed given $\mathscr{S}_k$. Finally, recall that $\barpp{m}{k}=(I-P_k)\,m^k=m^k-\sum_{\ell=1}^k\alpha_\ell^k\,m^{\ell-1}$, and that $(M_k^+)^\top M_k^\top m^\ell=P_k^\top m^\ell=P_k m^\ell=m^\ell$ for all $0\leq\ell\leq k-1$ by the definition of the projection matrix $P_k=M_k M_k^+$. It follows that $P_k^\perp(\tilde{\zeta}^{k+1}\barpp{m}{k})=\tilde{\zeta}^{k+1}\barpp{m}{k}$ and
\begin{align*}
(M_k^+)^\top H_k^\top\barpp{m}{k}&=(M_k^+)^\top\biggl(H_k^\top m^k-\sum_{\ell=1}^k\alpha_\ell^k\,H_k^\top m^{\ell-1}\biggr)\\
b_k m^{k-1}-\sum_{\ell=1}^k b_{\ell-1}\alpha_\ell^k\, m^{\ell-2}&=(M_k^+)^\top\biggl(b_k\,M_k^\top m^{k-1}-\sum_{\ell=1}^k\alpha_\ell^k\,b_{\ell-1}\,M_k^\top m^{\ell-2}\biggr).
\end{align*}
Thus, since $(M_k^+)^\top=M_k(M_k^\top M_k)^+$, the random variable $h^{k+1,k}$ defined in~\eqref{eq:hcond} is identical to that in~\eqref{eq:hcondid}, so we conclude from~\eqref{eq:hcondproof} that $h^{k+1}\eqdcond{\mathscr{S}_k} h^{k+1,k}$, as required.
\end{proof}
\umparskip
\subsection{Proof outline for the AMP master theorems in Section~\protect{\ref{sec:AMPsym}}}
\label{sec:inductive}
Recalling the definition~\eqref{eq:Sigmabar} of the limiting covariance matrices $\bar{\mathrm{T}}^{[k]}\in\R^{k\times k}$ in Theorem~\ref{thm:masterext}, we first outline a standard construction of a single random sequence $(\bar{G}_k:k\in\N)$ satisfying $(\bar{G}_1,\dotsc,\bar{G}_k)\sim N_k(0,\bar{\mathrm{T}}^{[k]})$ for each $k$. Let $\bar{\mathrm{T}}^{[k],k+1}:=(\bar{\mathrm{T}}_{1,k+1},\dotsc,\bar{\mathrm{T}}_{k,k+1})\in\R^k$ and
\begin{equation}
\label{eq:alphabar}
\bar{\alpha}^k\equiv(\bar{\alpha}_1^k,\dotsc,\bar{\alpha}_k^k):=\bigl(\bar{\mathrm{T}}^{[k]}\bigr)^{-1}\,\bar{\mathrm{T}}^{[k],k+1}\in\R^k
\end{equation}
for each $k\in\N$, where the latter is well-defined since $\bar{\mathrm{T}}^{[k]}$ is positive definite under~\ref{ass:A4} by Lemma~\ref{lem:posdef}. It is easily verified that if $(G_1,\dotsc,G_{k+1})\sim N_{k+1}(0,\bar{\mathrm{T}}^{[k+1]})$, then $G_{[k]}:=(G_1,\dotsc,G_k)$ and $\xi_{k+1}:=G_{k+1}-G_{[k]}^\top\,\bar{\alpha}^k=G_{k+1}-\sum_{\ell=1}^k\bar{\alpha}_\ell^k\,G_\ell$ are uncorrelated and hence independent. This means that
\[G_{[k]}^\top\,\bar{\alpha}^k=\sum_{\ell=1}^k\bar{\alpha}_\ell^k\,G_\ell=\E(G_{k+1}\,|\,G_1,\dotsc,G_k)\]
for each $k$. Moreover, since $\bar{\mathrm{T}}^{k+1}=\Cov(G_1,\dotsc,G_{k+1})$ is positive definite and $\xi_{k+1}$ is a non-trivial linear combination of $G_1,\dotsc,G_{k+1}$, it follows under~\ref{ass:A4} that
\begin{align}
0&<\Var(\xi_{k+1})=\Var(\xi_{k+1}\,|\,G_1,\dotsc,G_k)
=\Var(G_{k+1}\,|\,G_1,\dotsc,G_k)\notag\\
&=\Var(G_{k+1})-\Var\bigl(G_{[k]}^\top\,\bar{\alpha}^k\bigr)\notag\\
&=\bar{\mathrm{T}}_{k+1,k+1}-(\bar{\alpha}^k)^\top\bar{\mathrm{T}}^{[k]}\,\bar{\alpha}^k\notag\\
\label{eq:tauperpbar}
&=\tau_{k+1}^2-(\bar{\mathrm{T}}^{[k],k+1})^\top\bigl(\bar{\mathrm{T}}^{[k]}\bigr)^{-1}\,\bar{\mathrm{T}}^{[k],k+1}=:\barpp{\tau}{2}_{k+1}
\end{align}
for $k\in\N$, so that $\barp{\tau}_{k+1}\in (0,\infty)$ satisfies $\barpp{\tau}{2}_{k+1}=\Var(\xi_{k+1})\leq\Var(G_{k+1})=\tau_{k+1}^2$. 
Now let $\bar{G}_1\sim N(0,\tau_1^2)$, and for $k\in\N$, inductively define
\begin{equation}
\label{eq:Gkbar}
\bar{G}_{k+1}:=\sum_{\ell=1}^k\bar{\alpha}_\ell^k\,\bar{G}_\ell+\barp{\tau}_{k+1}\barp{\zeta}_{k+1},
\end{equation}
where $\barp{\zeta}_{k+1}\sim N(0,1)$ is independent of $(\bar{G}_1,\dotsc,\bar{G}_k)$. Then $(\bar{G}_k:k\in\N)$ is a random sequence with $(\bar{G}_1,\dotsc,\bar{G}_k)\sim N_k(0,\bar{\mathrm{T}}^{[k]})$ for each $k$, as desired.
With the above definitions in place, we record here some key identities. 
In view of~\eqref{eq:Sigmabar}, we certainly have
\begin{equation}
\label{eq:covid1}
\Cov(\bar{G}_k,\bar{G}_\ell)=\E(\bar{G}_k\bar{G}_\ell)=\bar{\mathrm{T}}_{k,\ell}=
\begin{cases}
\,\tau_1^2\quad&\text{if }k=\ell=1\\
\,\E\bigl(F_0(\bar{\gamma})\cdot f_{k-1}(\bar{G}_{k-1},\bar{\gamma})\bigr)\quad&\text{if }k>\ell=1\\
\,\E\bigl(f_{\ell-1}(\bar{G}_{\ell-1},\bar{\gamma})\cdot f_{k-1}(\bar{G}_{k-1},\bar{\gamma})\bigr)\quad&\text{if }k\geq\ell\geq 2,
\end{cases}
\end{equation}
where $f_1,f_2,\dotsc$ are the Lipschitz functions in the AMP recursion~\eqref{eq:AMPsym}, and $\tau_1$ and $F_0$ are as in~\ref{ass:A2} and~\ref{ass:A3} respectively. This fact underlies an important assertion (Proposition~\ref{prop:AMPproof}(e) below) in our inductive proof of the master theorems. Moreover, for $k,\ell \in \mathbb{N}$ and any Lipschitz function $\varphi\colon\R\to\R$ with weak derivative $\varphi'$, we have
\begin{equation}
\label{eq:steinid}
\E\bigl(\bar{G}_k\,\varphi(\bar{G}_\ell)\bigr)=\E\bigl(\varphi'(\bar{G}_\ell)\bigr)\E(\bar{G}_k\bar{G}_\ell)=\E\bigl(\varphi'(\bar{G}_\ell)\bigr)\bar{\mathrm{T}}_{k,\ell}.
\end{equation}
This follows from Stein's lemma, a general formulation of which can be found in~\citet[Lemma~3.6]{Tsy09} and Lemma~\ref{lem:steinsigma}.
\begin{lemma}[Stein's lemma]
\label{lem:stein}
If $Z\sim N(0,\sigma^2)$ and $\varphi\colon\R\to\R$ is an absolutely continuous function with weak derivative $\varphi'$ such that $\varphi'(Z)$ is integrable, then $\E\bigl(Z\varphi(Z)\bigr)=\sigma^2\,\E\bigl(\varphi'(Z)\bigr)$.
\end{lemma}

\unparskip
Indeed, the first equality in~\eqref{eq:steinid} follows from Lemma~\ref{lem:stein} upon writing $\bar{G}_k=(\bar{\mathrm{T}}_{k,\ell}/\bar{\mathrm{T}}_{\ell,\ell})\,\bar{G}_\ell+\xi_{k\ell}$, where $\xi_{k\ell}$ has zero mean and is independent of $\bar{G}_\ell$, so that $\E\bigl(\xi_{k\ell}\,\varphi(\bar{G}_\ell)\bigr)=0$.

Our choice of functions $(f_k)_{k=0}^\infty$ and $(f_k')_{k=0}^\infty$ in~\eqref{eq:AMPsym} ensures that for any fixed $y\in\R$, we can take $\varphi=f_\ell(\cdot\,,y)$ and $\varphi'=f_\ell'(\cdot\,,y)$ in~\eqref{eq:steinid} to see that $\E\bigl(\bar{G}_k\,f_\ell(\bar{G}_\ell,y)\bigr)=\E\bigl(f_\ell'(\bar{G}_\ell,y)\bigr)\E(\bar{G}_k\bar{G}_\ell)$ for $k,\ell \in \mathbb{N}$.
We deduce from this (and Lemma~\ref{lem:indeplem}) that if $\bar{\gamma}\sim\pi$ is independent of $\bar{G}_1,\bar{G}_2,\dotsc$, then
\begin{equation}
\label{eq:covid2}
\E\bigl(\bar{G}_k\,f_\ell(\bar{G}_\ell,\bar{\gamma})\bigr)=\E\bigl(f_\ell'(\bar{G}_\ell,\bar{\gamma})\bigr)\E(\bar{G}_k\bar{G}_\ell)=\bar{b}_\ell\,\bar{\mathrm{T}}_{k,\ell}
\end{equation}
for all $k,\ell\in \mathbb{N}$, where $\bar{b}_\ell:=\E\bigl(f_\ell'(\bar{G}_\ell,\bar{\gamma})\bigr)$. This forms part of assertion (f) in Proposition~\ref{prop:AMPproof} below.

To complete our technical preparations for the main derivations below, we will set up a more explicit connection between the Gaussian variables $\bar{G}_{k+1}\sim N(0,\tau_{k+1}^2)$ in~\eqref{eq:Gkbar} and the random vectors $h^{k+1,k}\equiv h^{k+1,k}(n)$ defined for $n\in\N$ and $k\in\{0,1,\dotsc,n-1\}$ in~\eqref{eq:condzero} and~\eqref{eq:hcond} in Section~\ref{sec:AMPcond} above. For such $n$ and $k$, Proposition~\ref{prop:AMPconddist} asserts that $h^{k+1}(n)$ and $h^{k+1,k}(n)$ are identically distributed given $\mathscr{S}_k\equiv\mathscr{S}_k(n)=\sigma(\gamma,m^0,h^j:1\leq j\leq k)$, and we now write $h^{k+1,k}(n)=\tilde{h}^{k+1}(n)+\Delta^{k+1}(n)$, where
\begin{equation}
\label{eq:sum10}
\tilde{h}^1\equiv\tilde{h}^1(n):=\tau_1\tilde{Z}^1\quad\text{and}\quad\Delta^1\equiv\Delta^1(n):=(\norm{m^0}_n-\tau_1)\tilde{Z}^1+\tilde{\zeta}^1 m^0,
\end{equation}
and
\begin{align}
\label{eq:htilde}
\tilde{h}^{k+1}\equiv\tilde{h}^{k+1}(n)&:=\sum_{\ell=1}^k\bar{\alpha}_\ell^k\,h^\ell+\barp{\tau}_{k+1}\tilde{Z}^{k+1},\\
\Delta^{k+1}\equiv\Delta^{k+1}(n)&:=\sum_{\ell=1}^k(\alpha_\ell^k-\bar{\alpha}_\ell^k)\,h^\ell+M_k(M_k^\top M_k)^+\biggl(v^{k,k}-\sum_{\ell=1}^k\alpha_\ell^k\,v^{k,\ell-1}\biggr)\notag\\
\label{eq:deltak}
&\hspace{1.15cm}-\norm{\barpp{m}{k}}_n(P_k\tilde{Z}^{k+1})+(\norm{\barpp{m}{k}}_n-\barp{\tau}_{k+1})\tilde{Z}^{k+1}+\tilde{\zeta}^{k+1}\barpp{m}{k}.
\end{align}
Recall that $(\tilde{Z}^{k+1},\tilde{\zeta}^{k+1})\equiv\bigl(\tilde{Z}^{k+1}(n),\tilde{\zeta}^{k+1}(n)\bigr)\sim N_n(0,I_n)\otimes N(0,1/n)$ was taken to be independent of $\mathscr{S}_k\equiv\mathscr{S}_k(n)$ in Proposition~\ref{prop:AMPconddist}, where we also defined $\alpha^k\equiv\alpha^k(n)$ and $v^{k,\ell}\equiv v^{k,\ell}(n)$ for $0\leq\ell\leq k$. 

In the decomposition above, we have defined $\tilde{h}^{k+1}$ in~\eqref{eq:htilde} to mimic the expression for the limiting Gaussian variable $\bar{G}_{k+1}$ in~\eqref{eq:Gkbar}.
Contrasting the definitions of $h^{k+1,k}$ and $\tilde{h}^{k+1}$ in~\eqref{eq:hcond} and~\eqref{eq:htilde} respectively for $k\in\{0,1,\dotsc,n-1\}$, we see that the random quantities $\alpha^k$ and $\norm{\barpp{m}{k}}_n$ in~\eqref{eq:hcond} are replaced in~\eqref{eq:htilde} with the deterministic $\bar{\alpha}^k\in\R^k$ and $\barp{\tau}_{k+1}\in (0,\infty)$ from~\eqref{eq:alphabar} and~\eqref{eq:tauperpbar} respectively; these turn out to be the correct limiting values in Proposition~\ref{prop:AMPproof}(i,\,j) below under the non-degeneracy assumption~\ref{ass:A4}. 

We are now in a position to state the main result of this subsection. To ease notation, we will often suppress the dependence on $n$ of quantities such as $h^k\equiv h^k(n)$, $v^{k,\ell}\equiv v^{k,\ell}(n)$, $\alpha^k\equiv\alpha^k(n)$ and $\Delta^k\equiv\Delta^k(n)$. 
\begin{proposition}
\label{prop:AMPproof}
For a sequence of symmetric AMP recursions~\eqref{eq:AMPsym} satisfying~\emph{\ref{ass:A0}}--\emph{\ref{ass:A5}} as well as~\emph{\ref{ass:A4}}, the following hold as $n\to\infty$ for each $k\in\N$:

\unparskip
\begin{enumerate}[label=(\alph*)]
\item $\norm{\Delta^k}_{n,r}\cvc 0$;
\item $\norm{h^j}_{n,r}=O_c(1)$ for $1\leq j\leq k$;

$\norm{m^j}_{n,r}=O_c(1)$ for $0\leq j\leq k$;
\item $\inv{n}\sum_{i=1}^n\psi(h_i^1,\dotsc,h_i^k,\gamma_i)\cvc\E\bigl(\psi(\bar{G}_1,\dotsc,\bar{G}_k,\bar{\gamma})\bigr)$ for every $\psi\in\PL_{k+1}(r)$;
\item $\inv{n}\sum_{i=1}^n m_i^0\,\phi(h_i^1,\dotsc,h_i^k,\gamma_i)\cvc\E\bigl(F_0(\bar{\gamma})\cdot\phi(\bar{G}_1,\dotsc,\bar{G}_k,\bar{\gamma})\bigr)$ for every $\phi\in\PL_{k+1}(1)$;
\item $\ipr{m^{j-1}}{m^{\ell-1}}_n\cvc\E(\bar{G}_j\bar{G}_\ell)=\bar{\mathrm{T}}_{j,\ell}$ for $1\leq j,\ell\leq k+1$;
\item $\ipr{h^j}{m^\ell}_n=\ipr{h^j}{f_\ell(h^\ell,\gamma)}_n\cvc\E\bigl(\bar{G}_jf_\ell(\bar{G}_\ell,\bar{\gamma})\bigr)=\E\bigl(f_\ell'(\bar{G}_\ell,\bar{\gamma})\bigr)\E(\bar{G}_j\bar{G}_\ell)=\bar{b}_\ell\,\bar{\mathrm{T}}_{j,\ell}$ for $1\leq j,\ell\leq k$;

$\ipr{h^j}{m^0}_n\cvc 0$ for $1\leq j\leq k$;
\item $b_k=\abr{f_k'(h^k,\gamma)}_n\cvc\E\bigl(f_k'(\bar{G}_k,\bar{\gamma})\bigr)=\bar{b}_k$;
\item $v^{k,\ell}/n=(H_k^\top m^\ell-b_\ell\,M_k^\top m^{\ell-1})/n\cvc 0$ for $0\leq\ell\leq k$;
\item $\alpha^k\cvc\bar{\alpha}^k$;
\item $\norm{\barpp{m}{k}}_n\cvc\barp{\tau}_{k+1}=\Var^{1/2}(\bar{G}_{k+1}\,|\,\bar{G}_1,\dotsc,\bar{G}_k)$.
\end{enumerate}

\unparskip
\end{proposition}
\begin{remark}
\label{rem:AMPproof}
Under~\ref{ass:A4} and the alternative hypotheses of Remark~\ref{rem:AMPconv}(a), the assertions (a)--(j) above remain valid if we replace $\cvc$ with $\cvp$ and $O_c(1)$ with $O_p(1)$ throughout.
\end{remark}

\unparskip
To establish Proposition~\ref{prop:AMPproof}, we proceed by induction on $k\in\N$ and prove the assertions (a)--(i) one at a time (in that order). Here, we will give a technical summary of the inductive argument (which can be read alongside the detailed proof in Section~\ref{sec:masterproofs}) to highlight its overall structure and key features. Henceforth, we write $\mathcal{H}_k(\cdots)$ for parts $(\cdots)$ of the inductive hypothesis for $k\in\N$. 

$\mathcal{H}_k(e,f)$: These are obtained as direct consequences of the inductive hypotheses $\mathcal{H}_k(c,d)$ by choosing suitable pseudo-Lipschitz functions $\psi\in\PL_{k+1}(2)\subseteq\PL_{k+1}(r)$ that depend on at most three of their $k+1$ arguments. We use $\mathcal{H}_k(d)$ to handle the inner products that feature $m^0$ and apply $\mathcal{H}_k(c)$ to those that do not. In $\mathcal{H}_k(e)$, the limiting value of $\ipr{m^{j-1}}{m^{\ell-1}}_n=\ipr{f_{j-1}(h^{j-1},\gamma)}{f_{\ell-1}(h^{\ell-1},\gamma)}_n$ is shown to be 
\[\begin{cases}
\E\bigl(f_{j-1}(\bar{G}_{j-1},\bar{\gamma})\cdot f_{\ell-1}(\bar{G}_{\ell-1},\bar{\gamma})\bigr)=\E(\bar{G}_j\bar{G}_\ell)\quad&\text{for }2\leq j,\ell\leq k+1\\
\E\bigl(F_0(\bar{\gamma})\cdot f_{j-1}(\bar{G}_{j-1},\bar{\gamma}))=\E(\bar{G}_1\bar{G}_\ell)\quad&\text{for }1=j<\ell\leq k+1,
\end{cases}
\] 
where the two equalities are drawn from~\eqref{eq:covid1} and form the basis of the definition of the limiting covariances $\bar{\mathrm{T}}_{j,\ell}$ in~\eqref{eq:Sigmabar}. Moreover, for $1\leq j,\ell\leq k$, the identity $\E\bigl(\bar{G}_jf_\ell(\bar{G}_\ell,\bar{\gamma})\bigr)=\E\bigl(f_\ell'(\bar{G}_\ell,\bar{\gamma})\bigr)\E(\bar{G}_j\bar{G}_\ell)$ in the first line of $\mathcal{H}_k(f)$ comes from~\eqref{eq:covid2}. 
These identities~\eqref{eq:covid1} and~\eqref{eq:covid2} ultimately provide the crucial link between the limiting values of $\ipr{h^j}{m^\ell}_n$ and $b_\ell\,\ipr{m^{j-1}}{m^{\ell-1}}_n$ in $\mathcal{H}_k(h)$. 

$\mathcal{H}_k(g,h)$: This is also derived from $\mathcal{H}_k(c)$, but since $f_k'\colon\R^2\to\R$ need not lie in $\PL_2(r)$, we instead apply the analytic Lemmas~\ref{lem:seqcoup} and~\ref{lem:weakconv} rather than imitate the proofs of $\mathcal{H}_k(e,f)$. See the proof of Corollary~\ref{cor:Wr}(b) for a similar argument. $\mathcal{H}_k(h)$ follows immediately from $\mathcal{H}_k(e,f,g)$.

$\mathcal{H}_k(i,j)$: We see from $\mathcal{H}_k(e)$ that the matrices $M_k^\top M_k/n\in\R^{k\times k}$ converge completely to the limiting covariance matrix $\bar{\mathrm{T}}^{[k]}=\Cov(\bar{G}_1,\dotsc,\bar{G}_k)\in\R^{k\times k}$, which is positive definite under~\ref{ass:A4}.
In $\mathcal{H}_k(i)$, we consider $\alpha^k=(M_k^\top M_k/n)^+(M_k^\top m^k/n)\in\R^k$, a vector of projection coefficients defined in~\eqref{eq:alphak}.
It follows from $\mathcal{H}_k(e)$ that $(M_k^\top M_k/n)^+$ and $M_k^\top m^k/n$ converge completely to $(\bar{\mathrm{T}}^{[k]})^{-1}$ and $\bar{\mathrm{T}}^{[k],k+1}$ respectively, and hence that $\alpha^k\cvc(\bar{\mathrm{T}}^{[k]})^{-1}\,\bar{\mathrm{T}}^{[k],k+1}=\bar{\alpha}^k$, as defined in~\eqref{eq:alphabar}. For $\mathcal{H}_k(j)$, we recall the definitions at the start of Section~\ref{sec:AMPcond} and write \[\norm{\barpp{m}{k}}_n^2=\norm{P_k^\perp m^k}_n^2=\norm{m^k}_n^2-\norm{P_k m^k}_n^2=\norm{m^k}_n^2-(\alpha^k)^\top(M_k^\top M_k/n)\,\alpha^k.\]
Applying $\mathcal{H}_k(e,i)$ to the individual terms on the right hand side above, 
we deduce that $\norm{\barpp{m}{k}}_n^2\cvc\bar{\mathrm{T}}_{k+1,k+1}-(\bar{\alpha}^k)^\top\bar{\mathrm{T}}^{[k]}\,\bar{\alpha}^k=\barpp{\tau}{2}_{k+1}$, as defined in~\eqref{eq:tauperpbar}.

$\mathcal{H}_{k+1}(a)$: 
It is thanks to the key fact $\mathcal{H}_k(h)$ and the presence of the Onsager term $-b_km^{k-1}$ in the original AMP recursion~\eqref{eq:AMPsym} (and subsequently in~\eqref{eq:hcond} in Proposition~\ref{prop:AMPconddist}) that the $\norm{{\cdot}}_{n,r}$ norm of the second term in~\eqref{eq:deltak}
converges completely to 0. 
Using $\mathcal{H}_k(b,i,j)$ to handle some of the remaining terms in this definition~\eqref{eq:deltak} of the deviation term $\Delta^{k+1}$, we conclude that $\norm{\Delta^{k+1}}_{n,r}\cvc 0$. 

$\mathcal{H}_{k+1}(b)$: Using the distributional equality $h^{k+1}\eqd h^{k+1,k}=\tilde{h}^{k+1}+\Delta^{k+1}=\sum_{\ell=1}^k\bar{\alpha}_\ell^k\,h^\ell+\barp{\tau}_{k+1}\tilde{Z}^{k+1}+\Delta^{k+1}$ from Proposition~\ref{prop:AMPconddist} and~(\ref{eq:htilde},\,\ref{eq:deltak}), we deduce from $\mathcal{H}_{k+1}(a)$ and the inductive hypothesis $\mathcal{H}_k(b)$ that $\norm{h^{k+1}}_{n,r}=O_c(1)$. Since $\norm{\gamma}_{n,r}=O_c(1)$ by~\ref{ass:A1} and $f_{k+1}$ is Lipschitz, this in turn implies that $\norm{m^{k+1}}_{n,r}=\norm{f_k(h^{k+1},\gamma)}_{n,r}=O_c(1)$.

$\mathcal{H}_{k+1}(c)$: This is the main assertion in Proposition~\ref{prop:AMPproof}; by Corollary~\ref{cor:Wr}(b), it is in fact equivalent to the conclusion~\eqref{eq:masterext} of Theorem~\ref{thm:masterext}. We first condition on $\mathscr{S}_k=\sigma(\gamma,m^0,h^j:1\leq j\leq k)$ and appeal to Proposition~\ref{prop:AMPconddist}, which asserts that for each $n>k$, the conditional distribution of $h^{k+1}\equiv h^{k+1}(n)$ given $\mathscr{S}_k$ is identical to that of $h^{k+1,k}\equiv h^{k+1,k}(n)$ from~\eqref{eq:hcond}. With $h^{k+1,k}=\tilde{h}^{k+1}+\Delta^{k+1}$ in place of $h^{k+1}$ on the left hand side of $\mathcal{H}_{k+1}(c)$, we use $\mathcal{H}_{k+1}(a,b)$ to show that the `deviation' term $\Delta^{k+1}\equiv\Delta^{k+1}(n)$ from~\eqref{eq:deltak} has asymptotically negligible effect, so that $h^{k+1,k}$ can in fact be replaced with $\tilde{h}^{k+1}$ in all relevant expressions. In~\eqref{eq:htilde}, $\tilde{h}^{k+1}$ was defined as $\sum_{\ell=1}^k\bar{\alpha}_\ell^k\,h^\ell+\barp{\tau}_{k+1}\tilde{Z}^{k+1}$, where $\sum_{\ell=1}^k\bar{\alpha}_\ell^k\,h^\ell$ is a deterministic linear combination of the previous iterates $h^1,\dotsc,h^k$, and $\barp{\tau}_{k+1}\tilde{Z}^{k+1}$ is a new Gaussian variable that has i.i.d.\ components and is independent of $\mathscr{S}_k$. 

In view of this, the proof of $\mathcal{H}_{k+1}(c)$ can be completed in two stages (given by~\eqref{eq:Hkc3} and~\eqref{eq:Hkc1} below): the influence of the latter Gaussian term can first be understood by appealing to $\mathcal{H}_{k+1}(b)$ and a general concentration result for sums of pseudo-Lipschitz functions of independent Gaussians (Lemma~\ref{lem:pseudoconc}), before we subsequently reintroduce the randomness in $\gamma,m^0,h^1,\dotsc,h^k$ and apply the inductive hypothesis $\mathcal{H}_k(c)$ to account for this.
The appearance of the new limiting Gaussian variable $\bar{G}_{k+1}$ on the right hand side of $\mathcal{H}_{k+1}(c)$ (in addition to the existing $\bar{G}_1,\dotsc,\bar{G}_k$ from $\mathcal{H}_k(c)$) can be explained through its definition in~\eqref{eq:Gkbar}, which matches up neatly with the definition~\eqref{eq:htilde} of $\tilde{h}^{k+1}$ and the two-stage argument we have just outlined; see~\eqref{eq:Hkc1} and~\eqref{eq:Hkc2} in the proof. 

$\mathcal{H}_{k+1}(d)$: The proof of this is similar in spirit to that of $\mathcal{H}_{k+1}(c)$, except that it also makes use of condition~\ref{ass:A3}. Note also that $\mathcal{H}_{k+1}(d)$ applies only to Lipschitz $\phi\colon\R^{k+2}\to\R$ rather than general $\phi\in\PL_{k+2}(r)$, but this is sufficient for our purposes in the subsequent proofs of $\mathcal{H}_{k+1}(e,f)$. 

The proofs we give for $\mathcal{H}_{k+1}(c,d)$ combine aspects of the asymptotic and finite-sample arguments (see Remark~\ref{rem:finitesample}) in the existing AMP literature. Proposition~E.1 in~\citet{Fan20} provides the basis for an alternative asymptotic approach, whose details we omit.
\hfparskip
\subsection{Proofs for Sections~\ref{sec:AMPsym} and~\ref{sec:AMPrem}}
\label{sec:masterproofs}
\begin{proof}[Proof of Proposition~\protect{\ref{prop:AMPproof}}]
Since we are carrying out an asymptotic analysis, we may assume without loss of generality that $n>k$ in the proofs of $\mathcal{H}_k(a,\dotsc,i)$ for each $k\in\N$; this enables us to apply the results on conditional distributions from Section~\ref{sec:AMPcond}. Note also that we use $T_n,T_{n1},T_{n1}',T_{n2}$ to refer to different quantities of interest in different parts of the proof. In Lemma~\ref{lem:slutsky} and Remark~\ref{rem:arithmetic}, we state versions of the continuous mapping theorem and Slutsky's lemma for complete convergence, as well the `arithmetic rules' for $o_c$ and $O_c$ symbols. We will apply these repeatedly in the arguments below, often without further comment or explanation.

First, we prove $\mathcal{H}_1(a,b,c,d)$, which form the base case for the induction.

$\mathcal{H}_1(a)$: Recall from~\eqref{eq:sum10} that $\Delta^1\equiv\Delta^1(n)=(\norm{m^0}_n-\tau_1)\tilde{Z}^1+\tilde{\zeta}^1 m^0$, where $(\tilde{Z}^1,\tilde{\zeta}^1)\equiv(\tilde{Z}^1(n),\tilde{\zeta}^1(n))\sim N_n(0,I_n)\otimes N(0,1/n)$ for each $n$. Taking $\zeta\sim N(0,1)$, we have $\abs{\tilde{\zeta}^1}\eqd n^{-1/2}\,\abs{\zeta}\cvc 0$ by Example~\ref{ex:exptails}(a), and $\norm{\tilde{Z}^1}_{n,r}=(n^{-1}\sum_{i=1}^n\,\abs{\tilde{Z}_i^1}^r)^{1/r}\cvc\E(\abs{\zeta}^r)^{1/r}\in (0,\infty)$ by Lemma~\ref{lem:pseudoconc} and Proposition~\ref{prop:compequiv}. Moreover, $\bigl|\norm{m^0}_n-\tau_1\bigr|\cvc 0$ and $\norm{m^0}_{n,r}=O_c(1)$ by~\ref{ass:A2}. Putting everything together, we recall from Remark~\ref{rem:arithmetic} the `arithmetic rules'~\eqref{eq:arithmetic} for $o_c$ and $O_c$ symbols, and conclude using the triangle inequality for $\norm{{\cdot}}_{n,r}$ that
\[\norm{\Delta^1}_{n,r}\leq\bigl|\norm{m^0}_n-\tau_1\bigr|\,\norm{\tilde{Z}^1}_{n,r}+\abs{\tilde{\zeta}^1}\,\norm{m^0}_{n,r}=o_c(1)\,O_c(1)+o_c(1)\,O_c(1)=o_c(1).\]
$\mathcal{H}_1(b)$: Recall from~\eqref{eq:condzero} in Proposition~\ref{prop:AMPconddist} and~\eqref{eq:sum10} that 
\begin{equation}
\label{eq:h1}
h^1\equiv h^1(n)\eqdcond{\mathscr{S}_0}\tilde{h}^1(n)+\Delta^1(n)=h^{1,0}(n)\equiv h^{1,0}
\end{equation}
for each $n\in\N$, where $\tilde{h}^1\equiv\tilde{h}^1(n)=\tau_1\tilde{Z}^1$ and $\tilde{Z}^1\sim N_n(0,I_n)$ is independent of $\mathscr{S}_0=\sigma(\gamma,m^0)$. Then $\norm{\Delta^1}_{n,r}=o_c(1)$ by $\mathcal{H}_1(a)$ and $\norm{\tilde{h}^1}_{n,r}=\tau_1\norm{\tilde{Z}^1}_{n,r}=O_c(1)$ as in the proof of $\mathcal{H}_1(a)$, so \[\norm{h^1}_{n,r}\eqd\norm{h^{1,0}}_{n,r}\leq\norm{\tilde{h}^1}_{n,r}+\norm{\Delta^1}_{n,r}=O_c(1)+o_c(1)=O_c(1).\]
We already have $\norm{m^0}_{n,r}=O_c(1)$ by~\ref{ass:A2}. In addition, $\norm{\gamma}_{n,r}=(n^{-1}\sum_{i=1}^n\,\abs{\gamma_i}^r)^{1/r}\cvc\E(\abs{\bar{\gamma}}^r)^{1/r}$ by~\ref{ass:A1}, so $\norm{\gamma}_{n,r}=O_c(1)$. Letting $L'>0$ be such that the function $f_1$ in the AMP recursion~\eqref{eq:AMPsym} lies in $\PL_2(1,L')$, we have $\abs{f_1(x,y)}\leq\abs{f_1(0,0)}+L'(\abs{x}+\abs{y})$ for all $(x,y)\in\R^2$, so we can apply the triangle inequality for $\norm{{\cdot}}_{n,r}$ to deduce that \[\norm{m^1}_{n,r}=\norm{f_1(h^1,\gamma)}_{n,r}\leq\abs{f_1(0,0)}\,\norm{\mathbf{1}_n}_{n,r}+L'(\norm{h^1}_{n,r}+\norm{\gamma}_{n,r})=O_c(1).\]
$\mathcal{H}_1(c)$: For each $n$, note that $(\gamma,h^1)\eqdcond{\mathscr{S}_0}(\gamma,h^{1,0})$ by~\eqref{eq:h1} and Lemma~\ref{lem:conddist}(c). Thus, for each fixed $\psi\in\PL_2(r)$, it follows that $\inv{n}\sum_{i=1}^n\psi(h_i^1,\gamma_i)\eqd \inv{n}\sum_{i=1}^n\psi(h_i^{1,0},\gamma_i)=:T_n$ for each $n$, so in view of the third bullet point in Remark~\ref{rem:AMPconv}, it is enough to show that $T_n\cvc\E\bigl(\psi(\bar{G}_1,\bar{\gamma})\bigr)$ as $n\to\infty$. To this end, we write
\[T_n=\frac{1}{n}\sum_{i=1}^n\psi(\tilde{h}_i^1,\gamma_i)+\frac{1}{n}\sum_{i=1}^n\,\bigl\{\psi(h_i^{1,0},\gamma_i)-\psi(\tilde{h}_i^1,\gamma_i)\bigr\}=:T_{n1}+T_{n2}\]
for each $n$, and aim to prove that $T_{n1}\cvc\E\bigl(\psi(\bar{G}_1,\bar{\gamma})\bigr)$ and $T_{n2}\cvc 0$, which together imply the desired conclusion.

Before proceeding, we briefly describe the techniques that we use to determine the limit of $(T_{n1})$ and also to prove $\mathcal{H}_1(d)$ and $\mathcal{H}_{k+1}(c,d)$ later on. It is instructive to consider the following two special cases where the claim is easier to establish. If $\psi$ depends only on its first argument, then since $h^1(n)=\tau_1\tilde{Z}^1(n)$ and $\tilde{Z}^1(n)\sim N_n(0,I_n)$ for each $n$, the result follows readily from the concentration inequality~\eqref{eq:pseudoconc} in Lemma~\ref{lem:pseudoconc} and the characterisation of complete convergence in Proposition~\ref{prop:compequiv}. On the other hand, if $\psi$ depends only on its second argument, then since $\bigl(\gamma\equiv\gamma(n):n\in\N\bigr)$ satisfies~\ref{ass:A1} by assumption, we can appeal directly to Corollary~\ref{cor:Wr}(b). 

For general $\psi\in\PL_2(r)$, we seek to combine these two different lines of reasoning by exploiting the independence of $\tilde{h}^1(n)$ and $\mathscr{S}_0\equiv\mathscr{S}_0(n)=\sigma(\gamma,m^0)$ for each $n$. This allows $\gamma(n)$ and $\tilde{h}^1(n)$ to be handled separately (to a large extent) when we decompose $T_{n1}$ as a sum of $\E(T_{n1}\,|\,\mathscr{S}_0)$ and $T_{n1}-\E(T_{n1}\,|\,\mathscr{S}_0)$ in~\eqref{eq:H1c1} and~\eqref{eq:H1c2} respectively. For the latter, it is helpful to first think of $\gamma(n)$ as being fixed when applying Lemma~\ref{lem:pseudoconc} to the Gaussian $\tilde{h}^1$, before subsequently accounting for the randomness of $\gamma(n)$ using~\ref{ass:A1}.

Define $\Psi\colon\R\to\R$ by $\Psi(y):=\E\bigl(\psi(\tau_1 Z,y)\bigr)$ with $Z\sim N(0,1)$. For each $n$, since $\tilde{Z}^1\equiv\tilde{Z}^1(n)\sim N_n(0,I_n)$ is independent of $\mathscr{S}_0\equiv\mathscr{S}_0(n)=\sigma(\gamma,m^0)$, we deduce from Lemma~\ref{lem:indeplem} that $\E\bigl(\psi(\tilde{h}_i^1,\gamma_i)\!\bigm|\!\mathscr{S}_0\bigr)=\E\bigl(\psi(\tau_1\tilde{Z}_i^1,\gamma_i)\!\bigm|\!\mathscr{S}_0\bigr)=\Psi(\gamma_i)$ almost surely, for every $1\leq i\leq n$. Since $\Psi\in\PL_1(r)$ by Lemma~\ref{lem:pseudocomp}(b), it follows from~\ref{ass:A1} and Corollary~\ref{cor:Wr}(b) that $\inv{n}\sum_{i=1}^n\Psi(\gamma_i)\cvc\E\bigl(\Psi(\bar{\gamma})\bigr)$ as $n\to\infty$, where $\bar{\gamma}\sim\pi$. A further application of Lemma~\ref{lem:indeplem} shows that if $\bar{G}_1\sim N(0,\tau_1^2)$ is independent of $\bar{\gamma}$, then $\E\bigl(\Psi(\bar{\gamma})\bigr)=\E\bigl(\E\bigl\{\psi(\bar{G}_1,\bar{\gamma})\!\bigm|\!\bar{\gamma}\bigr\}\bigl)=\E\bigl(\psi(\bar{G}_1,\bar{\gamma})\bigr)$, so in summary, we have
\begin{equation}
\label{eq:H1c1}
\frac{1}{n}\sum_{i=1}^n\E\bigl(\psi(\tilde{h}_i^1,\gamma_i)\!\bigm|\!\mathscr{S}_0\bigr)=
\frac{1}{n}\sum_{i=1}^n\Psi(\gamma_i)\cvc\E\bigl(\Psi(\bar{\gamma})\bigr)=\E\bigl(\psi(\bar{G}_1,\bar{\gamma})\bigr).
\end{equation}
To complete the proof that $T_{n1}\cvc\E\bigl(\psi(\bar{G}_1,\bar{\gamma})\bigr)$, we must therefore show that
\begin{equation}
\label{eq:H1c2}
T_{n1}':=\frac{1}{n}\sum_{i=1}^n\,\bigl\{\psi(\tilde{h}_i^1,\gamma_i)-\E\bigl(\psi(\tilde{h}_i^1,\gamma_i)\!\bigm|\!\mathscr{S}_0\bigr)\bigr\}\cvc 0
\end{equation}
as $n\to\infty$. To this end, let $L>0$ be such that $\psi\in\PL_2(r,L)$, and for each $y\in\R$, define $\psi_y,\bar{\psi}_y\colon\R\to\R$ by $\psi_y(z):=\psi(\tau_1 z,y)$ and $\bar{\psi}_y(z):=\psi_y(z)-\E\bigl(\psi_y(Z)\bigr)$, where $Z\sim N(0,1)$. Then by Lemma~\ref{lem:pseudocomp}(a), there exists $K_0>0$, depending only on $\tau_1$ and $r$, such that $\psi_y\in\PL_1(r,K_0L_y)$ with $L_y:=L(1\vee\abs{y}^{r-1})$. For fixed $n\in\N$ and $y_1,\dotsc,y_n\in\R$, define $\breve{L}\equiv\breve{L}(y_1,\dotsc,y_n):=(L_{y_1},\dotsc,L_{y_n})$. Let $r':=r/(r-1)\in (1,2]$ be the H\"older conjugate of $r$,
so that $1/r+1/r'=1$, and note that since $\norm{{\cdot}}_{p'}\leq\norm{{\cdot}}_p$ for $1\leq p\leq p'\leq\infty$, we have
\begin{align}
\frac{\norm{\breve{L}}_\infty}{n^{1/r'}}\leq\frac{\norm{\breve{L}}_2}{n^{1/r'}}\leq\frac{\norm{\breve{L}}_{r'}}{n^{1/r'}}=\norm{\breve{L}}_{n,r'}=\biggl(\frac{1}{n}\sum_{i=1}^n\,\abs{L_{y_i}}^{r'}\biggr)^{1/r'}&\leq L\biggl(1+\frac{1}{n}\sum_{i=1}^n\,\abs{y_i}^r\biggr)^{1/r'}\notag\\
\label{eq:H1cLbar}
&=L(1+\norm{y}_{n,r}^r)^{1/r'}.
\end{align}
By Lemma~\ref{lem:pseudoconc}, there exists a universal constant $C>0$ such that if $Z_1,\dotsc,Z_n\iid N(0,1)$, then
\begin{align}
P(n,t,y_1,\dotsc,y_n)&:=\Pr\biggl(\biggl|\frac{1}{n}\sum_{i=1}^n\,\bar{\psi}_{y_i}(Z_i)\biggr|\geq t\biggr)\notag\\
&\leq\exp\biggl(1-\min\biggl\{\biggl(\frac{nt}{(Cr)^r K_0\norm{\breve{L}}_2}\biggr)^2,\biggl(\frac{nt}{(Cr)^rK_0\norm{\breve{L}}_\infty}\biggr)^{2/r}\biggr\}\biggr)\notag\\
&\leq\exp\biggl(1-\min\biggl\{\biggl(\frac{n^{1/r}t}{(Cr)^r K_0\norm{\breve{L}}_{n,r'}}\biggr)^2,\biggl(\frac{n^{1/r}t}{(Cr)^r K_0\norm{\breve{L}}_{n,r'}}\biggr)^{2/r}\biggr\}\biggr)\notag\\
\label{eq:H1cEbar}
&=:E_r(n,t,K_0\norm{\breve{L}}_{n,r'})\equiv E_r\bigl(n,t,K_0\norm{\breve{L}(y_1,\dotsc,y_n)}_{n,r'}\bigr)
\end{align}
for every $t\geq 0$. Returning to~\eqref{eq:H1c2}, we see that 
\[\psi(\tilde{h}_i^1,\gamma_i)-\E\bigl(\psi(\tilde{h}_i^1,\gamma_i)\!\bigm|\!\mathscr{S}_0\bigr)=\psi_{\gamma_i}(\tilde{Z}_i^1)-\E\bigl(\psi_{\gamma_i}(\tilde{Z}_i^1)\!\bigm|\!\mathscr{S}_0\bigr)=\bar{\psi}_{\gamma_i}(\tilde{Z}_i^1)\]
for all $1\leq i\leq n$, where the final equality follows from Lemma~\ref{lem:indeplem} and the fact that $\tilde{Z}^1\equiv\tilde{Z}^1(n)$ is independent of $\mathscr{S}_0\equiv\mathscr{S}_0(n)=\sigma(\gamma,m^0)$. We deduce from this and~\eqref{eq:H1cEbar} that
\begin{equation}
\label{eq:H1cPneps}
\Pr(\abs{T_{n1}'}>\varepsilon\,|\,\mathscr{S}_0)=\Pr\biggl(\biggl|\frac{1}{n}\sum_{i=1}^n\bar{\psi}_{\gamma_i}(\tilde{Z}_i^1)\biggr|\geq\varepsilon\Bigm|\mathscr{S}_0\biggr)=P(n,\varepsilon,\gamma_1,\dotsc,\gamma_n)\leq E_r\bigl(n,\varepsilon,K_0\breve{L}_0(n)\bigr)
\end{equation}
for every $n$ and $\varepsilon>0$, where the second equality is again obtained using Lemma~\ref{lem:indeplem}, and
$\breve{L}_0(n):=\norm{\breve{L}(\gamma_1,\dotsc,\gamma_n)}_{n,r'}\leq L(1+\norm{\gamma}_{n,r}^r)^{1/r'}=O_c(1)$ by~\eqref{eq:H1cLbar} and~\ref{ass:A1}. Thus, by Proposition~\ref{prop:compequiv}, there exists $\bar{L}_0\in (0,\infty)$ such that for $n\in\N$, the events $A_0(n):=\{\breve{L}_0(n)\leq\bar{L}_0\}\in\mathscr{S}_0(n)$ satisfy $\sum_{n=1}^\infty\Pr\bigl(A_0(n)^c\bigr)<\infty$. Moreover, for each $n$ and $\varepsilon>0$, it follows from~\eqref{eq:H1cPneps} that
\begin{align*}
\Pr\bigl(\{\abs{T_{n1}'}>\varepsilon\}\cap A_0(n)\!\bigm|\!\mathscr{S}_0\bigr)=\Pr(\abs{T_{n1}'}>\varepsilon\,|\,\mathscr{S}_0)\Ind_{A_0(n)}&\leq P(n,\varepsilon,\gamma_1,\dotsc,\gamma_n)\Ind_{A_0(n)}\\
&\leq E_r\bigl(n,\varepsilon,K_0\breve{L}_0(n)\bigr)\Ind_{A_0(n)}\leq E_r(n,\varepsilon,K_0\bar{L}_0),
\end{align*}
where we have used the fact that $A_0(n)\in\mathscr{S}_0(n)$ to obtain the first equality above. Recalling the expression for $E_r(n,\varepsilon,K_0\bar{L}_0)$ in~\eqref{eq:H1cEbar}, we see that $\sum_{n=1}^\infty E_r(n,\varepsilon,K_0\bar{L}_0)<\infty$, and hence conclude that for every $\varepsilon>0$, we have
\begin{align}
\sum_{n=1}^\infty\,\Pr(\abs{T_{n1}'}>\varepsilon)&\leq\sum_{n=1}^\infty\,\Pr\bigl(\{\abs{T_{n1}'}>\varepsilon\}\cap A_0(n)\bigr)+\sum_{n=1}^\infty\,\Pr\bigl(A_0(n)^c\bigr)\notag\\
&=\sum_{n=1}^\infty\,\E\bigl\{\Pr\bigl(\{\abs{T_{n1}'}>\varepsilon\}\cap A_0(n)\!\bigm|\!\mathscr{S}_0\bigr)\bigr\}+\sum_{n=1}^\infty\,\Pr\bigl(A_0(n)^c\bigr)\notag\\
\label{eq:H1cTn1}
&\leq\sum_{n=1}^\infty\,E_r(n,\varepsilon,K_0\bar{L}_0)+\sum_{n=1}^\infty\,\Pr\bigl(A_0(n)^c\bigr)<\infty,
\end{align}
which together with Proposition~\ref{prop:compequiv} implies~\eqref{eq:H1c2}. Together with~\eqref{eq:H1c1}, this shows that $T_{n1}\cvc\E\bigl(\psi(\bar{G}_1,\bar{\gamma})\bigr)$, as claimed.

Next, we bound $\abs{T_{n2}}$ for each $n$. Letting $L>0$ be such that $\psi\in\PL_2(r,L)$, we can apply Lemma~\ref{lem:pseudoavg} to see that
\begin{align}
\abs{T_{n2}}&\leq \frac{1}{n}\sum_{i=1}^n\,\bigl|\psi(h_i^{1,0},\gamma_i)-\psi(h_i^{1,0}-\Delta_i^1,\gamma_i)\bigr|\notag\\
&\leq 2^{\frac{r}{2}-1}L\norm{\Delta^1}_{n,r}\bigl(1+\norm{h^{1,0}}_{n,r}^{r-1}+\norm{h^{1,0}-\Delta^1}_{n,r}^{r-1}+2\norm{\gamma}_{n,r}^{r-1}\bigr)\notag\\
\label{eq:H1cTn2}
&\lesssim_r L\norm{\Delta^1}_{n,r}\bigl(1+\norm{h^{1,0}}_{n,r}^{r-1}+\norm{\Delta^1}_{n,r}^{r-1}+\norm{\gamma}_{n,r}^{r-1}\bigr),
\end{align}
where the final bound is obtained using the triangle inequality for $\norm{{\cdot}}_{n,r}$ and the fact that $(a+b)^{r-1}\leq 2^{r-2}(a^{r-1}+b^{r-1})$ for $a,b\geq 0$. Now $\norm{h^{1,0}}_{n,r}\eqd\norm{h^1}_{n,r}=O_c(1)$ by $\mathcal{H}_1(b)$ and $\norm{\Delta^1}_{n,r}=o_c(1)$ by $\mathcal{H}_1(a)$, so $\abs{T_{n2}}=o_c(1)\bigl(1+O_c(1)+o_c(1)\bigr)=o_c(1)$. We conclude that $\inv{n}\sum_{i=1}^n\psi(h_i^1,\gamma_i)\eqd T_n=T_{n1}+T_{n2}\cvc\E\bigl(\psi(\bar{G}_1,\bar{\gamma})\bigr)$, as desired.

$\mathcal{H}_1(d)$: For each $n$, we have $(\gamma,m^0,h^1)\eqd(\gamma,m^0,h^{1,0})$ by~\eqref{eq:h1} and Lemma~\ref{lem:conddist}(c), so for each fixed $\phi\in\PL_2(1)$, it follows that
\begin{equation}
\label{eq:H1d}
\frac{1}{n}\sum_{i=1}^n m_i^0\,\phi(h_i^1,\gamma_i)\eqd\frac{1}{n}\sum_{i=1}^n m_i^0\,\phi(\tilde{h}_i^1,\gamma_i)+\frac{1}{n}\sum_{i=1}^n\,m_i^0\bigl\{\phi(h_i^{1,0},\gamma_i)-\phi(\tilde{h}_i^1,\gamma_i)\bigr\}=:T_{n1}+T_{n2}.
\end{equation}
By similar (and slightly simpler) arguments to those in $\mathcal{H}_1(c)$, we will prove that $T_{n1}\cvc\E\bigl(F_0(\bar{\gamma})\cdot\phi(\bar{G}_1,\bar{\gamma})\bigr)$ and $T_{n2}\cvc 0$ as $n\to\infty$.

For $T_{n1}$, define $\Phi\colon\R\to\R$ by $\Phi(y):=\E\bigl(\phi(\tau_1 Z,y)\bigr)$ with $Z\sim N(0,1)$. For each $n$, recalling once again that $\tilde{Z}^1\equiv\tilde{Z}^1(n)\sim N_n(0,I_n)$ is independent of $\mathscr{S}_0\equiv\mathscr{S}_0(n)=\sigma(\gamma,m^0)$, we deduce from Lemma~\ref{lem:indeplem} that $\E\bigl(\phi(\tilde{h}_i^1,\gamma_i)\!\bigm|\!\mathscr{S}_0\bigr)=\Phi(\gamma_i)$ almost surely, for every $1\leq i\leq n$. Now since $\Phi$ is Lipschitz by Lemma~\ref{lem:pseudocomp}(b), it follows from~\ref{ass:A3} that if $\bar{G}_1\sim N(0,\tau_1^2)$ is independent of $\bar{\gamma}\sim\pi$, then
\begin{align}
\frac{1}{n}\sum_{i=1}^n m_i^0\,\E\bigl(\phi(\tilde{h}_i^1,\gamma_i)\!\bigm|\!\mathscr{S}_0\bigr)&=
\frac{1}{n}\sum_{i=1}^n m_i^0\,\Phi(\gamma_i)\notag\\
&\cvc\E\bigl(F_0(\bar{\gamma})\Phi(\bar{\gamma})\bigr)=\E\bigl(F_0(\bar{\gamma})\cdot\E\bigl\{\phi(\bar{G}_1,\bar{\gamma})\!\bigm|\!\bar{\gamma}\bigr\}\bigl)=\E\bigl(F_0(\bar{\gamma})\cdot\phi(\bar{G}_1,\bar{\gamma})\bigr).
\end{align}
To complete the proof that $T_{n1}\cvc\E\bigl(F_0(\bar{\gamma})\cdot\phi(\bar{G}_1,\bar{\gamma})\bigr)$, we must therefore show that
\begin{equation}
\label{eq:H1d1}
T_{n1}':=n^{-1}\sum_{i=1}^n m_i^0\bigl\{\phi(\tilde{h}_i^1,\gamma_i)-\E\bigl(\phi(\tilde{h}_i^1,\gamma_i)\!\bigm|\!\mathscr{S}_0\bigr)\bigr\}\cvc 0.
\end{equation}
To this end, let $L>0$ be such that $\phi\in\PL_2(1,L)$ on $\R$. For $u,y\in\R$, define $\phi_{u,y}\colon\R\to\R$ by $\phi_{u,y}(z):=u\bigl\{\phi(\tau_1 z,y)-\E\bigl(\phi(\tau_1 Z,y)\bigr)\bigr\}$, where $Z\sim N(0,1)$, so that $\phi_{u,y}\in\PL_2(1,L\tau_1\abs{u})$. 
Since $\tilde{Z}^1\sim N_n(0,I_n)$, it follows from~\eqref{eq:pseudoconcr} in Remark~\ref{rem:pseudoconc} that for every $v\equiv (v_1,\dotsc,v_n)\in\R^n$ and $t\geq 0$, we have
\begin{align}
\tilde{P}(n,t,v):=\Pr\biggl(\biggl|\frac{1}{n}\sum_{i=1}^n\phi_{v_i,y_i}(\tilde{Z}_i^1)\biggr|\geq t\biggr)\leq\exp\biggl\{1-\biggl(\frac{nt}{CL\tau_1\norm{v}_2}\biggr)^2\biggr\}&\leq\exp\biggl\{1-\biggl(\frac{n^{1/2}t}{CL\tau_1\norm{v}_n}\biggr)^2\biggr\}\notag\\
\label{eq:H1d2}
&=:\tilde{E}_r(n,t,L\tau_1\norm{v}_n),
\end{align}
where $C>0$ is a suitable universal constant. Recalling once again that $\tilde{Z}^1$ is independent of $\mathscr{S}_0=\sigma(\gamma,m^0)$, we deduce using Lemma~\ref{lem:indeplem} that 
$m_i^0\bigl\{\phi(\tilde{h}_i^1,\gamma_i)-\E\bigl(\phi(\tilde{h}_i^1,\gamma_i)\!\bigm|\!\mathscr{S}_0\bigr)\bigr\}=\phi_{m_i^0,\gamma_i}(\tilde{Z}_i^1)$ for all $1\leq i\leq n$. Thus, for each $n$ and $\varepsilon>0$, we have
\begin{equation}
\label{eq:H1dTn1'}
\Pr(\abs{T_{n1}'}>\varepsilon\,|\,\mathscr{S}_0)=\Pr\biggl(\biggl|\frac{1}{n}\sum_{i=1}^n\phi_{m_i^0,\gamma_i}(\tilde{Z}_i^1)\biggr|\geq\varepsilon\Bigm|\mathscr{S}_0\biggr)=\tilde{P}(n,\varepsilon,m^0)\leq \tilde{E}_r(n,\varepsilon,L\tau_1\norm{m^0}_n).
\end{equation}
Since $\norm{m^0}_n\cvc\tau_1$ by~\ref{ass:A2}, Proposition~\ref{prop:compequiv} ensures that the events $\tilde{A}_0(n):=\{\norm{m^0}_n\leq\tau_1+1\}\in\mathscr{S}_0(n)$ satisfy $\sum_{n=1}^\infty\Pr\bigl(\tilde{A}_0(n)^c\bigr)<\infty$. Moreover, for each $n$ and $\varepsilon>0$, it follows from~\eqref{eq:H1d2} that
\begin{align}
\label{eq:H1d3}
\Pr\bigl(\{\abs{T_{n1}'}>\varepsilon\}\cap \tilde{A}_0(n)\!\bigm|\!\mathscr{S}_0\bigr)=\Pr(\abs{T_{n1}'}>\varepsilon\,|\,\mathscr{S}_0)\Ind_{\tilde{A}_0(n)}&\leq\tilde{P}(n,\varepsilon,m^0)\Ind_{\tilde{A}_0(n)}\\
&\leq \tilde{E}_r(n,\varepsilon,L\tau_1C_0)\Ind_{\tilde{A}_0(n)}\leq \tilde{E}_r(n,\varepsilon,L\tau_1C_0),\notag
\end{align}
where we have used the fact that $\tilde{A}_0(n)\in\mathscr{S}_0(n)$ to obtain the first equality above. Recalling the expression for $\tilde{E}_r(n,\varepsilon,C_0)$ in~\eqref{eq:H1d2}, we see that $\sum_{n=1}^\infty\tilde{E}_r(n,\varepsilon,L\tau_1C_0)<\infty$. Thus, for every $\varepsilon>0$, we conclude as in~\eqref{eq:H1cTn1} that
\begin{align}
\sum_{n=1}^\infty\,\Pr(\abs{T_{n1}'}>\varepsilon)&\leq\sum_{n=1}^\infty\,\Pr\bigl(\{\abs{T_{n1}'}>\varepsilon\}\cap \tilde{A}_0(n)\bigr)+\sum_{n=1}^\infty\,\Pr\bigl(\tilde{A}_0(n)^c\bigr)\notag\\
\label{eq:H1dTn1}
&\leq\sum_{n=1}^\infty\,\tilde{E}_r(n,\varepsilon,L\tau_1C_0)+\sum_{n=1}^\infty\,\Pr\bigl(\tilde{A}_0(n)^c\bigr)<\infty,
\end{align}
which implies~\eqref{eq:H1d1} in view of Proposition~\ref{prop:compequiv}, and hence that $T_{n1}\cvc\tilde{\tau}\,\E\bigl(\phi(\bar{G}_1)\bigr)$ in~\eqref{eq:H1d}.

As for $T_{n2}$ in~\eqref{eq:H1d}, let $L>0$ be as above, so that $\phi\in\PL_2(1,L)$.  For each $n$, recalling from~\eqref{eq:h1} that $h^{1,0}=\tilde{h}^1+\Delta^1$, we now apply the Cauchy--Schwarz inequality to see that
\begin{align*}
\abs{T_{n2}}\leq\frac{1}{n}\sum_{i=1}^n\,\abs{m_i^0}\,\bigl|\phi(h_i^{1,0},\gamma_i)-\phi(h_i^{1,0}-\Delta_i^1,\gamma_i)\bigr|\leq\frac{L}{n}\sum_{i=1}^n\,\abs{m_i^0}\,\abs{\Delta_i^1}\leq L\norm{m^0}_n\norm{\Delta^1}_n.
\end{align*}
Since $\norm{m^0}_n\cvc\tau_1$ by~\ref{ass:A2} and $\norm{\Delta^1}_n\leq\norm{\Delta^1}_{n,r}=o_c(1)$ by $\mathcal{H}_1(a)$, we conclude that $T_{n2}=O_c(1)\,o_c(1)=o_c(1)$. This completes the proof of $\mathcal{H}_1(d)$.

Turning to the inductive step, we consider a general $k\in\N$ and suppose that $\mathcal{H}_k(b,c,d)$ have already been established. The assertions $\mathcal{H}_k(e,\dotsc,j)$ and $\mathcal{H}_{k+1}(a,b,c,d)$ will now be proved, in that order. Note that $\PL_{k+1}(2)\subseteq\PL_{k+1}(r)$ since $r\geq 2$.

$\mathcal{H}_k(e)$: In the case $j=\ell=1$, we have $\bigl|\norm{m^0}_n-\tau_1\bigr|\cvc 0$ by~\ref{ass:A2}, so $\ipr{m^0}{m^0}_n\cvc\tau_1^2=\bar{\mathrm{T}}_{1,1}=\E(\bar{G}_1^2)$ by~\eqref{eq:covid1}. Now fix $j,\ell\in\{2,\dotsc,k+1\}$. Then $\tilde{\psi}_{j\ell}\colon(x_1,\dotsc,x_k,y)\mapsto f_{j-1}(x_{j-1},y)f_{\ell-1}(x_{\ell-1},y)$ lies in $\PL_{k+1}(2)\subseteq\PL_{k+1}(r)$ by Lemma~\ref{lem:pseudoprod} and the fact that $f_{j-1},f_{\ell-1}$ in the AMP recursion~\eqref{eq:AMPsym} are Lipschitz by assumption. Thus, by taking $\psi=\tilde{\psi}_{j\ell}$ in $\mathcal{H}_k(c)$, we see that
\begin{align*}
\ipr{m^{j-1}}{m^{\ell-1}}_n&=\frac{1}{n}\sum_{i=1}^n f_{j-1}(h_i^{j-1},\gamma_i)\,f_{\ell-1}(h_i^{\ell-1},\gamma_i)\\
&\cvc\E\bigl(f_{j-1}(\bar{G}_{j-1},\bar{\gamma})\cdot f_{\ell-1}(\bar{G}_{\ell-1},\bar{\gamma})\bigr)=\bar{\mathrm{T}}_{j,\ell}=\E(\bar{G}_j\bar{G}_\ell),
\end{align*}
where the final equalities are taken from~\eqref{eq:covid1}. To handle the remaining case where $\{j,\ell\}=\{1,k+1\}$, note that since $f_k$ is Lipschitz, the map $\phi_{k+1}\colon (x_1,\dotsc,x_k,y)\mapsto f_k(x_k,y)$ lies in $\PL_{k+1}(1)$. Thus, by taking $\phi=\phi_{k+1}$ in $\mathcal{H}_k(d)$, we deduce that
\[\ipr{m^0}{m^k}_n=\frac{1}{n}\sum_{i=1}^n m_i^0\,f_k(h_i^k,\gamma_i)\cvc\E\bigl(F_0(\bar{\gamma})\cdot f_k(\bar{G}_k,\bar{\gamma})\bigr)=\bar{\mathrm{T}}_{1,k+1}=\E(\bar{G}_1\bar{G}_{k+1}),\]
where the final equalities are again taken from~\eqref{eq:covid1}.

$\mathcal{H}_k(f)$: This proof is very similar to that of $\mathcal{H}_k(e)$. First fix $1\leq j,\ell\leq k$. By Lemma~\ref{lem:pseudoprod}, the function $(x_1,\dotsc,x_k,y)\mapsto x_j f_\ell(x_\ell,y)$ lies in $\PL_{k+1}(2)\subseteq\PL_{k+1}(r)$, so by applying $\mathcal{H}_k(c)$ again, we deduce that
\[\ipr{h^j}{m^\ell}_n=\frac{1}{n}\sum_{i=1}^n h_i^j\,f_\ell(h_i^\ell,\gamma_i)\cvc\E\bigl(\bar{G}_j f_\ell(\bar{G}_\ell,\bar{\gamma})\bigr)=\E\bigl(f_\ell'(\bar{G}_\ell,\bar{\gamma})\bigr)\E(\bar{G}_j\bar{G}_\ell)=\bar{b}_\ell\,\bar{\mathrm{T}}_{j,\ell},\]
where the final equalities are taken from~\eqref{eq:covid2}. For the second part of $\mathcal{H}_k(f)$, we fix $1\leq j\leq k$ and apply $\mathcal{H}_k(d)$ with the $\PL_{k+1}(1)$ function $(x_1,\dotsc,x_k,y)\mapsto x_j$ to see that
\[\ipr{h^j}{m^0}_n=\frac{1}{n}\sum_{i=1}^n m_i^0\,h_i^j\cvc\E\bigl(F_0(\bar{\gamma})\bar{G}_j\bigr)=0\]
by the independence of $\bar{G}_j\sim N(0,\tau_j^2)$ and $\bar{\gamma}\sim\pi$.

$\mathcal{H}_k(g)$: In view of Definition~\ref{def:compconv} of complete convergence, it suffices to show that if $(\beta_n)$ is any sequence of random variables with $\beta_n\eqd\abr{f_k'(h^k,\gamma)}_n$ for each $n$, then $\beta_n\cvas\E\bigl(f_k'(\bar{G}_k,\bar{\gamma})\bigr)$. For any such sequence $(\beta_n)$, we first seek to construct a random sequence $\bigl((\eta_n,\theta_n)\in\R^n\times\R^n:n\in\N\bigr)$ such that $(\eta_n,\theta_n)\eqd\bigl(h^k(n),\gamma(n)\bigr)$ for each $n$, and $\bigl(\abr{f_k'(\eta_n,\theta_n)}_n:n\in\N\bigr)=(\beta_n:n\in\N)$ almost surely as random sequences. This can be done by applying Lemma~\ref{lem:seqcoup}, where we take $g_n\colon\R^n\times\R^n\to\R$ to be the measurable function $(x,y)\mapsto\abr{f_k'(x,y)}_n=\inv{n}\sum_{i=1}^n f_k'(x_i,y_i)$ for each $n$.

Since $(\eta_n,\theta_n)\eqd\bigl(h^k(n),\gamma(n)\bigr)$ for each $n$ by construction, it follows from the inductive hypothesis $\mathcal{H}_k(c)$ that $\inv{n}\sum_{i=1}^n\varphi(\eta_{ni},\theta_{ni})\eqd\inv{n}\sum_{i=1}^n\varphi(h_i^k,\gamma_i)\cvc\E\bigl(\varphi(\bar{G}_k,\bar{\gamma})\bigr)$ for every $\varphi\in\PL_2(r)$,
where $(\bar{G}_k,\bar{\gamma})\sim N(0,\tau_k^2)\otimes\pi=:\bar{\mu}^k$. Consequently, denoting by $\tilde{\mu}_n^k:=\nu_n(\eta_n,\theta_n)=n^{-1}\sum_{i=1}^n\delta_{(\eta_{ni},\theta_{ni})}$ the joint empirical distribution of the components of $\eta_n$ and $\theta_n$ for each $n$, we deduce using Corollary~\ref{cor:Wr}(a) that $d_r(\tilde{\mu}_n^k,\bar{\mu}^k)\cvas 0$, and hence that $(\tilde{\mu}_n^k)$ converges weakly to $\bar{\mu}^k$ with probability 1. By~\ref{ass:A5}, $f_k'$ is bounded, Borel measurable and continuous $\bar{\mu}^k$-almost everywhere, so we may now apply Lemma~\ref{lem:weakconv} to conclude that $\beta_n=\abr{f_k'(\eta_n,\theta_n)}_n=\int_{\R^2}f_k'\,d\tilde{\mu}_n^k\to\int_{\R^2}f_k'\,d\bar{\mu}^k=\E\bigl(f_k'(\bar{G}_k,\bar{\gamma})\bigr)$ almost surely. This completes the proof of $\mathcal{H}_k(g)$.

$\mathcal{H}_k(h)$: For $1\leq\ell\leq k$, it follows from $\mathcal{H}_k(e,f,g)$ that \[v_j^{k,\ell}/n=\ipr{h^j}{m^\ell}_n-b_\ell\,\ipr{m^{j-1}}{m^{\ell-1}}_n\cvc\bar{b}_\ell\,\bar{\mathrm{T}}_{j,\ell}-\bar{b}_\ell\,\bar{\mathrm{T}}_{j,\ell}=0\]
for all $1\leq j\leq k$. For $\ell=0$, we have $m^{-1}=0$ by definition, so $v_j^{k,0}/n=\ipr{h^j}{m^0}_n\cvc 0$ for all $1\leq j\leq k$ by the second part of $\mathcal{H}_k(f)$.

$\mathcal{H}_k(i)$: Recall from~\eqref{eq:alphak} that $\alpha^k=(M_k^\top M_k/n)^+(M_k^\top m^k/n)\in\R^k$. It follows from $\mathcal{H}_k(e)$ that $(M_k^\top M_k/n)_{j\ell}=\ipr{m^{j-1}}{m^{\ell-1}}_n\cvc\bar{\mathrm{T}}_{j,\ell}$ and $(M_k^\top m^k/n)_j=\ipr{m^{j-1}}{m^k}_n\cvc\bar{\mathrm{T}}_{j,k+1}$ for all $1\leq j,\ell\leq k$. In the notation of Section~\ref{sec:inductive}, this means that $M_k^\top m^k/n\cvc\bar{\mathrm{T}}^{[k],k+1}\in\R^k$ and $M_k^\top M_k/n\cvc\bar{\mathrm{T}}^{[k]}\in\R^{k\times k}$. Under~\ref{ass:A4}, Lemma~\ref{lem:posdef} ensures that $\bar{\mathrm{T}}^{[k]}$ is positive definite and hence invertible, we now apply the continuous mapping theorem for complete convergence (Lemma~\ref{lem:slutsky}) to deduce that \[\alpha^k=(M_k^\top M_k/n)^+(M_k^\top m^k/n)\cvc\bigl(\bar{\mathrm{T}}^{[k]}\bigr)^{-1}\,\bar{\mathrm{T}}^{[k],k+1}=\bar{\alpha}^k,\]
as defined in~\eqref{eq:alphabar}.

$\mathcal{H}_k(j)$: Recalling~\eqref{eq:alphak} as well as the definitions at the start of Section~\ref{sec:AMPcond}, we can write
\[\norm{\barpp{m}{k}}_n^2=\norm{P_k^\perp m^k}_n^2=\norm{m^k}_n^2-\norm{P_km^k}_n^2=\norm{m^k}_n^2-(\alpha^k)^\top(M_k^\top M_k/n)\,\alpha^k.\]
Now $\norm{m^k}_n^2\cvc\bar{\mathrm{T}}_{k+1,k+1}=\tau_{k+1}^2$ and $M_k^\top M_k/n\cvc\bar{\mathrm{T}}^{[k]}\in\R^{k\times k}$ by $\mathcal{H}_k(e)$, and $\alpha^k\cvc\bar{\alpha}^k\in\R^k$ by $\mathcal{H}_k(i)$, so \[\norm{\barpp{m}{k}}_n^2=\norm{m^k}_n^2-(\alpha^k)^\top(M_k^\top M_k/n)\,\alpha^k\cvc\bar{\mathrm{T}}_{k+1,k+1}-(\bar{\alpha}^k)^\top\bar{\mathrm{T}}^{[k]}\,\bar{\alpha}^k=\barpp{\tau}{2}_{k+1},\] 
as defined in~\eqref{eq:tauperpbar}. 

$\mathcal{H}_{k+1}(a)$: Denote by $R_{n1},\dotsc,R_{n5}$ the individual summands (in the order in which they appear) in the definition~\eqref{eq:deltak} of $\Delta^{k+1}\equiv\Delta^{k+1}(n)\in\R^n$. To establish that $\norm{\Delta^{k+1}}_{n,r}\cvc 0$, it suffices to show that $\norm{R_{ns}}_{n,r}\cvc 0$ for $s=1,\dotsc,5$. Observe first that since $\alpha^k\cvc\bar{\alpha}^k\in\R^k$ by $\mathcal{H}_k(i)$ and $\norm{h^\ell}_{n,r}=O_c(1)$ for all $1\leq\ell\leq k$ by $\mathcal{H}_k(b)$, we have
$\norm{R_{n1}}_{n,r}\leq\sum_{\ell=1}^k\,\abs{\alpha_\ell^k-\bar{\alpha}_\ell^k}\,\norm{h^\ell}_{n,r}=\sum_{\ell=1}^k o_c(1)\,O_c(1)=o_c(1)$. 

As for $R_{n2}$, we know from $\mathcal{H}_k(e)$ that $(M_k^\top M_k/n)_{j\ell}=\ipr{m^{j-1}}{m^{\ell-1}}_n\cvc\bar{\mathrm{T}}_{j,\ell}$ for all $1\leq j,\ell\leq k$, so $M_k^\top M_k/n\cvc\bar{\mathrm{T}}^{[k]}\in\R^{k\times k}$, which is positive definite by Lemma~\ref{lem:posdef}. We can now apply the continuous mapping theorem for complete convergence (Lemma~\ref{lem:slutsky}) to deduce that $(M_k^\top M_k/n)^+\cvc (\bar{\mathrm{T}}^{[k]})^{-1}$; see $\mathcal{H}_k(i)$ above for a similar argument. By $\mathcal{H}_k(h)$, we have $v^{k,\ell}/n\cvc 0\in\R^k$ for all $0\leq\ell\leq k$, so
\[\tilde{w}^k\equiv\tilde{w}^k(n):=(M_k^\top M_k)^+\biggl(v^{k,k}-\sum_{\ell=1}^k\alpha_\ell^k\,v^{k,\ell-1}\biggr)\cvc 0,\]
formally by Slutsky's lemma for complete convergence (Lemma~\ref{lem:slutsky}). 
Since $\norm{m^{\ell-1}}_{n,r}=O_c(1)$ for $1\leq\ell\leq k$ by $\mathcal{H}_k(b)$, we have $\norm{R_{n2}}_{n,r}=\norm{M_k\tilde{w}^k}_{n,r}\leq\sum_{\ell=1}^k\,\abs{\tilde{w}_\ell^k}\,\norm{m^{\ell-1}}_{n,r}=\sum_{\ell=1}^k o_c(1)\,O_c(1)=o_c(1)$.

Turning to $R_{n3}$ and introducing $\xi_1,\dotsc,\xi_k\iid N(0,1)$, we see from Lemma~\ref{lem:projgaussian} that $\norm{P_k\tilde{Z}^{k+1}}_{n,r}$ is stochastically dominated by $\sum_{i=1}^k\,\abs{\xi_i}/n^{1/r}$ for each $n>k$. By Example~\ref{ex:exptails}(a), $\sum_{i=1}^k\,\abs{\xi_i}/n^{1/r}\leq k\max_{1\leq i\leq k}\,\abs{\xi_i}/n^{1/r}=o_c(1)$, so $\norm{P_k\tilde{Z}^{k+1}}_{n,r}\cvc 0$. Since $\norm{\barpp{m}{k}}_n\cvc\barp{\tau}_{k+1}\in (0,\infty)$ by $\mathcal{H}_k(j)$, we deduce that $\norm{R_{n3}}_{n,r}=\norm{\barpp{m}{k}}_n\norm{P_k\tilde{Z}^{k+1}}_{n,r}=o_c(1)$.

For the remaining summands $R_{n4}$ and $R_{n5}$, the arguments are similar to those in the proof of $\mathcal{H}_1(a)$. Recall that $(\tilde{Z}^{k+1},\tilde{\zeta}^{k+1})\equiv\bigl(\tilde{Z}^{k+1}(n),\tilde{\zeta}^{k+1}(n)\bigr)\sim N_n(0,I_n)\otimes N(0,1/n)$. Introducing $\zeta\sim N(0,1)$, we have $\abs{\tilde{\zeta}^{k+1}}\eqd n^{-1/2}\,\abs{\zeta}\cvc 0$ by Example~\ref{ex:exptails}(a), and $\norm{\tilde{Z}^{k+1}}_{n,r}=(n^{-1}\sum_{i=1}^n\,\abs{\tilde{Z}_i^{k+1}}^r)^{1/r}\cvc\E(\abs{\zeta}^r)^{1/r}\in (0,\infty)$ by Lemma~\ref{lem:pseudoconc} and Proposition~\ref{prop:compequiv}. By $\mathcal{H}_k(j)$, we have $\norm{\barpp{m}{k}}_n-\barp{\tau}_{k+1}=o_c(1)$. Moreover, $\alpha^k=\bar{\alpha}^k+o_c(1)=O_c(1)$ by $\mathcal{H}_k(i)$ and $\norm{m^\ell}_{n,r}=O_c(1)$ for $0\leq\ell\leq k$ by $\mathcal{H}_k(b)$, so it follows from~\eqref{eq:alphak} that
\[\norm{\barpp{m}{k}}_{n,r}=\norm{(I-P_k)\,m^k}_{n,r}\leq\norm{m^k}_{n,r}+\sum_{\ell=1}^k\,\abs{\alpha_\ell^k}\,\norm{m^{\ell-1}}_{n,r}=O_c(1)+\sum_{\ell=1}^k O_c(1)\,O_c(1)=O_c(1).\]
Putting everything together, we see that
\[\norm{R_{n4}}_{n,r}+\norm{R_{n5}}_{n,r}=\bigl|\norm{\barpp{m}{k}}_n-\barp{\tau}_{k+1}\bigr|\,\norm{\tilde{Z}^{k+1}}_{n,r}+\abs{\tilde{\zeta}^{k+1}}\,\norm{\barpp{m}{k}}_{n,r}=o_c(1)\,O_c(1)+o_c(1)\,O_c(1)=o_c(1).\]
We have now shown that $\norm{R_{ns}}_{n,r}\cvc 0$ for $s=1,\dotsc,5$, so $\norm{\Delta^{k+1}}_{n,r}\leq\sum_{s=1}^5\norm{R_{ns}}_{n,r}\cvc 0$.

$\mathcal{H}_{k+1}(b)$: By the inductive hypothesis $\mathcal{H}_k(b)$, we have $\norm{h^j}_{n,r}=O_c(1)$ for all $1\leq j\leq k$ and $\norm{m^j}_{n,r}=O_c(1)$ for all $0\leq j\leq k$. Now let $j=k+1$. For each integer $n>k$, recall from~\eqref{eq:hcond} in Proposition~\ref{prop:AMPconddist} and~\eqref{eq:htilde} that
\begin{equation}
\label{eq:hk+1}
h^{k+1}(n)\eqdcond{\mathscr{S}_k} h^{k+1,k}(n)=\tilde{h}^{k+1}(n)+\Delta^{k+1}(n)=\sum_{\ell=1}^k\bar{\alpha}_\ell^k\,h^\ell(n)+\barp{\tau}_{k+1}\tilde{Z}^{k+1}(n)+\Delta^{k+1}(n),
\end{equation}
where the deterministic $\bar{\alpha}^k\in\R^k$ is taken from~\eqref{eq:alphabar} and $\tilde{Z}^{k+1}\equiv\tilde{Z}^{k+1}(n)\sim N_n(0,I_n)$ is independent of $\mathscr{S}_k\equiv\mathscr{S}_k(n)=\sigma(\gamma,m^0,h^1,\dotsc,h^k)$. Then $\norm{\Delta^{k+1}}_{n,r}=o_c(1)$ by $\mathcal{H}_{k+1}(a)$ and $\norm{\tilde{Z}^{k+1}}_{n,r}=O_c(1)$, as in the last part of the proof of $\mathcal{H}_{k+1}(a)$ above. It follows from this and $\mathcal{H}_k(b)$ that
\[\norm{h^{k+1}}_{n,r}\eqd\norm{h^{k+1,k}}_{n,r}\leq\sum_{\ell=1}^k\,\abs{\bar{\alpha}_\ell^k}\,\norm{h^\ell}_{n,r}+\norm{\Delta^{k+1}}_{n,r}=O_c(1)+o_c(1)=O_c(1).\]
In addition, letting $L'>0$ be such that $f_{k+1}$ is $L'$-Lipschitz, we can argue as in the proof of $\mathcal{H}_1(b)$ to deduce that \[\norm{m^{k+1}}_{n,r}=\norm{f_{k+1}(h^{k+1},\gamma)}_{n,r}\leq\abs{f_{k+1}(0,0)}+L'(\norm{h^{k+1}}_{n,r}+\norm{\gamma}_{n,r})=O_c(1).\]
$\mathcal{H}_{k+1}(c)$: We again make use of the distributional equality~\eqref{eq:hk+1}, which together with Lemma~\ref{lem:conddist}(c) implies that $(\gamma,h^1,\dotsc,h^k,h^{k+1})\eqdcond{\mathscr{S}_k}(\gamma,h^1,\dotsc,h^k,h^{k+1,k})$ for each integer $n>k$. Thus, for any fixed $\psi\in\PL_{k+2}(r)$, it follows that $\inv{n}\sum_{i=1}^n\psi(h_i^1,\dotsc,h_i^k,h_i^{k+1},\gamma_i)\eqd\inv{n}\sum_{i=1}^n\psi(h_i^1,\dotsc,h_i^k,h_i^{k+1,k},\gamma_i)=:T_n$ for each such $n$, so it suffices to show that $T_n\cvc\E\bigl(\psi(\bar{G}_1,\dotsc,\bar{G}_k,\bar{G}_{k+1},\bar{\gamma})\bigr)$. We decompose $T_n$ as
\begin{equation}
\label{eq:Hkc}
\frac{1}{n}\sum_{i=1}^n\psi(h_i^1,\dotsc,h_i^k,\tilde{h}_i^{k+1},\gamma_i)+\frac{1}{n}\sum_{i=1}^n\,\{\psi(h_i^1,\dotsc,h_i^k,h_i^{k+1,k},\gamma_i)-\psi(h_i^1,\dotsc,h_i^k,\tilde{h}_i^{k+1},\gamma_i)\}=:T_{n1}+T_{n2}
\end{equation}
for each $n>k$, and seek to establish that $T_{n1}\cvc\E\bigl(\psi(\bar{G}_1,\dotsc,\bar{G}_k,\bar{G}_{k+1},\bar{\gamma})\bigr)$ and $T_{n2}\cvc 0$
by imitating and extending the analogous arguments in the proof of $\mathcal{H}_1(c)$. 

For $T_{n1}$, define $\Psi\colon\R^{k+1}\to\R$ by $\Psi(x_1,\dotsc,x_k,y):=\E\bigl\{\psi\bigl(x_1,\dotsc,x_k,\sum_{\ell=1}^k\bar{\alpha}_\ell^k\,x_\ell+\barp{\tau}_{k+1}Z,y\bigr)\bigr\}$ with $Z\sim N(0,1)$.
For each integer $n>k$, since $\tilde{Z}^{k+1}\equiv\tilde{Z}^{k+1}(n)\sim N_n(0,I_n)$ is independent of $\mathscr{S}_k\equiv\mathscr{S}_k(n)=\sigma(\gamma,m^0,h^1,\dotsc,h^k)$, we deduce from~\eqref{eq:hk+1} and Lemma~\ref{lem:indeplem} that
\[\E\bigl(\psi(h_i^1,\dotsc,h_i^k,\tilde{h}_i^{k+1},\gamma_i)\!\bigm|\!\mathscr{S}_k\bigr)=\E\bigl\{\psi\bigl(h_i^1,\dotsc,h_i^k,\textstyle\sum_{\ell=1}^k\bar{\alpha}_\ell^k\, h_i^\ell+\barp{\tau}_{k+1}\tilde{Z}_i^{k+1},\gamma_i\bigr)\!\bigm|\!\mathscr{S}_k\bigr\}=\Psi(h_i^1,\dotsc,h_i^k,\gamma_i)\]
almost surely, for every $1\leq i\leq n$. Now since $\Psi\in\PL_{k+1}(r)$ by Lemma~\ref{lem:pseudocomp}(b), it follows from the inductive hypothesis $\mathcal{H}_k(c)$ that
\begin{equation}
\label{eq:Hkc1}
\frac{1}{n}\sum_{i=1}^n\E\bigl(\psi(h_i^1,\dotsc,h_i^k,\tilde{h}_i^{k+1},\gamma_i)\!\bigm|\!\mathscr{S}_k\bigr)=
\frac{1}{n}\sum_{i=1}^n\Psi(h_i^1,\dotsc,h_i^k,\gamma_i)\cvc\E\bigl(\Psi(\bar{G}_1,\dotsc,\bar{G}_k,\bar{\gamma})\bigr),
\end{equation}
where the equality above holds almost surely for each $n>k$. Taking $\barp{\zeta}_{k+1}\sim N(0,1)$ to be independent of $\bar{G}_{[k]}:=(\bar{G}_1,\dotsc,\bar{G}_k)$ and $\bar{\gamma}$, we apply Lemma~\ref{lem:indeplem} again to see that
\begin{align}
\label{eq:Hkc2}
\E\bigl(\Psi(\bar{G}_1,\dotsc,\bar{G}_k,\bar{\gamma})\bigr)&=\E\bigl(\E\bigl\{\psi\bigl(\bar{G}_1,\dotsc,\bar{G}_k,\textstyle\sum_{\ell=1}^k\bar{\alpha}_\ell^k\,\bar{G}_\ell+\barp{\tau}_{k+1}\barp{\zeta}_{k+1},\bar{\gamma}\bigr)\!\bigm|\!\bar{G}_{[k]},\bar{\gamma}\bigr\}\bigl)\\
&=\E\bigl\{\psi\bigl(\bar{G}_1,\dotsc,\bar{G}_k,\textstyle\sum_{\ell=1}^k\bar{\alpha}_\ell^k\,\bar{G}_\ell+\barp{\tau}_{k+1}\barp{\zeta}_{k+1},\bar{\gamma}\bigr)\bigr\}=\E\bigl(\psi(\bar{G}_1,\dotsc,\bar{G}_k,\bar{G}_{k+1},\bar{\gamma})\bigr),\notag
\end{align}
where the final equality follows from the definition of $\bar{G}_{k+1}$ in~\eqref{eq:Gkbar}. To complete the proof that $T_{n1}\cvc\E\bigl(\psi(\bar{G}_1,\dotsc,\bar{G}_k,\bar{G}_{k+1},\bar{\gamma})\bigr)$, we must therefore show that
\begin{equation}
\label{eq:Hkc3}
T_{n1}':=\frac{1}{n}\sum_{i=1}^n\,\bigl\{\psi(h_i^1,\dotsc,h_i^k,\tilde{h}_i^{k+1},\gamma_i)-\E\bigl(\psi(h_i^1,\dotsc,h_i^k,\tilde{h}_i^{k+1},\gamma_i)\!\bigm|\!\mathscr{S}_k\bigr)\bigr\}\cvc 0.
\end{equation}
To this end, let $L>0$ be such that $\psi\in\PL_{k+2}(r,L)$, and for each $v\equiv (x_1,\dotsc,x_k,y)\in\R^{k+1}$, define $\psi_v,\bar{\psi}_v\colon\R\to\R$ by $\psi_v(z):=\psi\bigl(x_1,\dotsc,x_k,\sum_{\ell=1}^k\bar{\alpha}_\ell^k\,x_\ell+\barp{\tau}_{k+1} z,y\bigr)$ and $\bar{\psi}_v(z):=\psi_v(z)-\E\bigl(\psi_v(Z)\bigr)$, where $Z\sim N(0,1)$. Then by Lemma~\ref{lem:pseudocomp}(a), there exists $K_k>0$, depending only on the deterministic $\bar{\alpha}^k\equiv (\bar{\alpha}_1^k,\dotsc,\bar{\alpha}_k^k)$, $\barp{\tau}_{k+1}$ and $r$, such that $\psi_v\in\PL_1(r,K_kL_{\norm{v}})$ with $L_{\norm{v}}=L(1\vee\norm{v}^{r-1})$. For a fixed integer $n>k$ and $v^{(1)},\dotsc,v^{(n)}\in\R^{k+1}$, define $\breve{L}\equiv\breve{L}\bigl(v^{(1)},\dotsc,v^{(n)}\bigr):=\bigl(L_{\norm{v^{(1)}}},\dotsc,L_{\norm{v^{(n)}}}\bigr)$. Let $r'=r/(r-1)\in (1,2]$ be as in the proof of $\mathcal{H}_1(c)$,
so that $1/r+1/r'=1$, and note that since $\norm{{\cdot}}_{p'}\leq\norm{{\cdot}}_p$ for $1\leq p\leq p'\leq\infty$, we have
\begin{equation}
\label{eq:Lbar}
\frac{\norm{\breve{L}}_\infty}{n^{1/r'}}\leq\frac{\norm{\breve{L}}_2}{n^{1/r'}}\leq\frac{\norm{\breve{L}}_{r'}}{n^{1/r'}}=\norm{\breve{L}}_{n,r'}=\biggl(\frac{1}{n}\sum_{i=1}^n\,\bigl(L_{\norm{v^{(i)}}}\bigr)^{r'}\biggr)^{1/r'}\leq L\biggl(1+\frac{1}{n}\sum_{i=1}^n\norm{v^{(i)}}^r\biggr)^{1/r'}.
\end{equation}
By Lemma~\ref{lem:pseudoconc}, it follows as in~\eqref{eq:H1cEbar} that there exists a universal constant $C>0$ such that if $Z_1,\dotsc,Z_n\iid N(0,1)$, then
\begin{align}
P\bigl(n,t,v^{(1)},\dotsc,v^{(n)}\bigr)&:=\Pr\biggl(\biggl|\frac{1}{n}\sum_{i=1}^n\,\bar{\psi}_{v^{(i)}}(Z_i)\biggr|\geq t\biggr)\notag\\
&\leq\exp\biggl(1-\min\biggl\{\biggl(\frac{nt}{(Cr)^r K_k\norm{\breve{L}}_2}\biggr)^2,\biggl(\frac{nt}{(Cr)^r K_k\norm{\breve{L}}_\infty}\biggr)^{2/r}\biggr\}\biggr)\notag\\
&\leq\exp\biggl(1-\min\biggl\{\biggl(\frac{n^{1/r}t}{(Cr)^r K_k\norm{\breve{L}}_{n,r'}}\biggr)^2,\biggl(\frac{n^{1/r}t}{(Cr)^r K_k\norm{\breve{L}}_{n,r'}}\biggr)^{2/r}\biggr\}\biggr)\notag\\
\label{eq:Ebar}
&=E_r(n,t,K_k\norm{\breve{L}}_{n,r'})\equiv E_r\Bigl(n,t,K_k\bigl\|\breve{L}\bigl(v^{(1)},\dotsc,v^{(n)}\bigr)\bigr\|_{n,r'}\Bigr)
\end{align}
for every $t\geq 0$. Next, define the $\mathscr{S}_k$-measurable vectors $\upsilon_k^{(i)}:=(h_i^1,\dotsc,h_i^k,\gamma_i)$ for $1\leq i\leq n$.
Returning to~\eqref{eq:Hkc3} and recalling~\eqref{eq:hk+1}, we see that 
\begin{align*}
\psi(h_i^1,\dotsc,h_i^k,\tilde{h}_i^{k+1},\gamma_i)-\E\bigl(\psi(h_i^1,\dotsc,h_i^k,\tilde{h}_i^{k+1},\gamma_i)\!\bigm|\!\mathscr{S}_k\bigr)&=\psi_{\upsilon_k^{(i)}}(\tilde{Z}_i^{k+1})-\E\bigl(\psi_{\upsilon_k^{(i)}}(\tilde{Z}_i^{k+1})\!\bigm|\!\mathscr{S}_k\bigr)\\
&=\bar{\psi}_{\upsilon_k^{(i)}}(\tilde{Z}_i^{k+1})
\end{align*}
for all $1\leq i\leq n$, where the final equality follows from Lemma~\ref{lem:indeplem} and the fact that $\tilde{Z}^{k+1}\equiv\tilde{Z}^{k+1}(n)$ is independent of $\mathscr{S}_k\equiv\mathscr{S}_k(n)$. We deduce from this and~\eqref{eq:Ebar} that
\begin{align}
\Pr(\abs{T_{n1}'}>\varepsilon\,|\,\mathscr{S}_k)=\Pr\biggl(\biggl|\frac{1}{n}\sum_{i=1}^n\bar{\psi}_{\upsilon_k^{(i)}}(\tilde{Z}_i^{k+1})\biggr|\geq\varepsilon\Bigm|\mathscr{S}_k\biggr)&=P\bigl(n,\varepsilon,\upsilon_k^{(1)},\dotsc,\upsilon_k^{(n)}\bigr)\notag\\
\label{eq:Pneps}
&\leq E_r\Bigl(n,\varepsilon,K_k\bigl\|\breve{L}\bigl(\upsilon_k^{(1)},\dotsc,\upsilon_k^{(n)}\bigr)\bigr\|_{n,r'}\Bigr)
\end{align}
for every $n>k$ and $\varepsilon>0$, where the second equality is again obtained using Lemma~\ref{lem:indeplem}. Now by~\eqref{eq:Lbar} and H\"older's inequality, which ensures that $\norm{{\cdot}}\equiv\norm{{\cdot}}_2\leq (k+1)^{\frac{1}{2}-\frac{1}{r}}\,\norm{{\cdot}}_r$ on $\R^{k+1}$,
we have
\begin{align}
\breve{L}_k(n):=\bigl\|\breve{L}\bigl(\upsilon_k^{(1)},\dotsc,\upsilon_k^{(n)}\bigr)\bigr\|_{n,r'}&\leq L\biggl(1+\frac{1}{n}\sum_{i=1}^n\norm{\upsilon_k^{(i)}}^r\biggr)^{1/r'}\notag\\
&\lesssim_{k,r} L\,\Biggl\{1+\frac{1}{n}\sum_{i=1}^n\sum_{\ell=1}^k\,\bigl(\bigl|h_i^\ell\bigr|^r+\abs{\gamma_i}^r\bigr)\Biggr\}^{1/r'}\notag\\
\label{eq:Lbreve}
&=L\biggl(1+\sum_{\ell=1}^k\,\bigl(\norm{h^\ell}_{n,r}^r+\norm{\gamma}_{n,r}^r\bigr)\biggr)^{1/r'}
\end{align}
for each $n>k$. Since $\norm{h^\ell}_{n,r}=O_c(1)$ for $1\leq\ell\leq k$ by $\mathcal{H}_k(b)$ and $\norm{\gamma}_{n,r}=O_c(1)$ by~\ref{ass:A1}, this means that $\breve{L}_k(n)=O_c(1)$. Thus, by Proposition~\ref{prop:compequiv}, there exists $\bar{L}_k\in (0,\infty)$ such that for integers $n>k$, the events $A_k(n):=\{\breve{L}_k(n)\leq\bar{L}_k\}\in\mathscr{S}_k(n)$ satisfy $\sum_{n=k+1}^\infty\Pr\bigl(A_k(n)^c\bigr)<\infty$. Moreover, for each $n>k$ and $\varepsilon>0$, it follows from~\eqref{eq:Pneps} that
\begin{align*}
\Pr\bigl(\{\abs{T_{n1}'}>\varepsilon\}\cap A_k(n)\!\bigm|\!\mathscr{S}_k\bigr)=\Pr(\abs{T_{n1}'}>\varepsilon\,|\,\mathscr{S}_k)\Ind_{A_k(n)}&\leq P\bigl(n,\varepsilon,\upsilon_k^{(1)},\dotsc,\upsilon_k^{(n)}\bigr)\Ind_{A_k(n)}\\
&\leq E_r\bigl(n,\varepsilon,K_k\breve{L}_k(n)\bigr)\Ind_{A_k(n)}\leq E_r(n,\varepsilon,K_k\bar{L}_k),
\end{align*}
where we have used the fact that $A_k(n)\in\mathscr{S}_k(n)$ to obtain the first equality above. Recalling the expression for $E_r(n,\varepsilon,K_k\bar{L}_k)$ in~\eqref{eq:Ebar}, we see that $\sum_{n=k+1}^\infty E_r(n,\varepsilon,K_k\bar{L}_k)<\infty$, and hence conclude as in~\eqref{eq:H1cTn1} that for every $\varepsilon>0$, we have
\begin{align*}
\sum_{n=k+1}^\infty\Pr(\abs{T_{n1}'}>\varepsilon)&\leq\sum_{n=k+1}^\infty\Pr\bigl(\{\abs{T_{n1}'}>\varepsilon\}\cap A_k(n)\bigr)+\sum_{n=k+1}^\infty\Pr\bigl(A_k(n)^c\bigr)\\
&\leq\sum_{n=k+1}^\infty E_r(n,\varepsilon,K_k\bar{L}_k)+\sum_{n=k+1}^\infty\Pr\bigl(A_k(n)^c\bigr)<\infty,
\end{align*}
which implies~\eqref{eq:Hkc3} in view of Proposition~\ref{prop:compequiv}. Together with~\eqref{eq:Hkc1} and~\eqref{eq:Hkc2}, this shows that $T_{n1}\cvc\E\bigl(\psi(\bar{G}_1,\dotsc,\bar{G}_k,\bar{G}_{k+1},\bar{\gamma})\bigr)$ in~\eqref{eq:Hkc}, as claimed.

The final step in the proof of $\mathcal{H}_{k+1}(c)$ is to show that $T_{n2}\cvc 0$ in~\eqref{eq:Hkc}. Letting $L>0$ be such that $\psi\in\PL_{k+2}(r,L)$, we can apply Lemma~\ref{lem:pseudoavg} as in~\eqref{eq:H1cTn2} to see that
\begin{align}
\abs{T_{n2}}&\leq \frac{1}{n}\sum_{i=1}^n\,\bigl|\psi(h_i^1,\dotsc,h_i^k,h_i^{k+1,k},\gamma_i)-\psi(h_i^1,\dotsc,h_i^k,h_i^{k+1,k}-\Delta_i^{k+1},\gamma_i)\bigr|\notag\\
&\leq L(k+2)^{\frac{r}{2}-1}\,\norm{\Delta^{k+1}}_{n,r}\biggl(1+2\sum_{\ell=1}^k\norm{h^\ell}_{n,r}^{r-1}+2\norm{\gamma}_{n,r}^{r-1}+\norm{h^{k+1,k}}_{n,r}^{r-1}+\norm{h^{k+1,k}-\Delta^{k+1}}_{n,r}^{r-1}\biggr)\notag\\
\label{eq:HkcT2}
&\lesssim_{k,r} L\norm{\Delta^{k+1}}_{n,r}\biggl(1+\sum_{\ell=1}^k\norm{h^\ell}_{n,r}^{r-1}+\norm{h^{k+1,k}}_{n,r}^{r-1}+\norm{\gamma}_{n,r}^{r-1}+\norm{\Delta^{k+1}}_{n,r}^{r-1}\biggr)
\end{align}
for each integer $n>k$. Now $\norm{h^{k+1,k}}_{n,r}\eqd\norm{h^{k+1}}_{n,r}$ for each such $n$, and recall from $\mathcal{H}_{k+1}(a)$ that $\norm{\Delta^{k+1}}_{n,r}=o_c(1)$ and from $\mathcal{H}_{k+1}(b)$ that $\norm{h^\ell}_{n,r}=O_c(1)$ for $1\leq\ell\leq k+1$. Thus, $T_{n2}\cvc 0$ by~\eqref{eq:HkcT2}, and we conclude from~\eqref{eq:Hkc} that $\inv{n}\sum_{i=1}^n\psi(h_i^1,\dotsc,h_i^k,h_i^{k+1},\gamma_i)\eqd T_n=T_{n1}+T_{n2}\cvc\E\bigl(\psi(\bar{G}_1,\dotsc,\bar{G}_k,\bar{G}_{k+1},\bar{\gamma})\bigr)$, as required. 

$\mathcal{H}_{k+1}(d)$: The arguments in this proof are similar to those given for $\mathcal{H}_1(d)$ and $\mathcal{H}_{k+1}(c)$, so we outline the key steps without going into the full details. By~\eqref{eq:hk+1} and Lemma~\ref{lem:conddist}(c), we have $(\gamma,m^0,h^1,\dotsc,h^k,h^{k+1})\eqdcond{\mathscr{S}_k}(\gamma,m^0,h^1,\dotsc,h^k,h^{k+1,k})$ for each $n>k$, so for fixed $\phi\in\PL_{k+2}(1)$, it follows that $\inv{n}\sum_{i=1}^n m_i^0\,\phi(h_i^1,\dotsc,h_i^k,h_i^{k+1})\eqd\inv{n}\sum_{i=1}^n m_i^0\,\phi(h_i^1,\dotsc,h_i^k,h_i^{k+1,k})=:T_n$ for each such $n$. Using~\eqref{eq:hk+1}, we now write 
\begin{align}
T_n&=\frac{1}{n}\sum_{i=1}^n m_i^0\,\phi(h_i^1,\dotsc,h_i^k,\tilde{h}_i^{k+1},\gamma_i)+\frac{1}{n}\sum_{i=1}^n m_i^0\bigl\{\phi(h_i^1,\dotsc,h_i^k,h_i^{k+1,k},\gamma_i)-\phi(h_i^1,\dotsc,h_i^k,\tilde{h}_i^{k+1},\gamma_i)\bigr\}\notag\\
\label{eq:Hkd}
&=:T_{n1}+T_{n2}
\end{align}
for each $n>k$, and aim to prove that $T_{n1}\cvc\E\bigl(F_0(\bar{\gamma})\cdot\phi(\bar{G}_1,\dotsc,\bar{G}_k,\bar{\gamma})\bigr)$ and $T_{n2}\cvc 0$, which together imply the desired conclusion. 

For $T_{n1}$, recall once again from~\eqref{eq:htilde} or~\eqref{eq:hk+1} that $\tilde{h}^{k+1}(n)=\sum_{\ell=1}^k\bar{\alpha}_\ell^k\,h^\ell(n)+\barp{\tau}_{k+1}\tilde{Z}^{k+1}(n)$ for each $n$, where $\tilde{Z}^{k+1}\equiv\tilde{Z}^{k+1}(n)\sim N_n(0,I_n)$ is independent of $\mathscr{S}_k\equiv\mathscr{S}_k(n)=\sigma(\gamma,m^0,h^1,\dotsc,h^k)$. Define $\Phi\colon\R^{k+1}\to\R$ by $\Phi(x_1,\dotsc,x_k,y):=\E\bigl\{\phi\bigl(x_1,\dotsc,x_k,\sum_{\ell=1}^k\bar{\alpha}_\ell^k\,x_\ell+\barp{\tau}_{k+1}Z,y\bigr)\bigr\}$ with $Z\sim N(0,1)$.
Then $\Phi\in\PL_{k+1}(1)$ by Lemma~\ref{lem:pseudocomp}(b), and as in~\eqref{eq:Hkc1} and~\eqref{eq:Hkc2}, it follows from the inductive hypothesis $\mathcal{H}_k(d)$ and Lemma~\ref{lem:indeplem} that
\begin{align}
\frac{1}{n}\sum_{i=1}^n m_i^0\,\E\bigl(\phi(h_i^1,\dotsc,h_i^k,\tilde{h}_i^{k+1},\gamma_i)\!\bigm|\!\mathscr{S}_k\bigr)&=
\frac{1}{n}\sum_{i=1}^n m_i^0\,\Phi(h_i^1,\dotsc,h_i^k,\gamma_i)\notag\\
&\cvc\E\bigl(F_0(\bar{\gamma})\cdot\Phi(\bar{G}_1,\dotsc,\bar{G}_k,\bar{\gamma})\bigr)\notag\\
\label{eq:Hkd1}
&=\E\bigl(F_0(\bar{\gamma})\cdot\phi(\bar{G}_1,\dotsc,\bar{G}_k,\bar{G}_{k+1},\bar{\gamma})\bigr).
\end{align}
Next, we show as in~\eqref{eq:Hkc3} that
\begin{equation}
\label{eq:Hkd2}
T_{n1}':=\frac{1}{n}\sum_{i=1}^n m_i^0\bigl\{\phi(h_i^1,\dotsc,h_i^k,\tilde{h}_i^{k+1},\gamma_i)-\E\bigl(\phi(h_i^1,\dotsc,h_i^k,\tilde{h}_i^{k+1},\gamma_i)\!\bigm|\!\mathscr{S}_k\bigr)\bigr\}\cvc 0.
\end{equation}
To this end, for each $u\in\R$ and $v\equiv (x_1,\dotsc,x_k,y)\in\R^{k+1}$, define $\phi_{u,v},\bar{\phi}_{u,v}\colon\R\to\R$ by $\phi_{u,v}(z):=u\phi\bigl(x_1,\dotsc,x_k,\sum_{\ell=1}^k\bar{\alpha}_\ell^k\,x_\ell+\barp{\tau}_{k+1}z,y\bigr)$ and $\bar{\phi}_{u,v}(z):=\phi_{u,v}(z)-\E\bigl(\phi_{u,v}(Z)\bigr)$, where $Z\sim N(0,1)$. Since $\phi\in\PL_{k+2}(1)$, we deduce from Lemma~\ref{lem:pseudocomp}(a) that there exists $K'>0$, depending only on the deterministic $\bar{\alpha}^k\equiv (\bar{\alpha}_1^k,\dotsc,\bar{\alpha}_k^k)$, $\barp{\tau}_{k+1}$ and $r$, such that $\bar{\phi}_{u,v}\in\PL_1(1,LK'\abs{u})$ for each $u\in\R$ and $v\in\R^{k+1}$. Now define the $\mathscr{S}_k$-measurable vectors $\upsilon_k^{(i)}:=(h_i^1,\dotsc,h_i^k,\gamma_i)$ for $1\leq i\leq n$, as in~\eqref{eq:Pneps}, and let $\tilde{E}_r$ be as in~\eqref{eq:H1d2}. Then by Lemma~\ref{lem:indeplem} and~\eqref{eq:pseudoconcr} in Remark~\ref{rem:pseudoconc}, it follows as in~\eqref{eq:H1d2},~\eqref{eq:Ebar} and~\eqref{eq:Pneps} that for each $n>k$ and $\varepsilon>0$, we have
\begin{equation}
\label{eq:HkdTn1'}
\Pr(\abs{T_{n1}'}>\varepsilon\,|\,\mathscr{S}_k)=\Pr\biggl(\biggl|\frac{1}{n}\sum_{i=1}^n\bar{\phi}_{m_i^0,\,\upsilon_k^{(i)}}(\tilde{Z}_i^{k+1})\biggr|\geq\varepsilon\Bigm|\mathscr{S}_k\biggr)\leq \tilde{E}_r(n,\varepsilon,LK'\norm{m^0}_n).
\end{equation}
Now $m^0\equiv m^0(n)$ is measurable with respect to $\mathscr{S}_0\subseteq\mathscr{S}_k\equiv\mathscr{S}_k(n)$ for each $n$, and $\norm{m^0}_n=O_c(1)$ by~\ref{ass:A2}, so we conclude as in~\eqref{eq:H1d3} and~\eqref{eq:H1dTn1} that $\sum_n\Pr(\abs{T_{n1}'}>\varepsilon)<\infty$ for all $\varepsilon>0$. Thus, $T_{n1}'\cvc 0$ by Proposition~\ref{prop:compequiv}, as claimed in~\eqref{eq:Hkd2}.

Finally, we prove that $T_{n2}\cvc 0$. Let $L>0$ be such that $\phi\in\PL_{k+2}(1,L)$, and for $n>k$ and $1\leq i\leq n$, define $\upsilon_{k+1}^{(i)}=(h_i^1,\dotsc,h_i^k,h_i^{k+1,k},\gamma_i)$ and $\tilde{\upsilon}_{k+1}^{(i)}=(h_i^1,\dotsc,h_i^k,\tilde{h}_i^{k+1},\gamma_i)=(h_i^1,\dotsc,h_i^k,h_i^{k+1,k}-\Delta_i^{k+1},\gamma_i)$ as in~\eqref{eq:HkcT2}, where the final equality is obtained from~\eqref{eq:deltak}. As in the proof of $\mathcal{H}_1(d)$, we now apply the Cauchy--Schwarz inequality and the fact that $\norm{{\cdot}}_n\equiv\norm{{\cdot}}_{n,2}\leq\norm{{\cdot}}_{n,r}$ to see that
\[\abs{T_{n2}}\leq\frac{1}{n}\sum_{i=1}^n\,\abs{m_i^0}\bigl|\phi(\upsilon_{k+1}^{(i)})-\phi(\tilde{\upsilon}_{k+1}^{(i)})\bigr|\leq\frac{L}{n}\sum_{i=1}^n\,\abs{m_i^0}\abs{\Delta_i^{k+1}}\leq L\norm{m^0}_n\norm{\Delta^{k+1}}_n.\]
Since $\norm{m^0}_n\cvc\tau_1$ by~\ref{ass:A2} and $\norm{\Delta^{k+1}}_n\leq\norm{\Delta^{k+1}}_{n,r}=o_c(1)$ by $\mathcal{H}_{k+1}(a)$, we conclude that $T_{n2}=o_c(1)$, as required. Together with~\eqref{eq:Hkd},~\eqref{eq:Hkd1},~\eqref{eq:Hkd2}, this yields $\mathcal{H}_{k+1}(d)$, and hence completes the inductive step for Proposition~\ref{prop:AMPproof}.
\end{proof}

\deparskip
\begin{proof}[Proof of Remark~\ref{rem:AMPproof}]
Under~\ref{ass:A0},~\ref{ass:A4} and~\ref{ass:A5}, if instead~\ref{ass:A1}--\ref{ass:A3} hold with $\cvc$ and $O_c(1)$ replaced with $\cvp$ and $O_p(1)$ respectively, then as explained in the third bullet point in Remark~\ref{rem:AMPconv}, we can make the same replacements in the proof of Proposition~\ref{prop:AMPproof} and most of the arguments go through as before. However, a few alterations are required in the proofs of (c,\,d) and (g), which we now describe. 

First, in the proof of $\mathcal{H}_1(d)$, the goal in~\eqref{eq:H1d1} is now to show that $T_{n1}'\cvp 0$ as $n\to\infty$. Instead of proceeding as in~\eqref{eq:H1d3} and~\eqref{eq:H1dTn1}, we return to~\eqref{eq:H1dTn1'}, where we note that if $\varepsilon>0$ is fixed and~\ref{ass:A2} takes the form $\norm{m^0}_n\cvp\tau_1$, then $\Pr(\abs{T_{n1}'}>\varepsilon\,|\,\mathscr{S}_0)\leq \tilde{E}_r(n,\varepsilon,L\tau_1\norm{m^0}_n)=o_p(1)$ by Slutsky's lemma and the definition of $\tilde{E}_r$ in~\eqref{eq:H1d2}. Then for every $\varepsilon>0$, it follows from the bounded convergence theorem that $\Pr(\abs{T_{n1}'}>\varepsilon)=\E\bigl(\Pr(\abs{T_{n1}'}>\varepsilon\,|\,\mathscr{S}_0)\bigr)\to 0$ as $n\to\infty$, so $T_{n1}'\cvp 0$, as desired. 

In the proofs of $\mathcal{H}_{k+1}(c,d)$, the analogues of~\eqref{eq:Hkc3} and~\eqref{eq:Hkd2} can be derived from~(\ref{eq:Pneps},\,\ref{eq:Lbreve}) and~\eqref{eq:HkdTn1'} respectively in much the same way; for the former, since $\norm{h^\ell}_{n,r}=O_p(1)$ for $1\leq\ell\leq k$ by the modified $\mathcal{H}_k(b)$,~\eqref{eq:Lbreve} implies that $\breve{L}_k(n)\lesssim_{k,r}LK\bigl(1+\sum_{\ell=1}^k\,\norm{h^\ell}_{n,r}^r\bigr)^{1/r'}=O_p(1)$.

In addition, $\mathcal{H}_k(g)$ now reads $b_k=\abr{f_k'(h^k,\gamma)}_n\cvp\E\bigl(f_k'(\bar{G}_k,\bar{\gamma})\bigr)=\bar{b}_k$. To prove this, we can argue along subsequences, similarly to the proof of Corollary~\ref{cor:Wr}(b).
\end{proof}


\deparskip
\begin{proof}[Proofs of Theorems~\ref{thm:AMPmaster} and~\ref{thm:masterext}]
Theorem~\ref{thm:AMPmaster} follows from Theorem~\ref{thm:masterext}, which in turn is a immediate consequence of Proposition~\ref{prop:AMPproof}(c) and Corollary~\ref{cor:Wr}(b).
\end{proof}

\deparskip
\begin{proof}[Proofs for Remark~\ref{rem:AMPconv}] 
(a) \emph{Convergence in probability}: This is immediate from Remark~\ref{rem:AMPproof} and Corollary~\ref{cor:Wr}(b).

(b) \emph{Almost sure convergence}: The random sequences $\Upsilon:=\bigl(m^0(n):n\in\N\bigr)$ and $\Gamma:=\bigl(\gamma(n):n\in\N\bigr)$ take values in $E:=\prod_{n=1}^\infty\R^n$, whose cylindrical and Borel $\sigma$-algebras 
coincide by~\citet[Lemma~1.2]{Kal97}. Let $E^*$ be the set of all $(u,v)\in E\times E$ such that~\ref{ass:A1}--\ref{ass:A3} hold when $\Upsilon=u\equiv\bigl(u(n):n\in\N\bigr)$ and $\Gamma=v\equiv\bigl(v(n):n\in\N\bigr)$ are non-random. It can be verified that $E^*$ is a Borel subset of $E\times E$.

For $k\in\N$, let $\bar{\mu}^k:=N(0,\tau_k^2)\otimes\pi$ and $\mu_n^k:=\nu_n(h^k,\gamma)$ for $n\in\N$. In the special case where $(\Upsilon,\Gamma)\in E^*$ is deterministic,
Theorem~\ref{thm:AMPmaster} implies that for each $k$, the resulting sequence of AMP iterates $\bigl(h^k(n):n\in\N\bigr)$ satisfies $d_r(\mu_n^k,\bar{\mu}^k)\cvc 0$. Note that for each $n$, we can write $d_r(\mu_n^k,\bar{\mu}^k)=d_r\bigl(\nu_n(h^k,\gamma),\bar{\mu}^k\bigr)=g_n\bigl(m^0(n),\gamma(n),W(n)\bigr)$ for some (non-random) Borel measurable $g_n\colon\R^n\times\R^n\times\R^{n\times n}\to\R$. Indeed, we see from~\eqref{eq:AMPsym} that $h^k\equiv h^k(n)$ is a deterministic Borel measurable function of $m^0(n)$, $\gamma(n)$ and $W(n)$. Moreover for all $x,x'\in\R^n$ and the corresponding empirical distributions $\nu_n(x),\nu_n(x')$ of their components, we have \[\bigl|d_r\bigl(\nu_n(x),\bar{\mu}^k\bigr)-d_r\bigl(\nu_n(x'),\bar{\mu}^k\bigr)\bigr|\leq d_r\bigl(\nu_n(x),\nu_n(x')\bigr)\leq\bigl(n^{-1}\textstyle\sum_{i=1}^n\,\abs{x_i-x_i'}^r\bigr)^{1/r}=\norm{x-x'}_{n,r}\]
since $d_r$ is a metric, so $x\mapsto d_r\bigl(\nu_n(x),\bar{\mu}^k\bigr)$ is continuous on $\R^n$. 

Since $\{(a_1,a_2,\dotsc)\in\R^\N:\lim_{n\to\infty}a_n=0\}$ is a Borel subset of $\R^\N$, we conclude that $g\colon (u,v)\mapsto\Pr\bigl\{\lim_{n\to\infty}g_n\bigl(u(n),v(n),W(n)\bigr)=0\bigr\}$
is a well-defined Borel measurable function on $E\times E$ satisfying $g(u,v)=1$ for all $(u,v)\in E^*$. 

Now suppose more generally that $(\Upsilon,\Gamma)$ and $\bigl(W(n):n\in\N\bigr)$ are independent. If $(\Upsilon,\Gamma)\in E^*$ almost surely, then for the corresponding sequence of AMP iterates $\bigl(h^k(n):n\in\N\bigr)$ from~\eqref{eq:AMPsym},
\begin{align*}
\Pr\Bigl(\lim_{n\to\infty}d_r(\mu_n^k,\bar{\mu}^k)=0\Bigr)&=\Pr\Bigl(\lim_{n\to\infty}g_n\bigl(m^0(n),\gamma(n),W(n)\bigr)=0\Bigr)\\
&=\E\Bigl\{\Pr\Bigl(\lim_{n\to\infty}g_n\bigl(m^0(n),\gamma(n),W(n)\bigr)=0\Bigm|\Upsilon,\Gamma\Bigr)\Bigr\}\\
&=\E\bigl(g(\Upsilon,\Gamma)\bigr)\geq\E\bigl(g(\Upsilon,\Gamma)\,\Ind_{\{(\Upsilon,\Gamma)\in E^*\}}\bigr)=1,
\end{align*}
where the third equality follows from the independence assumption and Lemma~\ref{lem:indeplem}. Therefore, $d_r(\mu_n^k,\bar{\mu}^k)\cvas 0$ as $n\to\infty$, and moreover $\widetilde{d}_r(\mu_n^k,\bar{\mu}^k)\cvas 0$ by Corollary~\ref{cor:Wr}(a).
\end{proof}

\umparskip
\subsection{Auxiliary results and proofs for Section~\protect{\ref{Sec:Master}}}
\label{sec:addproofs}
\begin{proof}[Proof of Lemma~\ref{lem:posdef}]
We proceed by induction on $k$, noting first that the base case $k=1$ is trivial since $\bar{\mathrm{T}}^{[1]}=\tau_1^2>0$ by~\ref{ass:A4}. Now for $k\geq 2$ and $a\equiv (a_1,\dotsc,a_k)\in\R^k$, recall the expression~\eqref{eq:posdef} for $a^\top\bar{\mathrm{T}}^{[k]}a$. If $a_1\neq 0=a_2=\dotsb=a_k$, then $a^\top\bar{\mathrm{T}}^{[k]}a=(a_1\tau_1)^2>0$ as in the base case.
On the other hand, if $a_L\neq 0$ for some $2\leq L\leq k$, then by~\ref{ass:A4}, we can find
$B_L\subseteq\R$ with $\pi(B_L)>0$ such that $x_{L-1}\mapsto a_L f_{L-1}(x_{L-1},y)$ is non-constant whenever $y\in B_L$. For all such $y$, note that $(x_1,\dotsc,x_{k-1})\mapsto\sum_{\ell=2}^k a_\ell f_{\ell-1}(x_{\ell-1},y)$ is non-constant on $\R^{k-1}$. Now by the inductive hypothesis, $(G_1,\dotsc,G_{k-1})$ has a positive definite covariance matrix $\bar{\mathrm{T}}^{[k-1]}$, so the random variable $a_1 F_0(y)+\sum_{\ell=2}^k a_\ell f_{\ell-1}(G_{\ell-1},y)$ is non-degenerate whenever $y\in B_L$. Since $\bar{\gamma}\sim\pi$ is independent of $G_1,\dotsc,G_{k-1}$ and $\Pr(\bar{\gamma}\in B_L)=\pi(B_L)>0$, it follows that $a_1 F_0(\bar{\gamma})+\sum_{\ell=2}^k a_\ell f_{\ell-1}(G_{\ell-1},\bar{\gamma})$ is also non-degenerate. Thus, in all cases, it follows from~\eqref{eq:posdef} and~\ref{ass:A3} that $a^\top\bar{\mathrm{T}}^{[k]}a>0$ whenever $a\neq 0$, as claimed.
\end{proof}

\deparskip
\begin{proof}[Proof of Remark~\ref{rem:A4'}]
Since $\tilde{f}_0$ is Lipschitz and $\bar{\eta},\bar{\gamma}\in\mathcal{P}_1(2)$, we have $\E\bigl\{\bigl\|\bigl(\tilde{f}_0(\bar{\eta},\bar{\gamma}),\bar{\gamma}\bigr)\bigr\|^2\bigr\}<\infty$, so $\mu^0\in\mathcal{P}_2(2)$. By Corollary~\ref{cor:Wr}(b), an equivalent formulation of~\ref{ass:A1+} is that \[\frac{1}{n}\sum_{i=1}^n\psi(m_i^0,\gamma_i)\cvc\E\bigl\{\psi\bigl(\tilde{f}_0(\bar{\eta},\bar{\gamma}),\bar{\gamma}\bigr)\bigr\}\]
for all $\psi\in\PL_2(2)$. In particular, $(x,y)\mapsto x^2$ lies in $\PL_2(2)$, so $\norm{m^0}_n^2=n^{-1}\sum_{i=1}^n\,\abs{m_i^0}^2\cvc\E\bigl(\tilde{f}_0(\bar{\eta},\bar{\gamma})^2\bigr)=:\tau_1^2$, which yields the first part of~\ref{ass:A2}. Moreover, the function $F_0\colon\R\to\R$ defined by $F_0(y):=\E\bigl(\tilde{f}_0(\bar{\eta},y)\bigr)$ is Lipschitz, and since $\bar{\eta},\bar{\gamma}$ are independent, it follows from Lemma~\ref{lem:indeplem} and Jensen's inequality that \[\E\bigl(F_0(\bar{\gamma})^2\bigr)=\E\bigl\{\E\bigl(\tilde{f}_0(\bar{\eta},\bar{\gamma})\!\bigm|\!\bar{\gamma}\bigr)^2\bigr\}\leq\E\bigl\{\E\bigl(\tilde{f}_0(\bar{\eta},\bar{\gamma})^2\!\bigm|\!\bar{\gamma}\bigr)\bigr\}=\tau_1^2.\] For each Lipschitz $\phi\colon\R\to\R$, Lemma~\ref{lem:pseudoprod} ensures that $(x,y)\mapsto x\phi(y)$ belongs to $\PL_2(2)$, so 
\[\ipr{m^0}{\phi(\gamma)}_n=\frac{1}{n}\sum_{i=1}^n m_i^0\,\phi(\gamma_i)\cvc\E\bigl(\tilde{f}_0(\bar{\eta},\bar{\gamma})\cdot\phi(\bar{\gamma})\bigr)=\E\bigl\{\E\bigl(\tilde{f}_0(\bar{\eta},\bar{\gamma})\!\bigm|\!\bar{\gamma}\bigr)\phi(\bar{\gamma})\bigr\}=\E\bigl(F_0(\bar{\gamma})\phi(\bar{\gamma})\bigr),\]
where the final equality again follows from Lemma~\ref{lem:indeplem}. Therefore,~\ref{ass:A3} also holds.
\end{proof}

\unparskip
The following auxiliary result is used in the proof of Proposition~\ref{prop:AMPproof}(a) to control the third summand in the deviation term $\Delta^{k+1}$ defined in~\eqref{eq:deltak}.
\begin{lemma}
\label{lem:projgaussian}
For $n\in\N$ and $k\in\{0,1\dotsc,n-1\}$, let $\tilde{Z}^{k+1}\equiv\tilde{Z}^{k+1}(n)$ be as in Proposition~\ref{prop:AMPconddist}, so that $\tilde{Z}^{k+1}\sim N_n(0,I_n)$ is independent of $\mathscr{S}_k$. If $\xi_1,\dotsc,\xi_k\iid N(0,1)$ and $r\geq 1$, then $\norm{P_k\tilde{Z}^{k+1}}_{n,r}$ is stochastically dominated by $\sum_{i=1}^k\,\abs{\xi_i}/n^{1/(r\vee 2)}$.
\end{lemma}

\deparskip
\begin{proof}
Note that $P_k$ is an $\mathscr{S}_k$-measurable 
projection matrix of rank $r_k\leq k$. Since $\tilde{Z}^{k+1}$ is independent of $\mathscr{S}_k$, it therefore has conditional distribution $N_n(0,I_n)$ given $\mathscr{S}_k$ by Remark~\ref{rem:indepcond}. Now let $\xi_1,\dotsc,\xi_k\iid N(0,1)$ be independent of $\mathscr{S}_k$ and let $\{\tilde{m}^1,\dotsc,\tilde{m}^{r_k}\}$ be any $\mathscr{S}_k$-measurable orthonormal basis of $\Img(M_k)=\Img(P_k)$, as in Remark~\ref{rem:gs}. Recall that if $Z\sim N_n(0,I_n)$ and $P\in\R^{n\times n}$ is a deterministic projection matrix of rank $p$, then $PZ\sim N(0,P)$, which is also the distribution of $\sum_{i=1}^p\zeta_i u_i$ when $\zeta_1,\dotsc,\zeta_p\iid N(0,1)$ and $\{u_1,\dotsc,u_p\}$ is any orthonormal basis of $\Img(P)$. We deduce from this and Lemma~\ref{lem:conddist}(b) that $P_k\tilde{Z}^{k+1}$ and $\sum_{i=1}^{r_k}\xi_i\tilde{m}^i$ both have conditional distribution $N_n(0,P_k)$ given $\mathscr{S}_k$. This implies that $\norm{P_k\tilde{Z}^{k+1}}_{n,r}\eqd\norm{\sum_{i=1}^{r_k}\xi_i\tilde{m}^i}_{n,r}$.

Now for all $x\in\R^n$, we have $\norm{x}_{n,r}=n^{-1/r}\norm{x}_r\leq n^{-1/(r\vee 2)}\norm{x}_2$ by H\"older's inequality and the fact that $\norm{{\cdot}}_{p'}\leq\norm{{\cdot}}_p$ for $1\leq p\leq p'$. Since $\norm{\tilde{m}^i}_2=1$ for all $i$ by definition, it follows from this and the triangle inequality for $\norm{{\cdot}}_{n,r}$ that $\norm{\sum_{i=1}^{r_k}\xi_i\tilde{m}^i}_{n,r}\leq\sum_{i=1}^{r_k}\,\abs{\xi_i}\,\norm{\tilde{m}^i}_{n,r}\leq\sum_{i=1}^{r_k}\,\abs{\xi_i}/n^{1/(r\vee 2)}\leq\sum_{i=1}^k\,\abs{\xi_i}/n^{1/(r\vee 2)}$. Combining this with the conclusion of the previous paragraph yields the result.
\end{proof}

\umparskip
\subsection{AMP with matrix-valued iterates}
\label{sec:AMPmatrix}
As mentioned in Section~\ref{sec:AMPsym}, state evolution characterisations can be obtained for more general abstract AMP recursions in which the iterates are matrices rather than vectors. Here, we will briefly describe the extended version of the asymmetric iteration~\eqref{eq:AMPnonsym}, which is used to establish the master theorem for GAMP in Section~\ref{sec:GAMPmaster}.

For $n,p\in\N$, let $W\in\R^{n\times p}$, $\beta\in\R^p$ and $\gamma\in\R^n$ be as in~\ref{ass:B0}. For $\ell_E,\ell_H\in\N$, let $(g_k,f_{k+1}:k\in\N_0)$ be two sequences of Lipschitz functions $g_k\colon\R^{\ell_E}\times\R\to\R^{\ell_H}$ and $f_{k+1}\colon\R^{\ell_H}\times\R\to\R^{\ell_E}$, which are applied row-wise to matrices. Given $Q^{-1}:=0\in\R^{n\times \ell_H}$, $B_0\in\R^{\ell_E\times \ell_H}$ and $M^0\in\R^{p\times \ell_E}$, inductively define 
\begin{equation}
\label{eq:AMPmatrix}
\begin{alignedat}{3}
E^k&:=WM^k-Q^{k-1}B_k^\top,\quad\;\;&Q^k&:=g_k(E^k,\gamma),\quad\;\;&C_k&:=
n^{-1}\textstyle\sum_{i=1}^n g_k'(E_i^k,\gamma_i),\\ 
H^{k+1}&:=W^\top Q^k-M^k\,C_k^\top,\quad\;\;&M^{k+1}&:=f_{k+1}(H^{k+1},\beta),\quad\;\;&B_{k+1}&:=
n^{-1}\textstyle\sum_{j=1}^p f_{k+1}'(H_j^{k+1},\beta_j)
\end{alignedat}
\end{equation}
for $k\in\N_0$. Here, $E_i^k$ and $H_j^{k+1}$ denote the $i^{th}$ and $j^{th}$ rows of $E^k\in\R^{n\times\ell_E}$ and $H^{k+1}\in\R^{p\times\ell_H}$ respectively, and $g_k'\colon\R^{\ell_E}\times\R\to\R^{\ell_H\times \ell_E}$ and $f_{k+1}'\colon\R^{\ell_H}\times\R\to\R^{\ell_E\times \ell_H}$ are bounded, Borel measurable functions that agree with the derivatives (Jacobians) of $g_k,f_{k+1}$ respectively with respect to their first arguments, wherever the latter are defined. 

Consider now a sequence of recursions~\eqref{eq:AMPmatrix} indexed by $n$ and $p\equiv p_n$ with $n/p\to\delta\in (0,\infty)$ as $n\to\infty$, and assume appropriate analogues of~\ref{ass:B0}--\ref{ass:B5} with $r\in [2,\infty)$. In particular, suppose in place of~\ref{ass:B2} that $(M^0)^\top M^0/n\cvc\Sigma_0$ for some non-negative definite $\Sigma_0\in\R^{\ell_E\times\ell_E}$, and that $p^{-1}\sum_{i=1}^p\sum_{j=1}^{\ell_E}\,\abs{M_{ij}^0}^r=O_c(1)$.
The state evolution recursion for~\eqref{eq:AMPmatrix} is then defined analogously to that in~\eqref{eq:statevolns}, via
\begin{equation}
\label{eq:statevolmat}
\begin{split}
\mathrm{T}_{k+1}&:=\E\bigl(g_k(G_k^\sigma,\bar{\gamma})^\top g_k(G_k^\sigma,\bar{\gamma})\bigr)\in\R^{\ell_H\times\ell_H},\\
\Sigma_{k+1}&:=\inv{\delta}\,\E\bigl(f_{k+1}(G_{k+1}^\tau,\bar{\beta})^\top f_{k+1}(G_{k+1}^\tau,\bar{\beta})\bigr)\in\R^{\ell_E\times\ell_E}
\end{split}
\end{equation}
for $k\in\N_0$, where we take $G_k^\sigma\sim N_{\ell_E}(0,\Sigma_k)$ to be independent of $\bar{\beta}\sim\pi_{\bar{\beta}}$, and $G_{k+1}^\tau\sim N_{\ell_H}(0,\mathrm{T}_{k+1})$ to be independent of $\bar{\gamma}\sim\pi_{\bar{\gamma}}$.

For $k\in\N_0$, it can be shown that the empirical distributions of the rows of $(E^k\;\gamma)$ and $(H^{k+1}\;\beta)$ converge completely in $d_r$ to $N_{\ell_E}(0,\Sigma_k)\otimes\pi_{\gamma}$ and $N_{\ell_H}(0,\mathrm{T}_{k+1})\otimes\pi_{\beta}$ respectively as $n,p\to\infty$ with $n/p\to\delta$. Similarly as in Remark~\ref{rem:bk}, these limiting distributions remain unchanged if one or both of $C_k,B_{k+1}$ are replaced with the deterministic matrices $\bar{C}_k:=\E\bigl(g_k'(G_k^\Sigma,\bar{\gamma})\bigr)$ and $\bar{B}_{k+1}:=\E\bigl(f_{k+1}'(G_{k+1}^{\mathrm{T}},\bar{\beta})\bigr)$ respectively. Moreover, by generalising the definitions~\eqref{eq:Sigmabarns}--\eqref{eq:Taubarns} of the limiting covariance matrices in line with~\eqref{eq:statevolmat}, one can obtain the $d_r$ limits of the joint empirical distributions for~\eqref{eq:AMPmatrix} above.

The proofs of these results are conceptually very similar to that of Theorem~\ref{thm:AMPnonsym}. For further details, see~\citet{JM13}, who first consider a generalisation of the symmetric iteration~\eqref{eq:AMPsym} with matrix-valued iterates, and then handle the asymmetric case by a reduction argument.

\hfparskip
\subsection{Proofs for Section~\ref{Sec:LowRank}}
\label{sec:LowRankproofs}
\begin{proof}[Proof of Theorem~\ref{thm:AMPlowsym}]
As described in the proof sketch on page~\getpagerefnumber{sketch:rankone}, we introduce the recursion~\eqref{eq:matsymabs} given by $u^1\equiv u^1(n)=W\hat{v}^0=Wg_0(v^0)$ and
\begin{align*}
u^{k+1}\equiv u^{k+1}(n)&=Wg_k(u^k+\mu_k v)-\tilde{b}_kg_{k-1}(u^{k-1}+\mu_{k-1}v)\\
&=Wf_k(u^k,v)-\tilde{b}_k f_{k-1}(u^{k-1},v)
\end{align*}
for $k,n\in\N$, where $f_k(x,y)=g_k(x+\mu_k y)$ and $f_k'(x,y)=g_k'(x+\mu_k y)$ for $x,y\in\R$, and $\tilde{b}_k\equiv\tilde{b}_k(n)=\abr{g_k'(u^k+\mu_k v)}_n=\abr{f_k'(u^k,v)}_n$. First, we verify that this is an iteration of the form~\eqref{eq:AMPsym} to which we can apply the master theorems from Section~\ref{sec:AMPsym} for symmetric AMP. Indeed, it follows from~\ref{ass:S0} and~\ref{ass:S1} respectively that~\eqref{eq:matsymabs} satisfies~\ref{ass:A0} and~\ref{ass:A1+}, where the latter holds with $m^0=v$, $\gamma=v$, $\bar{\gamma}=V\sim\pi$, $\bar{\eta}=U$ and $\tilde{f}_0(x,y)=f_0(\mu_0 x+\sigma_0 y)$ for $x,y\in\R$. By Remark~\ref{rem:A4'},~\ref{ass:A1+} implies that~\ref{ass:A1}--\ref{ass:A3} hold with $r=2$ and $\tau_1=\clim_{n\to\infty}\norm{\hat{v}^0}_n^2=\sigma_1^2$. As verified in~\eqref{eq:statevolrankone}, the state evolution parameters $(\tau_k:k\in\N)$ for~\eqref{eq:matsymabs} satisfy $\tau_k^2=\sigma_k^2$ for all $k$ in view of~\eqref{eq:statevolsymat}. Finally, by~\ref{ass:S2}, each $f_k\colon\R^2\to\R$ is Lipschitz and the corresponding $f_k'$ satisfies~\ref{ass:A5}.

Consequently, for each $k\in\N$, it follows from Theorem~\ref{thm:masterext} that
\[\sup_{\psi\in\PL_{k+1}(2,1)}\;\biggl|\frac{1}{n}\sum_{i=1}^n\psi(v_i^0,v_i^1,\dotsc,v_i^k,v_i)-\E\bigl(\psi(\mu_0 V+\sigma_0 U,\sigma_1 G_1,\dotsc,\sigma_k G_k,V)\bigr)\biggr|\cvc 0\]
as $n\to\infty$, where $(\sigma_1 G_1,\dotsc,\sigma_k G_k)\sim N_k(0,\bar{\Sigma}^{[k]})$ is taken to be independent of $(U,V)$ from~\ref{ass:S1}. Since $\Phi_k\colon(x_1,\dotsc,x_k,y)\mapsto (x_1+\mu_1 y,\dotsc,x_k+\mu_k y,y)$ is a linear map with Lipschitz constant $\tilde{L}_k:=\norm{(\mu_1,\dotsc,\mu_k,1)}$, we have $\tilde{L}_k^{-2}(\psi\circ\Phi_k)\in\PL_{k+1}(2,1)$ whenever $\psi\in\PL_{k+1}(2,1)$,
so it follows from the display above that
\begin{align}
&\sup_{\psi\in\PL_{k+2}(2,1)}\;\biggl|\frac{1}{n}\sum_{i=1}^n\psi(v_i^0,v_i^1+\mu_1 v_i,\dotsc,v_i^k+\mu_k v_i,v_i)\notag\\
\label{eq:matsymdeltaconv}
&\hspace{3cm}-\E\bigl(\psi(\mu_0 V+\sigma_0 U,\mu_1 V+\sigma_1 G_1,\dotsc,\mu_k V+\sigma_k G_k,V)\bigr)\biggr|\cvc 0
\end{align}
as $n\to\infty$. Defining $\tilde{\Delta}^k\equiv\tilde{\Delta}^k(n):=v^k-(u^k+\mu_k v)\in\R^n$ for $k,n\in\N$, we can apply Lemma~\ref{lem:pseudoavg} to see that
\begin{align}
&\sup_{\psi\in\PL_{k+2}(2,1)}\;\biggl|\frac{1}{n}\sum_{i=1}^n\psi(v_i^0,v_i^1,\dotsc,v_i^k,v_i)-\psi(v_i^0,v_i^1+\mu_1 v_i,\dotsc,v_i^k+\mu_k v_i,v_i)\biggr|\notag\\
&\hspace{2.5cm}\leq \rbr{\sum_{\ell=1}^k\norm{\tilde{\Delta}^\ell}_n^2}^{1/2}\biggl(1+\sum_{\ell=1}^k\,\bigl(\norm{v^\ell}_n+\norm{u^\ell+\mu_\ell v}_n\bigr)+2\bigl(\norm{v^0}_n+\norm{v}_n\bigr)\biggr)\notag\\
\label{eq:Deltatilde}
&\hspace{2.5cm}\leq \rbr{\sum_{\ell=1}^k\norm{\tilde{\Delta}^\ell}_n^2}^{1/2}\biggl(1+\sum_{\ell=1}^k\,\bigl(\norm{\tilde{\Delta}^\ell}_n+2\norm{u^\ell+\mu_\ell v}_n\bigr)+2\bigl(\norm{v^0}_n+\norm{v}_n\bigr)\biggr)
\end{align}
for all $k$ and $n$, where $\norm{{\cdot}}_n\equiv\norm{{\cdot}}_{n,2}=n^{-1/2}\,\norm{{\cdot}}$ on $\R^n$. For every $\ell\in\N$, it follows from~\ref{ass:S1} and~\eqref{eq:matsymdeltaconv} that
\begin{align}
&\norm{v^0}_n\cvc\E\bigl((\mu_0 V+\sigma_0 U)^2\bigr)=\mu_0^2+\sigma_0^2,\qquad\norm{v}_n^2\to\E(V^2)=1\notag\\
\label{eq:vtildenorm}
\text{and}\quad&\norm{u^\ell+\mu_\ell v}_n^2=\frac{1}{n}\sum_{i=1}^n\,(u_i^\ell+\mu_\ell v_i)^2\cvc\E\bigl((\mu_\ell V+\sigma_\ell G_\ell)^2\bigr)=\mu_\ell^2+\sigma_\ell^2
\end{align}
as $n\to\infty$. We will now establish by induction on $k\in\N$ that
\begin{equation}
\label{eq:Deltatildenorm}
\norm{\tilde{\Delta}^k}_n=\norm{v^k-(u^k+\mu_k v)}_n\cvc 0\quad\text{as }n\to\infty
\end{equation}
and hence that the conclusion~\eqref{eq:AMPlowsym} of Theorem~\ref{thm:AMPlowsym} holds for every $k$. For the base case $k=1$, we have $\norm{v}_n\cvc 1$ by~\eqref{eq:vconv} and $\lambda\ipr{\hat{v}^0}{v}_n\cvc\mu_1$ by~\eqref{eq:musigma1}, so
\[\norm{\tilde{\Delta}^1}_n=\norm{v^1-(u^1+\mu_1 v)}_n=\norm{A\hat{v}^0-(W\hat{v}^0+\mu_1 v)}_n=\abs{\lambda\ipr{\hat{v}^0}{v}_n-\mu_1}\,\norm{v}_n\cvc 0\]
as $n\to\infty$. It follows from this and~\eqref{eq:matsymdeltaconv}--\eqref{eq:vtildenorm} that~\eqref{eq:AMPlowsym} holds when $k=1$. For a general $k\geq 2$, we write
\begin{align}
\tilde{\Delta}^{k+1}\equiv\tilde{\Delta}^{k+1}(n)&=v^{k+1}-(u^{k+1}+\mu_{k+1}v)\notag\\
&=Ag_k(v^k)-b_k g_{k-1}(v^{k-1})-\bigl(Wg_k(u^k+\mu_k v)-\tilde{b}_k g_{k-1}(u^{k-1}+\mu_{k-1}v)+\mu_{k+1}v\bigr)\notag\\
&=\bigl(\lambda\ipr{v}{g_k(v^k)}_n-\mu_{k+1}\bigr)v+W\bigl(g_k(v^k)-g_k(u^k+\mu_k v)\bigr)\notag\\
&\hspace{4.48cm}+\bigl(\tilde{b}_k g_{k-1}(u^{k-1}+\mu_{k-1}v)-b_k g_{k-1}(v^{k-1})\bigr)\notag\\
\label{eq:Rn123}
&=:R_{n1}+R_{n2}+R_{n3}
\end{align}
for each $n\in\N$, and consider $R_{n1},R_{n2},R_{n3}$ in turn. First, since $(x_1,\dotsc,x_k,y)\mapsto yg_k(x_k)$ belongs to $\PL_{k+1}(2)$ in view of Lemma~\ref{lem:pseudoprod}, it follows from the inductive hypothesis~\eqref{eq:AMPlowsym} and the definition of $\mu_{k+1}$ in~\eqref{eq:statevolsymat} that 
$\lambda\ipr{v}{g_k(v^k)}_n=n^{-1}\sum_{i=1}^n\lambda v_i\,g_k(v_i^k)\cvc\lambda\E\bigl(Vg_k(\mu_k V+\sigma_k G_k)\bigr)=\mu_{k+1}$. Together with~\eqref{eq:vtildenorm}, this implies that $\norm{R_{n1}}_n=\abs{\lambda\ipr{v}{g_k(v^k)}_n-\mu_{k+1}}\,\norm{v}_n\cvc 0$ as $n\to\infty$. 

Next, since $\norm{W}\equiv\norm{W}_{2\to 2}=O_c(1)$~\citep[e.g.][]{AGZ10,KY13} and $g_k$ is $L_k$-Lipschitz for some $L_k>0$, the inductive hypothesis~\eqref{eq:Deltatildenorm} ensures that
\[\norm{R_{n2}}_n\leq\norm{W}\,\norm{g_k(v^k)-g_k(u^k+\mu_k v)}_n\leq L_k\norm{W}\,\norm{v^k-(u^k+\mu_k v)}_n=L_k\norm{W}\,\norm{\tilde{\Delta}^k}_n\cvc 0\]
as $n\to\infty$. Similarly, $\norm{\tilde{\Delta}^{k-1}}_n\cvc 0$ by induction and $g_{k-1}$ is $L_{k-1}$-Lipschitz for some $L_{k-1}>0$, so as a first step towards controlling $\norm{R_{n3}}_n$, we have \[\norm{g_{k-1}(u^{k-1}+\mu_{k-1}v)-g_{k-1}(v^{k-1})}_n\leq L_{k-1}\norm{\tilde{\Delta}^{k-1}}_n\cvc 0.\]
Note also that since $(x_1,\dotsc,x_k,y)\mapsto g_{k-1}(x_{k-1})^2$ lies in $\PL_{k+1}(2)$ by Lemma~\ref{lem:pseudoprod}, it follows from~\eqref{eq:matsymdeltaconv} that \[\norm{g_{k-1}(u^{k-1}+\mu_{k-1}v)}_n^2\cvc\E\bigl(g_{k-1}(\mu_{k-1}V+\sigma_{k-1}G_{k-1})^2\bigr)=\sigma_k^2\]
as $n\to\infty$. Furthermore, since $g_k'$ satisfies~\ref{ass:S2}, we can apply the inductive hypothesis~\eqref{eq:AMPlowsym} and argue as in the proof of Proposition~\ref{prop:AMPproof}(f) to see that
$b_k\equiv b_k(n)=\abr{g_k'(v^k)}_n=n^{-1}\sum_{i=1}^n g_k'(v_i^k)\cvc\E\bigl(g_k'(\mu_kV+\sigma_k G_k)\bigr)$. Similar reasoning based on~\eqref{eq:matsymdeltaconv} yields $\tilde{b}_k\equiv\tilde{b}_k(n)=n^{-1}\sum_{i=1}^n g_k'(v_i^k+\mu_k v_i)\cvc\E\bigl(g_k'(\mu_kV+\sigma_k G_k)\bigr)$, so $\tilde{b}_k(n)-b_k(n)\cvc 0$ as $n\to\infty$.
Putting everything together, we conclude that \[\norm{R_{n3}}_n\leq\abs{\tilde{b}_k-b_k}\norm{g_{k-1}(u^{k-1}+\mu_{k-1}v)}_n+\abs{b_k}\norm{g_{k-1}(u^{k-1}+\mu_{k-1}v)-g_{k-1}(v^{k-1})}_n\cvc 0,\]
and hence that $\norm{\tilde{\Delta}^{k+1}}_n\leq\norm{R_{n1}}_n+\norm{R_{n2}}_n+\norm{R_{n3}}_n\cvc 0$ as $n\to\infty$. Combining this with~\eqref{eq:matsymdeltaconv}--\eqref{eq:vtildenorm} yields the desired conclusion~\eqref{eq:AMPlowsym}, so the inductive step is complete.
\end{proof}

\deparskip
\begin{proof}[Proof of Corollary~\ref{cor:AMPlowsym}]
For $\psi\in\PL_2(2)$, note that since $g_k\colon\R\to\R$ is Lipschitz by assumption, $(x_0,x_1,\dotsc,x_k,y)\mapsto\psi\bigl(g_k(x_k),y\bigr)$ is a $\PL_{k+2}(2)$ function to which we can apply~\eqref{eq:AMPlowsym}. This yields~\eqref{eq:asymdev}, which specialises to~\eqref{eq:asymmse} when $\psi=\psi_2\colon (x,y)\mapsto (x-y)^2$ is squared error loss. Finally, by considering the $\PL_2(2)$ functions $(x,y)\mapsto yg_k(x)$, $(x,y)\mapsto g_k(x)^2$ and $(x,y)\mapsto y^2$, we deduce from~\eqref{eq:asymdev} that as $n\to\infty$, we have \[\ipr{\hat{v}^k}{v}_n\cvc\E\bigl(Vg_k(\mu_k V+\sigma_k G_k)\bigr)=\mu_{k+1}\]
as in the paragraph after~\eqref{eq:Rn123} above,
\[\norm{\hat{v}^k}_n^2=\frac{1}{n}\sum_{i=1}^n g_k(v_i^k)^2\cvc\E\bigl(g_k(\mu_k V+\sigma_k G_k)^2\bigr)=\sigma_{k+1}^2\] 
and $\norm{v}_n^2\cvc\E(V^2)=1$ as in~\eqref{eq:vconv}. Combining these, we obtain~\eqref{eq:asymcorr}.
\end{proof}

\deparskip
\begin{proof}[Proof of Lemma~\ref{lem:tweedie}]
Fix $\mu\neq 0$ and $\sigma>0$, and let $V\sim\pi$ and $G\sim N(0,1)$ be independent. Then $\mu V+\sigma G$ has Lebesgue density $y\mapsto p(y):=\int_\R\phi_\sigma(y-\mu x)\,d\pi(x)>0$, where $\phi_\sigma\colon z\mapsto (\sqrt{2\pi}\sigma)^{-1}e^{-z^2/(2\sigma^2)}$ is the density of $\sigma G\sim N(0,\sigma^2)$. Moreover, since all the derivatives of $\phi_\sigma$ are bounded on $\R$, we can differentiate repeatedly under the integral sign to see that $p^{(j)}(y)=\int_\R\phi_\sigma^{(j)}(y-\mu x)\,d\pi(x)$ for all $y$ and $j\in\N_0$, so $p$ is a smooth function on $\R$. For each $y\in\R$, define $\pi_y$ to be the distribution on $\R$ with density (i.e.\ Radon--Nikodym derivative)
\begin{equation}
\label{eq:posterior}
\frac{d\pi_y}{d\pi}\colon x\mapsto\frac{\phi_\sigma(y-\mu x)}{\int_\R\phi_\sigma(y-\mu x')\,d\pi(x')}=\frac{\phi_\sigma(y-\mu x)}{p(y)}
\end{equation}
with respect to $\pi$. It is easily verified that $\pi_y$ is the ``conditional distribution of $V$ given $\mu V+\sigma G=y$'', formally in the sense of Remark~\ref{rem:condkern}(II). 
It follows from this and~\citet[Theorem~10.2.5]{Dud02} that taking $V_y\sim\pi_y$ and defining $g(y):=\E(V_y)=\int_\R x\,d\pi_y(x)$ for $y\in\R$, we have $\E(V\,|\,\mu V+\sigma G)=g(\mu V+\sigma G)$. For each $y$, note that
\begin{equation}
\label{eq:tweedie1}
g(y)=\frac{\int_\R x\,\phi_\sigma(y-\mu x)\,d\pi(x)}{\int_\R\phi_\sigma(y-\mu x)\,d\pi(x)}=\frac{\int_\R\mu^{-1}\bigl(y\phi_\sigma(y-\mu x)+\phi_\sigma'(y-\mu x)\bigr)\,d\pi(x)}{\int_\R\phi_\sigma(y-\mu x)\,d\pi(x)}=\frac{y+\sigma^2(\log p)'(y)}{\mu},
\end{equation}
so~\eqref{eq:tweedie} holds and $g$ is infinitely differentiable on $\R$, and by similar calculations,
\[g'(y)
=\frac{1+\sigma^2(\log p)''(y)}{\mu}=\frac{\mu}{\sigma^2}\frac{\sigma^2\bigl(1+\sigma^2(\log p)''(y)\bigr)}{\mu^2}=\frac{\mu}{\sigma^2}\Var(V_y)\geq 0.\]
We now consider in turn the two conditions on $\pi$ in the statement of the lemma.

\unparskip
\begin{enumerate}[label=(\roman*)]
\item If $V\sim\pi$ has a log-concave density, then the density $p$ of $\mu V+\sigma G$ is also log-concave~\citep{Pre80}, so $(\log p)''\leq 0$ on $\R$. Thus, $0\leq g'\leq\abs{\mu}^{-1}$ on $\R$, so $g$ is 
Lipschitz with constant $\abs{\mu}^{-1}$.
\item Suppose first that $\pi$ is supported on a compact interval $[a,b]$. Then for each $y\in\R$, the distribution $\pi_y$ has a density with respect to $\pi$ (by definition), so it is also supported on $[a,b]$. Thus, $\Var(V_y)\leq\E\bigl\{\bigl(V_y-(a+b)/2\bigr)^2\bigr\}\leq (b-a)^2/4$ for all $y$, whence $g$ is Lipschitz with constant $\abs{\mu}(b-a)^2/(4\sigma^2)$, and
\begin{equation}
\label{eq:tweedie2}
{-1}\leq\sigma^2(\log p)''(y)\leq\frac{\mu^2(b-a)^2}{4\sigma^2}-1.
\end{equation}
More generally, suppose that $\pi$ is the distribution of $U_0+V_0$, where $U_0\sim N(0,\sigma_0^2)$ with $\sigma_0\geq 0$, and $V_0\sim\pi_0$ is independent of $U_0$ and supported on some compact interval $[a,b]$. Then $p$ is the density of $\mu V+\sigma G\eqd\mu U_0+\sqrt{\sigma^2+\mu^2\sigma_0^2}\,G$, so it follows from~\eqref{eq:tweedie1} and~\eqref{eq:tweedie2} that
\[\frac{1}{\abs{\mu}}\biggl(1-\frac{\sigma^2}{\sigma^2+\mu^2\sigma_0^2}\biggr)\leq \abs{g'}\leq\frac{1}{\abs{\mu}}\,\biggl\{1+\frac{\sigma^2}{\sigma^2+\mu^2\sigma_0^2}\biggl(\frac{\mu^2(b-a)^2}{4(\sigma^2+\mu^2\sigma_0^2)}-1\biggr)\biggr\}\]
on $\R$. The expression on the right hand side is therefore a Lipschitz constant for $g$.
\end{enumerate}

\vspace{-1.2cm}
\end{proof}

\deparskip
\begin{proof}[Proof of Lemma~\ref{lem:snrmax}]
Let $\psi\colon\R^2\to [0,\infty)$ be any measurable loss function, and fix $s_1,s_2\in (0,\infty)$ with $s_1>s_2$. Taking $G'\sim N(0,s_1^2-s_2^2)$, $V\sim\pi$ and $G\sim N(0,1)$ to be jointly independent, we first claim that
\begin{equation}
\label{eq:snrmax}
R_{\pi,\psi}(s_2^{-2})=\inf_g\,\E\bigl\{\psi\bigl(g(V+s_2 G),V\bigr)\bigr\}=\inf_{\tilde{g}}\,\E\bigl\{\psi\bigl(\tilde{g}(V+s_2 G,G'),V\bigr)\bigr\},
\end{equation}
where the infima are taken over all measurable $g\colon\R\to\R$ and $\tilde{g}\colon\R^2\to\R$ respectively. Indeed, the first equality holds since $V+s_2 G=s_2(\sqrt{\rho}V+G)$ when $\rho=s_2^{-2}$, and the middle expression is clearly bounded below by the final one, so it remains to prove the reverse inequality. For any fixed $\tilde{g}\colon\R^2\to\R$, we have \[\tilde{\Psi}(a):=\E\bigl\{\psi\bigl(\tilde{g}(V+s_2 G,a),V\bigr)\bigr\}\geq\inf_g\E\bigl\{\psi\bigl(g(V+s_2 G),V\bigr)\bigr\}=R_{\pi,\psi}(s_2^{-2})\] 
for all $a\in\R$, so it follows from Lemma~\ref{lem:indeplem} that $\E\bigl\{\psi\bigl(\tilde{g}(V+s_2 G,G'),V\bigr)\bigr\}=\E\bigl(\tilde{\Psi}(G')\bigr)\geq R_{\pi,\psi}(s_2^{-2})$, and hence that~\eqref{eq:snrmax} holds. Since $(V,V+s_2 G+G')\eqd (V,V+s_1 G)$, we deduce that
\begin{align*}
R_{\pi,\psi}(s_2^{-2})=\inf_{\tilde{g}}\,\E\bigl\{\psi\bigl(\tilde{g}(V+s_2 G,G'),V\bigr)\bigr\}&\leq\inf_g\,\E\bigl\{\psi\bigl(g(V+s_2 G+G'),V\bigr)\bigr\}\\
&=\inf_g\,\E\bigl\{\psi\bigl(g(V+s_1 G),V\bigr)\bigr\}=R_{\pi,\psi}(s_1^{-2}).
\end{align*}
In addition, arguing as above for~\eqref{eq:snrmax}, we have
\[R_{\pi,\psi}(s_1^{-2})=\inf_g\,\E\bigl\{\psi\bigl(g(V+s_1 G),V\bigr)\bigr\}\leq\inf_{a\in\R}\,\E\bigl(\psi(a,V)\bigr)=\inf_g\,\E\bigl\{\psi\bigl(g(G),V\bigr)\bigr\}=R_{\pi,\psi}(0).\]
Thus, $\rho\mapsto R_{\pi,\psi}(\rho)$ is non-increasing on $[0,\infty)$.

Finally, fix $\rho\in (0,\infty)$ and for each $y\in\R$, let $V_y\sim\pi_y$ be a random variable whose density with respect to $\pi$ is given by~\eqref{eq:posterior} with $\mu=\sqrt{\rho}$ and $\sigma=1$, so that $\pi_y$ is the conditional (i.e.\ posterior) distribution of $V$ given $\sqrt{\rho}V+G=y$. It follows from~\citet[Theorem~3]{BP73} that if the posterior risk function \[r_y\colon a\mapsto\E\bigl(\psi(a,V_y)\bigr)\]
attains its infimum on $\R$ for Lebesgue almost every $y\in\R$, then there exists a measurable $g^*\equiv g_\rho^*\colon\R\to\R$ with $g^*(y)\in\argmin_{a\in\R}\E\bigl(\psi(a,V_y)\bigr)$ for Lebesgue almost every $y\in\R$, whence $R_{\pi,\psi}(\rho)=\E\bigl\{\psi\bigl(g^*(\sqrt{\rho}V+G),V\bigr)\bigr\}$. This is the case (for every $\rho$) if $\psi(x,y)=\Psi(x-y)$ for some convex function $\Psi$ with $\Psi(u)\to\infty$ as $\abs{u}\to\infty$, in which case $r_y\colon a\mapsto\E\bigl(\psi(V_y,a)\bigr)$ is convex with $r_y(a)\to\infty$ as $\abs{a}\to\infty$, for each $y\in\R$.
\end{proof}

\unparskip
\begin{corollary}
\label{cor:mmse}
Given independent random variables $V\sim\pi$ and $G\sim N(0,1)$, the function $\rho\mapsto\mmse(\rho):=\E\bigl\{\bigl(V-\E(V\,|\,\sqrt{\rho}V+G)\bigr)^2\bigr\}$ is non-increasing on $[0,\infty)$. Moreover, if $V$ satisfies one of the conditions of Lemma~\ref{lem:tweedie}, then $\rho\mapsto\mmse(\rho)$ is also continuous on $(0,\infty)$.
\end{corollary}

\deparskip
\begin{proof}[Proof of Corollary~\ref{cor:mmse}]
Recall that whenever $X,Y$ are random variables with $\E(X^2)<\infty$, it follows from an orthogonal decomposition of the type~\eqref{eq:ortho} that
$\E\bigl\{\bigl(X-\E(X\,|\,Y)\bigr)^2\bigr\}=\min_g\E\bigl\{\bigl(X-g(Y)\bigr)^2\bigr\}$, where the minimum is over all measurable functions $g\colon\R\to\R$. Thus,  by Lemma~\ref{lem:snrmax}, $\rho\mapsto\mmse(\rho)$ is non-increasing on $[0,\infty)$.

Now fix $s_1>s_2>0$, and as in the proof of Lemma~\ref{lem:snrmax}, let $G'\sim N(0,s_1^2-s_2^2)$, $V\sim\pi$ and $G\sim N(0,1)$ be jointly independent, so that $(V,V+s_2 G+G')\eqd (V,V+s_1 G)$. Then under the conditions of Lemma~\ref{lem:tweedie}, it follows from (i) and (ii) in its proof that there exists a Lipschitz $g_2^*\colon\R\to\R$ with $g_2^*(V+s_2 G)=\E(V\,|\,V+s_2 G)=\E(V\,|\,s_2^{-1}V+G)$ and Lipschitz constant $L_{s_2}\leq C_{\pi}(1\vee s_2^{-2})$, where $C_{\pi}>0$ depends only on $\pi$. Thus,
\begin{align}
\mmse(s_2^{-2})\leq\mmse(s_1^{-2})&=\min_g\,\E\bigl\{\bigl(V-g(V+s_1 G)\bigr)^2\bigr\}\notag\\
&\leq\E\bigl\{\bigl(V-g_2^*(V+s_1 G)\bigr)^2\bigr\}\notag\\
&=\E\bigl\{\bigl(V-g_2^*(V+s_2 G+G')\bigr)^2\bigr\}\notag\\
&=\E\bigl\{\bigl(V-g_2^*(V+s_2 G)\bigr)^2\bigr\}+\E\bigl\{\bigl(g_2^*(V+s_2 G)-g_2^*(V+s_2 G+G')\bigr)^2\bigr\}\notag\\
\label{eq:mmse}
&\leq\mmse(s_2^{-2})+L_{s_2}\,\E(\abs{G'}^2)\notag\\
&\leq\mmse(s_2^{-2})+C_{\pi}\,(s_1^2s_2^2\vee s_1^2)(s_2^{-2}-s_1^{-2}).
\end{align}
To justify the equality in the third-last line, note that $g_2^*(V+s_2 G)=\E(V\,|\,V+s_2 G,G')$ by the independence of $G'$ and $(V,G)$, so for any measurable $\tilde{g}\colon\R^2\to\R$ with $\E\bigl(\tilde{g}(V+s_2 G,G')^2\bigr)<\infty$, we have $\E\bigl\{\bigl(V-g_2^*(V+s_2 G)\bigr)\,\tilde{g}(V+s_2 G,G')\bigr\}=0$. We deduce from~\eqref{eq:mmse} that $\rho\mapsto\mmse(\rho)$ is Lipschitz on $(\rho',\infty)$ for every $\rho'>0$, and hence that it is continuous on $(0,\infty)$.
\end{proof}

\deparskip
\begin{proof}[Proof of Corollary~\ref{cor:snrbayes}]
Given any sequence of functions $(g_k)$ for which the corresponding AMP iterations~\eqref{eq:AMPmatsym} satisfy the hypotheses of Theorem~\ref{thm:AMPlowsym} or~\ref{thm:spectral}, we prove~\eqref{eq:snrbayes} by induction on $k\in\N_0$. For each such $k$, it follows from~\eqref{eq:asymcorr} and~\eqref{eq:snrCS} that as $n\to\infty$, we have
\[\frac{\abs{\ipr{\hat{v}^k}{v}_n}}{\norm{\hat{v}^k}_n\norm{v}_n}\cvc\frac{\sqrt{\rho_{k+1}}}{\lambda}=\frac{\bigl|\E\bigl(Vg_k(\mu_k V+\sigma_k G_k)\bigr)\bigr|}{\sqrt{\E\bigl(g_k(\mu_k V+\sigma_k G_k)^2\bigr)}}\leq\sqrt{\E\bigl(\E(V\,|\,\mu_k V+\sigma_k G_k)^2\bigr)}=\sqrt{1-\mmse_k(\rho_k)},\]
where we set $\rho_0\equiv\rho_0^*=(\mu_0/\sigma_0)^2$ and write $G_0$ for the random variable $U$ from~\ref{ass:S1} when $k=0$. Now $\sqrt{1-\mmse_0(\rho_0^*)}=\sqrt{\rho_1^*}/\lambda$ by~\eqref{eq:gammak}, so $\rho_1\leq\rho_1^*$ and~\eqref{eq:snrbayes} holds when $k=0$. For a general $k\in\N$, we have $\rho_k\leq\rho_k^*$ by induction, so since $\rho\mapsto\mmse(\rho)\equiv\mmse_k(\rho)$ is non-increasing by Corollary~\ref{cor:mmse}, we deduce that $\sqrt{1-\mmse(\rho_k)}\leq\sqrt{1-\mmse(\rho_k^*)}=\sqrt{\rho_{k+1}^*}/\lambda$ and hence that $\rho_{k+1}\leq\rho_{k+1}^*$. This completes the inductive step for~\eqref{eq:snrbayes}.

As for~\eqref{eq:bayesdev}, we can apply~\eqref{eq:asymdev}, the definition of $R_{\pi,\psi}(\rho)$, the fact that $\rho_k\leq\rho_k^*$ and Lemma~\ref{lem:snrmax} (in that order) to conclude that $n^{-1}\sum_{i=1}^n\psi(\hat{v}_i^k,v_i)\cvc\E\bigl\{\psi\bigl(g_k(\mu_k V+\sigma_k G_k),V\bigr)\bigr\}\geq R_{\psi,\pi}(\rho_k)\geq R_{\psi,\pi}(\rho_k^*)$, as required.
\end{proof}

\deparskip
\begin{proof}[Proof of Theorem~\ref{thm:bayesamp}]
Under the conditions of Lemma~\ref{lem:tweedie}, each $g_k^*\colon\R\to\R$ is Lipschitz and satisfies~\ref{ass:S2}, and by Corollary~\ref{cor:mmse}, $\rho\mapsto\lambda^2\bigl(1-\mmse(\rho)\bigr)=:m_\lambda(\rho)$ is non-decreasing on $[0,\infty)$ and continuous on $(0,\infty)$. 

(a) We will show that if either (i) or (ii) holds, then
\begin{equation}
\label{eq:snrfixed}
\begin{alignedat}{2}
&\rho_\AMP^*\equiv\rho_\AMP^*(\lambda):=\inf\{\rho>0:\rho=m_\lambda(\rho)\}>0,\qquad&&\rho_0^*\in [0,\rho_\AMP^*],\qquad\rho_1^*>0,\\
&m_\lambda(\rho)\geq\rho\;\text{for all }\rho\in [0,\rho_\AMP^*],
\qquad&&\rho_{k+1}^*=m_\lambda(\rho_k^*)\;\text{for all }k\in\N_0.
\end{alignedat}
\end{equation}
Note that since $m_\lambda(\rho)\leq\lambda^2$ for all $\rho\in [0,\infty)$, we always have $\rho_\AMP^*(\lambda)\leq\lambda^2$.

\unparskip
\begin{enumerate}[label=(\roman*)]
\item \emph{Non-spectral initialisation}: In this case,~\ref{ass:S1} holds with $\mu_0=0$, $\sigma_0=1$, $\rho_0^*=0$ and $U=1$. Since $\E(V)\neq 0$ and $\E(V^2)=1$ by~\ref{ass:S1}, $\rho_1^*=\lambda^2\bigl(1-\mmse_k(0)\bigr)=\lambda^2\bigl(1-\Var(V)\bigr)=\lambda^2\,\E(V)^2=m_\lambda(0)>0$, and therefore $\rho_{k+1}^*=m_\lambda(\rho_k^*)$ for all $k\in\N_0$. In addition, $m_\lambda(\rho)\geq m_\lambda(0)=\rho_1^*>\rho$ for all $\rho\in [0,\rho_1^*)$, so 
$\rho_\AMP^*\geq\rho_1^*>0$ and $m_\lambda(\rho)>\rho$ for all $\rho\in [0,\rho_\AMP^*)$.
\item \emph{Spectral initialisation}: By Proposition~\ref{prop:power},~\ref{ass:S1} holds with $U\equiv G_0\sim N(0,1)$, so $\rho_{k+1}^*=m_\lambda(\rho_k^*)$ for all $k\in\N_0$. Since $\mmse(\rho)$ is the minimum value of $\E\bigl\{\bigl(V-g(\sqrt{\rho}V+G)\bigr)^2\bigr\}$ as $g$ ranges over all measurable functions,
\[m_\lambda(\rho)\geq\lambda^2\Bigl(1-\inf_{a,b\in\R}\E\bigl\{\bigl(V-a(\sqrt{\rho}V+G)-b\bigr)^2\bigr\}\Bigr)=\lambda^2\,\biggl(1-\frac{1-\E(V)^2}{1+\rho\,\{1-\E(V)^2\}}\biggr)\geq\frac{\lambda^2\rho}{1+\rho}\]
for all $\rho\in [0,\infty)$, where the equality above can be verified by a routine calculation. Recalling that $\lambda>1$, we have $m_\lambda(\rho)\geq\lambda^2\rho/(1+\rho)>\rho$ for all $\rho\in [0,\lambda^2-1)$, so $\rho_\AMP^*\geq\lambda^2-1=\rho_0^*>0$ and $m_\lambda(\rho)>\rho$ for all $\rho\in (0,\rho_\AMP^*)$. 
Moreover, $\rho_1^*=m_\lambda(\rho_0^*)\geq\rho_0^*>0$.
\end{enumerate}

\unparskip
To complete the proof of (a), note that if $0\leq\rho_k^*\leq\rho_\AMP^*$ for some $k\in\N_0$, then by~\eqref{eq:snrfixed} and the fact that $\rho\mapsto m_\lambda(\rho)$ is non-decreasing, we have $0\leq\rho_k^*\leq m_\lambda(\rho_k^*)=\rho_{k+1}^*\leq m_\lambda(\rho_\AMP^*)=\rho_\AMP^*$. Since $\rho_1^*>0$, it follows by induction that $(\rho_k^*)$ is an increasing sequence that converges to some $\rho^*\in (0,\rho_\AMP^*]$. By the continuity of $m_\lambda$ on $(0,\infty)$, we conclude that $\rho^*=\lim_{k\to\infty}\rho_{k+1}^*=\lim_{k\to\infty}m_\lambda(\rho_k^*)=m_\lambda(\rho^*)$ and hence that $\rho^*=\rho_\AMP^*$.

(b) Since $g_{k,\psi}^*$ is Lipschitz by assumption, this follows directly from Theorem~\ref{thm:AMPlowsym}; see the proof of Corollary~\ref{cor:AMPlowsym}.

(c) Since each $g_k^*$ is Lipschitz and $\rho_k^*\nearrow\rho_\AMP^*$ by (a), this is an immediate consequence of~\eqref{eq:asymmse} and~\eqref{eq:asymcorr} from Corollary~\ref{cor:AMPlowsym}.
\end{proof}

\umparskip
\subsection{Proofs for Section~\ref{Sec:GAMP}}
\label{sec:GAMPproofs}
The proof of Lemma~\ref{lem:GAMPstein} makes use of the following multivariate version of Stein's lemma.
\begin{lemma}
\label{lem:steinsigma}
Let $g\colon\R^d\to\R$ be such that for $j=1,\dotsc,d$, the function $x_j\mapsto g(x_1,\dotsc,x_d)$ is absolutely continuous for Lebesgue almost every $(x_i:i\neq j)\in\R^{d-1}$, with weak derivative $D_jg\colon\R^d\to\R$ satisfying $\E\bigl(\abs{D_jg(X)}\bigr)<\infty$. Let $\nabla g(x):=(D_1(x),\dotsc,D_d(x))$ for $x\in\R^d$. If $X\sim N_d(0,\Sigma)$ with $\Sigma$ positive definite, then \[\E\bigl(Xg(X)\bigr)=\Sigma\,\E\bigl(\nabla g(X)\bigr).\]
\end{lemma}

\deparskip
\begin{proof}
The result for $\Sigma=I_d$ is stated as~\citet[Lemma~3.6]{Tsy09}. For a general non-negative definite $\Sigma\in\R^{d\times d}$, let $\tilde{g}(z)=g(\Sigma^{1/2}z)$ for $z\in\R^d$. Then $\nabla\tilde{g}(z)=\Sigma^{1/2}\,\nabla g(\Sigma^{1/2}z)$ for all $z$ (\citealp[Theorem~2.1]{FSW18};~\citealp[Proposition~E.5]{Fan20}), so by taking $Z\sim N_d(0,I_d)$, we conclude that
\[\E\bigl(Xg(X)\bigr)=\Sigma^{1/2}\,\E\bigl(Z\tilde{g}(Z)\bigr)=\Sigma^{1/2}\,\E\bigl(\nabla \tilde{g}(Z)\bigr)=\Sigma\,\E\bigl(\nabla g(X)\bigr).\] 

\vspace{-1.2cm}
\end{proof}

\deparskip
\begin{proof}[Proof of Lemma~\ref{lem:GAMPstein}]
The first assertion follows from a general fact about Gaussian random vectors $(X_1,X_2)$, for which we write $\Sigma_{ij}:=\Cov(X_i,X_j)$ for $i,j\in\{1,2\}$: if $\E(X_1)=\E(X_2)=0$ and $\Sigma_{11}=\Var(X_1)$ is invertible, then $AX_1$ and $X_2-AX_1$ are uncorrelated and hence independent when $A:=\Sigma_{21}\Sigma_{11}^{-1}$. Thus, 
$(X_1,X_2)\eqd (X_1,AX_1+G)$, where $AX_1=\E(X_2\,|\,X_1)$ and $G\sim N(0,\Sigma_{22}-\Sigma_{21}\Sigma_{11}^{-1}\Sigma_{12})$ is independent of $X_1$. For~\eqref{eq:zkmusigma}, we deduce from Lemma~\ref{lem:steinsigma} that
\[\E\bigl(Z\tilde{g}_k(Z,Z_k,v)\bigr)=\Sigma_{11}\,\E\bigl(D_1\tilde{g}_k(Z,Z_k,v)\bigr)+\Sigma_{12}\,\E\bigl(D_2\tilde{g}_k(Z,Z_k,v)\bigr)\]
for all $v\in\R$, where $\Sigma\equiv\Sigma_k$ is as in~\eqref{eq:GAMPstatevol2}. Thus, since $(Z,Z_k)\eqd (Z,\mu_{Z,k}Z+\sigma_{Z,k}\tilde{G}_k)$ is independent of $\bar{\varepsilon}$ in~\eqref{eq:GAMPstatevol1}, and $\mu_{Z,k}=\Sigma_{21}/\Sigma_{11}$ by the first part of the lemma,
\begin{align}
\label{eq:GAMPstein1}
\E\bigl(Zg_k(Z_k,Y)\bigr)=\E\bigl(Z\tilde{g}_k(Z,Z_k,\bar{\varepsilon})\bigr)&=\Sigma_{11}\,\E\bigl(D_1\tilde{g}_k(Z,Z_k,\bar{\varepsilon})\bigr)+\Sigma_{12}\,\E\bigl(D_2\tilde{g}_k(Z,Z_k,\bar{\varepsilon})\bigr)\\
&=\frac{\E(\bar{\beta}^2)}{\delta}\bigl\{\mu_{k+1}+\mu_{Z,k}\,\E\bigl(g_k'(Z_k,Y)\bigr)\bigr\},\notag
\end{align}
which yields the first equality. Next, by the tower property of expectation, the final expression in~\eqref{eq:zkmusigma} can be written as
\[\E\biggl(\frac{\E(Z\,|\,Z_k,Y)-\E(Z\,|\,Z_k)}{\Var(Z\,|\,Z_k)}\,g_k(Z_k,Y)\biggr)=\E\biggl(\frac{Z-\E(Z\,|\,Z_k)}{\Var(Z\,|\,Z_k)}\,\tilde{g}_k(Z,Z_k,\bar{\varepsilon})\biggr).\]
Since $Z$ is conditionally Gaussian given $Z_k$ and $\bar{\varepsilon}$ is independent of $(Z,Z_k)$, a further (conditional) application of Stein's lemma yields
\[\E\biggl(\frac{Z-\E(Z\,|\,Z_k)}{\Var(Z\,|\,Z_k)}\,\tilde{g}_k(Z,Z_k,\bar{\varepsilon})\biggm|Z_k\biggr)=\frac{\E\bigl\{\bigl(Z-\E(Z\,|\,Z_k)\bigr)\,\tilde{g}_k(Z,Z_k,\bar{\varepsilon})\!\bigm|\!Z_k\bigr\}}{\Var(Z\,|\,Z_k)}=\E\bigl(D_1\tilde{g}_k(Z,Z_k,\bar{\varepsilon})\!\bigm|\!Z_k\bigr),\]
so by taking expectations, we obtain the second identity for $\mu_{k+1}$.
\end{proof}

\deparskip
\begin{proof}[Proof of Proposition~\ref{prop:GAMPopt}]
Consider the right hand side of~\eqref{eq:Lcond1} and write $\bar{J}(\tilde{\beta}):=\sum_{j=1}^p J(\tilde{\beta}_j)$ for $\tilde{\beta}\in\R^p$. Using the expression for the Lagrangian~\eqref{eq:lagrangian}, and ignoring terms that do not depend on $\tilde{\beta}$, we obtain
\begin{equation}
\begin{split}
\argmin_{\tilde{\beta}\in\R^p}\,\Bigl\{L(\tilde{\beta},\hat{\theta}^k, \hat{s}^k)-\frac{\bar{c}_k}{2}\norm{\tilde{\beta}-\hat{\beta}^k}^2\Bigr\}&=
\argmin_{\tilde{\beta}\in\R^p}\,\Bigl\{\bar{J}(\tilde{\beta})-\tilde{\beta}^\top X^\top\hat{s}^k-\frac{\bar{c}_k}{2}\norm{\tilde{\beta}-\hat{\beta}^k}^2\Bigr\}\\
&=\argmin_{\tilde{\beta}\in\R^p}\,\biggl\{\bar{J}(\tilde{\beta})-\frac{\bar{c}_k}{2}\biggl\|\tilde{\beta}+\frac{X^\top\hat{s}^k-\bar{c}_k\hat{\beta}^k}{\bar{c}_k}\biggr\|^2\biggr\}\\
&=\argmin_{\tilde{\beta}\in\R^p}\,\biggl\{\bar{J}(\tilde{\beta})-\frac{\bar{c}_k}{2}\biggl\|\tilde{\beta}+\frac{\beta^{k+1}}{\bar{c}_k}\biggr\|^2\biggr\}=\hat{\beta}^{k+1},
\end{split}
\end{equation}
where the third and final equalities follow from the definitions of $\beta^{k+1}$ and $\hat{\beta}^{k+1}=f_{k+1}(\beta^{k+1})$ respectively in~\eqref{eq:GAMPopt}, with $f_{k+1}$ as in~\eqref{eq:fkbar}. Similarly, in view of the definition of $\bar{g}_k$ in~\eqref{eq:gkbar}, we can obtain~\eqref{eq:Lcond2} by completing the square. For~\eqref{eq:Lcond3}, we can apply~\eqref{eq:GAMPopt} to see that
\begin{equation}
\hat{s}^{k+1}=\frac{\hat{\theta}^{k+1}-\theta^{k+1}}{\bar{b}_{k+1}}=\frac{\hat{\theta}^{k+1}-(X\hat{\beta}^{k+1}-\bar{b}_{k+1}\hat{s}^k)}{\bar{b}_{k+1}}=\hat{s}^k+\frac{(\hat{\theta}^{k+1}-X\hat{\beta}^{k+1})}{\bar{b}_{k+1}}.
\end{equation}
For the final assertion of Proposition~\ref{prop:GAMPopt}, if $(\beta^*,\theta^*,\hat{\beta}^*,\hat{\theta}^*,\hat{s}^*)$ is a fixed point of the algorithm~\eqref{eq:GAMPopt}, then $\hat{\theta}^*=X\hat{\beta}^*$ by~\eqref{eq:Lcond3}, and (for example by considering subgradients) it follows from~\eqref{eq:Lcond1} and~\eqref{eq:Lcond2} respectively that
\begin{align*}
\hat{\beta}^*&=\argmin_{\tilde{\beta}\in\R^p}\,L(\tilde{\beta},\hat{\theta}^*,\hat{s}^*)=\argmin_{\tilde{\beta}\in\R^p}\,\bigl\{\bar{J}(\tilde{\beta})-(X\tilde{\beta})^\top\hat{s}^*\bigr\},\\
\hat{\theta}^*&=\argmin_{\tilde{\theta}\in\R^n}\,L(\hat{\beta}^*,\tilde{\theta},\hat{s}^*)=\argmin_{\tilde{\theta}\in\R^n}\,\bigl\{\bar{\ell}(\tilde{\theta},y)+\tilde{\theta}^\top\hat{s}^*\bigr\},
\end{align*}
where $\bar{\ell}(\tilde{\theta},y):=\sum_{i=1}^n\ell(\tilde{\theta}_i,y_i)$. Thus, for all $(\tilde{\beta},\tilde{\theta})\in\R^p\times\R^n$ with $\tilde{\theta}=X\tilde{\beta}$, we have
\begin{align*}
\bar{J}(\tilde{\beta})+\bar{\ell}(\tilde{\theta},y)&=\bigl(\bar{J}(\tilde{\beta})-(X\tilde{\beta})^\top\hat{s}^*\bigr)+\bigl(\bar{\ell}(\tilde{\theta},y)+\tilde{\theta}^\top\hat{s}^*\bigr)\\
&\geq\bigl(\bar{J}(\hat{\beta}^*)-(X\hat{\beta}^*)^\top\hat{s}^*\bigr)+\bigl(\bar{\ell}(\hat{\theta}^*,y)+(\hat{\theta}^*)^\top\hat{s}^*\bigr)=\bar{J}(\hat{\beta}^*)+\bar{\ell}(\hat{\theta}^*,y),
\end{align*}
so $(\hat{\beta}^*,\hat{\theta}^*)$ is a solution to the optimisation problem~\eqref{eq:constropt}, as required.
\end{proof}

\section{Supplementary mathematical background}
\label{Sec:Supp}
\subsection{Basic properties of complete convergence}
\label{sec:cc}
\begin{proof}[Proof of Proposition~\ref{prop:compequiv}]
For (a), suppose that $\sum_n\Pr(\norm{X_n}_E>\varepsilon)<\infty$ for all $\varepsilon>0$. Then for any sequence $(Y_n)$ of $E$-valued random elements with $Y_n\eqd X_n$ for all $n$, the first Borel--Cantelli lemma implies that $\Pr(\norm{Y_n}_E>\varepsilon\text{ infinitely often})=0$ for all $\varepsilon>0$ 
and hence that $Y_n\to 0$ almost surely. This shows that $X_n\cvc 0$. Conversely, suppose that $\sum_n\Pr(\norm{X_n}_E>\varepsilon)=\infty$ for some $\varepsilon>0$. Then for a sequence $(Y_n)$ of independent $E$-valued random elements with $Y_n\eqd X_n$ for all $n$, the second Borel--Cantelli lemma implies that $\Pr(\norm{Y_n}_E>\varepsilon\text{ infinitely often})=1$ and hence that $Y_n\nrightarrow 0$ almost surely. Thus, $X_n\not\cvc 0$.

The argument for (b) is similar. If $\sum_n\Pr(\norm{X_n}_E>C)<\infty$ for all $C>0$, then $X_n=O_c(1)$ by the first Borel--Cantelli lemma. Conversely, suppose that $\sum_n\Pr(\norm{X_n}_E>C)=\infty$ for all $C>0$. Then for a sequence $(Y_n)$ of independent $E$-valued random elements with $Y_n\eqd X_n$ for all $n$, the second Borel--Cantelli lemma implies that $\Pr(\norm{Y_n}_E>C\text{ infinitely often})=1$ for all $C>0$
and hence that $\limsup_{n\to\infty}\norm{Y_n}_E=\infty$ almost surely. Thus, $(X_n)$ is not $O_c(1)$.
\end{proof}
\unparskip

\begin{remark}
\label{rem:compdist}
For a random sequence $(X_n)$ taking values in a Euclidean space $(E,\|\cdot\|_E)$, it can be seen from Definition~\ref{def:compconv} and Proposition~\ref{prop:compequiv} that complete convergence (to a degenerate limit) is a property of the marginal distributions of the random elements $X_1,X_2,\dotsc$ and not of their joint dependence structure (i.e.\ the specific coupling between them), so $X_1,X_2,\dotsc$ need not be defined on the same probability space. Thus, just as for weak convergence or convergence in probability to a degenerate limit (but not almost sure convergence), there is a meaningful notion of complete convergence 
for sequences $(\mu_n)$ of Borel probability measures on $E$: defining $\bar{B}(x,\varepsilon):=\{x'\in E:\norm{x'-x}_E\leq\varepsilon\}$ for $\varepsilon>0$, we write $\mu_n\cvc\delta_x$ if $\sum_n\mu_n\bigl(\bar{B}(x,\varepsilon)^c\bigr)<\infty$ for all $\varepsilon>0$.
\end{remark}

\unparskip
\begin{example}
\label{ex:exptails}
Let $(X_n)$ be any sequence of random variables for which there exist $c_1,c_2,\beta>0$ such that $\Pr(\abs{X_n}>t)\leq c_1\exp(-c_2 t^\beta)$ for all $t>0$ and $n\in\N$. Let $(a_n)$ be a deterministic sequence of real numbers. If $a_n=o(1)$, then clearly $a_nX_n=o_p(1)$, and if $a_n=O(1)$, then $a_nX_n=O_p(1)$. Moreover:

\unparskip
\begin{enumerate}[label=(\alph*)]
\item If $\abs{a_n}^\beta\log n\to 0$, then for every $t>0$, there exists $N\in\N$ such that $c_2(t/\abs{a_n})^\beta\geq 2\log n$ for all $n>N$, so $\sum_n\Pr(\abs{a_nX_n}>t)\leq N+\sum_{n>N} c_1 e^{-2\log n}<\infty$. Thus, $a_nX_n=o_c(1)$ by Proposition~\ref{prop:compequiv}(a).
\item If $\limsup_{n\to\infty}\,\abs{a_n}^\beta\log n<\infty$, then there exists $t>0$ and $N\in\N$ such that $c_2(t/\abs{a_n})^\beta\geq 2\log n$ for all $n>N$, so $\sum_n\Pr(\abs{a_nX_n}>t)<\infty$ as in (i). Thus, $a_nX_n=O_c(1)$ by Proposition~\ref{prop:compequiv}(b).
\end{enumerate}

\unparskip
Suppose in addition that there exist $c_1',c_2',\beta'>0$ such that $\Pr(\abs{X_n}>t)\geq c_1'\exp(-c_2' t^{\beta'})$ for all $t>0$ and $n\in\N$.

\unparskip
\begin{enumerate}[resume,label=(\alph*)]
\item If $\liminf_{n\to\infty}\,\abs{a_n}^{\beta'}\log n>0$, then there exist $t>0$ and $N\in\N$ such that $c_2'(t/\abs{a_n})^{\beta'}\leq\log n$ for all $n>N$, so $\sum_n\Pr(\abs{a_nX_n}>t)\geq\sum_{n\geq N} c_1 e^{-\log n}=\infty$. Thus, $(a_nX_n)$ is not $o_c(1)$ in view of Proposition~\ref{prop:compequiv}(a).
\item If $\abs{a_n}^{\beta'}\log n\to\infty$, then for every $t>0$, there exists $N\in\N$ such that $c_2'(t/\abs{a_n})^{\beta'}\leq\log n$ for all $n>N$, so $\sum_n\Pr(\abs{a_nX_n}>t)=\infty$ as in (iii). Thus, $(a_nX_n)$ is not $O_c(1)$ in view of Proposition~\ref{prop:compequiv}(b).
\end{enumerate}

\unparskip
For instance, suppose that $X_n=X\sim N(0,1)$ for all $n$. Then $X_n\to X$ almost surely and $X_n=O_p(1)$ but $(X_n)$ is not $O_c(1)$, and $X_n/\log^{1/2}n\to 0$ almost surely but $(X_n/\log^{1/2}n)$ is not $o_c(1)$.  
\end{example}

\unparskip
Using Proposition~\ref{prop:compequiv}, it is straightforward to verify that the continuous mapping theorem and Slutsky's lemma remain valid when stated in terms of complete convergence. 
\begin{lemma}
\label{lem:slutsky}
Let $(X_n),(Y_n)$ be sequences of random elements taking values in Euclidean spaces $E,E'$ respectively such that $X_n\cvc x$ and $Y_n\cvc y$ for some deterministic limits $x\in E$ and $y\in E'$. Then $(X_n,Y_n)\cvc (x,y)$ in $E\times E'$
and $g(X_n)\cvc g(x)$ in $E'$ for any function $g\colon E\to E'$ that is continuous at $x$. If in addition $x$ lies in some open set $U\subseteq E$, then $\Ind_{\{X_n\notin U\}}\cvc 0$.

Consequently, $X_n+Y_n\cvc x+y$ when $E=E'$. Moreover, $X_nY_n\cvc xy$ when $E=\R$ and $E'$ is any Euclidean space (in the case of scalar multiplication), or when $E=\R^{k\times\ell}$ and $E'=\R^\ell$ for some $k,\ell\in\N$ (in the case of matrix multiplication). If in addition $k=\ell$ and $x^{-1}\in E$ is well-defined, then $\Ind_{\{X_n\text{ is not invertible}\}}\cvc 0$ and $X_n^+ Y_n\cvc\inv{x}y$.
\end{lemma}

\deparskip
\begin{proof}
For the first part of the lemma, we apply Proposition~\ref{prop:compequiv}(a). Since $\norm{(\tilde{x},\tilde{y})}_{E\times E'}^2=\norm{\tilde{x}}_E^2+\norm{\tilde{y}}_{E'}^2$ for any $(\tilde{x},\tilde{y})\in E\times E'$, we have \[\sum_n\Pr\bigl(\norm{(X_n,Y_n)-(x,y)}_{E\times E'}>\varepsilon\bigr)\leq\sum_n\,\bigl\{\Pr(\norm{X_n-x}_E>\varepsilon/\sqrt{2})+\Pr(\norm{Y_n-y}_{E'}>\varepsilon/\sqrt{2})\bigr\}<\infty\]
for all $\varepsilon>0$, so $(X_n,Y_n)\cvc (x,y)$ by Proposition~\ref{prop:compequiv}(a). If $g\colon E\to E'$ is continuous at $x\in E$, then for each $\varepsilon>0$, there exists $\delta>0$ such that $\norm{g(\tilde{x})-g(x)}_{E'}<\varepsilon$ whenever $\norm{\tilde{x}-x}_E<\delta$, so
\[\sum_n\Pr(\norm{g(X_n)-g(x)}_{E'}>\varepsilon)\leq\sum_n\Pr(\norm{X_n-x}_E>\delta)<\infty.\]
This holds for all $\varepsilon>0$, so $g(X_n)\cvc g(x)$ by Proposition~\ref{prop:compequiv}(a). When $x$ lies in some open set $U\subseteq E$, there exists $\varepsilon>0$ such that $\Pr(X_n\notin U)\leq\Pr(\norm{X_n-X}_E>\varepsilon)$, so $\Ind_{\{X_n\notin U\}}\cvc 0$, again by Proposition~\ref{prop:compequiv}(a).

Having established the first part of the lemma, we can now apply the facts above to deduce the remaining assertions. Indeed, when $E=E'$, the function $g\colon (\tilde{x},\tilde{y})\mapsto \tilde{x}+\tilde{y}$ is continuous on $E\times E'$ and we know that $(X_n,Y_n)\cvc (x,y)$, so it follows that $X_n+Y_n\cvc x+y$. When $E=\R$ or when $E=\R^{k\times\ell}$ and $E'=\R^\ell$ for some $k,\ell\in\N$, the scalar and matrix multiplication maps (respectively) are continuous on $E\times E'$. Therefore, it follows similarly that $X_nY_n\cvc xy$. 

If in addition $k=\ell$, then $\tilde{x}^+=\tilde{x}^{-1}$ for all invertible $\tilde{x}\in E=\R^{k\times k}$, so the map $\tilde{x}\mapsto\tilde{x}^+$ is continuous on the set $U$ of all invertible $\tilde{x}\in E$, which is open. A further application of the continuous mapping result above shows that if $x$ is invertible, then $\Ind_{\{X_n\text{ is not invertible}\}}\cvc 0$ and $X_n^+ Y_n\cvc\inv{x}y$, as required.
\end{proof} 

\unparskip

\begin{remark}
\label{rem:arithmetic}
By a similar application of Proposition~\ref{prop:compequiv}, it can be shown that the stochastic $o_c$ and $O_c$ symbols obey the `arithmetic rules' of standard $O$ notation. Written in compact form, some examples of these are as follows (for sequences defined on spaces with compatible dimensions):
\begin{equation} 
\label{eq:arithmetic}
\begin{alignedat}{3}
O_c(1)+O_c(1)&=O_c(1),\qquad\quad& o_c(1)+o_c(1)&=o_c(1),\qquad\quad& O_c(1)+o_c(1)&=O_c(1),\\
O_c(1)\,O_c(1)&=O_c(1),\qquad\quad& o_c(1)\,o_c(1)&=o_c(1),\qquad\quad& O_c(1)\,o_c(1)&=O_c(1),
\end{alignedat}
\end{equation}
where the assertions in the second line apply to scalar multiplication or matrix multiplication as appropriate. (The proofs are straightforward and are therefore omitted.) To give another example of a basic fact that follows directly from Definition~\ref{def:compconv} or Proposition~\ref{prop:compequiv}, let $E,E'$ be Euclidean spaces and suppose that $g\colon E\to E'$ is bounded on every bounded subset of $E$. Then for any sequence $(X_n)$ of $E$-valued random elements such that $X_n=O_c(1)$, we also have $g(X_n)=O_c(1)$.
\end{remark}

\umparskip
\subsection{Regular conditional distributions and conditional independence}
\label{sec:conddists}
First, we recall the notion of conditional expectation: if $(\Omega,\mathcal{F},\Pr)$ is a probability space and $\mathcal{G}\subseteq\mathcal{F}$ is a sub-$\sigma$-algebra, we write $\restr{\Pr}{\mathcal{G}}$ for the restricted probability measure on $(\Omega,\mathcal{G})$ given by $\restr{\Pr}{\mathcal{G}}(B):=\mathbb{P}(B)$ for $B\in\mathcal{G}$.  If $Y\colon(\Omega,\mathcal{F},\Pr)\rightarrow\R$ is a random variable with $\mathbb{E}(|Y|) < \infty$, then there exists a $\mathcal{G}$-measurable random variable $Z = \mathbb{E}(Y\,|\,\mathcal{G})$ with the property that $\mathbb{E}(Z\mathbbm{1}_E) = \mathbb{E}(Y\mathbbm{1}_E)$ for all $E \in \mathcal{G}$ \citep[][Chapter~10.1]{Dud02}.  We call $Z$ the \emph{conditional expectation of $Y$ given $\mathcal{G}$}, noting that it is unique up to $\restr{\Pr}{\mathcal{G}}$-almost sure equivalence.  For $F \in \mathcal{F}$, we also write $\mathbb{P}(F\,|\,\mathcal{G}) := \mathbb{E}(\Ind_F\,|\,\mathcal{G})$.

If $X$ is a measurable function from $(\Omega,\mathcal{F},\Pr)$ to a measurable space $(\mathscr{X},\mathcal{A})$, we say that $P_{X|\mathcal{G}}\colon\Omega\times\mathcal{A}\to [0,1]$ is a \emph{(regular) conditional distribution for $X$ given $\mathcal{G}$} if

\unparskip 
\begin{enumerate}[label=(\roman*)]
\item for every
$\omega\in\Omega$, the set function $P_{X|\mathcal{G}}(\omega,\cdot)$ is a probability measure on $\mathcal{A}$;
\item for each $A\in\mathcal{A}$, the map $P_{X|\mathcal{G}}(\cdot,A)$ is $\mathcal{G}$-measurable, and $P_{X|\mathcal{G}}(\omega,A)=\Pr(X^{-1}(A)\,|\,\mathcal{G})(\omega)$ for $\restr{\Pr}{\mathcal{G}}$-almost every $\omega\in\Omega$, so that $\Pr(X^{-1}(A)\cap E)=\int_E P_{X|\mathcal{G}}(\omega,A)\,d\Pr(\omega)$ for all $E\in\mathcal{G}$.
\end{enumerate}

\unparskip
We say that $(\mathscr{X},\mathcal{A})$ is a \emph{Borel space} if there exist a Borel subset $S\subseteq [0,1]$ (equipped with the restriction $\mathcal{B}_S$ of the Borel $\sigma$-algebra on $[0,1]$ to $S$) and a bijection $f\colon (\mathscr{X},\mathcal{A})\to (S,\mathcal{B}_S)$ such that both $f$ and $f^{-1}$ are measurable. Examples of Borel spaces $(\mathscr{X},\mathcal{A})$ include \emph{Polish spaces} (i.e.\ separable, completely metrisable topological spaces)
$\mathscr{X}$ equipped with their Borel $\sigma$-algebras $\mathcal{A}$~\citep[Theorem A1.6]{Kal97}. 

Whenever $(\mathscr{X},\mathcal{A})$ is a Borel space, there exists a conditional distribution $P_{X|\mathcal{G}}$, and moreover, if $P'$ is another such conditional distribution, then $P'(\omega,\cdot) = P_{X|\mathcal{G}}(\omega,\cdot)$ for $\restr{\Pr}{\mathcal{G}}$-almost every $\omega\in\Omega$; see~\citet[Theorem~5.3]{Kal97} and~\citet[Theorem~10.2.2]{Dud02}. For brevity, we will write `$X$ has conditional distribution $P\equiv P_\omega$ given $\mathcal{G}$ on an event $\Omega_0 \in \mathcal{G}$' to mean that there exists a conditional distribution $P_{X|\mathcal{G}}$, and we can take $P_{X|\mathcal{G}}(\omega,\cdot)=P_\omega(\cdot)$ for $\restr{\Pr}{\mathcal{G}}$-almost every $\omega\in\Omega_0$.  When we omit the phrase `on an event $\Omega_0 \in \mathcal{G}$', we mean that the statement holds for $\restr{\Pr}{\mathcal{G}}$-almost every $\omega\in\Omega$. 

For measurable $X,X'\colon (\Omega,\mathcal{F},\Pr)\to (\mathscr{X},\mathcal{A})$, we say that $X,X'$ are \emph{identically distributed given $\mathcal{G}$}, and write $X\eqdcond{\mathcal{G}}X'$, if there exist conditional distributions $P_{X|\mathcal{G}}$ and $P_{X'|\mathcal{G}}'$ for $X,X'$ respectively, and $P_\omega(\cdot)\equiv P_{X|\mathcal{G}}(\omega,\cdot)=P_{X'|\mathcal{G}}'(\omega,\cdot)\equiv P_\omega'(\cdot)$ for $\restr{\Pr}{\mathcal{G}}$-almost every $\omega\in\Omega$.
\begin{remark}
\label{rem:indepcond}
For example, $X$ has distribution $Q$ on $(\mathscr{X},\mathcal{A})$ and is independent of $\mathcal{G}$ if and only if $X$ has conditional distribution $P_\omega=Q$ for all $\omega\in\Omega$.
\end{remark}
\begin{remark}
\label{rem:condkern}
Let $X\colon (\Omega,\mathcal{F},\Pr)\to(\mathscr{X},\mathcal{A})$ be as above and consider the important special case where $\mathcal{G}=\sigma(Y)$ for some measurable map $Y$ from $(\Omega,\mathcal{F},\Pr)$ to a measurable space $(\mathscr{Y},\mathcal{B})$. Denote by $P$ the joint distribution of $(X,Y)\colon (\Omega,\mathcal{F},\Pr)\to(\mathscr{X}\times\mathscr{Y},\mathcal{A}\otimes\mathcal{B})$ and by $P^Y$ the (marginal) distribution of $Y$ on $(\mathscr{Y},\mathcal{B})$. We note here that a random variable $Z\colon (\Omega,\mathcal{F},\Pr)\to\R$ is $\sigma(Y)$-measurable if and only if $Z=g\circ Y$ for some measurable function $g\colon (\mathscr{Y},\mathcal{B})\to\R$ (i.e.\ `$Z(\omega)$ depends on $\omega$ only through $Y(\omega)$'); see for example~\citet[Theorem~4.2.8]{Dud02}.
Using this fact and the defining property (ii) above, it can be verified~\citep[as in][Theorem~10.2.1]{Dud02} that there exists a regular conditional distribution $P_{X|\sigma(Y)}\colon\Omega\times\mathcal{A}\to [0,1]$ if and only if there is a family of probability distributions $(Q_y)_{y\in\mathscr{Y}}$ on $(\mathscr{X},\mathcal{A})$ such that the following hold for every $A\in\mathcal{A}$:

\unparskip
\begin{enumerate}[label=(\Roman*)]
\item $y\mapsto Q_y(A)$ is a measurable function from $(\mathscr{Y},\mathcal{B})$ to $\R$;
\item $P(A\times B)=\Pr\bigl(X^{-1}(A)\cap Y^{-1}(B)\bigr)=\int_B Q_y(A)\,dP^Y(y)$ for all $B\in\mathcal{B}$.
\end{enumerate}

\unparskip
In this case, for $\restr{\Pr}{\sigma(Y)}$-almost every $\omega\in\Omega$, we have $P_{X|\sigma(Y)}(\omega,A)=Q_{Y(\omega)}(A)$ for all $A\in\mathcal{A}$. Note that $(Q_y)_{y\in\mathscr{Y}}$ is only unique up to $P^Y$-almost sure equivalence, in the sense that if $(Q_y')_{y\in\mathscr{Y}}$ satisfies (I) and $Q_y=Q_y'$ for $P^Y$-almost every $y\in\mathscr{Y}$, then $(Q_y')_{y\in\mathscr{Y}}$ also satisfies (II). In view of (I), the map $(y,A)\mapsto Q_y(A)$ is said to be a \emph{probability kernel}. An interpretation of (II) is that it makes precise the notion of \emph{disintegrating} the joint distribution $P$ of $(X,Y)$ into the marginal distribution $P^Y$ of $Y$ and the distributions $(Q_y)_{y\in\mathscr{Y}}$, where (for $P^Y$-almost every $y\in\mathscr{Y}$) we can view $Q_y$ as the ``conditional distribution of $X$ given $Y=y$''. Indeed, by analogy with the construction of the usual product measure and Fubini's theorem~\citep[e.g.][Chapter~4.4]{Dud02}, it can be shown that if $\phi\colon (\mathscr{X}\times\mathscr{Y},\mathcal{A}\otimes\mathcal{B})\to\R$ is $P$-integrable (i.e.\ $\phi$ is measurable and $\E\bigl(\abs{\phi(X,Y)}\bigr)<\infty$), then

\unparskip
\begin{enumerate}[resume,label=(\Roman*)]
\item $x\mapsto \phi(x,y)$ is $\mathcal{A}$-measurable for all $y\in\mathcal{Y}$ and $Q_y$-integrable for $P^Y$-almost every $y\in\mathcal{Y}$; 
\item $y\mapsto\int_{\mathscr{X}}\phi(x,y)\,dQ_y(x)$ is $\mathcal{B}$-measurable and $P^Y$-integrable;
\item $\int_{\mathscr{X}\times\mathscr{Y}}\,\phi(x,y)\,dP(x,y)=\int_{\mathscr{Y}}\,\bigl(\int_{\mathscr{X}}\phi(x,y)\,dQ_y(x)\bigr)\,dP^Y(y)$.
\end{enumerate}

\unparskip
This generalisation of Fubini's theorem is sometimes known as the \emph{disintegration theorem}, and is derived from (II) using a monotone class argument; see~\citet[Theorem~10.2.1(II)]{Dud02} and~\citet[Theorem~5.4]{Kal97}.
\end{remark}
\begin{lemma}
\label{lem:conddist}
Let $(\mathscr{X},\mathcal{A}),(\mathscr{Y},\mathcal{B})$ be measurable spaces and let $(\mathscr{Z},\mathcal{C})$ be a Borel space. 
Let $\phi\colon (\mathscr{X}\times\mathscr{Y},\mathcal{A}\otimes\mathcal{B})\to (\mathscr{Z},\mathcal{C})$ be a measurable function and let $\mathcal{G}\subseteq\mathcal{F}$ be a $\sigma$-algebra.

\unparskip
\begin{enumerate}[label=(\alph*)]
\item If $E\in\mathcal{G}$ and $X_1,X_2\colon (\Omega,\mathcal{F},\Pr)\to(\mathscr{X},\mathcal{A})$ are measurable functions with conditional distributions $P\equiv P_\omega$ and $Q\equiv Q_\omega$ respectively given $\mathcal{G}$, then the measurable function $X\colon (\Omega,\mathcal{F},\Pr)\to(\mathscr{X},\mathcal{A})$ satisfying $X=X_1$ on $E$ and $X=X_2$ on $E^c$ has conditional distribution $R(\cdot)\equiv R_\omega(\cdot):=P_\omega(\cdot)\Ind_{\{\omega\in E\}}+Q_\omega(\cdot)\Ind_{\{\omega\in E^c\}}$ given $\mathcal{G}$.
\item For $D\in\mathcal{A}\otimes\mathcal{B}$ and $y\in\mathscr{Y}$, let $D^y:=\{x\in\mathscr{X}:(x,y)\in D\}=\iota_y^{-1}(D)$, where $\iota_y\colon\mathscr{X}\to\mathscr{X}\times\mathscr{Y}$ denotes the map $x\mapsto (x,y)$.  Fix $\Omega_0 \in \mathcal{G}$.  Suppose that $X\colon(\Omega,\mathcal{F},\Pr)\to(\mathscr{X},\mathcal{A})$ has conditional distribution $P\equiv P_\omega$ given $\mathcal{G}$ on $\Omega_0$, and that $Y\colon(\Omega,\mathcal{F},\Pr)\to(\mathscr{Y},\mathcal{B})$ is $\mathcal{G}$-measurable. If $Z\colon(\Omega,\mathcal{F},\mathbb{P}) \rightarrow (\mathscr{Z},\mathcal{C})$ is a measurable map that agrees with $\phi(X,Y)$ on $\Omega_0$, then $Z$ has conditional distribution $\tilde{P}\equiv\tilde{P}_\omega=P_\omega\circ(\phi\circ\iota_{Y(\omega)})^{-1}$ given $\mathcal{G}$ on $\Omega_0$, so that $\tilde{P}_\omega(C)=P_\omega\bigl(\phi^{-1}(C)^{Y(\omega)}\bigr)$ for all $C\in\mathcal{C}$ and $\omega\in\Omega_0$.
\item Suppose that $X,X'\colon (\Omega,\mathcal{F},\Pr)\to(\mathscr{X},\mathcal{A})$ are measurable functions satisfying $X\eqdcond{\mathcal{G}}X'$, and that $Y\colon(\Omega,\mathcal{F},\Pr)\to(\mathscr{Y},\mathcal{B})$ is $\mathcal{G}$-measurable. Then $\phi(X,Y)\eqdcond{\mathcal{G}}\phi(X',Y)$.
\end{enumerate}
\end{lemma}

\unparskip
The result in (b) has an intuitive interpretation. Suppose for simplicity that $\Omega_0=\Omega$, and fix $\omega\in\Omega$. Let $\mu:=P_\omega$ be taken from the conditional distribution of $X$ given $\mathcal{G}$, and assume that $Y$ is $\mathcal{G}$-measurable. To obtain the corresponding $\tilde{P}_\omega$ from the conditional distribution of $\phi(X,Y)$ given $\mathcal{G}$, Lemma~\ref{lem:conddist}(b) tells us that we can take $\tilde{P}_\omega$ to be the distribution of $\phi(U,y)$, where $U\sim\mu$ and $y:=Y(\omega)$. In essence, the reason for this is that since $Y$ is $\mathcal{G}$-measurable, we can think of $Y$ as being `fixed' once we have conditioned on $\mathcal{G}$. 

\unparskip
\begin{proof}
(a) The fact that $R_\omega(\cdot)$ is a probability measure on $\mathcal{A}$ for $\restr{\Pr}{\mathcal{G}}$-almost every $\omega\in\Omega$ follows immediately from the corresponding facts for $P_\omega(\cdot)$ and $Q_\omega(\cdot)$. For each $A \in \mathcal{A}$, the map $\omega \mapsto R_\omega(A)$ is a composition of $\mathcal{G}$-measurable functions (since $E \in \mathcal{G}$ by assumption), so is $\mathcal{G}$-measurable.

For $A \in \mathcal{A}$, let $\chi_A\colon\mathscr{X} \rightarrow \{0,1\}$ denote the indicator function of $A$.  Then
\[
\chi_A \circ X = (\chi_A \circ X_1)\mathbbm{1}_E + (\chi_A \circ X_2)\mathbbm{1}_{E^c}.
\]
Since $\mathbbm{1}_E$ and $\mathbbm{1}_{E^c}$ are $\mathcal{G}$-measurable, it follows that
\begin{align*}
\mathbb{P}\bigl(X^{-1}(A) \bigm| \mathcal{G}\bigr)(\omega) &= \mathbb{E}\bigl((\chi_A \circ X_1)\mathbbm{1}_E \bigm| \mathcal{G}\bigr)(\omega) + \mathbb{E}\bigl((\chi_A \circ X_2)\mathbbm{1}_{E^c} \bigm| \mathcal{G}\bigr)(\omega) \\
&= \mathbb{E}\bigl(\chi_A \circ X_1 \bigm| \mathcal{G}\bigr)(\omega)\mathbbm{1}_E(\omega) + \mathbb{E}\bigl((\chi_A \circ X_2) \bigm| \mathcal{G}\bigr)(\omega)\mathbbm{1}_{E^c}(\omega)\\
&= P_\omega(A)\mathbbm{1}_E(\omega) + Q_\omega(A)\mathbbm{1}_{E^c}(\omega) = R_\omega(A)
\end{align*}
for $\restr{\Pr}{\mathcal{G}}$-almost every $\omega\in\Omega$, as required.

(b) This can be deduced from~\citet[Theorem~5.4]{Kal97} and part (a) above, but we give a direct proof here for completeness. Note that for $\restr{\Pr}{\mathcal{G}}$-almost every $\omega\in\Omega_0$, the set function $\tilde{P}_\omega$ is the push-forward (image measure) of $P_\omega$ induced by the measurable map $x \mapsto \phi \circ \iota_{Y(\omega)}(x)$ from $(\mathscr{X},\mathcal{A})$ to $(\mathscr{Z},\mathcal{C})$; thus, $\tilde{P}_\omega$ is indeed a probability measure for $\restr{\Pr}{\mathcal{G}}$-almost every $\omega\in\Omega_0$.

Now let $\mathcal{D}$ denote the collection of all $D \in \mathcal{A}\otimes\mathcal{B}$ for which $\omega \mapsto P_\omega\bigl(D^{Y(\omega)}\bigr)$ is $\mathcal{G}$-measurable and $P_\omega\bigl(D^{Y(\omega)}\bigr) = \mathbb{P}\bigl((X,Y)^{-1}(D) \bigm|\mathcal{G}\bigr)(\omega)$ for $\restr{\Pr}{\mathcal{G}}$-almost every $\omega\in\Omega_0$.  If $D = A \times B$ for some $A \in \mathcal{A}$ and $B \in \mathcal{B}$, then $D^{Y(\omega)} = A$ if $Y(\omega) \in B$, and $D^{Y(\omega)} = \emptyset$ if $Y(\omega) \notin B$. Thus,
\begin{align*}
P_\omega\bigl(D^{Y(\omega)}\bigr) &= P_\omega(A)\mathbbm{1}_{\{Y(\omega) \in B\}} = \Pr\bigl(X^{-1}(A)\bigm|\mathcal{G}\bigr)(\omega)\mathbbm{1}_{\{Y(\omega) \in B\}} \\
&= \mathbb{E}\bigl(\chi_A \circ X \bigm| \mathcal{G}\bigr)(\omega)\cdot(\chi_B \circ Y)(\omega) = \mathbb{E}\bigl((\chi_A \circ X) \cdot(\chi_B \circ Y)\bigm| \mathcal{G}\bigr)(\omega) \\
&= \mathbb{P}\bigl((X,Y)^{-1}(D) \bigm|\mathcal{G}\bigr)(\omega)
\end{align*}
for $\restr{\Pr}{\mathcal{G}}$-almost every $\omega\in\Omega_0$, where we have used the fact that $\chi_B \circ Y$ is $\mathcal{G}$-measurable in the penultimate equality.  Thus $\mathcal{D} \supseteq\{A \times B : A \in \mathcal{A},\,B \in \mathcal{B}\}$, which is a $\pi$-system that generates $\mathcal{D}$.  Now suppose that $D_1,D_2 \in \mathcal{D}$ with $D_1 \subseteq D_2$.  Then
\begin{align*}
P_\omega\bigl((D_2 \setminus D_1)^{Y(\omega)}\bigr) &= P_\omega\bigl(D_2^{Y(\omega)} \setminus D_1^{Y(\omega)}\bigr) = P_\omega\bigl(D_2^{Y(\omega)}\bigr) - P_\omega\bigl(D_1^{Y(\omega)}\bigr) \\
&= \mathbb{P}\bigl((X,Y)^{-1}(D_2) \bigm|\mathcal{G}\bigr)(\omega) - \mathbb{P}\bigl((X,Y)^{-1}(D_1) \bigm|\mathcal{G}\bigr)(\omega) \\
&= \mathbb{P}\bigl((X,Y)^{-1}(D_2 \setminus D_1) \bigm|\mathcal{G}\bigr)(\omega) 
\end{align*}
for $\restr{\Pr}{\mathcal{G}}$-almost every $\omega\in\Omega_0$, so $D_2 \setminus D_1 \in \mathcal{D}$.  Finally, let $(D_n)$ be an increasing sequence of sets in $\mathcal{D}$, and let $D := \bigcup_{n=1}^\infty D_n$.  Then $D^{Y(\omega)} = \bigcup_{n=1}^\infty D_n^{Y(\omega)}$ and $(X,Y)^{-1}(D) = \bigcup_{n=1}^\infty (X,Y)^{-1}(D_n)$, so that
\begin{align*}
P_\omega\bigl(D^{Y(\omega)}\bigr) &= \lim_{n \rightarrow \infty} P_\omega\bigl(D_n^{Y(\omega)}\bigr) = \lim_{n \rightarrow \infty} \mathbb{P}\bigl((X,Y)^{-1}(D_n) \bigm|\mathcal{G}\bigr)(\omega)= \mathbb{P}\bigl((X,Y)^{-1}(D) \bigm|\mathcal{G}\bigr)(\omega)
\end{align*}
for $\restr{\Pr}{\mathcal{G}}$-almost every $\omega\in\Omega_0$, where we have used the conditional monotone convergence theorem in the final equality \citep[][Theorem~10.1.7]{Dud02}.  Thus, $D \in \mathcal{D}$, and it follows from Dynkin's lemma that $\mathcal{D} = \mathcal{A} \otimes \mathcal{B}$.

Finally, if $C \in \mathcal{C}$, then $D := \phi^{-1}(C) \in \mathcal{A} \otimes \mathcal{B}$, so that for $\restr{\Pr}{\mathcal{G}}$-almost every $\omega\in\Omega_0$,
\begin{align*}
P_\omega\bigl(\phi^{-1}(C)^{Y(\omega)}\bigr) &= P_\omega\bigl(D^{Y(\omega)}\bigr) = \mathbb{P}\bigl((X,Y)^{-1}(D) \bigm|\mathcal{G}\bigr)(\omega) \\
&= \mathbb{P}\bigl((X,Y)^{-1}(D) \bigm|\mathcal{G}\bigr)(\omega)\cdot \mathbbm{1}_{\Omega_0}(\omega) = \mathbb{P}\bigl((X,Y)^{-1}(D) \cap \Omega_0 \bigm|\mathcal{G}\bigr)(\omega) \\
&= \mathbb{P}\bigl(Z^{-1}(C) \cap \Omega_0 \bigm|\mathcal{G}\bigr)(\omega) = \mathbb{P}\bigl(Z^{-1}(C)\bigm|\mathcal{G}\bigr)(\omega),
\end{align*}
as required, since $\Omega_0 \in \mathcal{G}$.

(c) This follows directly from (b) on setting $\Omega_0=\Omega$.
\end{proof}

\unparskip
The following useful result is a special case of~\citet[Theorem~5.4]{Kal97} and can be derived
using the definition of conditional expectation~\citep[Problem~10.1.9]{Dud02}, or alternatively using regular conditional distributions and standard measure-theoretic devices (similarly to the proofs of Lemma~\ref{lem:conddist}(b) above and~\citet[Theorem~10.2.5]{Dud02}).
\begin{lemma}
\label{lem:indeplem}
Let $X,Y$ be measurable functions from $(\Omega,\mathcal{F},\Pr)$ to measurable spaces $(\mathscr{X},\mathcal{A}),(\mathscr{Y},\mathcal{B})$ respectively, and let $\phi\colon (\mathscr{X}\times\mathscr{Y},\mathcal{A}\otimes\mathcal{B})\to\R$ be a measurable function satisfying $\E\bigl(\abs{\phi(X,Y)}\bigr)<\infty$. Let $\mathcal{G}\subseteq\mathcal{F}$ be a $\sigma$-algebra, and suppose that $Y$ is $\mathcal{G}$-measurable. If $X$ has distribution $Q$ on $(\mathscr{X},\mathcal{A})$ and is independent of $\mathcal{G}$, then $\E\bigl(\phi(X,Y)\!\bigm|\!\mathcal{G}\bigr)(\omega)=\int_{\mathscr{X}}\phi\bigl(x,Y(\omega)\bigr)\,dQ(x)$ for $\restr{\Pr}{\mathcal{G}}$-almost every $\omega\in\Omega$.
\end{lemma}

\unparskip
Next, for $\sigma$-algebras $\mathcal{G}_1,\mathcal{G}_2,\mathcal{G}_3\subseteq\mathcal{F}$, we say that $\mathcal{G}_1$ and $\mathcal{G}_2$ are \emph{conditionally independent given $\mathcal{G}_3$}, and write $\mathcal{G}_1\indep\mathcal{G}_2\,|\,\mathcal{G}_3$, if $\Pr(A_1\cap A_2\,|\,\mathcal{G}_3)=\Pr(A_1\,|\,\mathcal{G}_3)\,\Pr(A_2\,|\,\mathcal{G}_3)$ almost surely for all $A_1\in\mathcal{G}_1$ and $A_2\in\mathcal{G}_2$, or equivalently if $\Pr\bigl(A_1\,|\,\sigma(\mathcal{G}_2,\mathcal{G}_3)\bigr)=\Pr(A_1\,|\,\mathcal{G}_3)$ almost surely for all $A_1\in\mathcal{G}_1$~\citep[Proposition~5.6]{Kal97}. If this holds with $\mathcal{G}_1=\sigma(X)$ for some random variable $X$, we also say that $X$ and $\mathcal{G}_2$ are conditionally independent given $\mathcal{G}_3$, and write $X\indep\mathcal{G}_2\,|\,\mathcal{G}_3$ (and similarly for $\mathcal{G}_2$ and $\mathcal{G}_3$). The following basic facts follow straightforwardly from the definition of conditional independence.
\begin{lemma}[\protect{\citealp[Corollary~5.7(i)]{Kal97}}]
\label{lem:condindequiv}
We have $\mathcal{G}_1 \indep \mathcal{G}_2 \,|\,\mathcal{G}_3$ if and only if $\sigma(\mathcal{G}_1,\mathcal{G}_3)\indep \mathcal{G}_2 \,|\,\mathcal{G}_3$.
\end{lemma}
\begin{lemma}
\label{lem:condind}
Let $(\mathscr{X},\mathcal{A})$ and $(\mathscr{Y},\mathcal{B})$ be measurable spaces and let $\mathcal{G}\subseteq\mathcal{F}$ be a $\sigma$-algebra.

\unparskip
\begin{enumerate}[label=(\alph*)]
\item For $i=1,2$, let $X_i\colon (\Omega,\mathcal{F},\Pr)\to (\mathscr{X},\mathcal{A})$ and $Y_i\colon (\Omega,\mathcal{F},\Pr)\to (\mathscr{Y},\mathcal{B})$ be measurable functions such that $X_i\indep Y_i\,|\,\mathcal{G}$. For $E\in\mathcal{G}$, let $X\colon(\Omega,\mathcal{F},\Pr)\to (\mathscr{X},\mathcal{A})$ be the measurable function satisfying $X=X_1$ on $E$ and $X=X_2$ on $E^c$, and define $Y$ similarly. Then $X\indep Y\,|\,\mathcal{G}$.
\item Suppose that the measurable maps $X\colon (\Omega,\mathcal{F},\Pr)\to (\mathscr{X},\mathcal{A})$ and $Y\colon (\Omega,\mathcal{F},\Pr)\to (\mathscr{Y},\mathcal{B})$ have conditional distributions $P\equiv P_\omega$ and $Q\equiv Q_\omega$ respectively given $\mathcal{G}$. Then $X\indep Y\,|\,\mathcal{G}$ if and only if $(X,Y)\colon (\Omega,\mathcal{F},\Pr)\to (\mathscr{X}\times\mathscr{Y},\mathcal{A}\otimes\mathcal{B})$ has conditional distribution $R\equiv R_\omega:=P_\omega\otimes Q_\omega$ given $\mathcal{G}$.
\end{enumerate}
\end{lemma}

\deparskip
\begin{proof}
(a) For fixed $A\in\mathcal{A}$ and $B\in\mathcal{B}$, write $\chi_A\colon (\mathscr{X},\mathcal{A})\to\{0,1\}$ and $\chi_B\colon (\mathscr{Y},\mathcal{B})\to\{0,1\}$ for the respective indicator functions, and for $i=1,2$, note that 
\begin{align*}
\E\bigl((\chi_A\circ X_i)\cdot (\chi_B\circ Y_i)\bigm| \mathcal{G}\bigr)=\Pr\bigl(X_i^{-1}(A)\cap Y_i^{-1}(B)\bigm|\mathcal{G}\bigr)&=\Pr\bigl(X_i^{-1}(A)\bigm|\mathcal{G}\bigr)\Pr\bigl(Y_i^{-1}(B)\bigm|\mathcal{G}\bigr)\\
&=\E(\chi_A\circ X_i\,|\,\mathcal{G})\,\E(\chi_B\circ Y_i\,|\,\mathcal{G})
\end{align*}
almost surely, since $X_i\indep Y_i\,|\,\mathcal{G}$. As in the proof of Lemma~\ref{lem:conddist}(a), we have $\chi_A\circ X=(\chi_A\circ X_1)\Ind_E+(\chi_{A^c}\circ X_2)\Ind_{E^c}$ and $\chi_B\circ Y=(\chi_B\circ Y_1)\Ind_E+(\chi_{B^c}\circ Y_2)\Ind_{E^c}$, so it follows that
\begin{align*}
\Pr\bigl(X^{-1}(A)\cap Y^{-1}(B)\bigm|\mathcal{G}\bigr)&=\E\bigl((\chi_A\circ X)\cdot (\chi_B\circ Y)\bigm| \mathcal{G}\bigr)\\
&=\E\bigl((\chi_A\circ X_1)\cdot (\chi_B\circ Y_1)\Ind_E+(\chi_A\circ X_2)\cdot (\chi_B\circ Y_2)\Ind_{E^c}\bigm| \mathcal{G}\bigr)\\
&=\E\bigl((\chi_A\circ X_1)\cdot (\chi_B\circ Y_1)\bigm| \mathcal{G}\bigr)\Ind_E+\E\bigl((\chi_A\circ X_2)\cdot (\chi_B\circ Y_2)\bigm| \mathcal{G}\bigr)\Ind_{E^c}\\
&=\E(\chi_A\circ X_1\,|\,\mathcal{G})\,\E(\chi_B\circ Y_1\,|\,\mathcal{G})\Ind_E+\E(\chi_A\circ X_2\,|\,\mathcal{G})\,\E(\chi_B\circ Y_2\,|\,\mathcal{G})\Ind_{E^c}\\
&=\E\bigl((\chi_A\circ X_1)\Ind_E+(\chi_A\circ X_2)\Ind_{E^c}\bigm|\mathcal{G}\bigr)\,\E\bigl((\chi_B\circ Y_1)\Ind_E+(\chi_B\circ Y_2)\Ind_{E^c}\,|\,\mathcal{G}\bigr)\\
&=\E(\chi_A\circ X\,|\,\mathcal{G})\,\E(\chi_B\circ Y\,|\,\mathcal{G})=\Pr\bigl(X^{-1}(A)\bigm|\mathcal{G}\bigr)\,\Pr\bigl(Y^{-1}(B)\bigm|\mathcal{G}\bigr),
\end{align*}
where we have used the fact that $E\in\mathcal{G}$ to obtain the third-last equality. Since this holds for all $A\in\mathcal{A}$ and $B\in\mathcal{B}$, the result follows.

(b) For $A\in\mathcal{A}$ and $B\in\mathcal{B}$, note that
\begin{align}
\label{eq:cap}
&\Pr\bigl(X^{-1}(A)\cap Y^{-1}(B)\bigm|\mathcal{G}\bigr)(\omega)=\Pr\bigl((X,Y)^{-1}(A\times B)\bigm|\mathcal{G}\bigr)(\omega)\\
\label{eq:prod}
&\Pr\bigl(X^{-1}(A)\bigm|\mathcal{G}\bigr)(\omega)\cdot\Pr\bigl(Y^{-1}(B)\bigm|\mathcal{G}\bigr)(\omega)=P_\omega(A)\,Q_\omega(B)=R_\omega(A\times B)
\end{align}
for $\restr{\Pr}{\mathcal{G}}$-almost every $\omega\in\Omega$. Thus, if $(X,Y)$ has conditional distribution $R\equiv R_\omega=P_\omega\otimes Q_\omega$ given $\mathcal{G}$, then for any $A\in\mathcal{A}$ and $B\in\mathcal{B}$, the right hand sides of~\eqref{eq:cap} and~\eqref{eq:prod} agree for $\restr{\Pr}{\mathcal{G}}$-almost every $\omega\in\Omega$, so the same is true of the left hand sides. This shows that $X\indep Y\,|\,\mathcal{G}$. 

Conversely, suppose that $X\indep Y\,|\,\mathcal{G}$ and let $\mathcal{D}$ be the collection of all $D\in\mathcal{A}\otimes\mathcal{B}$ such that $R_\omega(D)=(P_\omega\otimes Q_\omega)(D)=\Pr\bigl((X,Y)^{-1}(D)\bigm|\mathcal{G}\bigr)(\omega)$ for $\restr{\Pr}{\mathcal{G}}$-almost every $\omega\in\Omega$. Then for any $A\in\mathcal{A}$ and $B\in\mathcal{B}$, the left hand sides of~\eqref{eq:cap} and~\eqref{eq:prod} agree for $\restr{\Pr}{\mathcal{G}}$-almost every $\omega\in\Omega$, so $\mathcal{D}$ contains a $\pi$-system $\{A\times B:A\in\mathcal{A},\,B\in\mathcal{B}\}$ that generates $\mathcal{A}\otimes\mathcal{B}$. Similarly to the proof of Lemma~\ref{lem:conddist}(b), it can be verified that $\mathcal{D}$ is a $d$-system,
so it follows from Dynkin's lemma that $\mathcal{D}=\mathcal{A}\otimes\mathcal{B}$, and hence that $(X,Y)$ has conditional distribution $R\equiv R_\omega=P_\omega\otimes Q_\omega$, as required.
\end{proof}

\umparskip
\subsection{Auxiliary probabilistic results}
\label{sec:inductiveaux}
The following general result is used in the proofs of some important complete convergence statements in Sections~\ref{sec:inductive} and~\ref{sec:Wr}, specifically Proposition~\ref{prop:AMPproof}(g) and Corollary~\ref{cor:Wr}(b).
\begin{lemma}
\label{lem:seqcoup}
Let $(X_n),(Y_n)$ be sequences of random elements defined on $(\Omega,\mathcal{F},\Pr)$ such that $X_n,Y_n$ take values in Polish spaces $E_n,E_n'$ respectively for each $n\in\N$, and suppose that there exist Borel measurable functions $g_n\colon E_n\to E_n'$ such that $Y_n\eqd g_n(X_n)$ for each $n$. Then there exists a sequence of random elements $\tilde{X}_n\colon\Omega\to E_n$ such that $\tilde{X}_n\eqd X_n$ for all $n$ and $\bigl(g_1(\tilde{X}_1),g_2(\tilde{X}_2),\dotsc\bigr)=(Y_1,Y_2,\dotsc)$ almost surely (viewed as random sequences taking values in $\prod_{n=1}^\infty E_n'$, equipped with its cylindrical (i.e.\ Borel) $\sigma$-algebra).
\end{lemma}

\unparskip
This is an extension to random sequences of the following result for pairs of random elements: given random elements $X_1,X_2$ taking values in $E_1,E_2$ respectively, let $(Y_1,Y_2)\sim\pi$ be any coupling of $g_1(X_1),g_2(X_2)$. Then there exists a coupling $(X_1',X_2')\sim\pi'$ of $X_1,X_2$ such that $\bigl(g_1(X_1'),g_2(X_2')\bigr)\eqd (Y_1,Y_2)$, i.e.\ $\pi=\pi'\circ (g_1,g_2)^{-1}$. This can be proved by applying the gluing lemma from optimal transport~\citep[Lemma~7.6]{Vil03} or a simpler version of the general argument below.

Given an arbitrary coupling $(Y_1,Y_2,\dotsc)$ of the random elements $g_1(X_1),g_2(X_2),\dotsc$, the first (and most important) step in the proof below is to `lift' this to produce a suitable coupling $(X_1',X_2',\dotsc)$ of the random elements $X_1,X_2,\dotsc$, in such a way that $\bigl(g_1(X_1'),g_2(X_2'),\dotsc\bigr)\eqd(Y_1,Y_2,\dotsc)$ as random sequences. Intuitively, the key construction can be interpreted as the output of the following two-stage procedure: 

\unparskip
\begin{enumerate}[label=(\Alph*)]
\item Denoting by $\pi$ the (given) distribution of $(Y_1,Y_2,\dotsc)$ on $\prod_{n=1}^\infty E_n'$, we first draw $(Y_1',Y_2',\dotsc)\sim\pi$;
\item Having obtained $(Y_1',Y_2',\dotsc)=(y_1,y_2,\dotsc)$ from Step A, we then generate $X_1',X_2',\dotsc$ by sampling independently from $Q_{y_1}^1,Q_{y_2}^2,\dotsc$, where $Q_{y_n}^n$ denotes the ``conditional distribution of $X_n$ given $g_n(X_n)=y_n$''. 
\end{enumerate}

\unparskip
Step B ensures that $X_1',X_2',\dotsc$ are conditionally independent given $(Y_1',Y_2',\dotsc)$. To make rigorous sense of this informal description and to validate the construction, we use the language of \emph{disintegration of measures}, as outlined in Remark~\ref{rem:condkern}. There are similarities here with the proof of the gluing lemma~\citep[Lemma~7.6]{Vil03}. 
To verify that the random sequences $\bigl(g_n(X_n')\bigr)$ and $(Y_n)$ have the same distribution on $\prod_{n=1}^\infty E_n'$, it suffices to show that they have the same finite-dimensional distributions, i.e.\ that $\bigl(g_1(X_1'),\dotsc,g_n(X_n')\bigr)\eqd(Y_1,\dotsc,Y_n)$ for all $n$. Finally, to upgrade all the distributional equalities above to almost-sure equalities, we appeal to a general result from abstract probability theory~\citep[Corollary~5.11]{Kal97}, 
which is also proved using disintegration techniques.
\begin{remark}
\label{rem:randseq}
To guarantee the existence of a random sequence $(\tilde{X}_1,\tilde{X}_2,\dotsc)$ with a given distribution on $\prod_{n=1}^\infty E_n'$, we require the underlying probability space $(\Omega,\mathcal{F},\Pr)$ to be rich enough to support a sequence of independent $U[0,1]$ random variables. This can be assumed without loss of generality, since otherwise we can work with the 
product space $(\Omega\times [0,1],\mathcal{F}\otimes\mathcal{B}_{[0,1]},\Pr\otimes\mu_{[0,1]})$, where $\mathcal{B}_{[0,1]}$ and $\mu_{[0,1]}$ denote the Borel $\sigma$-algebra and Lebesgue measure on $[0,1]$ respectively.
\end{remark}

\deparskip
\begin{proof}[Proof of Lemma~\ref{lem:seqcoup}]
For each $n$, denote by $\mathcal{B}_n,\mathcal{B}_n'$ the Borel $\sigma$-algebras of $E_n,E_n'$ respectively. It follows from~\citet[Theorem~2.5.7]{Dud02} and~\citet[Lemma~1.2]{Kal97} that $\prod_{j=1}^n E_j'$ and $\prod_{j=1}^n E_j$ are Polish spaces with Borel $\sigma$-algebras $\bigotimes_{j=1}^n\mathcal{B}_j'$ and $\bigotimes_{j=1}^n\mathcal{B}_j$ respectively. Denote by $\mu_n,\pi_n$ the distributions of $Y_n$ and $(Y_1,\dotsc,Y_n)$ on $(E_n',\mathcal{B}_n')$ and $\bigl(\prod_{j=1}^n E_j',\bigotimes_{j=1}^n\mathcal{B}_j'\bigr)$ respectively. Since $E_n$ is a Polish space, we know from Section~\ref{sec:conddists} that there exists a regular conditional distribution for $X_n$ given $\sigma\bigl(g_n(X_n)\bigr)$. Equivalently, there is a family of probability distributions $(Q_y^n)_{y\in E_n'}$ on $E_n$ satisfying conditions (I) and (II) in Remark~\ref{rem:condkern}, where we take $X:=X_n$, $Y:=g_n(X_n)\eqd Y_n$ and $P^Y:=\mu_n$. It follows from Remark~\ref{rem:condkern}(II) that $\Pr(X_n\in A)=\int_{E_n'}Q_y^n(A)\,d\mu_n(y)$ for all $A\in\mathcal{B}_n$, and moreover that
\begin{equation}
\label{eq:gnxn}
\int_{B'}\Ind_B(y)\,d\mu_n(y)=\Pr\bigl(g_n(X_n)\in B\cap B'\bigr)=\Pr\bigl(X_n\in g_n^{-1}(B),\,g_n(X_n)\in B'\bigr)=\int_{B'}Q_y^n\bigl(g_n^{-1}(B)\bigr)\,d\mu_n(y)
\end{equation}
for $B,B'\in\mathcal{B}_n'$. Thus, for all $B\in\mathcal{B}_n'$, we have $Q_y^n\bigl(g_n^{-1}(B)\bigr)=\Ind_B(y)$ for $\mu_n$-almost every $y\in E_n'$.

For each $n\in\N$, we now define a new measure $\pi_n'$ on $\bigl(\prod_{j=1}^n E_j,\bigotimes_{j=1}^n\mathcal{B}_j\bigr)$ by
\begin{equation}
\label{eq:pitilde}
\pi_n'(A):=\int_{\,\prod_{j=1}^n E_j'}\biggl(\,\int_{E_1}\dotsi\int_{E_n}\Ind_A(x_1,\dotsc,x_n)\,dQ_{y_n}^n(x_n)\;\dotsi\, dQ_{y_1}^1(x_1)\biggr)\,d\pi_n(y_1,\dotsc,y_n)
\end{equation}
for $A\in\bigotimes_{j=1}^n\mathcal{B}_j$. That this a well-defined probability measure follows from Remark~\ref{rem:condkern}(III,\,IV) 
and the monotone convergence theorem. For each $n$, we claim that

\unparskip
\begin{enumerate}[label=(\roman*)]
\item $\pi_{n+1}'(A\times E_{n+1})=\pi_n'(A)$ for every $A\in\bigotimes_{j=1}^n\mathcal{B}_j$;
\item $\Pr(X_n\in A_n)=\pi_n'\bigl(\prod_{j=1}^{n-1}E_j\times A_n\bigr)$ for every $A_n\in\mathcal{B}_n$;
\item $\pi_n=\pi_n'\circ (g_1,\dotsc,g_n)^{-1}$, where $(g_1,\dotsc,g_n)\colon\bigl(\prod_{j=1}^n E_j,\bigotimes_{j=1}^n\mathcal{B}_j\bigr)\to\bigl(\prod_{j=1}^n E_j',\bigotimes_{j=1}^n\mathcal{B}_j'\bigr)$ denotes the measurable map $(x_1,\dotsc,x_n)\mapsto \bigl(g_1(x_1),\dotsc,g_n(x_n)\bigr)$.
\end{enumerate}

\unparskip
Property (i) is immediate from~\eqref{eq:pitilde} and the fact that $\pi_{n+1}(B\times E_n')=\pi_n(B)$ for all $B\in\bigotimes_{j=1}^n\mathcal{B}_j'$. To verify (ii), observe that 
\begin{align*}
\pi_n'\bigl(\textstyle\prod_{j=1}^{n-1}E_j\times A_n\bigr)&=\int_{\,\prod_{j=1}^n E_j'}\,\int_{E_n}\Ind_{A_n}(x_n)\,dQ_{y_n}^n(x_n)\,d\pi_n(y_1,\dotsc,y_n)\\
&=\int_{E_n'}Q_{y_n}^n(A_n)\,d\mu_n(y_n)=\Pr(X_n\in A_n),
\end{align*}
where the final equality is obtained from Remark~\ref{rem:condkern}(II) as above. As for (iii), fix $B_j\in\mathcal{B}_j'$ for $1\leq j\leq n$ and note that by~\eqref{eq:gnxn} and~\eqref{eq:pitilde}, we have
\begin{align*}
\pi_n'\bigl(\textstyle\prod_{j=1}^n g_j^{-1}(B_j)\bigr)&=\int_{\,\prod_{j=1}^n E_j'}Q_{y_1}^1\bigl(g_1^{-1}(B_1)\bigr)\dotsm\,Q_{y_n}^n\bigl(g_n^{-1}(B_n)\bigr)\,d\pi_n(y_1,\dotsc,y_n)\\
&=\int_{\,\prod_{j=1}^n E_j'}\Ind_{B_1}(y_1)\dotsm\Ind_{B_n}(y_n)\,d\pi_n(y_1,\dotsc,y_n)=\pi_n\bigl(\textstyle\prod_{j=1}^n B_j\bigr).
\end{align*}
This means that $\pi_n$ and $\pi_n'\circ (g_1,\dotsc,g_n)^{-1}$ agree on $\bigl\{\prod_{j=1}^n B_j:B_j\in\mathcal{B}_j'\text{ for all }1\leq j\leq n\bigr\}$, a $\pi$-system that generates $\bigotimes_{j=1}^n\mathcal{B}_j$, so (iii) holds.

Since the distributions $\pi_1',\pi_2',\dotsc$ on the Polish spaces $E_1,E_1\times E_2,\dotsc$ satisfy the consistency condition (i), we deduce from the Daniell--Kolmogorov extension theorem~\citep[Theorem~5.14]{Kal97} and Remark~\ref{rem:randseq} that exists a sequence $(X_n')_{n\in\N}$ of random elements $X_n'\colon\Omega\to E_n$ such that $(X_1',\dotsc,X_n')\sim\pi_n'$ on $\prod_{j=1}^n E_j$ for each $n$. Then by (ii) and (iii) above, we have $X_n'\eqd X_n$ and $\bigl(g_1(X_1'),\dotsc,g_n(X_n')\bigr)\sim\pi_n'\circ (g_1,\dotsc,g_n)^{-1}=\pi_n$ for each $n$, where $\pi_n$ was defined to be the distribution of $(Y_1,\dotsc,Y_n)$. Thus, the sequences $\bigl(g_n(X_n')\bigr)$ and $(Y_n)$ have the same finite-dimensional distributions; in other words, their distributions agree on $\bigl\{\prod_{j=1}^N B_j\times\prod_{j=N+1}^\infty E_j':N\in\N,\,B_j\in\mathcal{B}_j'\text{ for all }1\leq j\leq N\bigr\}$, a collection of cylindrical sets that generate the cylindrical $\sigma$-algebra $\mathcal{B}'$ of $\prod_{n=1}^\infty E_n'$. (By~\citet[Lemma~1.2]{Kal97}, $\mathcal{B}'$ is the Borel $\sigma$-algebra of $\prod_{n=1}^\infty E_n'$.) 
We conclude that $\bigl(g_1(X_1'),g_2(X_2'),\dotsc\bigr)\eqd (Y_1,Y_2,\dotsc)$ as random sequences taking values in $\bigl(\prod_{n=1}^\infty E_n',\mathcal{B}'\bigr)$. 

Finally, we apply~\citet[Corollary~5.11]{Kal97} with $T=\prod_{n=1}^\infty E_n$, $S=\prod_{n=1}^\infty E_n'$, $\eta=(X_n')$, $\xi=(Y_n)$
and $f\colon T\to S$ given by $f(x_1,x_2,\dotsc)=\bigl(g_1(x_1),g_2(x_2),\dotsc\bigr)$; note that $T,S$ are Polish spaces~\citep[e.g.][Theorem~2.5.7]{Dud02}
and that $f$ is Borel measurable. Having already shown that $f(\eta)\eqd\xi$, we deduce from~\citet[Corollary~5.11]{Kal97} that there exists $(\tilde{X}_n)\equiv\tilde{\eta}\eqd\eta=(X_n')$ satisfying $\bigl(g_n(\tilde{X}_n)\bigr)=f(\tilde{\eta})=\xi=(Y_n)$ almost surely, as required.
\end{proof}

\unparskip
In the proofs of Proposition~\ref{prop:AMPproof}(a,\,c), we apply the concentration inequality below for sums of pseudo-Lipschitz functions of independent Gaussian random variables.
\begin{lemma}
\label{lem:pseudoconc}
There exists a universal constant $C>0$ such that the following holds for all $n\in\N$, $r\geq 2$ and $t\geq 0$: if $Z_1,\dotsc,Z_n\iid N(0,1)$, $L\equiv (L_1,\dotsc,L_n)\in (0,\infty)^n$ and $f_i\in\PL_1(r,L_i)$ for $1\leq i\leq n$, then
\begin{equation}
\label{eq:pseudoconc}
\Pr\biggl(\biggl|\frac{1}{n}\sum_{i=1}^n\,\bigl\{f_i(Z_i)-\E\bigl(f_i(Z_i)\bigr)\bigr\}\biggr|\geq t\biggr)\leq\exp\biggl(1-\min\,\biggl\{\biggl(\frac{nt}{(Cr)^r\norm{L}_2}\biggr)^2,\biggl(\frac{nt}{(Cr)^r\norm{L}_\infty}\biggr)^{2/r}\biggr\}\biggr).
\end{equation}
\end{lemma}

\deparskip
\begin{proof}
We first consider the case $n=1$. For arbitrary $r\geq 2$ and $L>0$, we may assume without loss of generality that $f\equiv f_1\in\PL(r,L)$ satisfies $f(0)=0$, so that $\abs{f(x)}=\abs{f(x)-f(0)}\leq L(\abs{x}+\abs{x}^r)\leq 2L(\abs{x}\vee\abs{x}^r)$ for all $x\in\R$. Thus, if $Z\sim N(0,1)$, then
\begin{equation}
\label{eq:plbd1}
\Pr(\abs{f(Z)}\geq s)\leq\Pr\bigl(\abs{Z}\vee\abs{Z}^r\geq s/(2L)\bigr)\leq e^{-\frac{1}{2}\min\cbr{\rbr{\frac{s}{2L}}^2,\,\rbr{\frac{s}{2L}}^{2/r}}}
\end{equation}
for all $s\geq 0$, and \[\frac{\E\bigl(\abs{f(Z)}\bigl)}{L}\leq\E(\abs{Z}+\abs{Z}^r)=\rbr{\sqrt{\frac{2}{\pi}}+\frac{2^{r/2}}{\sqrt{\pi}}\,\Gamma\rbr{\frac{r+1}{2}}}=:\upsilon_r\]
by direct computation. Now $\Gamma(x)<e^{1/(12x)}(x/e)^x\sqrt{2\pi/x}$ for all $x>0$ by a non-asymptotic version of Stirling's formula; see for example~\citet[Theorem~5]{Gor94} and~\citet[Lemma~10]{DSW21}. Since $r\geq 2$, we have $(r+1)/e<r$ and $(\sqrt{2}-1)r^{r/2}\geq 2(\sqrt{2}-1)>1/\sqrt{\pi}$. Therefore,
\begin{equation}
\label{eq:plbd2}
\frac{\upsilon_r}{2}\leq\frac{1}{\sqrt{2}}\rbr{\frac{1}{\sqrt{\pi}}+\rbr{\frac{r+1}{e}}^{r/2}e^{\frac{1}{6(r+1)}-\frac{1}{2}}}\leq\frac{1}{\sqrt{2}}\rbr{\frac{1}{\sqrt{\pi}}+r^{r/2}}<r^{r/2}.
\end{equation}
Thus, for $t\geq L\upsilon_r$, we deduce from~\eqref{eq:plbd1} and~\eqref{eq:plbd2} that
\begin{align}
\Pr\bigl\{\bigl|f(Z)-\E\bigl(f(Z)\bigr)\bigr|\geq t\bigr\}\leq\Pr\bigl\{\abs{f(Z)}\geq t-\E\bigl(\abs{f(Z)}\bigl)\bigr\}&\leq e^{-\frac{1}{2}\min\cbr{\rbr{\frac{t}{2L}-\frac{\upsilon_r}{2}}^2,\,\rbr{\frac{t}{2L}-\frac{\upsilon_r}{2}}^{2/r}}}\notag\\
&\leq e^{\frac{1}{2}-\frac{1}{2}\rbr{\frac{t}{2L}-\frac{\upsilon_r}{2}}^{2/r}}\notag\\
&\leq e^{\frac{1+r}{2}-\frac{1}{2}\rbr{\frac{t}{2L}}^{2/r}}\notag\\
\label{eq:plbd3}
&\leq e^{1-\rbr{\frac{t}{2L}}^{2/r}\frac{1}{r+1}},
\end{align}
where the third inequality follows from the fact that $a^{2/r}\leq\abs{a-b}^{2/r}+b^{2/r}$ for $r\geq 2$ as above and any $a,b\geq 0$. Now~\eqref{eq:plbd3} holds trivially for all $t\in [0,L\upsilon_r)$ since $1-(r+1)^{-1}\{t/(2L)\}^{2/r}>1-(r+1)^{-1}(\upsilon_r/2)^{2/r}>0$ by~\eqref{eq:plbd2}, so~\eqref{eq:pseudoconc} holds with $C=3$ when $n=1$.

We now derive~\eqref{eq:pseudoconc} for general $n\geq 2$ with the aid of Theorem~3.1 and Proposition~A.3 in~\citet{KC18}; see also Theorem~1 and Corollary~2 in~\citet{BMP20}. As in Sections~2 and~3 of~\citet{KC18}, we begin by defining $\vartheta_\beta\colon [0,\infty)\to [0,\infty)$ for each $\beta>0$ by $\vartheta_\beta(x):=\exp(x^\beta)-1$.
Moreover, for $\beta,\lambda>0$, let $\vartheta_{\beta,\lambda}\colon [0,\infty)\to [0,\infty)$ be the continuous, strictly increasing function with inverse given by $\vartheta_{\beta,\lambda}^{-1}(t):=\log^{1/2}(1+t)+\lambda\log^{1/\beta}(1+t)$ for $t\geq 0$. For a random variable $X$ and a strictly increasing function $g\colon [0,\infty)\to [0,\infty)$ satisfying $g(0)=0$, we write $\Xi_g(X):=\inf\bigl\{\theta>0:\E\bigl(g(\abs{X}/\theta)\bigr)\leq 1\bigr\}\in [0,\infty]$, setting $\inf\emptyset=\infty$ by convention. Note that $\Xi_g(X)$ is precisely the $g$-Orlicz norm of $X$ when $g$ is convex, but that $\Xi_g$ does not in general define a norm when $g$ is not convex (for example when $g=\vartheta_\beta$ for $\beta\in (0,1)$, as in the proof below).

For arbitrary $n\geq 2$, $r\geq 2$ and $L\equiv(L_1,\dotsc,L_n)\in (0,\infty)^n$, let $f_1,\dotsc,f_n\in\PL(r,L_i)$ and $Z_1,\dotsc, Z_n\iid N(0,1)$, and assume without loss of generality that $X_i:=f_i(Z_i)$ satisfies $\E(X_i)=0$ for all $1\leq i\leq n$. Setting $\beta:=2/r\in [0,1]$ and $\theta_i:=2\{4(r+1)\}^{r/2}\,L_i$ for $1\leq i\leq n$, we now integrate up the bound~\eqref{eq:plbd3} to see that
\begin{align}
\E\bigl(\vartheta_\beta(\abs{X_i}/\theta_i)\bigr)=\int_0^\infty\Pr\bigl(\vartheta_\beta(\abs{X_i}/\theta_i)\geq t\bigr)\,dt&=\int_0^\infty\Pr\bigl(\abs{X_i}\geq\theta_i\vartheta_\beta^{-1}(t)\bigr)\,dt\notag\\
&=\int_0^\infty\Pr\bigl(\abs{X_i}\geq 2(r+1)^{r/2}L_i\,\{4\log(1+t)\}^{1/\beta}\bigr)\,dt\notag\\
\label{eq:orlicz}
&\leq\int_0^\infty e(1+t)^{-4}\,dt=e/3<1,
\end{align}
whence $\Xi_{\vartheta_\beta}(X_i)\leq\theta_i=2\{4(r+1)\}^{r/2}\,L_i<\infty$. This shows that $X_1,\dotsc,X_n$ are independent, centred \emph{sub-Weibull} random variables of order $\beta=2/r$, in the sense of Definition 2.2 in~\citet{KC18}. Then applying~\citet[Theorem~3.1]{KC18} with $a=(1/n,\dotsc,1/n)\in\R^n$ and $b:=\bigl(\Xi_{\vartheta_\beta}(X_1)/n,\dotsc,\Xi_{\vartheta_\beta}(X_n)/n\bigr)$ in their notation, we deduce from~\eqref{eq:orlicz} that 
\begin{equation}
\label{eq:KC18}
\Xi_{\vartheta_{\beta,\lambda_\beta}}\biggl(\frac{1}{n}\sum_{i=1}^n X_i\biggr)\leq 2eC_\beta\norm{b}_2,\text{ where }
\begin{cases}
C_\beta:=(2e^{2/e}/\beta)^{1/\beta}(128\pi)^{1/4}\,e^{3+\frac{1}{24}}\\
\lambda_\beta:=(4^{1/\beta}/\sqrt{2})\,\norm{b}_\infty/\norm{b}_2.
\end{cases}
\end{equation}
It then follows from Proposition~A.3 in~\citet{KC18} that
\[\Pr\biggl(\biggl|\frac{1}{n}\sum_{i=1}^n X_i\biggr|\geq 4eC_\beta\norm{b}_2\max\bigl(s^{1/2},\lambda_\beta s^{1/\beta}\bigr)\biggr)\leq e^{1-s}\]
for all $s\geq 0$, and hence that
\[\Pr\biggl(\biggl|\frac{1}{n}\sum_{i=1}^n X_i\biggr|\geq t\biggr)\leq\exp\biggl(1-\min\biggl\{\biggl(\frac{t}{4eC_\beta\norm{b}_2}\biggr)^2,\biggl(\frac{t}{4eC_\beta'\norm{b}_\infty}\biggr)^\beta\biggr\}\biggr)\]
for all $t\geq 0$, where $C_\beta':=(4^{1/\beta}/\sqrt{2})\,C_\beta$. Since $\beta=2/r$ and $\Xi_{\vartheta_\beta}(X_i)\leq 2\{4(r+1)\}^{r/2}\,L_i$ for $1\leq i\leq n$, we have
\[n\norm{b}_p=\bigl\|\bigl(\Xi_{\vartheta_\beta}(X_1),\dotsc,\Xi_{\vartheta_\beta}(X_n)\bigr)\bigr\|_p\leq 2\{4(r+1)\}^{r/2}\,\norm{L}_p\] 
for $p\in\{2,\infty\}$. Moreover, $2\{4(r+1)\}^{r/2}\,C_\beta\leq 2\{4(r+1)\}^{r/2}\,C_\beta'\lesssim \{4e^{1/e}(r+1)\}^r$, so we can indeed find a suitable universal constant $C>0$ in~\eqref{eq:pseudoconc} such that the desired conclusion holds for all $n\in\N$, $r\geq 2$, $L\equiv(L_1,\dotsc,L_n)\in (0,\infty)^n$ and $t\geq 0$, as required.
\end{proof}
\begin{remark}
\label{rem:pseudoconc}
When $r>2$, $f\in\PL_1(r)$ and $Z\sim N(0,1)$, the moment generating function of $f(Z)$ may not be finite anywhere except at 0 if $f(Z)$ has heavier tails than an exponential random variable (for example when $f(z)=\sgn(z)\abs{z}^r$ for $z\in\R$). In these situations, the standard Chernoff method fails, which is why we apply different techniques that can handle general sub-Weibull random variables.

While we are primarily concerned with the case $r\geq 2$ in the proof of Proposition~\ref{prop:AMPproof}, there is an analogue of~\eqref{eq:pseudoconc} when $r\in [1,2)$, namely
\begin{equation}
\label{eq:pseudoconcr}
\Pr\biggl(\biggl|\frac{1}{n}\sum_{i=1}^n\,\bigl\{f_i(Z_i)-\E\bigl(f_i(Z_i)\bigr)\bigr\}\biggr|\geq t\biggr)\leq\exp\biggl(1-\min\,\biggl\{\biggl(\frac{nt}{C\norm{L}_2}\biggr)^2,\biggl(\frac{nt}{C\norm{L}_{\tilde{r}}}\biggr)^{2/r}\biggr\}\biggr),
\end{equation}
where $C>0$ is a suitable universal constant and $\tilde{r}:=2/(2-r)\in [2,\infty)$ is the H\"older conjugate of $2/r$. This can be proved using a Chernoff bound~\citep[e.g.][Exercise~2.27]{BLM13}, or alternatively using~\citet[Theorem~3.1]{KC18} once again, where we instead take $\beta:=2/r$, $C_\beta:=4e+2(2\log 2)^{r/2}$ and $\lambda_\beta:=(4^{1+1/\beta}C_\beta^{-1}e/\sqrt{2})\,\norm{b}_{\tilde{r}}/\norm{b}_2$ in~\eqref{eq:KC18}.
\end{remark}

\unparskip
The proof of Proposition~\ref{prop:AMPproof}(g) makes use of the following straightforward consequence of the definition of weak convergence.
\begin{lemma}
\label{lem:weakconv}
On a Euclidean space $E$, if $(\mu_n)$ is a sequence of Borel probability measures that converges weakly to a Borel probability measure $\mu$, then $\int_E g\,d\mu_n\to\int_E g\,d\mu$ for any bounded, Borel measurable $g\colon E\to\R$ that is continuous $\mu$-almost everywhere (in the sense that the set of discontinuities of $g$ has $\mu$-measure 0).
\end{lemma}

\deparskip
\begin{proof}
Writing $A\subseteq E$ for the set of discontinuities of $g$, we have $\mu(A)=0$ by assumption. By Skorokhod's representation theorem~\citep[e.g.][Theorem~3.30]{Kal97}, there exist random variables $X,X_1,X_2,\dotsc$ defined on a common probability space such that $X\sim\mu$, $X_n\sim\mu_n$ for all $n$ and $X_n\to X$ almost surely. Then $g(X_n)\to g(X)$ almost surely on the event $\{X\in A^c\}$, which has probability $\mu(A^c)=1$, so an application of the dominated (or bounded) convergence theorem shows that $\int_E g\,d\mu_n=\E\bigl(g(X_n)\bigr)\to\E\bigl(g(X)\bigr)=\int_E g\,d\mu$, as required.
\end{proof}

\unparskip
\begin{remark}
\label{rem:lipderiv}
For each Lipschitz function $f_k\colon\R^2\to\R$ in the AMP recursion~\eqref{eq:AMPsym}, we assume in~\ref{ass:A5} that there exists some $f_k'$ that 
satisfies the hypotheses of Lemma~\ref{lem:weakconv} above with $\mu=\lambda\otimes\pi$; recall that $\lambda$ denotes Lebesgue measure on $\R$ and the probability distribution $\pi$ is as in~\ref{ass:A1}. To see why~\ref{ass:A5} is a non-vacuous (albeit very mild) condition, consider Borel probability measures on $\R^D$ of the form $\mu=\lambda\otimes\nu$, where $D\geq 2$
and $\nu$ is some probability measure on $\R^{D-1}$. We will now give an example of a Lipschitz function $G\colon\R^D\to\R$ whose partial derivative $\frac{\partial G}{\partial x_1}$ cannot be extended beyond its domain of definition to a function $g\colon\R^D\to\R$ that is continuous $\mu$-almost everywhere, for any $\mu$ of the above form. 

Denote by $C\subseteq [0,1]$ the \emph{fat Cantor set}~\citep[e.g.][pp.\ 140--141]{AB98}, which has the property that for all $x\in C$ and $\varepsilon>0$, both $(x-\varepsilon,x+\varepsilon)\cap C$ and $(x-\varepsilon,x+\varepsilon)\cap C^c$ have positive Lebesgue measure.
Then for any $f\colon\R\to\R$ with $f=\Ind_C$ Lebesgue almost everywhere, we have $\{f(u):u\in(x-\varepsilon,x+\varepsilon)\}=\{0,1\}$ for all $x\in C$ and $\varepsilon>0$, so $f$ is discontinuous on $C$, which has Lebesgue measure $1/2>0$.
Note that $F\colon x\mapsto\int_{-\infty}^x\Ind_C(t)\,dt$ is a Lipschitz function on $\R$ with $F'(x)=\Ind_C(x)$ for Lebesgue almost every $x\in\R$. Thus, for general $D\in\N$, the function $G\colon (x_1,\dotsc,x_D)\mapsto F(x_1)$ is Lipschitz on $\R^D$, and if $g\colon\R^D\to\R$ agrees with $\frac{\partial G}{\partial x_1}$ everywhere where the latter is defined, then $g$ is discontinuous on $C\times\R^{D-1}$,
which has strictly positive $\mu$-measure when $\mu=\lambda\otimes\nu$ as above.
\end{remark}

\umparskip
\subsection{Wasserstein convergence and pseudo-Lipschitz functions}
\label{sec:Wr}
Throughout this subsection, we fix $D\in\N$ and $r\in [1,\infty)$, and write $\mathcal{P}(r)\equiv\mathcal{P}_D(r)$ for the set of probability measures $P$ on $\R^D$ with $\int_{\R^D}\norm{x}^r\,dP(x)<\infty$ (i.e.\ a finite $r^{th}$ moment). For $P,Q\in\mathcal{P}(r)$, recall from Section~\ref{sec:notation} the definitions of $\widetilde{d}_r(P,Q)$ and the $r$-Wasserstein distance $d_r(P,Q)$.

The primary purpose of this subsection is to establish Theorem~\ref{thm:Wr} and its probabilistic Corollary~\ref{cor:Wr}, which can be viewed as extensions of~\citet[Theorem~7.12]{Vil03}. These show in particular that $\widetilde{d}_r$ and $d_r$ are metrics on $\mathcal{P}(r)$ that induce the same topology (Remark~\ref{rem:Wrmetric}), and also formalise the link between functions in $\PL_D(r)$ and convergence in $d_r$ (or equivalently $\widetilde{d}_r$).

As a first step towards the proof of Theorem~\ref{thm:Wr}, it is helpful to establish the following.
\begin{proposition}
\label{prop:Wrcount}
There exists a countable set $T_r'$ of bounded Lipschitz functions on $\R^D$ with the property that $\widetilde{d}_r(P,Q)=\sup_{\psi\in T_r'}\,\bigl|\int_{\R^D}\psi\,dP-\int_{\R^D}\psi\,dQ\bigr|\in [0,\infty)$ for all $P,Q\in\mathcal{P}(r)$.
\end{proposition}

\unparskip
A key property of the set $T_r'$ we construct is that for any $\psi\in\PL_D(r)$, there exists a sequence $(\psi_\ell)$ in $T_r'$ that converges uniformly to $\psi$ on compact subsets of $\R^D$. In subsequent proofs, we will write $Q(f)$ as shorthand for $\int_{\R^D}f\,dQ$ when $Q$ is a signed Borel measure on $\R^D$ and $f\colon\R^D\to\R$ is a $Q$-integrable function.

\unparskip
\begin{proof}
For $N\in\N$, let $B_N\equiv\bar{B}_D(0,N):=\{x\in\R^D:\norm{x}\leq N\}$ and define $f_N(x):=(N-\norm{x})\vee 0\wedge 1$ for $x\in\R^D$, so that $f_N$ is 1-Lipschitz on $\R^D$, $f_N=1$ on $B_{N-1}$ and $f_N=0$ on $B_N^c$. In the argument below (and in the proof of Theorem~\ref{thm:Wr}), we will use $f_N$ as a substitute for the (discontinuous) indicator function $\mathbbm{1}_{B_N}$ in several places. Note in particular that if $\tilde{g}\colon B_N\to\R$ is Lipschitz on $B_N$, then the function $g\colon\R^D\to\R$ defined by $g(x):=\tilde{g}(x)f_N(x)$ is Lipschitz and supported on the compact set $B_N$. 

Recalling the definitions of $\widetilde{d}_r,\mathcal{P}(r)$ from~\eqref{eq:Wrtilde} and writing $\widetilde{\PL}_D(r,1)$ for the set of all $\phi\in\PL_D(r,1)$ satisfying $\phi(0)=0$, we see from~\eqref{eq:PL} that
\begin{equation}
\label{eq:Wrtildefinite}
\widetilde{d}_r(P,Q)=\sup_{\phi\in\widetilde{\PL}_D(r,1)}\abs{(P-Q)(\phi)}\leq\sup_{\phi\in\widetilde{\PL}_D(r,1)}(P+Q)(\abs{\phi})\leq\int_{\R^D}(\norm{x}+\norm{x}^r)\,d(P+Q)(x)<\infty
\end{equation}
for all $P,Q\in\mathcal{P}(r)$. If $\phi\in\widetilde{\PL}_D(r,1)$, then
\[\bigl|\phi(x)-\phi(y)\bigr|\leq\norm{x-y}\bigl(1+\norm{x}^{r-1}+\norm{y}^{r-1}\bigr)\leq (1+2N^{r-1})\norm{x-y}\]
for all $x,y\in B_N$, so $\restr{\phi}{B_N}$ belongs to the set of $(1+2N^{r-1})$-Lipschitz functions $g\colon B_N\to\R$ satisfying $g(0)=0$, which we denote by $\mathcal{G}_N$. Since $B_N$ is compact and $\mathcal{G}_N$ is uniformly bounded and equicontinuous, $\mathcal{G}_N$ is therefore compact for the supremum norm on $B_N$ by the Arzel\`a--Ascoli theorem~\citep[e.g.][Theorem~2.4.7]{Dud02}. It is therefore totally bounded, so for each $m\in\N$, we can find a finite subset $\widetilde{\mathcal{H}}_{N,m}\subseteq\mathcal{G}_N$ such that for any $g\in\mathcal{G}_N$, there exists $\tilde{h}\in\widetilde{\mathcal{H}}_{N,m}$ with $\sup_{x \in B_N}|g(x)-\tilde{h}(x)|<1/m$. Each $\tilde{h}\in\widetilde{\mathcal{H}}_{N,m}$ can be associated with a function $h\colon\R^D\to\R$ defined by $h(x):=\tilde{h}(x)f_N(x)$. By the reasoning in the previous paragraph, the collection $\mathcal{H}_{N,m}$ of all such $h$ is a finite set of bounded Lipschitz functions supported on $B_N$.

Consequently, $T_r':=\bigcup_{N,m\in\N}\mathcal{H}_{N,m}$ is a countable set of bounded Lipschitz functions on $\R^D$, and we claim that this has the desired property that $\widetilde{d}_r(P,Q)=\sup_{\psi\in T_r'}\,\abs{(P-Q)(\psi)}$ for any two probability measures $P,Q\in\mathcal{P}(r)$. Indeed, for fixed $P,Q\in\mathcal{P}(r)$, the function $\psi_r\colon x\mapsto\norm{x}+\norm{x}^r$ is integrable with respect to both $P$ and $Q$ on $\R^D$, so by the dominated convergence theorem, we have $P\bigl(\psi_r\mathbbm{1}_{B_{N-1}^c}\bigr)\to 0$ and $Q\bigl(\psi_r\mathbbm{1}_{B_{N-1}^c}\bigr)\to 0$ as $N\to\infty$. Thus, for an arbitrary $\varepsilon>0$, there exists a sufficiently large $N\equiv N_{\varepsilon,r}\in\N$ such that $P\bigl(\psi_r\mathbbm{1}_{B_{N-1}^c}\bigr)<\varepsilon/4$ and $Q\bigl(\psi_r\mathbbm{1}_{B_{N-1}^c}\bigr)<\varepsilon/4$. Choosing $m\equiv m_\varepsilon\in\N$ such that $1/m<\varepsilon/4$, we deduce from the previous paragraph that for any $\phi\in\widetilde{\PL}_D(r,1)$, there exists $\tilde{h}\in\widetilde{\mathcal{H}}_{N,m}$ such that $\sup_{x \in B_N}|\phi(x)-\tilde{h}(x)|<1/m<\varepsilon/4$. Letting $h$ be the corresponding function in $\mathcal{H}_{N,m}\subseteq T_r'$, we have
\begin{align}
\abs{(P-Q)(\phi)}&\leq\abs{(P-Q)\bigl(\phi(1-f_N)\bigr)}+\abs{(P-Q)(\phi f_N-h)}+\abs{(P-Q)(h)}\notag\\
\label{eq:Wrcount}
&\leq(P+Q)\bigl(\abs{\phi}(1-f_N)\bigr)+(P+Q)(\abs{\phi f_N-h})+\textstyle\sup_{\psi\in T_r'}\abs{(P-Q)(\psi)}
\end{align}
by the triangle inequality. Since $\phi\in\widetilde{\PL}_D(r,1)$, we have $\abs{\phi(x)}=\abs{\phi(x)-\phi(0)}\leq\norm{x}+\norm{x}^r=\psi_r(x)$ for all $x\in\R^D$, whence 
\begin{align*}
(P+Q)\bigl(\abs{\phi}(1-f_N)\bigr)&\leq (P+Q)\bigl(\abs{\phi}\mathbbm{1}_{B_{N-1}^c}\bigr)\leq (P+Q)\bigl(\psi_r\mathbbm{1}_{B_{N-1}^c}\bigr)<\varepsilon/2
\end{align*}
by our choice of $N$ and the fact that $0\leq 1-f_N\leq\mathbbm{1}_{B_{N-1}^c}$. Moreover, 
\begin{align*}
(P+Q)(\abs{\phi f_N-h)}&\leq 2\textstyle\sup_{x \in B_N}|\phi(x)f_N(x)-h(x)| \leq 2\textstyle\sup_{x \in B_N}|\phi(x)-\tilde{h}(x)|<2/m<\varepsilon/2
\end{align*}
by our choice of $h$, so it follows from~\eqref{eq:Wrcount} that $\abs{(P-Q)(\phi)}<\varepsilon+\sup_{\psi\in T_r'}\,\abs{(P-Q)(\psi)}$. Since this holds for every $\phi\in\widetilde{\PL}_D(r,1)$ and all $\varepsilon>0$, the result follows.
\end{proof}

\unparskip
\begin{theorem}
\label{thm:Wr}
Let $P\in\mathcal{P}_D(r)$ and let $(P_n)$ be a sequence of probability measures in $\mathcal{P}_D(r)$. Then there exists a countable set $T_r\subseteq\PL_D(r)$ such that the following are equivalent:

\unparskip
\begin{enumerate}[label=(\roman*)]
\item $\int_{\R^D}\psi\,dP_n\to\int_{\R^D}\psi\,dP$ for all $\psi\in T_r$;
\item $\widetilde{d}_r(P_n,P)\to 0$;
\item $d_r(P_n,P)\to 0$.
\end{enumerate}

\unparskip
A suitable set $T_r\subseteq\PL_D(r)$ can be constructed by enlarging the set $T_r'$ of bounded Lipschitz functions defined in (the proof of) Proposition~\ref{prop:Wrcount}. 
\end{theorem}
\begin{remark}
\label{rem:Wrmetric}
Using Theorem~\ref{thm:Wr}, we can verify that $\widetilde{d}_r$ is a metric on $\mathcal{P}(r)\equiv\mathcal{P}_D(r)$ that generates the same topology as $d_r$. Indeed, it is clear from~\eqref{eq:Wrtilde} and~\eqref{eq:Wrtildefinite} that $\widetilde{d}_r$ takes values in $[0,\infty)$ and satisfies the triangle inequality on $\mathcal{P}(r)$. In addition, if $P,Q\in\mathcal{P}(r)$ are such that $\widetilde{d}_r(P,Q)=0$, then by taking $P_n=Q$ for all $n$ in (ii) above, we deduce that $d_r(P,Q)=0$. Since $d_r$ is a metric on $\mathcal{P}(r)$~\citep[e.g.][Theorem~7.3]{Vil03}, this yields $P=Q$, as required. In fact, $\bigl(\mathcal{P}(r),d_r\bigr)$ is a separable, complete metric space~\citep[e.g.][Theorem~2.2.7 and Proposition~2.2.8]{PZ20}, so by the equivalence (ii) $\Leftrightarrow$ (iii) in Theorem~\ref{thm:Wr}, the same is true of $\bigl(\mathcal{P}(r),\widetilde{d}_r\bigr)$.
\end{remark}

\deparskip
\begin{proof}
(i) $\Rightarrow$ (ii): As in the proof of Proposition~\ref{prop:Wrcount}, the function $f_N\colon x\mapsto (N-\norm{x})\vee 0\wedge 1$ once again serves as a Lipschitz surrogate for the indicator function $\mathbbm{1}_{B_N}$ of $B_N\equiv\bar{B}_D(0,N)=\{x\in\R^D:\norm{x}\leq N\}$ for each $N\in\N$ in the argument below; note that $f_N=1$ on $B_{N-1}$, $f_N=0$ on $B_N^c$ and $f_N$ is 1-Lipschitz on $\R^D$. In view of this and the fact that $\psi_r\colon x\mapsto\norm{x}+\norm{x}^r$ belongs to $\PL_D(r)$, the function $\psi_r(1-f_N)$ also lies in $\PL_D(r)$ for every $N\in\N$.

Let $\widetilde{\mathcal{H}}_{N,m}$ and $\mathcal{H}_{N,m}$ be the finite sets constructed in the proof of Proposition~\ref{prop:Wrcount} for each $N,m\in\N$, and let $T_r':=\bigcup_{N,m\in\N}\mathcal{H}_{N,m}$. Since $T_r'$ is a set of bounded Lipschitz functions, we certainly have $T_r'\subseteq\PL_D(r)$. We claim that $T_r:=T_r'\cup\{\psi_r(1-f_N):N\in\N\}$ is a countable subset of $\PL_D(r)$ with the required property. To see this, suppose that (i) holds for this set $T_r$, i.e.\ that $P_n(\psi)\to P(\psi)$ for all $\psi\in T_r$. As noted in~\eqref{eq:Wrtildefinite}, we have $\widetilde{d}_r(P_n,P)=\sup_{\phi\in\widetilde{\PL}_D(r,1)}\abs{(P_n-P)(\phi)}$ for all $n$, where $\widetilde{\PL}_D(r,1)$ denotes the set of all $\phi\in\PL_D(r,1)$ satisfying $\phi(0)=0$, so it suffices to show that the latter quantity converges to 0.

We will consider a decomposition~\eqref{eq:Wr1} similar to~\eqref{eq:Wrcount} in the proof of Proposition~\ref{prop:Wrcount}, taking particular care in this instance to ensure that the subsequent bounds hold uniformly over $\phi\in\widetilde{\PL}_D(r,1)$. Observe that since $\psi_r(1-f_N)\to 0$ pointwise on $\R^D$ as $N\to\infty$, and $\psi_r(1-f_N)$ is dominated by the $P$-integrable function $\psi_r$ on $\R^D$ for each $N$, we have $P\bigl(\psi_r(1-f_N)\bigr)\to 0$ as $N\to\infty$ by the dominated convergence theorem. Thus, for an arbitrary $\varepsilon>0$, there exists a sufficiently large $N\equiv N_{\varepsilon,r}\in\N$ such that $P\bigl(\psi_r(1-f_N)\bigr)<\varepsilon/4$, and we also fix $m\equiv m_\varepsilon\in\N$ such that $1/m<\varepsilon/4$. With this choice of $N$ and $m$, it follows from the defining property of $\tilde{\mathcal{H}}_{N,m}$ that for any $\phi\in\widetilde{\PL}_D(r,1)$, there exists $\tilde{h}_\phi\in\widetilde{\mathcal{H}}_{N,m}$ such that $\sup_{x \in B_N}|\phi(x)-\tilde{h}_\phi(x)|<1/m<\varepsilon/4$. Letting $h_\phi$ be the corresponding function in $\mathcal{H}_{N,m}$ as above, we have
\begin{align}
\abs{(P_n-P)(\phi)}&\leq\bigl|(P_n-P)\bigl(\phi(1-f_N)\bigr)\bigr|+\abs{(P_n-P)(\phi f_N-h_\phi)}+\abs{(P_n-P)(h_\phi)}\notag\\
\label{eq:Wr1}
&\leq(P_n+P)\bigl(\abs{\phi}(1-f_N)\bigr)+(P_n+P)(\abs{\phi f_N-h_\phi})+\max_{\psi\in\mathcal{H}_{N,m}}\abs{(P_n-P)(\psi)}
\end{align}
by the triangle inequality. Now for every $\phi\in\widetilde{\PL}_D(r,1)$, we have $\abs{\phi(x)}=\abs{\phi(x)-\phi(0)}\leq\norm{x}+\norm{x}^r=\psi_r(x)$ for all $x\in\R^D$. Since $P_n\bigl(\psi_r(1-f_N)\bigr)\to P\bigl(\psi_r(1-f_N)\bigr)$ as $n\to\infty$ by assumption, this implies that
\begin{equation}
\label{eq:Wr2}
\limsup_{n\to\infty}\sup_{\phi\in\widetilde{\PL}_D(r,1)} (P_n+P)\bigl(\abs{\phi}(1-f_N)\bigr)\leq\limsup_{n\to\infty}\,(P_n+P)\bigl(\psi_r(1-f_N)\bigr)=2P\bigl(\psi_r(1-f_N)\bigr)<\varepsilon/2.
\end{equation}
Moreover, for any $\phi\in\widetilde{\PL}_D(r,1)$, the functions $\phi f_N$ and $h_\phi$ are both supported on $B_N$, and $\abs{\phi f_N-h_\phi}=|\phi-\tilde{h}_\phi|\,f_N\leq|\phi-\tilde{h}_\phi|<\varepsilon/4$ on $B_N$, so
\begin{equation}
\label{eq:Wr3}
\limsup_{n\to\infty}\sup_{\phi\in\widetilde{\PL}_D(r,1)}(P_n+P)(\abs{\phi f_N-h_\phi})\leq 2\sup_{\phi\in\widetilde{\PL}_D(r,1)}\sup_{x \in B_N}\,\abs{\phi(x) f_N(x)-h_\phi(x)}<\varepsilon/2.
\end{equation}
Finally, since $\mathcal{H}_{N,m}$ is finite and $P_n(\psi)\to P(\psi)$ for all $\psi\in\mathcal{H}_{N,m}\subseteq T_r$ by assumption, we have $\max_{\psi\in\mathcal{H}_{N,m}}\abs{(P_n-P)(\psi)}\to 0$. Combining this with~\eqref{eq:Wr1},~\eqref{eq:Wr2} and~\eqref{eq:Wr3}, we conclude that
\[\limsup_{n\to\infty}\,\widetilde{d}_r(P_n,P)=\limsup_{n\to\infty}\sup_{\phi\in\widetilde{\PL}_D(r,1)}\abs{(P_n-P)(\phi)}<\varepsilon/2+\varepsilon/2=\varepsilon.\]
Since $\varepsilon>0$ was arbitrary, the desired conclusion follows.

(ii) $\Rightarrow$ (iii): Suppose that $\widetilde{d}_r(P_n,P) \rightarrow 0$ and let $\psi\colon\R^D\to\R$ be a (bounded) $L$-Lipschitz function, for some $L > 0$. Then $\tilde{\psi}(\cdot) := \psi(\cdot)/L \in \PL_D(r,1)$, so $P_n(\psi)=LP_n(\tilde{\psi})\rightarrow LP(\tilde{\psi}) =P(\psi)$. Hence $P_n \stackrel{d}{\rightarrow} P$. Moreover, the function $x\mapsto\|x\|^r$ belongs to $\PL(r,(r/2)\vee 1)$ since by Lemma~\ref{lem:plpower} below,
\begin{equation}
\label{eq:plpower}
\bigl|\norm{x}^r-\norm{y}^r\bigr|\leq\frac{r\vee 2}{2}\,\bigl|\norm{x}-\norm{y}\bigr|\bigl(\norm{x}^{r-1}+\norm{y}^{r-1}\bigr)\leq\frac{r\vee 2}{2}\norm{x-y}\bigl(\norm{x}^{r-1}+\norm{y}^{r-1}\bigr)
\end{equation}
for all $x,y\in\R^D$, so $\int_{\R^D} \|x\|^r \, dP_n(x) \rightarrow \int_{\R^D} \|x\|^r \, dP(x)$. We conclude that $d_r(P_n,P) \rightarrow 0$.

(iii) $\Rightarrow$ (i): We will show here that if (iii) holds, then $P_n(\psi)\to P(\psi)$ for all $\psi\in\PL_D(r)$. Indeed, suppose that $P_n\stackrel{d}{\rightarrow} P$ and $\int_{\R^D}\norm{x}^r\,dP_n(x)\to\int_{\R^D}\norm{x}^r\,dP(x)$. Now for $L>0$ and any $\psi\in\widetilde{\PL}_D(r,L)$, we have $\abs{\psi(x)}\leq L\norm{x}(1+\norm{x}^{r-1})\leq 2L(1+\norm{x}^r)$ for all $x\in\R^D$. Thus, since $\psi$ is continuous on $\R^D$ and $x\mapsto\abs{\psi(x)}/(1+\norm{x}^r)$ is bounded on $\R^D$, it follows from (iii) and~\citet[Lemma~4.5]{DSS11} that $P_n(\psi)\to P(\psi)$.
\end{proof}
\begin{remark}
The proof of the implication (i) $\Rightarrow$ (ii) in Theorem~\ref{thm:Wr} is similar to the argument in~\citet{Dud02} showing that (b) implies (c) in his Theorem~11.3.3, where it is established that the bounded Lipschitz metric induces the topology of weak convergence (of probability measures on a separable metric space).
\end{remark}

\unparskip
To obtain a sharp pseudo-Lipschitz constant for $x\mapsto\norm{x}^r$ in~\eqref{eq:plpower} above, we apply the following elementary inequality.
\begin{lemma}
\label{lem:plpower}
If $a,b\geq 0$ and $r\geq 1$, then $\abs{a^r-b^r}\leq\max(1,r/2)\,\abs{a-b}\,(a^{r-1}+b^{r-1})$.
\end{lemma}

\deparskip
\begin{proof}
Suppose without loss of generality that $0\leq b\leq a$. If $r\geq 2$, then $t\mapsto rt^{r-1}$ is convex on $[0,\infty)$, so
\begin{align*}
a^r-b^r=\int_a^b rt^{r-1}\,dt\leq\int_a^b r\rbr{\frac{t-b}{a-b}\,a^{r-1}+\frac{a-t}{a-b}\,b^{r-1}}dt=\frac{r}{2}(a-b)(a^{r-1}+b^{r-1}).
\end{align*}
If $r\in [1,2]$, then $0\leq (ab)^{r-1}(a^{2-r}-b^{2-r})=ab^{r-1}-ba^{r-1}$, so $a^r-b^r\leq (a-b)(a^{r-1}+b^{r-1})$.
\end{proof}

\unparskip
When we have a sequence of possibly random probability measures $P_n\equiv P_n(\omega)$ on $\R^D$, we can apply the deterministic Theorem~\ref{thm:Wr} to obtain Corollary~\ref{cor:Wr} below, in which we equip $\mathcal{P}(r)$ with the Borel $\sigma$-algebra $\mathcal{B}_r\equiv\mathcal{B}\bigl(\mathcal{P}(r)\bigr)$ associated with the $d_r$ (or equivalently the $\widetilde{d}_r$) metric. Note that $\widetilde{d}_r(P_n,P)$ is measurable (i.e.\ a bona fide random variable) for each $n$ by Proposition~\ref{prop:Wrcount}. The measurability of $d_r(P_n,P)$ is guaranteed by~\citet[Corollary~5.22]{Vil09}; see also~\citet[Lemma~2.4.6]{PZ20}.
\begin{corollary}
\label{cor:Wr}
Fix $P\in\mathcal{P}(r)\equiv\mathcal{P}_D(r)$ and let $(P_n)$ be a sequence of random elements $P_n\colon\Omega\to\mathcal{P}(r)$.

\unparskip
\begin{enumerate}[label=(\alph*)]
\item 
Then the following are equivalent:

\unparskip
\begin{enumerate}[label=(\roman*)]
\item $\int_{\R^D}\psi\,dP_n\cvas\int_{\R^D}\psi\,dP$ for every $\psi\in\PL_D(r)$;
\item $\widetilde{d}_r(P_n,P)\cvas 0$;
\item $d_r(P_n,P)\cvas 0$.
\end{enumerate}

\unparskip
\item The same equivalences hold if the mode of convergence in (i)--(iii) is instead taken to be either convergence in probability or complete convergence.
\end{enumerate}

\unparskip
\end{corollary}

\unparskip
Thus, to establish the seemingly stronger conclusions in (ii) and (iii) for a random sequence of distributions $P_n$, a putative limit $P\in\mathcal{P}_D(r)$ and any of the above modes of stochastic convergence, it is sufficient (and sometimes more convenient) to show that the appropriate version of (i) holds for each $\psi\in\PL_D(r)$ in turn. This is the approach we take in the proofs of the master theorems for symmetric AMP (Theorems~\ref{thm:AMPmaster} and~\ref{thm:masterext}).

\unparskip
\begin{proof}
(a) The implications (ii) $\Rightarrow$ (iii) $\Rightarrow$ (i) are immediate from Theorem~\ref{thm:Wr}. As for (i) $\Rightarrow$ (ii), note that for each $\psi\in\PL_D(r)$ in (i), the event $\Omega(\psi)$ of probability 1 on which $\int_{\R^D}\psi\,dP_n\to\int_{\R^D}\psi\,dP$ may depend (a priori) on $\psi$. The key point is that under (i), Theorem~\ref{thm:Wr} ensures that this convergence is actually uniform over $\PL_D(r,1)$ on a \emph{countable} intersection of such events $\Omega(\psi)$. More precisely, letting $T_r\subseteq\PL_D(r)$ be as in Theorem~\ref{thm:Wr}, we see that $\bigcap_{\,\psi\in T_r}\Omega(\psi)$ is an event of probability 1 on which (ii) and (iii) hold.

(b) \emph{Convergence in probability}:

(i) $\Rightarrow$ (ii): First, we prove that if $P_n(\psi)\cvp P(\psi)$ for each $\psi\in\PL_D(r)$, then $\widetilde{d}_r(P_n,P)\cvp 0$, or equivalently that every subsequence of $\bigl(\widetilde{d}_r(P_n,P):n\in\N\bigr)$ has a further subsequence that converges almost surely to 0. It suffices to show that for any subsequence $(Q_k)\equiv(P_{n_k})$, there is a further subsequence $(Q_{k_\ell})$ such that with probability 1, we have $Q_{k_\ell}(\psi)\to P(\psi)$ for all $\psi\in T_r\subseteq\PL_D(r)$; indeed, the desired conclusion then follows directly from (a). To this end, enumerate the elements of the countable set $T_r$ as $\psi_1,\psi_2,\dotsc$ and apply a diagonal argument: since $Q_k(\psi_1)\cvp P(\psi_1)$, we can extract a subsequence $(Q_{k_{1,\ell}})$ of $(Q_k)$ such that $Q_{k_{1,\ell}}(\psi_1)\cvas P(\psi_1)$ as $\ell\to\infty$. Continuing inductively, we see that for each $J\in\N$, there exist a subsequence $(Q_{k_{J,\ell}})$ of $(Q_{k_{J-1,\ell}})$ and an event of probability 1 on which $Q_{k_{J,\ell}}(\psi_j)\to P(\psi_j)$ as $\ell\to\infty$ for all $1\leq j\leq J$. Finally, let $Q_{k_\ell}:=Q_{k_{\ell,\ell}}$ for $\ell\in\N$, and observe that with probability 1, we have $Q_{k_{J,\ell}}(\psi_j)\to P(\psi_j)$ as $\ell\to\infty$ for all $j\in\N$, as required.

(ii) $\Rightarrow$ (iii) $\Rightarrow$ (i): As above, we can argue along subsequences of $(P_n)$ and then appeal directly to the corresponding implications in (a).

\emph{Complete convergence}:

(i) $\Rightarrow$ (ii): Suppose that $P_n(\psi)\cvc P(\psi)$ for every $\psi\in\PL_D(r)$. In view of Definition~\ref{def:compconv} of complete convergence, it suffices to show that if $(\beta_n)$ is any sequence of random variables with $\beta_n\eqd\widetilde{d}_r(P_n,P)$ for each $n$, then $\beta_n\cvas 0$. For any such sequence $(\beta_n)$, we first seek to construct a sequence $\bigl(\tilde{P}_n)$ of random elements $\tilde{P}_n\colon\Omega\to\bigl(\mathcal{P}(r),\widetilde{d}_r\bigr)$ such that $\tilde{P}_n\eqd P_n$ on $\bigl(\mathcal{P}(r),\mathcal{B}_r\bigr)$ for each $n$ and $\bigl(\widetilde{d}_r(\tilde{P}_n,P):n\in\N\bigr)=(\beta_n:n\in\N)$ almost surely as random sequences. Since $\bigl(\mathcal{P}(r),\widetilde{d}_r)$ is a Polish space, a suitable $\bigl(\tilde{P}_n)$ can be obtained by applying Lemma~\ref{lem:seqcoup}, where for each $n$, we take $g_n\colon\bigl(\mathcal{P}(r),\widetilde{d}_r\bigr)\to\R$ to be the 1-Lipschitz (and hence Borel measurable) function $Q\mapsto\widetilde{d}_r(P,Q)$.

For each $\psi\in\PL_D(r)$, we see from the definition of $\widetilde{d}_r$ in~\eqref{eq:Wrtilde} that $Q\mapsto Q(\psi)=\int_{\R^D}\psi\,dQ$ is also a 1-Lipschitz (and hence Borel measurable) function from $\bigl(\mathcal{P}(r),\widetilde{d}_r\bigr)$ to $\R$, so $\tilde{P}_n(\psi)\colon\Omega\to\R$ is measurable (i.e.\ a random variable). Now $\tilde{P}_n\eqd P_n$ for each $n$ by construction, so for every $\psi\in\PL_D(r)$, it follows that $\tilde{P}_n(\psi)\eqd\tilde{P}_n(\psi)$ for each $n$ and hence that $\tilde{P}_n(\psi)\cvas P(\psi)$. Thus, by the implication (i) $\Rightarrow$ (ii) in (a) above, we conclude that $\beta_n=\widetilde{d}_r(\tilde{P}_n,P)\to 0$ almost surely, as required.

(ii) $\Rightarrow$ (iii) $\Rightarrow$ (i): To establish these remaining implications, observe that it suffices to show the following: if $F_n,G_n\colon\mathcal{P}(r)\to\R$ are Borel measurable functions for which it is known from (a) that $F_n(P_n)\cvas 0$ implies $G_n(P_n)\cvas 0$, then $F_n(P_n)\cvc 0$ implies $G_n(P_n)\cvc 0$. To prove this, we can proceed as in the argument for (i) $\Rightarrow$ (ii): given any random sequence $(\beta_n)$ such that $\beta_n\eqd G_n(P_n)$ for each $n$, Lemma~\ref{lem:seqcoup} yields a sequence $(\tilde{P}_n)$ of random elements $\tilde{P}_n\colon\Omega\to\mathcal{P}(r)$ such that $F_n(\tilde{P}_n)\eqd F_n(P_n)$ and $\beta_n=G_n(\tilde{P}_n)$ almost surely for each $n$. Then $F_n(\tilde{P}_n)\cvas 0$, so (a) implies that $\beta_n=G_n(\tilde{P}_n)\to 0$ almost surely. This completes the proof.
\end{proof}

\unparskip
We conclude this subsection with some straightforward results on pseudo-Lipschitz functions.
\begin{lemma}
\label{lem:pseudoprod}
For $D\in\N$, if $f\in\PL_D(r)$ and $g\in\PL_D(s)$ for some $r,s\geq 1$, then $fg\in\PL_D(r+s)$ and $\abs{f}^p\in\PL_D(pr)$ for all $p\geq 1$.
\end{lemma}

\deparskip
\begin{proof}
There exists $L>0$ such that $f\in\PL_D(r,L)$ and $g\in\PL_D(s,L)$. Letting $L':=L\vee\abs{f(0)}\vee\abs{g(0)}$, we have 
\begin{equation}
\label{eq:pseudoprod}
\begin{alignedat}{2}
\abs{f(x)}&\leq\abs{f(0)}+\abs{f(x)-f(0)}&&\leq L'(1+\norm{x}+\norm{x}^r)\leq 2L'(1+\norm{x}^r)\\ \abs{g(x)}&\leq\abs{g(0)}+\abs{g(x)-g(0)}&&\leq L'(1+\norm{x}+\norm{x}^s)\leq 2L'(1+\norm{x}^s)
\end{alignedat}
\end{equation}
for all $x\in\R^D$. Therefore, fixing arbitrary $x,y\in\R^D$ and setting $a:=\norm{x}\vee\norm{y}$, we see that
\begin{align*}
&\abs{f(x)g(x)-f(y)g(y)}\\
&\hspace{1cm}\leq\abs{f(x)}\,\abs{g(x)-g(y)}+\abs{g(x)}\,\abs{f(x)-f(y)}\\
&\hspace{1cm}\leq 2L'L\,\norm{x-y}\,\bigl\{\bigl(1+\norm{x}^r\bigr)\bigl(1+\norm{x}^{s-1}+\norm{y}^{s-1}\bigr)+\bigl(1+\norm{x}^s\bigr)\bigl(1+\norm{x}^{r-1}+\norm{y}^{r-1}\bigr)\bigr\}\\
&\hspace{1cm}\leq 2L'L\,\norm{x-y}\,(2+2a^{r-1}+a^r+2a^{s-1}+a^s+4a^{r+s-1})\\
&\hspace{1cm}\leq 20L'L\,\norm{x-y}\,(1+a^{r+s-1})\\
&\hspace{1cm}\leq 20L'L\,\norm{x-y}\,(1+\norm{x}^{r+s-1}+\norm{y}^{r+s-1}).
\end{align*}
This shows that $fg\in\PL_D(r+s)$. For $p\geq 1$, we have $(a+b)^{p-1}\leq (1\vee 2^{p-2})(a^{p-1}+b^{p-1})$ for $a,b\geq 0$, and it follows from Lemma~\ref{lem:plpower} and~\eqref{eq:pseudoprod} that
\begin{align*}
\bigl|\abs{f(x)}^p-\abs{f(y)}^p\bigr|&\leq\frac{p\vee 2}{2}\,\abs{f(x)-f(y)}\,\bigl(\abs{f(x)}^{p-1}+\abs{f(y)}^{p-1}\bigr)\\
&\leq\frac{L(2L')^{p-1}(p\vee 2)}{2}\,\norm{x-y}\,(1+\norm{x}^r+\norm{y}^r)\bigl((1+\norm{x}^r)^{p-1}+(1+\norm{y}^r)^{p-1}\bigr)\\
&\lesssim_p L(L')^{p-1}\,\norm{x-y}\,(1+\norm{x}^r+\norm{y}^r)\bigl(1+\norm{x}^{(p-1)r}+\norm{y}^{(p-1)r}\bigr)\\
&\lesssim_p L(L')^{p-1}\,\norm{x-y}\,(1+\norm{x}^{pr}+\norm{y}^{pr})
\end{align*}
for all $x,y\in\R^D$. Thus, $\abs{f}^p\in\PL_D(pr)$, as required.
\end{proof}

\unparskip
\begin{lemma}
\label{lem:pseudocomp}
Let $\psi\in\PL_{D+1}(r,L)$ for some $D\in\N$, $r\geq 1$ and $L>0$. Fix $c\equiv (c_1,\dotsc,c_D)\in\R^D$ and $\tau>0$.

\unparskip
\begin{enumerate}[label=(\alph*)]
\item For fixed $x\equiv (x_1,\dotsc,x_D)\in\R^D$, define $\psi_x\colon\R\to\R$ by $\psi_x(z):=\psi\bigl(x_1,\dotsc,x_D,\sum_{\ell=1}^D c_\ell\,x_\ell+\tau z\bigr)$. Then $\psi_x\in\PL_1(r,L_{\norm{x},\tau})$, where $L_{a,\tau}:=L\tau\max\{1+(2\vee 2^{r-1})(1+\norm{c})^{r-1}a^{r-1},(1\vee 2^{r-2})\tau^{r-1}\}$ for $a\geq 0$.
\item Let $Z\sim N(0,1)$ and define $\Psi\colon\R^D\to\R$ by $\Psi(x_1,\dotsc,x_D):=\E\bigl\{\psi\bigl(x_1,\dotsc,x_D,\sum_{\ell=1}^D c_\ell\,x_\ell+\tau Z\bigr)\bigr\}$. Then $\Psi\in\PL_D(r,L_\tau)$, where $L_\tau:=L(1+\norm{c})\max\{1+(2\vee 2^{r-1})\,\E(\abs{\tau Z}^{r-1}),(1\vee 2^{r-2})(1+\norm{c})^{r-1}\}$.
\end{enumerate}

\unparskip
\end{lemma}

\deparskip
\begin{proof}
For $x\equiv (x_1,\dotsc,x_D)\in\R^D$ and $z\in\R$, note first that
\begin{align}
\bigl\|\bigl(x_1,\dotsc,x_D,\textstyle\sum_{\ell=1}^D c_\ell\,x_\ell+\tau z\bigr)\bigr\|^{r-1}&\leq\bigl(\norm{x}+\abs{\textstyle\sum_{\ell=1}^D c_\ell\,x_\ell}+\tau\abs{z}\bigr)^{r-1}\notag\\
\label{eq:normxz1}
&\leq\{(1+\norm{c})\norm{x}+\tau\abs{z}\}^{r-1}\\
\label{eq:normxz2}
&\leq(1\vee 2^{r-2})\bigl\{(1+\norm{c})^{r-1}\norm{x}^{r-1}+\tau^{r-1}\abs{z}^{r-1}\bigr\},
\end{align}
where the three bounds above are obtained using the triangle inequality, the Cauchy--Schwarz inequality and the fact that $(a+b)^{r-1}\leq (1\vee 2^{r-2})(a^{r-1}+b^{r-1})$ for $a,b\geq 0$.

(a) For $z,z'\in\R$, we have
\begin{align*}
&\abs{\psi_x(z)-\psi_x(z')}\\
&\hspace{1cm}=\bigl|\psi\bigl(x_1,\dotsc,x_D,\textstyle\sum_{\ell=1}^D c_\ell\,x_\ell+\tau z\bigr)-\psi\bigl(x_1,\dotsc,x_D,\textstyle\sum_{\ell=1}^D c_\ell\,x_\ell+\tau z'\bigr)\bigl|\\
&\hspace{1cm}\leq L\tau\abs{z-z'}\,\bigl\{1+2(1\vee 2^{r-2})(1+\norm{c})^{r-1}\norm{x}^{r-1}+(1\vee 2^{r-2})\,\tau^{r-1}\bigl(\abs{z}^{r-1}+\abs{z'}^{r-1}\bigr)\bigr\}\\
&\hspace{1cm}\leq L_{\norm{x},\tau}\,\abs{z-z'}\,\bigl(1+\abs{z}^{r-1}+\abs{z'}^{r-1}\bigr),
\end{align*}
where the first bound follows from~\eqref{eq:normxz2} and the fact that $\psi\in\PL_{D+1}(r,L)$.

(b) For $x,y\in\R^D$, we have
\begin{align*}
&\abs{\Psi(x)-\Psi(y)}\\
&\hspace{1cm}\leq\E\bigl\{\bigl|\psi\bigl(x_1,\dotsc,x_D,\textstyle\sum_{\ell=1}^D c_\ell\,x_\ell+\tau Z\bigr)-\psi\bigl(y_1,\dotsc,y_D,\textstyle\sum_{\ell=1}^D c_\ell\,y_\ell+\tau Z\bigr)\bigr|\bigr\}\\
&\hspace{1cm}\leq L(1+\norm{c})\norm{x-y}\,\bigl\{1+2(1\vee 2^{r-2})\,\E(\abs{\tau Z}^{r-1})+(1\vee 2^{r-2})(1+\norm{c})^{r-1}\bigl(\norm{x}^{r-1}+\norm{y}^{r-1}\bigr)\bigr\}\\
&\hspace{1cm}\leq L_\tau\norm{x-y}\,\bigl(1+\norm{x}^{r-1}+\norm{y}^{r-1}\bigr),
\end{align*}
where the second bound again follows from~\eqref{eq:normxz1},~\eqref{eq:normxz2} and the fact that $\psi\in\PL_{D+1}(r,L)$.
\end{proof}

\unparskip
\begin{lemma}
\label{lem:pseudoavg}
Suppose that $\psi\in\PL_D(r,L)$ for some $D\in\N$, $r\in [2,\infty)$ and $L>0$. Then for any $n\in\N$ and vectors $x^\ell\equiv (x_1^\ell,\dotsc,x_n^\ell)$ and $y^\ell\equiv (y_1^\ell,\dotsc,y_n^\ell)$ for $1\leq\ell\leq D$, we have
\[\frac{1}{n}\sum_{i=1}^n\,\abs{\psi(x_i^1,\dotsc,x_i^D)-\psi(y_i^1,\dotsc,y_i^D)}\leq LD^{\frac{r}{2}-1}\rbr{\sum_{\ell=1}^D\norm{x^\ell-y^\ell}_{n,r}^r}^{1/r}\biggl(1+\sum_{\ell=1}^D\,\bigl(\norm{x^\ell}_{n,r}^{r-1}+\norm{y^\ell}_{n,r}^{r-1}\bigr)\biggr).\]
\end{lemma}

\deparskip
\begin{proof}
For $1\leq i\leq n$, define $X^{(i)}:=(x_i^1,\dotsc,x_i^D)$ and $Y^{(i)}:=(y_i^1,\dotsc,y_i^D)$, and let $r':=r/(r-1)\in (1,2]$ be the H\"older conjugate of $r$, so that $1/r+1/r'=1$. Then since $\psi\in\PL_D(r,L)$, an application of H\"older's inequality yields the bound
\begin{align}
&\frac{1}{n}\sum_{i=1}^n\,\abs{\psi(x_i^1,\dotsc,x_i^D)-\psi(y_i^1,\dotsc,y_i^D)}=\frac{1}{n}\sum_{i=1}^n\,\abs{\psi(X^{(i)})-\psi(Y^{(i)})}\notag\\
&\hspace{1.5cm}\leq\frac{1}{n}\sum_{i=1}^n L\norm{X^{(i)}-Y^{(i)}}\,\bigl(1+\norm{X^{(i)}}^{r-1}+\norm{Y^{(i)}}^{r-1}\bigr)\notag\\
\label{eq:pseudoavg1}
&\hspace{1.5cm}\leq L\,\biggl(\frac{1}{n}\sum_{i=1}^n\norm{X^{(i)}-Y^{(i)}}^r\biggr)^{1/r}\biggl(\frac{1}{n}\sum_{i=1}^n\,\bigl(1+\norm{X^{(i)}}^{r-1}+\norm{Y^{(i)}}^{r-1}\bigr)^{r'}\biggr)^{1/r'}.
\end{align}
Since $\norm{{\cdot}}\equiv\norm{{\cdot}}_2\leq D^{\frac{1}{2}-\frac{1}{r}}\norm{{\cdot}}_r$ on $\R^D$, we see that
\begin{equation}
\label{eq:pseudoavg2}
\frac{1}{n}\sum_{i=1}^n\norm{X^{(i)}-Y^{(i)}}^r\leq\frac{D^{\frac{r}{2}-1}}{n}\sum_{i=1}^n\sum_{\ell=1}^D\,\abs{x_i^\ell-y_i^\ell}^r=D^{\frac{r}{2}-1}\sum_{\ell=1}^D\norm{x^\ell-y^\ell}_{n,r}^r.
\end{equation}
In addition, by applying the triangle inequality for $\norm{{\cdot}}_{n,r'}$ and arguing as in~\eqref{eq:pseudoavg2}, we have
\begin{align}
\biggl(\frac{1}{n}\sum_{i=1}^n\,\bigl(1+\norm{X^{(i)}}^{r-1}+\norm{Y^{(i)}}^{r-1}\bigr)^{r'}\biggr)^{1/r'}&\leq 1+\biggl(\frac{1}{n}\sum_{i=1}^n\norm{X^{(i)}}^r\biggr)^{1/r'}+\biggl(\frac{1}{n}\sum_{i=1}^n\norm{Y^{(i)}}^r\biggr)^{1/r'}\notag\\
&\leq 1+\biggl(D^{\frac{r}{2}-1}\sum_{\ell=1}^D\norm{x^\ell}_{n,r}^r\biggr)^{\frac{r-1}{r}}+\biggl(D^{\frac{r}{2}-1}\sum_{\ell=1}^D\norm{y^\ell}_{n,r}^r\biggr)^{\frac{r-1}{r}}\notag\\
\label{eq:pseudoavg3}
&\leq 1+(D^{\frac{r}{2}-1})^{\frac{r-1}{r}}\sum_{\ell=1}^D\,\bigl(\norm{x^\ell}_{n,r}^{r-1}+\norm{y^\ell}_{n,r}^{r-1}\bigr),
\end{align}
where the final bound follows since $\norm{{\cdot}}_r\leq\norm{{\cdot}}_{r-1}$ on $\R^D$. Combining~\eqref{eq:pseudoavg1}--\eqref{eq:pseudoavg3} yields the desired conclusion.
\end{proof}

\umparskip


\begin{thebibliography}{75}
\bibitem[{Agresti}(2015)]{Agr15}
Agresti, A. (2015).
\newblock \emph{Foundations of Linear and Generalized Linear Models}.
\newblock Wiley, New Jersey.

\bibitem[{Albert and Anderson}(1984)]{AA84}
Albert, A. and Anderson J. A. (1984).
\newblock On the existence of maximum likelihood estimates in logistic regression models.
\newblock \emph{Biometrika}, \textbf{71}, 1--10.

\bibitem[{Aliprantis and Burkinshaw}(1998)]{AB98}
Aliprantis, C. D. and Burkinshaw, O. (1998).
\newblock \emph{Principles of Real Analysis}, 3rd edition.
\newblock Academic Press, San Diego.

\bibitem[{Alon et~al.}(1998)]{alon1998clique}
Alon, N., Krivelevich, M. and Sudakov, B. (1998).
\newblock Finding a large hidden clique in a random graph.
\newblock \textit{Random Struct. Algorithms}, \textbf{13}, 457--466.

\bibitem[{Anderson et al.}(2010)]{AGZ10}
Anderson, G., Guionnet, A. and Zeitouni, O. (2010). 
\newblock \emph{An Introduction to Random Matrices}.
\newblock Cambridge University Press, Cambridge.

\bibitem[{Bai and Silverstein(2010)}]{BaiSilverstein}
Bai, Z. and Silverstein, J. (2010).
\newblock \textit{Spectral Analysis of Large Dimensional Random Matrices}, 2nd edition.
\newblock Springer, New York.

\bibitem[{Baik et~al.(2005)}]{BBP05}
Baik, J., Ben~Arous, G. and P\'ech\'e, S. (2005).
\newblock Phase transition of the largest eigenvalue for nonnull complex sample covariance matrices.
\newblock \textit{Ann. Probab.}, \textbf{33}, 1643--1697.

\bibitem[{Baik and Silverstein(2006)}]{BS06}
Baik, J. and Silverstein, J.~W. (2006).
\newblock Eigenvalues of large sample covariance matrices of spiked population models.
\newblock \textit{J. Multivariate Anal.}, \textbf{97}, 1382--1408.

\bibitem[{Bakhshizadeh et al.}(2020)]{BMP20}
Bakhshizadeh, M., Maleki, A. and de la Pena, V. H. (2020). Sharp concentration results for heavy-tailed distributions. Available at \url{https://arxiv.org/pdf/2003.13819.pdf}.

\bibitem[{Barata and Hussein}(2012)]{BH12}
Barata, J. C. A. and Hussein, M. S. (2012). The Moore--Penrose pseudoinverse: A tutorial review of the theory.
\newblock \emph{Braz. J. Phys.}, \textbf{42}, 146--165.

\bibitem[{Barbier et~al.}(2016)]{barbier2016mutual}
Barbier, J., Dia, M., Macris, N., Krzakala, F., Lesieur, T. and Zdeborov\'a, L. (2016).
\newblock Mutual information for symmetric rank-one matrix estimation: a proof of the replica formula.
\newblock In \textit{Advances in Neural Information Processing Systems}, \textbf{29}, 424--432.

\bibitem[{Barbier and Krzakala(2017)}]{Barbier2017}
Barbier, J. and Krzakala, F. (2017).
\newblock Approximate message-passing decoder and capacity achieving sparse superposition codes.
\newblock \textit{IEEE Trans. Inf. Theory}, \textbf{63}, 4894--4927.

\bibitem[Barbier et~al.(2019)]{barbier2019optimal}
Barbier, J., Krzakala, F., Macris, N., Miolane, L., and Zdeborov\'a, L. (2019).
\newblock Optimal errors and phase transitions in high-dimensional generalized linear models.
\newblock \emph{Proc. Natl. Acad. Sci. U.S.A.}, \textbf{116}, 5451--5460.

\bibitem[Barbier et~al.(2020)]{BMR20}
Barbier, J., Macris, N. and Rush, C. (2020). 
\newblock All-or-nothing statistical and computational phase transitions in sparse spiked matrix estimation.
\newblock Available at \url{https://arxiv.org/pdf/2006.07971.pdf}.

\bibitem[{Bayati et~al.(2015)}]{bayati2015universality}
Bayati, M., Lelarge, M., Montanari, A. (2015).
\newblock Universality in polytope phase transitions and message passing algorithms.
\newblock \textit{Ann. Appl. Probab.}, \textbf{25}, 753--822.

\bibitem[{Bayati and Montanari}(2011)]{BM11}
Bayati, M. and Montanari, A. (2011). The dynamics of message passing on dense graphs, with applications to compressed sensing.
\newblock \emph{IEEE Trans.\ Inf.\ Theory}, \textbf{57}, 764--785.

\bibitem[{Bayati and Montanari(2012)}]{BM12}
Bayati, M. and Montanari, A. (2012).
\newblock The LASSO risk for Gaussian matrices.
\newblock \textit{IEEE Trans. Inf. Theory}, \textbf{58}, 1997--2017.

\bibitem[{Beck and Teboulle(2009)}]{BT09}
Beck, A. and Teboulle, M. (2009).
\newblock A fast iterative shrinkage-thresholding algorithm for linear inverse problems.
\newblock \emph{SIAM J. Imaging Sci}, \textbf{2}, 183--202.

\bibitem[{Bellec et~al.}(2018)]{bellec2018slope}
Bellec, P.~C., Lecu\'e, G. and Tsybakov, A.~B. (2018).
\newblock SLOPE meets LASSO: improved oracle bounds and optimality.
\newblock \textit{Ann. Statist.}, \textbf{46}, 3603--3642.

\bibitem[{Benaych-Georges and Nadakuditi}(2011)]{BN11}
Benaych-Georges, F. and Nadakuditi, R.~R. (2011).
\newblock The eigenvalues and eigenvectors of finite, low rank perturbations of large random matrices.
\newblock \textit{Adv. Math.}, \textbf{227}, 494--521.

\bibitem[{Berthier et~al.(2020)}]{BMN20}
Berthier, R., Montanari, A. and Nguyen, P.-M. (2020).
\newblock State evolution for approximate message passing with non-separable functions.
\newblock \textit{Inf. Inference}, \textbf{9}, 33--79.

\bibitem[{Blei et al.(2003)}]{BNJ2003}
Blei, D. M., Ng, A. Y. and Jordan, M. I. (2003). Latent Dirichlet allocation.
\newblock \emph{J. Mach. Learn. Res.}, \textbf{3}, 993--1022.

\bibitem[{Bogdan et~al.(2015)}]{bogdan2015slope}
Bogdan, M., van den Berg, E., Sabatti, C., Su, W. and Cand\`es, E. (2015).
\newblock SLOPE---Adaptive variable selection via convex optimization.
\newblock \emph{Ann. Appl. Stat.}, \textbf{9}, 1103--1140.

\bibitem[Bolthausen(2014)]{Bol14}
Bolthausen, E. (2014). An iterative construction of solutions of the TAP equations for the Sherrington--Kirkpatrick model.
\newblock \emph{Comm. Math. Phys.}, \textbf{325}, 333--366.

\bibitem[{Boucheron et al.}(2013)]{BLM13}
Boucheron, S., Lugosi, G. and Massart, P. (2013). \emph{Concentration Inequalities: A Nonasymptotic Theory of Independence}. 
\newblock Oxford University Press, Oxford.

\bibitem[{Boyd et~al.}(2011)]{boyd2011distributed}
Boyd, S., Parikh, N., Chu, E., Peleato, B. and Eckstein, J. (2011).
\newblock Distributed optimization and statistical learning via the alternating direction method of multipliers.
\newblock \textit{Found. Trends Mach. Learn.}, \textbf{3}, 1--122.

\bibitem[Brown and Purves(1973)]{BP73}
Brown, L. D. and Purves, R. (1973). Measurable selections of extrema.
\newblock \emph{Ann. Statist.}, \textbf{1}, 902--912.

\bibitem[{Bu et~al.}(2021)]{bu2019slope}
Bu, Z., Klusowski, J., Rush, C. and Su, W. (2021).
\newblock Algorithmic analysis and statistical estimation of SLOPE via approximate message passing.
\newblock \textit{IEEE Trans. Inf. Theory}, \textbf{67}, 506--537.

\bibitem[{B\"uhlmann and van de Geer}(2011)]{BvG11}
B\"uhlmann, P. and van de Geer, S. (2011).
\newblock \textit{Statistics for High-Dimensional Data: Methods, Theory and Applications}.
\newblock Springer, Berlin.

\bibitem[{\c{C}akmak and Opper}(2019)]{CO19}
\c{C}akmak, B. and Opper, M. (2019).
\newblock Memory-free dynamics for the Thouless--Anderson--Palmer equations of Ising models with arbitrary rotation-invariant ensembles of random coupling matrices.
\newblock \emph{Phys. Rev. E}, \textbf{99}, 062140.

\bibitem[{Cand\`es and Recht(2009)}]{candes2009exact}
Cand\`es, E. J. and Recht, B. (2009). Exact matrix completion via convex optimization. \emph{Found. Comput. Math.}, \textbf{9}, 717--772.

\bibitem[{Cand\`es and Sur(2020)}]{candes2020}
Cand\`es, E. J. and Sur, P. (2020). The phase transition for the existence of the maximum likelihood estimate in high-dimensional logistic regression.
\newblock \emph{Ann. Statist.}, \textbf{48}, 27--42.

\bibitem[{Capitaine et~al.(2009)}]{CDF09}
Capitaine, M., Donati-Martin, C. and F\'eral, D. (2009).
\newblock The largest eigenvalues of finite rank deformation of large Wigner matrices: convergence and nonuniversality of the fluctuations.
\newblock \textit{Ann. Probab.}, \textbf{37}, 1--47.

\bibitem[{Celentano and Montanari(2019)}]{CM19}
Celentano, M. and Montanari, A. (2019). 
\newblock Fundamental barriers to high-dimensional regression with convex penalties.
\newblock Available at \url{https://arxiv.org/pdf/1903.10603.pdf}.

\bibitem[{Celentano et al.(2020)}]{CMW20}
Celentano, M., Montanari, A. and Wu, Y. (2020). 
\newblock The estimation error of general first order methods.
\newblock \emph{Proc. Mach. Learn. Res.}, \textbf{125}, 1--64.

\bibitem[{Chen and Lam(2021)}]{CL20}
Chen, W-K. and Lam, W-K. (2021).
\newblock Universality of approximate message passing algorithms.
\newblock \textit{Electron. J. Probab.}, \textbf{26}, 1--44.

\bibitem[{Deshpande et~al.}(2016)]{deshpande2016asymptotic}
Deshpande, Y., Abbe, E. and Montanari, A. (2016).
\newblock Asymptotic mutual information for the balanced binary stochastic block model.
\newblock \textit{Inf. Inference}, \textbf{6}, 125--170.

\bibitem[{Deshpande and Montanari(2014)}]{Deshpande2014}
Deshpande, Y. and Montanari, A. (2014).
\newblock Information-theoretically optimal sparse PCA.
\newblock In \textit{2014 IEEE International Symposium on Information Theory}, pp.\ 2197--2201.

\bibitem[{Deshpande and Montanari(2015)}]{deshpande2015clique}
Deshpande, Y. and Montanari, A. (2015).
\newblock Finding hidden cliques of size $\sqrt{N/e}$ in nearly linear time.
\newblock \textit{Found. Comput. Math.}, \textbf{15}, 1069--1128.

\bibitem[{Donoho and Montanari(2015)}]{DM15}
Donoho, D. and Montanari, A. (2015). 
\newblock Variance breakdown of Huber (M)-estimators: $n/p\to m\in(1,\infty)$.
\newblock Available at \url{https://arxiv.org/pdf/1503.02106.pdf}.

\bibitem[{Donoho and Montanari(2016)}]{donohoMest16}
Donoho, D. and Montanari, A. (2016).
\newblock High dimensional robust M-estimation: asymptotic variance via approximate message passing.
\newblock \textit{Probab. Theory Related Fields}, \textbf{166}, 935--969.

\bibitem[{Donoho et~al.}(2013)]{donohoJM_SC2013}
Donoho, D.~L., Javanmard, A. and Montanari, A. (2013).
\newblock Information-theoretically optimal compressed sensing via spatial coupling and approximate message passing.
\newblock \textit{IEEE Trans. Inf. Theory}, \textbf{59}, 7434--7464.

\bibitem[{Donoho and Johnstone(1994)}]{DJ94a}
Donoho, D.~L. and Johnstone, I.~M. (1994).
\newblock Minimax risk over $l_p$ balls for $l_q$ error.
\newblock \textit{Prob. Theory Related Fields}, \textbf{99}, 277--303.

\bibitem[{Donoho and Johnstone(1998)}]{DJ98}
Donoho, D.~L. and Johnstone, I.~M. (1998).
\newblock Minimax estimation via wavelet shrinkage.
\newblock \textit{Ann. Statist.}, \textbf{26}, 879--921.

\bibitem[{Donoho et~al.}(2009)]{donohoMM2009}
Donoho, D.~L., Maleki, A. and Montanari, A. (2009).
\newblock Message-passing algorithms for compressed sensing.
\newblock \textit{Proc. Natl. Acad. Sci. U.S.A.}, \textbf{106}, 18914--18919.

\bibitem[{Dudley(2002)}]{Dud02}
Dudley, R. M. (2002). \emph{Real Analysis and Probability}, 2nd edition. 
\newblock Cambridge University Press, Cambridge.

\bibitem[{D\"umbgen et al.(2011)}]{DSS11}
D\"umbgen, L., Samworth, R. and Schuhmacher, D. (2011). Approximation by log-concave distributions, with applications to regression.
\newblock \emph{Ann. Statist.}, \textbf{39}, 702--730.

\bibitem[{D\"umbgen et al.(2021)}]{DSW21}
D\"umbgen, L., Samworth, R. J. and Wellner, J. A. (2021).  
\newblock Bounding distributional errors via density ratios. 
\newblock \emph{Bernoulli}, \textbf{27}, 818--852.

\bibitem[{Efron(2011)}]{Efr11}
Efron, B. (2011).
\newblock Tweedie's formula and selection bias.
\newblock \textit{J. Amer. Statist. Assoc.}, \textbf{106}, 1602--1614.

\bibitem[{Emami et al.}(2020)]{ESPRF20}
Emami, M., Sahraee-Ardakan, M., Pandit, P., Rangan, S. and Fletcher, A. K. (2020).
\newblock Generalization error of generalized linear models in high dimensions.
\newblock \emph{Proc. Mach. Learn. Res.}, \textbf{119}, 2892--2901.

\bibitem[{Fan(2020)}]{Fan20}
Fan, Z. (2020).
\newblock Approximate message passing algorithms for rotationally invariant matrices.
\newblock Available at \url{https://arxiv.org/pdf/2008.11892.pdf}.

\bibitem[Federer(1996)]{Fed96}
Federer, H. (1996). 
\newblock \emph{Geometric Measure Theory}.
\newblock Springer--Verlag, New York.

\bibitem[{F\'eral and P\'ech{\'e}(2007)}]{FP07}
F\'eral, D. and P\'ech\'e, S. (2007).
\newblock The largest eigenvalue of rank one deformation of large Wigner matrices.
\newblock \textit{Comm. Math. Phys.}, \textbf{272}, 185--228.

\bibitem[{Fletcher and Rangan(2014)}]{fletcher2014scalable}
Fletcher, A.~K. and Rangan, S. (2014).
\newblock Scalable inference for neuronal connectivity from calcium imaging.
\newblock In \textit{Advances in Neural Information Processing Systems}, \textbf{27}, 2843--2851.

\bibitem[{Fourdrinier et al.}(2018)]{FSW18}
Fourdrinier, D., Strawderman, W. E. and Wells, M. T. (2018). 
\newblock \emph{Shrinkage Estimation}. 
\newblock Springer, New York.

\bibitem[{Gataric et al.(2020)}]{gataric2020sparse}
Gataric, M., Wang, T. and Samworth, R. J. (2020). 
\newblock Sparse principal component analysis via axis-aligned random projections. 
\newblock \emph{J. Roy. Statist. Soc., Ser B}, \textbf{82}, 329--359.

\bibitem[Gordon(1994)]{Gor94}
Gordon, L. (1994). A stochastic approach to the gamma function. 
\newblock \emph{Am. Math. Mon.}, \textbf{101}, 858--865. 

\bibitem[{Guo and Verd\'u(2005)}]{GuoV05} 
Guo, D. and Verd\'u, S. (2005). 
Randomly spread {CDMA}: Asymptotics via statistical physics.
\newblock \emph{IEEE Trans. Inf. Theory}, \textbf{51}, 1983--2010.

\bibitem[{Hsu and Robbins}(1947)]{HR47}
Hsu, P. L. and Robbins, H. (1947). Complete convergence and the law of large numbers. 
\newblock \emph{Proc. Natl. Acad. Sci. U.S.A.}, \textbf{33}, 25--31.

\bibitem[Huber(1964)]{Hub64}
Huber, P. J. (1964). Robust estimation of a location parameter.
\newblock \emph{Ann. Math. Statist.}, \textbf{35}, 73--101.

\bibitem[Huber(1973)]{Hub73}
Huber, P. J. (1973). Robust regression: asymptotics, conjectures and Monte Carlo.
\newblock \emph{Ann. Statist.}, \textbf{1}, 799--821.

\bibitem[{Huber and Ronchetti}(2009)]{HR09}
Huber, P. J. and Ronchetti, E. (2009). 
\newblock \emph{Robust Statistics}, 2nd edition.
\newblock Wiley, New York.

\bibitem[{Kabashima and Vehkaper\"a}(2014)]{KV14}
Kabashima, Y. and Vehkaper\"a, M. (2014). 
\newblock Signal recovery using expectation consistent approximation for linear observations.
\newblock In \textit{2014 IEEE International Symposium on Information Theory}, pp.\ 226--230.

\bibitem[{Knowles and Yin}(2013)]{KY13}
Knowles, A. and Yin, J. (2013).
\newblock The isotropic semicircle law and deformation of Wigner matrices.
\newblock \emph{Comm. Pure Appl. Math.}, \textbf{66}, 1663--1749.

\bibitem[{Javanmard and Montanari}(2013)]{JM13}
Javanmard, A. and Montanari, A. (2013). State evolution for general approximate message passing algorithms, with applications to spatial coupling.
\newblock \emph{Inf. Inference}, \textbf{2}, 115--144.

\bibitem[{Jeon et al.}(2015)]{JGMS15}
Jeon, C., Ghods, R., Maleki, A. and Studer, C. (2015).
\newblock Optimality of large MIMO detection via approximate message passing.
\newblock In \textit{2015 IEEE International Symposium on Information Theory}, pp.\ 1227--1231.

\bibitem[{Johnstone(2006)}]{JohnstoneICM}
Johnstone, I.~M. (2006).
\newblock High Dimensional Statistical Inference and Random Matrices.
\newblock In \textit{Proceedings of the International Congress of Mathematicians, Madrid 2006}, pp.\ 307--333.

\bibitem[{Johnstone and Lu(2009)}]{johnstoneLu09consistency}
Johnstone, I.~M. and Lu, A.~Y. (2009).
\newblock On consistency and sparsity for principal components analysis in high dimensions.
\newblock \textit{J. Amer. Statist. Assoc.}, \textbf{104}, 682--693.

\bibitem[{Johnstone and Paul(2018)}]{JP18}
Johnstone, I.~M. and Paul, D. (2018).
\newblock PCA in high dimensions: an orientation.
\newblock \textit{Proc. IEEE}, \textbf{106}, 1277--1292.

\bibitem[{Jolliffe et al.(2003)}]{JTU03}
Jolliffe, I. T., Trendafilov, N. T. and Uddin, M. (2003). 
A modified principal component technique based on the LASSO.
\newblock \emph{J. Comput. Graph. Statist.}, \textbf{12}, 531--547. 

\bibitem[{Kabashima et~al.}(2016)]{kabashima2016phase}
Kabashima, Y., Krzakala, F., M\'ezard, M., Sakata, A. and Zdeborov\'a, L. (2016).
\newblock Phase transitions and sample complexity in Bayes optimal matrix factorization.
\newblock \textit{IEEE Trans. Inf. Theory}, \textbf{62}, 4228--4265.

\bibitem[Kallenberg(1997)]{Kal97}
Kallenberg, O. (1997). \emph{Foundations of Modern Probability}. 
\newblock Springer--Verlag, New York.

\bibitem[{Koller and Friedman(2009)}]{KF2009}
Koller, D. and Friedman, N. (2009). 
\newblock \emph{Probabilistic Graphical Models: Principles and Techniques}.
\newblock MIT Press, Cambridge, Massachusetts.

\bibitem[{Krzakala et~al.}(2012)]{krzakala2012}
Krzakala, F., M\'ezard, M., Sausset, F., Sun, Y. and Zdeborov\'a, L. (2012).
\newblock Probabilistic reconstruction in compressed sensing: algorithms, phase diagrams, and threshold achieving matrices.
\newblock \textit{J. Stat. Mech. Theory Exp.}, P08009.

\bibitem[{Kuchibhotla and Chakrabortty}(2018)]{KC18}
Kuchibhotla, A. and Chakrabortty A. (2018). Moving beyond sub-Gaussianity in high-dimensional statistics: applications in covariance estimation and linear regression. Available at \url{https://arxiv.org/pdf/1804.02605.pdf}.

\bibitem[{Lelarge and Miolane(2019)}]{lelarge2019fundamental}
Lelarge, M. and Miolane, L. (2019).
\newblock Fundamental limits of symmetric low-rank matrix estimation.
\newblock \textit{Probab. Theory Related Fields}, \textbf{173}, 859--929.

\bibitem[{Lesieur et~al.}(2017)]{lesieur2017constrained}
Lesieur, T., Krzakala, F. and Zdeborov\'a, L. (2017).
\newblock Constrained low-rank matrix estimation: phase transitions, approximate message passing and applications.
\newblock \textit{J. Stat. Mech. Theory Exp.}, 073403.

\bibitem[{Liang and Sur}(2020)]{LS20}
Liang, T. and Sur, P. (2020).
\newblock A precise high-dimensional asymptotic theory for boosting and minimum-$\ell_1$-norm interpolated classifiers.
\newblock Available at \url{https://arxiv.org/pdf/2002.01586.pdf}.

\bibitem[{Ma and Ping(2017)}]{ma2017}
Ma, J. and Ping, L. (2017).
\newblock Orthogonal AMP.
\newblock \textit{IEEE Access}, \textbf{5}, 2020--2033.

\bibitem[{Ma et~al.}(2019)]{ma2019optimization}
Ma, J., Xu, J. and Maleki, A. (2019).
\newblock Optimization-based AMP for phase retrieval: the impact of initialization and $\ell_2$ regularization.
\newblock \textit{IEEE Trans. Inf. Theory}, \textbf{65}, 3600--3629.

\bibitem[{Ma et al.(2019)}]{MRB19}
Ma, Y., Rush, C. and Baron, D. (2019). 
\newblock Analysis of approximate message passing with non-separable denoisers and Markov random field priors.
\newblock \textit{IEEE Trans. Inf. Theory}, \textbf{65}, 7367--7389.

\bibitem[{Ma and Wu(2015)}]{ma2015computational}
Ma, Z. and Wu, Y. (2015). Computational barriers in minimax submatrix detection. 
\newblock \emph{Ann. Statist.}, \textbf{43}, 1089--1116. 

\bibitem[{Matsushita and Tanaka(2013)}]{Ryosuke2013}
Matsushita, R. and Tanaka, T. (2013).
\newblock Low-rank matrix reconstruction and clustering via approximate message passing.
\newblock In \textit{Advances in Neural Information Processing Systems}, \textbf{26}, 917--925.

\bibitem[{McCullagh and Nelder(1989)}]{MN89}
McCullagh, P. and Nelder, J. A. (1989).
\newblock \emph{Generalized Linear Models}, 2nd edition.
\newblock Chapman \& Hall/CRC, Boca Raton.

\bibitem[{Mehta(2004)}]{Meh04}
Mehta, M. L. (2004).
\newblock \emph{Random Matrices}, 3rd edition.
\newblock Elsevier, San Diego.


\bibitem[{Metzler et~al.}(2017)]{metzler2017learned}
Metzler, C., Mousavi, A. and Baraniuk, R. (2017).
\newblock Learned D-AMP: Principled neural network based compressive image recovery.
\newblock In \textit{Advances in Neural Information Processing Systems}, \textbf{30}, 1772--1783.

\bibitem[{M\'ezard and Montanari(2009)}]{MM2009}
M\'ezard, M. and Montanari, M. (2009).
\newblock \emph{Information, Physics, and Computation}.
\newblock Oxford University Press, Oxford.

\bibitem[{M\'ezard et al.(1987)}]{MPV1987}
M\'ezard, M., Parisi, G., Virasoro, M. A. (1987). \emph{Spin Glass Theory and Beyond}.
\newblock World Scientific Lecture Notes in Physics, \textbf{9}.

\bibitem[{Miolane and Montanari}(2018)]{MM18}
Miolane, L. and Montanari, A. (2018).
\newblock The distribution of the Lasso: uniform control over sparse balls and adaptive parameter tuning.
\newblock Available at \url{https://arxiv.org/pdf/1811.01212.pdf}.

\bibitem[{Mondelli et al.}(2020)]{MTM20}
Mondelli, M., Thrampoulidis, C. and Venkataramanan, R. (2020). Optimal combination of linear and spectral estimators for generalized linear models.
\newblock Available at \url{https://arxiv.org/pdf/2008.03326.pdf}.

\bibitem[{Mondelli and Venkataramanan}(2020)]{MM20}
Mondelli, M. and Venkataramanan, R. (2020). Approximate message passing with spectral initialization for generalized linear models.
\newblock \emph{Proc. Mach. Learn. Res.}, \textbf{130}, 397--405.

\bibitem[{Montanari(2012)}]{Mon12}
Montanari, A. (2012).
\newblock Graphical Models Concepts in Compressed Sensing.
\newblock In \textit{Compressed Sensing: Theory and Applications} (Y.~Eldar and
G.~Kutyniok, eds.). 
\newblock Cambridge University Press, Cambridge.

\bibitem[{Montanari and Richard(2016)}]{montanari2016non}
Montanari, A. and Richard, E. (2016). 
\newblock Non-negative principal component analysis: Message passing algorithms and sharp asymptotics.
\newblock \textit{IEEE Trans. Inf. Theory}, \textbf{62}, 1458--1484.

\bibitem[{Montanari and Venkataramanan(2021)}]{MV21}
Montanari, A. and Venkataramanan, R. (2021).
\newblock Estimation of low-rank matrices via approximate message passing.
\newblock \textit{Ann. Statist.}, \textbf{49}, 321--345.

\bibitem[{Mousavi et al.}(2018)]{MMB18}
Mousavi, A., Maleki, A., Baraniuk, R. G. (2018).
\newblock Consistent parameter estimation for LASSO and approximate message passing.
\newblock \emph{Ann. Statist.}, \textbf{46}, 119--148.

\bibitem[{Opper et al.(2016)}]{OCW16}
Opper, M., \c{C}akmak, B. and Winther, O. (2016). 
\newblock A theory of solving TAP equations for Ising models with general invariant random matrices. 
\newblock \emph{J. Phys. A.}, \textbf{49}, 114002.

\bibitem[{Opper and Winther}(2005)]{OW05}
Opper, M. and Winther, O. (2005).
\newblock Expectation consistent approximate inference.
\newblock \emph{J. Mach. Learn. Res.}, \textbf{6}, 2177--2204.

\bibitem[{Pace and Salvan}(1997)]{PS97}
Pace, L. and Salvan, A. (1997). \emph{Principles of Statistical Inference: From a Neo-Fisherian Perspective}. 
\newblock World Scientific, Singapore.

\bibitem[{Panaretos and Zemel}(2020)]{PZ20}
Panaretos, V. M. and Zemel, Y. (2020). \emph{An Invitation to Statistics in Wasserstein Space}. \newblock Springer--Verlag, New York.

\bibitem[{Pandit et~al.}(2019)]{pandit2019asymptotics}
Pandit, P., Sahraee, M., Rangan, S. and Fletcher, A.~K. (2019).
\newblock Asymptotics of MAP inference in deep networks.
\newblock In \textit{2019 IEEE International Symposium on Information Theory}, pp.\ 842--846.

\bibitem[{Pandit et~al.}(2020)]{PSRSF20}
Pandit, P., Sahraee-Ardakan, M., Rangan, S., Schniter, P. and Fletcher, A.~K. (2020).
\newblock Inference with deep generative priors in high dimensions.
\newblock \emph{IEEE J. Sel. Areas Inf. Theory}, \textbf{1}, 336--347.

\bibitem[{Parikh and Boyd}(2013)]{PB13}
Parikh, N. and Boyd, S. (2013).
\newblock Proximal algorithms.
\newblock \textit{Found. Trends Optim.}, \textbf{1}, 123--231.

\bibitem[{Parker et~al.(2014a)}]{PSC14a}
Parker, J.~T., Schniter, P. and Cevher, V. (2014a).
\newblock Bilinear generalized approximate message passing---Part I: Derivation.
\newblock \textit{IEEE Trans. Signal Process.}, \textbf{62}, 5839--5853.

\bibitem[{Parker et~al.(2014b)}]{PSC14b}
Parker, J.~T., Schniter, P. and Cevher, V. (2014b).
\newblock Bilinear generalized approximate message passing---Part II: Applications.
\newblock \textit{IEEE Trans. Signal Process.}, \textbf{62}, 5854--5867.

\bibitem[{Paul(2007)}]{paul2007asymptotics}
Paul, D. (2007).
\newblock Asymptotics of sample eigenstructure for a large dimensional spiked covariance model.
\newblock \textit{Statist. Sinica}, \textbf{17}, 1617--1642.

\bibitem[Peng(2012)]{Pen12}
Peng, M. (2012).
\newblock Eigenvalues of deformed random matrices.
\newblock Available at \url{https://arxiv.org/pdf/1205.0572.pdf}.

\bibitem[{Perry et al.}(2018)]{PWBM18}
Perry, A., Wein, A. S., Bandeira, A. S. and Moitra, A. (2018).
\newblock Optimality and sub-optimality of PCA I: Spiked random matrix models.
\newblock \emph{Ann. Statist.}, \textbf{46}, 2416--2451.



\bibitem[Portnoy(1984)]{Por84}
Portnoy, S. (1984). Asymptotic behavior of $M$-estimators of $p$ regression parameters when $p^2/n$ is large. I. Consistency.
\newblock \emph{Ann. Statist.}, \textbf{12}, 1298--1309.

\bibitem[Portnoy(1985)]{Por85}
Portnoy, S. (1985). Asymptotic behavior of $M$-estimators of $p$ regression parameters when $p^2/n$ is large; II. Normal approximation.
\newblock \emph{Ann. Statist.}, \textbf{13}, 1403--1417.

\bibitem[Portnoy(1988)]{Por88}
Portnoy, S. (1988). Asymptotic behavior of likelihood methods for exponential families when the number of parameters tends to infinity.
\newblock \emph{Ann. Statist.}, \textbf{16}, 356--366.

\bibitem[{Pr\'ekopa}(1980)]{Pre80}
Pr\'ekopa, A. (1980). Logarithmic concave measures and related topics. In \emph{Stochastic Programming (Proc.\ Internat.\ Conf., Univ.\ Oxford, Oxford, 1974, M. A. H. Dempster ed.)}, pp.\ 63--82. 
\newblock Academic Press, London.

\bibitem[{Rangan(2011)}]{RanganGAMP}
Rangan, S. (2011).
\newblock Generalized approximate message passing for estimation with random linear mixing.
\newblock In \textit{2011 IEEE International Symposium on Information Theory}, pp.\ 2168--2172.

\bibitem[{Rangan and Fletcher(2012)}]{rangan2012iterative}
Rangan, S. and Fletcher, A.~K. (2012).
\newblock Iterative estimation of constrained rank-one matrices in noise.
\newblock In \textit{2012 IEEE International Symposium on Information Theory}, pp.\ 1246--1250.

\bibitem[{Rangan and Fletcher(2018)}]{rangan2018iterative}
Rangan, S. and Fletcher, A.~K. (2018).
\newblock Iterative reconstruction of rank-one matrices in noise.
\newblock \textit{Inf. Inference}, \textbf{7}, 1246--1250.

\bibitem[{Rangan et al.(2009)}]{RFG2009}
Rangan, S., Fletcher, A. K. and Goyal, V. K. (2009).
\newblock Asymptotic analysis of MAP estimation via the replica method and applications to compressed sensing.
\newblock In \textit{Advances in Neural Information Processing Systems}, \textbf{22}, 1545--1553.

\bibitem[{Rangan et~al.}(2019a)]{rangan2019convergence}
Rangan, S., Schniter, P., Fletcher, A.~K. and Sarkar, S. (2019a).
\newblock On the convergence of approximate message passing with arbitrary matrices.
\newblock \textit{IEEE Trans. Inf. Theory}, \textbf{65}, 5339--5351.

\bibitem[{Rangan et~al.}(2019b)]{ranganVAMP19}
Rangan, S., Schniter, P. and Fletcher, A.~K. (2019b).
\newblock Vector approximate message passing.
\newblock \textit{IEEE Trans. Inf. Theory}, \textbf{65}, 6664--6684.

\bibitem[{Rangan et~al.}(2016)]{rangan2016fixed}
Rangan, S., Schniter, P., Riegler, E., Fletcher, A.~K. and Cevher, V. (2016).
\newblock Fixed points of generalized approximate message passing with arbitrary matrices.
\newblock \textit{IEEE Trans. Inf. Theory}, \textbf{62}, 7464--7474.

\bibitem[{Reeves and Pfister(2019)}]{reeves2019replica}
Reeves, G. and Pfister, H. D. (2019). 
\newblock The replica-symmetric prediction for random linear estimation with Gaussian matrices is exact.
\newblock \textit{IEEE Trans. Inf. Theory}, \textbf{65}, 2252--2283.

\bibitem[Robbins(1956)]{Rob56}
Robbins, H. (1956). An empirical Bayes approach to statistics.
\newblock \emph{Proc. Third Berkeley Symp. Math. Statist. Prob.}, \textbf{1}, 157--163.

\bibitem[Rockafellar(1997)]{Rock97}
Rockafellar, R. T. (1997). \textit{Convex Analysis}.
\newblock Princeton University Press, Princeton.

\bibitem[{Rush et~al.}(2017)]{rushGV2017}
Rush, C., Greig, A. and Venkataramanan, R. (2017).
\newblock Capacity-achieving sparse superposition codes via approximate message passing decoding.
\newblock \textit{IEEE Trans. Inf. Theory}, \textbf{63}, 1476--1500.

\bibitem[{Rush and Venkataramanan}(2018)]{RV18}
Rush, C. and Venkataramanan, R. (2018). Finite sample analysis of approximate message passing algorithms.
\newblock \emph{IEEE Trans.\ Inf.\ Theory}, \textbf{64}, 7264--7286.

\bibitem[{Schniter(2011)}]{SchniterBICM2011}
Schniter, P. (2011).
\newblock A message-passing receiver for BICM-OFDM over unknown clustered-sparse channels.
\newblock \textit{IEEE J. Sel. Top. Signal Process.}, \textbf{5}, 1462--1474.

\bibitem[{Schniter(2020)}]{schniter2019simple}
Schniter, P. (2020).
\newblock A simple derivation of AMP and its state evolution via first-order cancellation.
\newblock \textit{IEEE Trans. Signal Process.}, \textbf{68}, 4283--4292.

\bibitem[{Schniter and Rangan(2014)}]{schniter2014compressive}
Schniter, P. and Rangan, S. (2014).
\newblock Compressive phase retrieval via generalized approximate message passing.
\newblock \textit{IEEE Trans. Signal Process.}, \textbf{63}, 1043--1055.

\bibitem[{Schniter et al.(2016)}]{SRF16}
Schniter, P., Rangan, S. and Fletcher, A. K. (2016).
\newblock Vector approximate message passing for the generalized linear model.
\newblock In \emph{50th Asilomar Conference on Signals, Systems and Computers}, pp.~1525--1529.

\bibitem[Serfling(1980)]{Ser80}
Serfling, R. J. (1980). \emph{Approximation Theorems of Mathematical Statistics}. 
\newblock Wiley, New York.

\bibitem[{Su et~al.(2017)}]{su2017false}
Su, W., Bogdan, M. and Cand\`es, E. (2017).
\newblock False discoveries occur early on the LASSO path.
\newblock \emph{Ann. Statist.}, \textbf{45}, 2133--2150.

\bibitem[{Su and Cand\`es(2016)}]{su2016slope}
Su, W. and Cand\`es, E. (2016).
\newblock SLOPE is adaptive to unknown sparsity and asymptotically minimax.
\newblock \emph{Ann. Statist.}, \textbf{44}, 1038--1068.

\bibitem[{Su and Khoshgoftaar(2009)}]{SK09}
Su, X. and Khoshgoftaar, T. M. (2009). 
\newblock A survey of collaborative filtering techniques.
\newblock \emph{Adv. Artif. Intelligence}, Volume 2009, 1--19.

\bibitem[{Sur and Cand\`es(2019a)}]{SC19a}
Sur, P. and Cand\`es, E.~J. (2019a).
\newblock A modern maximum-likelihood theory for high-dimensional logistic regression.
\newblock \emph{Proc. Natl. Acad. Sci. U.S.A.}, \textbf{116}, 14516--14525.

\bibitem[{Sur and Cand\`es(2019b)}]{SC19b}
Sur, P. and Cand\`es, E.~J. (2019b).
\newblock Additional supplementary materials for `A modern maximum-likelihood theory for high-dimensional logistic regression'. 
\newblock Available at \url{https://sites.fas.harvard.edu/~prs499/papers/proofs_LogisticAMP.pdf}.

\bibitem[{Sur et~al.(2017)Sur, Chen and Cand\`es}]{sur2017likelihood}
Sur, P., Chen, Y. and Cand\`es, E.~J. (2017).
\newblock The likelihood ratio test in high-dimensional logistic regression is asymptotically a rescaled chi-square.
\newblock \emph{Probab. Theory Related Fields}, \textbf{175}, 487--558.

\bibitem[Takeuchi(2020)]{Tak19}
Takeuchi, K. (2020).
\newblock Rigorous dynamics of expectation-propagation-based signal recovery from unitarily invariant measurements.
\newblock \emph{IEEE Trans. Inf. Theory}, \textbf{66}, 368--386.

\bibitem[{Talagrand(2011)}]{Tal11}
Talagrand, M. (2011). 
\newblock \emph{Mean Field Models for Spin Glasses, Vol I: Basic Examples}. 
\newblock Springer, New York.

\bibitem[{Tanaka(2002)}]{Tan02}
Tanaka, T. (2002).
\newblock A statistical-mechanics approach to large-system analysis of CDMA multiuser detectors.
\newblock \emph{IEEE Trans. Inf. Theory}, \textbf{48}, 2888--2910.

\bibitem[{Thrampoulidis et al.}(2018)]{TAH18}
Thrampoulidis, C., Abbasi, E. and Hassibi, B. (2018).
Precise error analysis of regularized $M$-estimators in high dimensions.
\newblock \emph{IEEE Trans. Inf. Theory}, \textbf{64}, 5592--5628.

\bibitem[{Thrampoulidis et al.}(2015)]{TOH15}
Thrampoulidis, C., Oymak, S. and Hassibi, B. (2015).
\newblock Regularized linear regression: a precise analysis of the estimation error.
\newblock \emph{Proc. Mach. Learn. Res.}, \textbf{40}, 1683--1709.

\bibitem[{Tibshirani(1996)}]{Tib96}
Tibshirani, R. (1996). 
\newblock Regression shrinkage and selection via the Lasso.
\newblock \emph{J. Roy. Statist. Soc., Ser. B}, \textbf{58}, 267--288.  

\bibitem[{Tramel et al.}(2014)]{TKGM14}
Tramel, E. W., Kumar, S., Giurgiu, A. and Montanari, A. (2014).
\newblock Statistical estimation: from denoising to sparse regression and hidden cliques.
\newblock Available at \url{https://arxiv.org/pdf/1409.5557.pdf}.

\bibitem[Tsybakov(2009)]{Tsy09}
Tsybakov, A. B. (2009). 
\newblock \emph{Introduction to Nonparametric Estimation}. 
\newblock Springer--Verlag, New York.

\bibitem[{van der Vaart}(1998)]{vdV98}
van der Vaart, A. W. (1998). \textit{Asymptotic Statistics}.
\newblock Cambridge University Press, Cambridge.

\bibitem[{Vila et~al.(2015)}]{VSM15}
Vila, J., Schniter, P. and Meola, J. (2015). 
\newblock Hyperspectral unmixing via turbo bilinear approximate message passing.
\newblock \textit{IEEE Trans. Comput. Imaging}, \textbf{1}, 143--158.

\bibitem[Villani(2003)]{Vil03}
Villani, C. (2003).
\newblock \emph{Topics in Optimal Transportation. Graduate Studies in Mathematics}. 
\newblock American Mathematical Society, Providence, RI.

\bibitem[Villani(2009)]{Vil09}
Villani, C. (2009). 
\newblock \emph{Optimal Transport, Old and New}. 
\newblock Springer--Verlag, New York.

\bibitem[{von Luxburg(2007)}]{vonLux07}
von Luxburg, U. (2007). 
\newblock A tutorial on spectral clustering. 
\newblock \emph{Statist. Comput.}, \textbf{17}, 395--416.

\bibitem[{Vu and Lei(2013)}]{VuLei2013}
Vu, V. Q. and Lei, J. (2013).
\newblock Minimax sparse principal subspace estimation in high dimensions.
\newblock \emph{Ann. Statist.}, \textbf{41}, 2905--2947.

\bibitem[{Wang et al.(2016)}]{WBS16}
Wang, T., Berthet, Q. and Samworth, R. J. (2016). 
\newblock Statistical and computational trade-offs in estimation of sparse principal components. 
\newblock \emph{Ann. Statist.}, \textbf{44}, 1896--1930.

\bibitem[Yang(2019)]{Yan19}
Yang, G. (2019).
\newblock Scaling limits of wide neural networks with weight sharing: Gaussian process behavior, gradient independence, and neural tangent kernel derivation.
\newblock Available at \url{https://arxiv.org/pdf/1902.04760.pdf}.

\bibitem[{Zhu et al.(2019)}]{ZWS19}
Zhu, Z., Wang, T. and Samworth, R. J. (2019). 
\newblock High-dimensional principal component analysis with heterogeneous missingness. 
\newblock Available at \url{https://arxiv.org/pdf/1906.12125.pdf}.

\bibitem[{Zdeborov\'a and Krzakala(2016)}]{ZK16}
Zdeborov\'a, L. and Krzakala, F. (2016). 
\newblock Statistical physics of inference: thresholds and algorithms. 
\newblock \emph{Adv. Phys.}, \textbf{65}, 453--552.

\bibitem[{Zou et al.}(2006)]{ZHT06}
Zou, H., Hastie, T. and Tibshirani, R. (2006). Sparse principal component analysis. 
\newblock \emph{J. Comput. Graph. Statist.}, \textbf{15}, 265--286.

\end{thebibliography}
\end{document}